\documentclass[a4paper,reqno]{paper}
\usepackage[latin1]{inputenc}
\usepackage[T1]{fontenc}
\DeclareSymbolFont{rsfscript}{OMS}{rsfs}{m}{b}
\DeclareSymbolFontAlphabet{\mathrsfs}{rsfscript}
\usepackage{amsfonts,enumerate}
\usepackage{amsmath}
\usepackage{amssymb,pgf}
\usepackage{colortbl}
\usepackage{amsthm}
\usepackage[absolute]{textpos}
\usepackage{stmaryrd}
\usepackage{subcaption}
\usepackage{dsfont}
\usepackage{float}
\usepackage{pstricks,tikz}
\usepackage{graphicx}
\usepackage{pst-grad}
\usepackage{multido}

\newenvironment{psmallmatrix}
  {\left(\begin{smallmatrix}}
  {\end{smallmatrix}\right)}

\usepackage{geometry}
\def\finl{~$\SS{\blacksquare}$}
\def\rfinl{\begin{right}\finl\end{right}}
\newcommand{\mn}{\hspace{0mm}\scalebox{.6}{$-$}\hspace{0.3mm}}
\newcommand{\pn}{\hspace{0mm}\scalebox{.6}{$+$}\hspace{0.3mm}}

\newtheorem{Th}{Theorem}[section]
\newtheorem{Lem}[Th]{Lemma}
\newtheorem{Cor}[Th]{Corollary}
\newtheorem{Prop}[Th]{Proposition}
\newtheorem{Def-Prop}[Th]{Definition-Proposition}
\newtheorem{conjecture}[Th]{Conjecture}

\newtheorem{Obs}[Th]{Observation}
\theoremstyle{definition}
\newtheorem{Def}[Th]{Definition}
\newtheorem{Exa}[Th]{Example}

\newtheorem{Rem}[Th]{Remark}
\numberwithin{equation}{section}

\definecolor{purple}{rgb}{0.8,0.12,0.8}
\definecolor{orange}{rgb}{1.0,0.7,0.0}
\definecolor{pink}{rgb}{1,0.5,0.8}
\definecolor{blackg}{rgb}{0.1,0.25,0.1}
\definecolor{ForestGreen}{cmyk}{0.91,0,0.88,0.42}
\definecolor{Turquoise}{cmyk}{0.85,0,0.20,0}

\definecolor{GreenYellow}{cmyk}{0.15,0,0.69,0} % PANTONE 388
\definecolor{Yellow}{cmyk}{0,0,1.,0} % PANTONE YELLOW
\definecolor{Goldenrod}{cmyk}{0,0.10,0.84,0} % PANTONE 109
\definecolor{Dandelion}{cmyk}{0,0.29,0.84,0} % PANTONE 123
\definecolor{Apricot}{cmyk}{0,0.32,0.52,0} % PANTONE 1565
\definecolor{Peach}{cmyk}{0,0.50,0.70,0} % PANTONE 164
\definecolor{Melon}{cmyk}{0,0.46,0.50,0} % PANTONE 177
\definecolor{YellowOrange}{cmyk}{0,0.42,1.,0} % PANTONE 130
\definecolor{Orange}{cmyk}{0,0.61,0.87,0} % PANTONE ORANGE-021
\definecolor{BurntOrange}{cmyk}{0,0.51,1.,0} % PANTONE 388
\definecolor{Bittersweet}{cmyk}{0,0.75,1.,0.24} % PANTONE 167
\definecolor{RedOrange}{cmyk}{0,0.77,0.87,0} % PANTONE 179
\definecolor{Mahogany}{cmyk}{0,0.85,0.87,0.35} % PANTONE 484
\definecolor{Maroon}{cmyk}{0,0.87,0.68,0.32} % PANTONE 201
\definecolor{BrickRed}{cmyk}{0,0.89,0.94,0.28} % PANTONE 1805
\definecolor{Red}{cmyk}{0,1.,1.,0} % PANTONE RED
\definecolor{OrangeRed}{cmyk}{0,1.,0.50,0} % No PANTONE,match
\definecolor{RubineRed}{cmyk}{0,1.,0.13,0} % PANTONE RUBINE-RED
\definecolor{WildStrawberry}{cmyk}{0,0.96,0.39,0} % PANTONE 206
\definecolor{Salmon}{cmyk}{0,0.53,0.38,0} % PANTONE 183
\definecolor{CarnationPink}{cmyk}{0,0.63,0,0} % PANTONE 218
\definecolor{Magenta}{cmyk}{0,1.,0,0} % PANTONE PROCESS-MAGENTA
\definecolor{VioletRed}{cmyk}{0,0.81,0,0} % PANTONE 219
\definecolor{Rhodamine}{cmyk}{0,0.82,0,0} % PANTONE RHODAMINE-RED
\definecolor{Mulberry}{cmyk}{0.34,0.90,0,0.02} % PANTONE 241
\definecolor{RedViolet}{cmyk}{0.07,0.90,0,0.34} % PANTONE 234
\definecolor{Fuchsia}{cmyk}{0.47,0.91,0,0.08} % PANTONE 248
\definecolor{Lavender}{cmyk}{0,0.48,0,0} % PANTONE 223
\definecolor{Thistle}{cmyk}{0.12,0.59,0,0} % PANTONE 245
\definecolor{Orchid}{cmyk}{0.32,0.64,0,0} % PANTONE 252
\definecolor{DarkOrchid}{cmyk}{0.40,0.80,0.20,0} %  No PANTONE match
\definecolor{Purple}{cmyk}{0.45,0.86,0,0} % PANTONE PURPLE
\definecolor{Plum}{cmyk}{0.50,1.,0,0} % PANTONE 518
\definecolor{Violet}{cmyk}{0.79,0.88,0,0} % PANTONE VIOLET
\definecolor{RoyalPurple}{cmyk}{0.75,0.90,0,0} % PANTONE 267
\definecolor{BlueViolet}{cmyk}{0.86,0.91,0,0.04} % PANTONE 2755
\definecolor{Periwinkle}{cmyk}{0.57,0.55,0,0} % PANTONE 2715
\definecolor{CadetBlue}{cmyk}{0.62,0.57,0.23,0} % PANTONE (534+535)
\definecolor{CornflowerBlue}{cmyk}{0.65,0.13,0,0} % PANTONE 292
\definecolor{MidnightBlue}{cmyk}{0.98,0.13,0,0.43} % PANTONE 302
\definecolor{NavyBlue}{cmyk}{0.94,0.54,0,0} % PANTONE 293
\definecolor{RoyalBlue}{cmyk}{1.,0.50,0,0} % No PANTONE match
\definecolor{Blue}{cmyk}{1.,1.,0,0} % PANTONE BLUE-072
\definecolor{Cerulean}{cmyk}{0.94,0.11,0,0} % PANTONE 3005
\definecolor{Cyan}{cmyk}{1.,0,0,0} % PANTONE PROCESS-Black
\definecolor{ProcessBlue}{cmyk}{0.96,0,0,0} % PANTONE PROCESS-BLUE
\definecolor{SkyBlue}{cmyk}{0.62,0,0.12,0} % PANTONE 2985
\definecolor{Turquoise}{cmyk}{0.85,0,0.20,0} % PANTONE (312+313)
\definecolor{TealBlue}{cmyk}{0.86,0,0.34,0.02} % PANTONE 3145
\definecolor{Aquamarine}{cmyk}{0.82,0,0.30,0} % PANTONE 3135
\definecolor{BlueGreen}{cmyk}{0.85,0,0.33,0} % PANTONE 320
\definecolor{Emerald}{cmyk}{1.,0,0.50,0} % No PANTONE match
\definecolor{JungleGreen}{cmyk}{0.99,0,0.52,0} % PANTONE 328
\definecolor{SeaGreen}{cmyk}{0.69,0,0.50,0} % PANTONE 3268
\definecolor{Green}{cmyk}{1.,0,1.,0} % PANTONE GREEN
\definecolor{ForestGreen}{cmyk}{0.91,0,0.88,0.12} % PANTONE 349
\definecolor{PineGreen}{cmyk}{0.92,0,0.59,0.25} % PANTONE 323
\definecolor{LimeGreen}{cmyk}{0.50,0,1.,0} % No PANTONE match
\definecolor{YellowGreen}{cmyk}{0.44,0,0.74,0} % PANTONE 375
\definecolor{SpringGreen}{cmyk}{0.26,0,0.76,0} % PANTONE 381
\definecolor{OliveGreen}{cmyk}{0.64,0,0.95,0.40} % PANTONE 582
\definecolor{RawSienna }{cmyk}{0,0.72,1.,0.45} % PANTONE 154
\definecolor{Sepia}{cmyk}{0,0.83,1.,0.70} % PANTONE 161
\definecolor{Brown}{cmyk}{0,0.81,1.,0.60} % PANTONE 1615
\definecolor{Tan}{cmyk}{0.14,0.42,0.56,0} % No PANTONE match
\definecolor{Gray}{cmyk}{0,0,0,0.50} % PANTONE COOL-GRAY-8
\definecolor{Black}{cmyk}{0,0,0,1.} % PANTONE PROCESS-BLACK
\definecolor{White}{cmyk}{0,0,0,0} % No PANTON

\definecolor{pp}{RGB}{215,25,28}
\definecolor{pm}{RGB}{253,174,97}
\definecolor{mp}{RGB}{171,221,164}
\definecolor{mm}{RGB}{43,131,186}

\definecolor{green1}{RGB}{153,216,201}
\definecolor{green2}{RGB}{44,162,95}

\definecolor{blue1}{RGB}{158,202,225}
\definecolor{blue2}{RGB}{49,130,189}

\newcommand{\cB}{\mathcal{B}}

\newcommand{\cD}{\mathcal{D}}
\newcommand{\cE}{\mathcal{E}}

\newcommand{\cH}{\mathcal{H}}

\newcommand{\cL}{\mathcal{L}}
\newcommand{\cLR}{\mathcal{LR}}
\newcommand{\cM}{\mathcal{M}}

\newcommand{\cP}{\mathcal{P}}
\newcommand{\cQ}{\mathcal{Q}}
\newcommand{\cR}{\mathcal{R}}

\newcommand{\cU}{\mathcal{U}}

\newcommand{\cZ}{\mathcal{Z}}

\newcommand{\sH}{\mathsf{H}}

\newcommand{\ba}{\mathbf{a}}

\newcommand{\fc}{\mathfrak{c}}

\newcommand{\sw}{\mathsf{w}}

\newcommand{\sbb}{\mathsf{b}}
\newcommand{\su}{\mathsf{u}}
\newcommand{\sv}{\mathsf{v}}
\newcommand{\sa}{\mathsf{a}}

\renewcommand{\ss}{{\mathsf s}}
\newcommand{\st}{{\mathsf t}}

\newcommand{\se}{\mathsf{e}}

\newcommand{\sq}{\mathsf{q}}

\newcommand{\sh}{\mathsf{h}}

\newcommand{\sT}{\mathsf{T}}
\newcommand{\sB}{\mathsf{B}}
\newcommand{\sP}{\mathsf{P}}

\newcommand{\tsc}{\mathsf{\Lambda}}

\newcommand{\nZ}{\mathbb{Z}}

\newcommand{\nN}{\mathbb{N}}
\newcommand{\nQ}{\mathbb{Q}}

\newcommand{\al}{\alpha}

\newcommand{\si}{\sigma}
\newcommand{\la}{\lambda}
\newcommand{\ga}{\gamma}
\newcommand{\eps}{\varepsilon}

\newcommand{\om}{\omega}

\newcommand{\Ga}{\Gamma}
\newcommand{\De}{\Delta}

\newcommand{\Up}{\Upsilon}

\newcommand{\tG}{\tilde{G}}

\newcommand{\ra}{\rightarrow}

\newcommand{\lra}{\longrightarrow}

\newcommand{\eq}{\Longleftrightarrow}
\newcommand{\lmt}{\longmapsto}

\newcommand{\barr}{\begin{array}{cccccccccc}}
\newcommand{\ear}{\end{array}}

\newcommand{\pmat}{\begin{pmatrix}}
\newcommand{\emat}{\end{pmatrix}}

\DeclareMathOperator{\Tr}{Tr}

\newcommand{\ben}{\begin{enumerate}}
\newcommand{\een}{\end{enumerate}}

\newcommand{\bit}{\begin{itemize}}
\newcommand{\eit}{\end{itemize}}

{ \begin{list}%
	{$\rightarrow$}%
	{\setlength{\labelwidth}{30pt}%
	 \setlength{\leftmargin}{20pt}%
	 \setlength{\itemsep}{.05cm}}}%
{ \end{list} }

\newcommand{\ov}{\overline}

\newcommand{\sg}{\langle}
\newcommand{\sd}{\rangle}

\newcommand{\conj}[1]{{\bf P#1}}

\newcommand{\B}[1]{{\bf B#1}}

\newcommand{\quand}{\quad\text{ and }\quad}
\newcommand{\quwhere}{\quad\text{ where }\quad}

\def\SS{\scriptstyle}

\def\finl{~$\SS{\blacksquare}$}
\def\rfinl{\begin{right}\finl\end{right}}

\newcommand{\wt}{\text{wt}}

\newcommand{\irreps}{\mathrm{Irrep}(\overline{\cH})}

\newcommand{\tbu}{{\tiny $\bullet$}}

\newenvironment{maliste}%
{ \begin{list}%
	{\tbu}%
	{\setlength{\labelwidth}{30pt}%
	 \setlength{\leftmargin}{20pt}%
	 \setlength{\itemsep}{.05cm}}}%
{ \end{list} }

\newcommand{\bem}{\begin{maliste}}
\newcommand{\eem}{\end{maliste}}

\usepackage{color}
\definecolor{asse}{HTML}{03AC5F}

\def\sR{\mathsf{R}}

\colorlet{ye}{red!5!yellow!42!}
\colorlet{bl}{blue!70!black!80!}
\colorlet{gr}{green!80!black!70!}
\colorlet{rd}{yellow!10!red!80!}
\colorlet{ora}{orange}
\colorlet{fgr}{ForestGreen}
\definecolor{tur}{cmyk}{0.85,0,0.20,0}
\definecolor{pur}{rgb}{0.8,0.12,0.8}

\usepackage{lscape}
\usepackage{array}
\newcolumntype{M}[1]{>{\centering\arraybackslash}m{#1}}
\newcolumntype{N}{@{}m{0pt}@{}}
\usepackage{multirow}
\colorlet{ye}{red!5!yellow!42!}
\colorlet{bl}{blue!70!black!80!}
\colorlet{gr}{green!80!black!70!}
\colorlet{rd}{yellow!10!red!80!}
\colorlet{ora}{orange}
\colorlet{fgr}{ForestGreen}
\definecolor{tur}{cmyk}{0.85,0,0.20,0}
\definecolor{pur}{rgb}{0.8,0.12,0.8}
\newcommand{\tba}{\tilde{\ba}}

\begin{document}

\title{A proof of Lusztig's conjectures for affine type $G_2$ with arbitrary parameters}

\author{J\'er\'emie Guilhot, James Parkinson \thanks{This work was partially supported by the regional project MADACA. Moreover, both authors would like to thank the Universit\'e de Tours for its invited researcher program under which a significant part of this research was undertaken.}}
%\ead{jeremie.guilhot@lmpt.univ-tours.fr}
\date{\today}
\maketitle

\def\finl{~$\SS{\blacksquare}$}
\def\rfinl{\begin{right}\finl\end{right}}

\parindent=0mm

\abstract{We prove Lusztig's conjectures ${\bf P1}$--${\bf P15}$ for the affine Weyl group of type $\tilde{G}_2$ for all choices of parameters. Our approach to compute Lusztig's $\mathbf{a}$-function is based on the notion of a ``balanced system of cell representations'' for the Hecke algebra. We show that for arbitrary Coxeter type the existence of balanced system of cell representations is sufficient to compute the $\mathbf{a}$-function and we explicitly construct such a system in type $\tilde{G}_2$ for arbitrary parameters.  We then investigate the connection between Kazhdan-Lusztig cells and the Plancherel Theorem in type $\tilde{G}_2$, allowing us to prove~${\bf P1}$ and determine the set of Duflo involutions. From there, the proof of the remaining conjectures follows very naturally, essentially from the combinatorics of Weyl characters of types $G_2$ and $A_1$, along with some explicit computations for the finite cells. }

\section*{Introduction}

The theory of Kazhdan-Lusztig cells plays a fundamental role in the representation theory of Coxeter groups and Hecke algebras. In their celebrated paper~\cite{KL1} Kazhdan and Lusztig introduced the theory in the equal parameter case, and in \cite{Lus1p} Lusztig generalised the construction to the case of arbitrary parameters. A very specific feature in the equal parameter case is the geometric interpretation of Kazhdan-Lusztig theory, which implies certain ``positivity properties'' (such as the positivity of the structure constants with respect to the Kazhdan-Lusztig basis). This was proved in the finite and affine cases by Kazhdan and Lusztig in~\cite{KL2}, and the case of arbitrary Coxeter groups was settled only very recently by Elias and Williamson in~\cite{EW:14}. However, simple examples show that these positivity properties no longer hold for unequal parameters, hence the need to develop new methods to deal with the general case. 
\medskip

A major step in this direction was achieved by Lusztig in his book on Hecke algebras with unequal parameters \cite[Chapter~14]{bible} where he introduced 15 conjectures \conj{1}--\conj{15} which capture essential properties of cells for all choices of parameters. In the case of equal parameters these conjectures can be proved using the above mentioned geometric interpretation. In the case of more general parameters {\bf P1}--{\bf P15} are only known to hold in the following situations: 
\bem
\item the \textit{quasisplit} case where a geometric interpretation is also available~\cite[Chapter~16]{bible};
\item finite dihedral type~\cite{Geck:11} and infinite dihedral type~\cite[Chapter~17]{bible} for arbitrary parameters;
\item universal Coxeter groups for arbitrary parameters~\cite{SY:15};
\item type $B_n$ in the ``asymptotic'' case \cite{B-I,Geck:11};
\item $F_4$ for arbitrary parameters \cite{Geck:11}.
\eem
Note that the only infinite Coxeter groups for which conjectures \conj{1}--\conj{15} are known to hold for arbitrary parameters are the universal Coxeter groups (including the infinite dihedral group), where the proof proceeds by explicit computations. In this paper we prove Lusztig's conjectures in type~$\tilde{G}_2$ for arbitrary parameters. This provides the very first example of an affine Coxeter group of rank greater than~$1$ in which the conjectures have been proved. Furthermore, our methods provide a theoretical framework that one may hope to apply to other types of affine Coxeter groups. For instance, the approach outlined in this paper can be applied to the~$\tilde{C}_2$ case, however the analysis is rather involved in this $3$ parameter setting and so we provide the details in~\cite{GP:18}.
\medskip

One of the main challenges in proving Lusztig's conjectures is to compute Lusztig's $\ba$-function since, in principle, it requires us to have information on all the structure constants with respect to the Kazhdan-Lusztig basis. Our approach to this problem is based on the notion of a \textit{balanced system of cell representations} inspired by the work of Geck~\cite{Geck:11} in the finite case. This notion can be defined for an arbitrary Coxeter group~$(W,S)$ with weight function $L:W\to\mathbb{N}$ and associated multi-parameter Hecke algebra $\cH$ defined over $\nZ[\sq,\sq^{-1}]$. Let $\tsc$ be the set of two-sided cells of $W$ with respect to~$L$, and let $\Ga\in \tsc$. We say that a representation $\pi$ is $\Ga$-balanced if it admits a basis such that (1) the maximal degree of the coefficients that appear in the matrix $\pi(T_w)$ is bounded by a constant  $\ba_{\pi}$ (here  $T_w$ denotes the standard basis of~$\cH$) and (2) this bound is attained if and only if $w\in \Ga$. A balanced systems of cell representations is a family~$(\pi_\Ga)_{\Ga\in \tsc}$ of~$\Ga$-balanced representations that satisfy some extra axioms (see Section~\ref{sec:def-balanced}). We show that the existence of such a system is sufficient to compute Lusztig's $\ba$-function, and as a byproduct we obtain an explicit construction of Lusztig's asymptotic algebra~$\mathcal{J}$ \cite[Chapter 18]{bible}.

\medskip

In the case of $\tG_2$, we construct a balanced system of cell representations  for each parameter regime. Our starting point is the partition of $W$ into Kazhdan-Lusztig cells that was proved by the first author in~\cite{guilhot4}. It turns out that the representations associated to finite cells naturally give rise to balanced representations and so most of our work is concerned with the infinite cells. In type $\tilde{G}_2$ there are 3 such cells for each choice of parameters, the lowest two-sided cell $\Ga_0$ and two other cells $\Ga_1$ and $\Ga_2$. To each of these cells we associate a natural finite dimensional representation $\pi_i$ admitting an elegant combinatorial description in terms of alcove walks, which allows us to establish the balancedness of these representations.

\medskip

Next we investigate connections between Kazhdan-Lusztig cells and the Plancherel Theorem, using the explicit formulation of the Plancherel Theorem in type $\tilde{G}_2$ obtained by the second author in~\cite{Par:14}. In particular, we show that in type $\tilde{G}_2$ there is a natural correspondence, in each parameter range, between two-sided cells appearing in the cell decomposition and the representations appearing in the Plancherel Theorem (these are the \textit{tempered} representations of~$\cH$). Moreover we define a \textit{$\sq$-valuation} on the Plancherel measure, and show that in type $\tilde{G}_2$ the $\sq$-valuation of the mass of a tempered representation is twice the value of Lusztig's $\ba$-function on the associated cell. This observation allows us to introduce an \textit{asymptotic Plancherel measure}, and we use this measure to prove \conj{1} and determine the set $\cD$ of Duflo involutions. Moreover we show that the Plancherel theorem ``descends'' to give an inner product on Lusztig's asymptotic algebra~$\mathcal{J}$.

\medskip

Once we have established the existence of a balanced system of cell representations for $\tG_2$ for each choice of parameters, proved \conj{1}, and determined the set $\cD$, conjectures \conj{2}--\conj{14} follow very naturally, essentially from combinatorics of Weyl characters of types $G_2$ and $A_1$ and explicit computations for the finite cells. Conjecture \conj{15} is slightly different in nature and follows from the generalised induction process \cite{guilhot3}.

\medskip

We note that in \cite{Xie:15}, Xie uses a decomposition formula for the Kazhdan-Lusztig basis to reduce \conj{1}--\conj{15} to proving \conj{8} and determining Lusztig's $\ba$-function. However we note that indeed the main work of this paper is precisely the determination of Lusztig's $\ba$-function, which uses our balanced system of cell representations. 

\medskip

We conclude this introduction with an outline of the structure of this paper. In Section~\ref{sec:1} we recall the basics of Kazhdan-Lusztig theory. In Section~\ref{sec:def-balanced} we introduce the axioms of a balanced system of representations, and show in Theorem~\ref{thm:afn} that given these axioms Lusztig's $\ba$-function can be computed. Section~\ref{sec:2} provides background on affine Weyl groups and the affine Hecke algebra. In Section~\ref{sec:KL-cells-G2} we recall the partition of $\tilde{G}_2$ into cells for all choices of parameters, and discuss cell factorisation properties for the infinite cells. In Section~\ref{sec:balanced} we prove that each finite cell admits a balanced cell representation. Some of the results in this section requires explicit computations. These have been carried out using {\sf gap3} \cite{GAP} and the package {\sf CHEVIE} \cite{chevie2,chevie}.
\medskip

Section~\ref{sec:4} deals with the case of the lowest two-sided cell. We note that this case has already been investigated by Xie in~\cite{Xie:17}, however we include our analysis here since it illustrates very clearly in this simpler case the combinatorial methodology that we will employ for the remaining more complicated infinite cells. Section~\ref{sec:5} deals with these remaining cells. We introduce a model based on alcove walks to study the representations associated to these cells. This allows us to give combinatorial proofs of bounds for matrix coefficients and to compute leading matrices for these representations. The analysis of this section is involved due in part to interesting complications arising in the case of non-generic parameters. 

\medskip

In Section~\ref{sec:plancherel} we analyse the connections between the Plancherel Theorem and the cell decomposition. We use the Plancherel Theorem to prove~\conj{1} and determine the set of Duflo involution $\cD$. We also define the ``asymptotic Plancherel measure'', and show that it induces an inner product on Lusztig's asymptotic algebra in type $\tilde{G}_2$. Finally, in Section~\ref{sec:lusztig} we provide our proof of the remaining conjectures \conj{2}--\conj{15}.

\medskip

%%%%%%%%%%%%%%%%%%%%%%%%%%%%%%%%%%%
%%%%%%%%%%%%%%%%%%%%%%%%%%%%%%%%%%%
%%%%%%%%%%%%%%%%%%%%%%%%%%%%%%%%%%%
%%%%%%%%%%%%%%%%%%%%%%%%%%%%%%%%%%%
%%%%%%%%%%%%%%%%%%%%%%%%%%%%%%%%%%%
%%%%%%%%%%%%%%%%%%%%%%%%%%%%%%%%%%%
%%%%%%%%%%%%%%%%%%%%%%%%%%%%%%%%%%%

\section{Kazhdan-Lusztig theory}\label{sec:1}

In this section we recall the setup of Kazhdan-Lusztig theory, including the Kazhdan-Lusztig basis, Kazhdan-Lusztig cells, and the Lusztig's conjectures ${\bf P1}$--${\bf P15}$. In this section $(W,S)$ denotes an arbitrary Coxeter system (with $|S|<\infty$) with length function $\ell:W\to\mathbb{N}=\{0,1,2,\ldots\}$. For $I\subseteq S$ let $W_I$ be the standard parabolic subgroup generated by~$I$. Let $L:W\to\mathbb{N}$ be a \textit{positive weight function} on~$W$. Thus $L(w)>0$ for all $w\in W$ different from the identity, and $L(ww')=L(w)+L(w')$ whenever $\ell(ww')=\ell(w)+\ell(w')$. Let $\sq$ be an indeterminate and let $\sR=\nZ[\sq,\sq^{-1}]$ be the ring of Laurent polynomial in~$\sq$.

\subsection{The Kazhdan-Lusztig basis}\label{sec:1.1}

 The \textit{Hecke algebra} $\cH$ associated to $(W,S,L)$ is the algebra over $\sR$ with basis $\{T_w\mid w\in W\}$ and multiplication given by 
$$T_wT_s=\begin{cases}
T_{ws} &\mbox{if $\ell(ws)=\ell(w)+1$}\\
T_{ws}+(\sq^{L(s)}\mn \sq^{-L(s)})T_{w}  &\mbox{if $\ell(ws)=\ell(w)-1$}.
\end{cases}$$
The basis $\{T_w\mid w\in W\}$ is called the \textit{standard basis} of~$\cH$. We set $\sq_s=\sq^{L(s)}$ for $s\in S$.

\medskip
The involution $\bar{\ }$ on $\sR$ which sends $\sq$ to $\sq^{-1}$ can be extended to an involution on $\cH$ by setting 
$$\ov{\sum_{w\in W} a_w T_w}=\sum_{w\in W} \ov{a_w}\, T_{w^{-1}}^{-1}.$$
In \cite{KL1}, Kazhdan and Lusztig proved that there exists a unique basis $\{C_w\mid w\in W\}$ of $\cH$ such that, for all $w\in W$,
$$
\ov{C_w}=C_w\quad\text{and}\quad C_w=T_w+\sum_{y<w} P_{y,w}T_y\quad\text{where $P_{y,w}\in \sq^{-1}\nZ[\sq^{-1}]$}. 
$$
This basis is called the \textit{Kazhdan-Lusztig basis} of~$\cH$ (the KL basis for short). The polynomials $P_{y,w}$ are called the \textit{Kazhdan-Lusztig polynomials}, and to complete the definition we set $P_{y,w}=0$ whenever $y\not<w$ (here $\leq$ denotes Bruhat order on~$W$). We note that the Kazhdan-Lusztig polynomials, and hence the elements $C_w$, depend on the the weight function $L$. For example, in the dihedral group $I_2(2m)$ with $m\geq 2$, $L(s_1)=a$, and $L(s_2)=b$, we have 
$$
P_{s_1,s_1s_2s_1}=
\begin{cases}
\sq^{-(b-a)}+\sq^{-a-b}&\text{if $a<b$}\\
\sq^{-2a}&\text{if $a=b$}\\
\sq^{-(a+b)}-\sq^{-(a-b)}&\text{if $a>b$}.
\end{cases}
$$ 
In particular, this example shows that the positivity properties enjoyed by $P_{y,z}$ in the equal parameter case (that is, $L=\ell$) do not transfer across to the general case. 
\medskip

Let $x,y\in W$. We denote by $h_{x,y,z}\in \sR$ the structure constants associated to the Kazhdan-Lusztig basis:
$$
C_{x}C_y=\sum_{z\in W} h_{x,y,z}C_z.
$$

\begin{Def}[{\cite[Chapter 13]{bible}}]
The \textit{Lusztig $\ba$-function} is the function $\ba:W\to\mathbb{N}$ defined by
$$\ba(z):=\min\{n\in \nN\mid \sq^{-n}h_{x,y,z}\in \nZ[\sq^{-1}]\text{ for all $x,y\in W$}\}.$$
\end{Def}

When $W$ is infinite it is, in general, unknown whether the $\sa$-function is well-defined. However in the case of affine Weyl groups it is known that $\ba$ is well-defined, and that $\ba(z)\leq L(\sw_0)$ where $\sw_0$ is the longest element of an underlying finite Weyl group $W_0$ (see \cite{bible}). The $\ba$-function is a very important tool in the representation theory of Hecke algebras, and plays a crucial role in the work of Lusztig on the unipotent characters of reductive groups.

\begin{Def}
For $x,y,z\in W$ let $\ga_{x,y,z^{-1}}$ denote the constant term of $\sq^{-\ba(z)}h_{x,y,z}$.
\end{Def}

The coefficients $\gamma_{x,y,z^{-1}}$ are the structure constants of the \textit{asymptotic algebra} $\mathcal{J}$ introduced by Lusztig in \cite[Chapter~18]{bible}.

\subsection{Kazhdan-Lusztig cells and associated representations}\label{sec:1.2}

Define preorders $\leq_{\cL},\leq_{\cR},\leq_{\cLR}$ on $W$ extending the following by transitivity:
\begin{align*}
x&\leq_{\cL}y&&\Longleftrightarrow&&\text{there exists $h\in\cH$ such that $C_x$ appears in the decomposition in the KL basis of $hC_y$}\\
x&\leq_{\cR}y&&\Longleftrightarrow&&\text{there exists $h\in\cH$ such that $C_x$ appears in the decomposition in the KL basis of $C_yh$}\\
x&\leq_{\cLR}y&&\Longleftrightarrow&&\text{there exists $h,h'\in\cH$ such that $C_x$ appears in the decomposition in the KL basis of $hC_yh'$.}
\end{align*}
We associate to these preorders equivalence relations $\sim_{\cL}$, $\sim_{\cR}$, and $\sim_{\cLR}$ by setting (for $*\in \{\cL,\cR,\cLR\}$)
$$\text{$x\sim_{*} y$ if and only if $x\leq_{*} y$ and $y\leq_{*} x$}.$$
The equivalence classes of $\sim_{\cL}$, $\sim_{\cR}$, and $\sim_{\cLR}$ are called \textit{left cells}, \textit{right cells}, and \textit{two-sided cells}.  
\medskip

We denote by $\tsc$ the set of all two-sided cells (note that $\tsc$ depends on the choice of parameters). Given any cell $\Ga$ (left, right, or two-sided) we set
$$\Ga_{\leq_\ast}:=\{w\in W\mid\text{there exists $x\in \Ga$} \text{ such that } w\leq_{\ast} x\}$$
and we define $\Ga_{\geq_\ast}$, $\Ga_{>_\ast}$ and $\Ga_{<_\ast}$ similarly.

\begin{Exa}
Table~\ref{dihedral-cell} records the decomposition of the dihedral group $W=I_2(m)=\langle s_1,s_2\rangle$ into two-sided cells for all choices of weight function $L(s_1)=a$ and $L(s_2)=b$ (up to duality). Lusztig's conjectures are known to hold for dihedral groups. In particular the $\ba$-function is constant on two-sided cells, and we list these values below. This example turns out to be particularly useful -- for all affine rank~$3$ (dimension~$2$) Weyl groups every two-sided cell intersects a finite parabolic subgroup (hence a dihedral group), and so assuming the Lusztig conjectures $\mathbf{P4}$ and $\mathbf{P12}$ the table below gives conjectural values of the $\ba$-function on all cells. These `conjectures' become `theorems' for type $\tilde{G}_2$ due to the results of this paper.
\vspace{-0.3cm}

\begin{table}[H]
$$
\renewcommand{\arraystretch}{1.2}
\begin{array}{|l|l|l|}
\hline
W&\text{two-sided cells}&\text{values of the }\ba\text{-function}\\
\hline\hline
I_2(2),\,\,a\geq b&\{e\},\,\{s_1\},\,\{s_2\},\,\{\sw_0\}&0,\,a,\,b,\,a+b\\
\hline
I_2(m),\,\,a=b,\,\,m\geq 2&\{e\},\,W\backslash\{e,\sw_0\},\,\{\sw_0\}&0,\,a,\,ma\\
\hline
I_2(2m),\,\,a>b,\,\,m\geq 2&\{e\},\,W\backslash\{e,s_2,\sw_0s_2,\sw_0\},\,\{s_2\},\,\{\sw_0s_2\},\,\{\sw_0\}&0,\,a,\,b,\,ma-(m-1)b,\,ma+mb\\
\hline
I_2(\infty),\,\,a=b&\{e\},\,W\backslash\{e\}&0,\,a\\
\hline
I_2(\infty),\,\,a>b&\{e\},\,W\backslash\{e,s_2\},\,\{s_2\}&0,\,a,\,b\\
\hline
\end{array}
$$
\caption{Cells and the $\ba$-function for dihedral groups} 
\label{dihedral-cell}
\end{table}
\end{Exa}

\medskip

To each right cell $\Up$ of $W$ there is a natural right $\cH$-module $\cH_{\Up}$ constructed as follows. Let $\cH_{\leq \Up}$ and $\cH_{<\Up}$ be the $\sR$-modules
\begin{align*}
\cH_{\leq \Up}&=\sg C_{x}\mid x\in \Up_{\leq_{\cR}}\sd\quand \cH_{<\Up}=\sg C_{x}\mid  x\in \Up_{<_{\cR}}\sd.
\end{align*}
Then $\cH_{\leq\Up}$ and $\cH_{<\Up}$ are naturally right $\cH$-modules. For $\cH_{\leq\Up}$ this is immediate from the definition of $\leq_{\cR}$. For $\cH_{<\Up}$ we note that if $x\leq_{\cR}y$ with $y\in\Up$ and if $x\notin\Up$ then, for $h\in\cH$,
$$
C_xh=\sum_{z\leq_{\cR}x}a_zC_z.
$$
If $z\in\Gamma$ then necessarily $a_z=0$ (for otherwise $y\sim_{\cR}z\leq_{\cR}x$ and so $y\leq_{\cR}x$ and $x\leq_{\cR}y$ giving $x\in\Gamma$). Thus $\cH_{<\Gamma}$ is a right $\cH$-module. Hence the quotient
$$\cH_{\Up}:=\cH_{\leq \Up}\slash \cH_{<\Up}$$
is naturally a right $\cH$-module with basis $\{\ov{C}_w\mid w\in \Up\}$ where $\ov{C}_w$ is the class of $C_w$ in $\cH_{\Up}$.

\subsection{Lusztig conjectures}\label{sec:1.3}

Define $\De:W\to \nN$ and $n_z\in \mathbb{Z}\backslash\{0\}$ by the relation
$$P_{e,z}=n_z\sq^{-\De(z)}+\text{ strictly smaller powers of $\sq$.}$$
This is well defined because $P_{x,y}\in \sq^{-1}\nZ[\sq^{-1}]$ for all $x,y\in W$. Let
$$\cD=\{w\in W\mid \De(w)=\ba(w)\}.$$
The elements of $\cD$ are called \textit{Duflo elements} (or, somewhat prematurely, \textit{Duflo involutions}; see \conj{6} below). 
\medskip

In \cite[Chapter 13]{bible},  Lusztig has formulated  the following 15 conjectures, now known as ${\bf P1}$--${\bf P15}$.
\begin{itemize}
\item[\bf P1.] For any $z\in W$ we have $\ba(z)\leq \Delta(z)$.
\item[\bf P2.] If $d \in \cD$ and $x,y\in W$ satisfy $\gamma_{x,y,d}\neq 0$,
then $y=x^{-1}$.
\item[\bf P3.] If $y\in W$ then there exists a unique $d\in \cD$ such that
$\gamma_{y,y^{-1},d}\neq 0$.
\item[\bf P4.] If $z'\leq_{\cLR} z$ then $\ba(z')\geq \ba(z)$. In particular the $\ba$-function is constant on two-sided cells.
\item[\bf P5.] If $d\in \cD$, $y\in W$, and $\gamma_{y,y^{-1},d}\neq 0$, then
$\gamma_{y,y^{-1},d}=n_d=\pm 1$.
\item[\bf P6.] If $d\in \cD$ then $d^2=1$.
\item[\bf P7.] For any $x,y,z\in W$, we have $\gamma_{x,y,z}=\gamma_{y,z,x}$.
\item[\bf P8.] Let $x,y,z\in W$ be such that $\gamma_{x,y,z}\neq 0$. Then
$y\sim_{\cR} x^{-1}$, $z \sim_{\cR} y^{-1}$, and $x\sim_{\cR} z^{-1}$.
\item[\bf P9.] If $z'\leq_{\cL} z$ and $\ba(z')=\ba(z)$, then $z'\sim_{\cL}z$.
\item[\bf P10.] If $z'\leq_{\cR} z$ and $\ba(z')=\ba(z)$, then $z'\sim_{\cR}z$.
\item[\bf P11.] If $z'\leq_{\cLR} z$ and $\ba(z')=\ba(z)$, then
$z'\sim_{\cLR}z$.
\item[\bf P12.] If $I\subseteq S$ then the $\ba$-function of $W_I$ is the restriction of the $\ba$-function of $W$.
\item[\bf P13.] Each right cell $\Ga$ of $W$ contains a unique element
$d\in \cD$. We have $\gamma_{x,x^{-1},d}\neq 0$ for all $x\in \Ga$.
\item[\bf P14.] For each $z\in W$ we have $z \sim_{\cLR} z^{-1}$.
\item[\bf P15.] If $x,x',y,w\in W$ are such that $\ba(w)=\ba(y)$ then 
$$\sum_{y'\in W} h_{w,x',y'}\otimes h_{x,y',y}=\sum_{y'\in W} h_{y',x',y}\otimes h_{x,w,y'} \text{ in } \sR\otimes_{\nZ} \sR.$$
\end{itemize}

As noted in the introduction, these conjectures have been established in the following cases: (1) when $W$ is a Weyl group or an affine Weyl group with equal parameters (see \cite{bible} and the updated version available on $\mathsf{ArXiv}$), (2) in the ``quasisplit case'' \cite[Chapter~16]{bible} where $W$ is obtained by ``twisting'' a larger Coxeter group~$\tilde{W}$, and $L$ is the restriction of the length function on $\tilde{W}$ to $W$, (3) when $W$ is a dihedral group (finite or infinite) for all choices of parameters (see \cite{Geck:11,bible}), (4) when $W=B_n$ in the ``asymptotic case'' (see~\cite{B-I,Geck:11}),  (5) when $W=F_4$ for any choices of parameters (see \cite{Geck:11}). We note that in case (1) and (2) the proof relies on deep results including the positivity of the Kazhdan-Lusztig polynomials in equal parameters. This approach cannot work for the general case, since we have seen that the positivity no longer holds in this case.

\medskip

In this paper we prove all conjectures $\conj{1}$--$\conj{15}$ for $\tilde{G}_2$ for all choices of parameters. Our approach extends naturally to all rank $2$ affine Weyl groups. The analysis for the three parameter case $\tilde{C}_2$ becomes rather involved due to the large number of distinct regimes of cell decompositions, and therefore we will provide the details elsewhere.

%%%%%%%%%%%%%%%%%%%%%%%%%%%%%%%%%%%
%%%%%%%%%%%%%%%%%%%%%%%%%%%%%%%%%%%
%%%%%%%%%%%%%%%%%%%%%%%%%%%%%%%%%%%
%%%%%%%%%%%%%%%%%%%%%%%%%%%%%%%%%%%
%%%%%%%%%%%%%%%%%%%%%%%%%%%%%%%%%%%
%%%%%%%%%%%%%%%%%%%%%%%%%%%%%%%%%%%
%%%%%%%%%%%%%%%%%%%%%%%%%%%%%%%%%%%
%%%%%%%%%%%%%%%%%%%%%%%%%%%%%%%%%%%
%%%%%%%%%%%%%%%%%%%%%%%%%%%%%%%%%%%
%%%%%%%%%%%%%%%%%%%%%%%%%%%%%%%%%%%
%%%%%%%%%%%%%%%%%%%%%%%%%%%%%%%%%%%

\section{Systems of balanced representations and Lusztig $\ba$-function}
\label{sec:def-balanced}

In this section we define a \textit{balanced system of cell representations}, inspired by the work of Geck~\cite{Geck:02,Geck:11} in the finite case. We show that the existence of such a system, plus one additional axiom, is sufficient for the computation of Lusztig's $\ba$-function. This gives us our primary strategy for resolving Lusztig's conjectures in type $\tilde{G}_2$. 

\subsection{Balanced system of cell representations}

In this section we introduce the central notion of the paper, inspired by the work of Geck~\cite{Geck:02,Geck:11} in the finite case.
\medskip

Recall that $\sR=\mathbb{Z}[\sq,\sq^{-1}]$. If $\mathsf{S}$ is an $\sR$-polynomial ring (including the possibility $\mathsf{S}=\sR$), we write $\mathsf{S}^{\leq 0}$ and $\mathsf{S}^0$ for the associated $\mathbb{Z}[\sq^{-1}]$-polynomial and $\mathbb{Z}$-polynomial subrings of $\mathsf{S}$, respectively. In particular $\sR^{\leq 0}=\mathbb{Z}[\sq^{-1}]$ and $\sR^0=\mathbb{Z}$. Let 
$$
\text{sp}_{|_{\sq^{-1}=0}}:\mathsf{S}^{\leq 0}\to \mathsf{S}^0\quad\text{denote the specialisation at $\sq^{-1}=0$}.
$$ 

By a \textit{matrix representation} of $\cH$ we shall mean a triple $(\pi,\cM,\sB)$ where $\cM$ is a right $\cH$-module over an $\sR$-polynomial ring $\mathsf{S}$, and $\sB$ is a basis of~$\cM$. We write (for $h\in\cH$ and $u,v\in\sB$)
$$
\pi(h)\quad\text{and}\quad [\pi(h)]_{u,v}
$$
for the matrix of $\pi(h)$ with respect to the basis $\sB$, and the $(u,v)^{th}$ entry of the matrix $\pi(h)$, respectively. We call a matrix representation $(\pi,\cM,\sB)$ \textit{bounded} if there exists an integer $n\geq 0$ such that 
$$
\sq^{-n}[\pi(C_w)]_{u,v}\in\mathsf{S}^{\leq 0}\quad\text{ for all $u,v\in \sB$ and all $w\in W$}.
$$
In this case we call the integer
$$
\ba_\pi:=\min\{n\in\mathbb{N}\mid \sq^{-n}[\pi(C_w)]_{u,v}\in\mathsf{S}^{\leq 0}\text{ for all $u,v\in \sB$ and all $w\in W$}\}
$$
the \textit{bound} of the matrix representation and we define  the \textit{leading matrices} by
\begin{center}
$\displaystyle{\fc_{\pi_\Ga,w}=\text{sp}_{|_{\sq^{-1}=0}}\left(\sq^{-\ba_{\pi}}\pi(C_w)\right)}\quad\text{for $w\in W$}.$
\end{center}

\begin{Def}\label{def:balanced}
We say that $\cH$ admits a \textit{balanced system of cell representations} if for each two-sided cell $\Gamma\in\tsc$ there exists a matrix representation $(\pi_{\Gamma},\cM_{\Gamma},\sB_{\Ga})$ defined over an $\sR$-polynomial ring $\sR_\Ga$ (where we could have $\sR_\Ga=\sR$) such that the following properties hold:
\begin{enumerate}
\item[\B{1}\textbf{.}] If $w\notin\Gamma_{\geq_{\mathcal{LR}}}$ then $\pi_{\Gamma}(C_w)=0$. 
\item[\B{2}\textbf{.}] The matrix representation $(\pi_{\Gamma},\cM_{\Gamma},\sB_{\Ga})$ is bounded by $\ba_{\pi_\Ga}$. 
\item[\B{3}\textbf{.}] We have $\fc_{\pi_{\Ga},w}\neq 0$ if and only if $w\in \Ga$. 
\item[\B{4}\textbf{.}] The leading matrices $\fc_{\pi_\Ga,w}$ ($w\in \Ga$) are free over $\nZ$.
\item[\B{5}\textbf{.}] If $\Gamma'\leq_{\cLR}\Gamma$ then $\ba_{\pi_{\Ga'}}\geq \ba_{\pi_{\Gamma}}$. 
\end{enumerate}
\end{Def}

%
%\begin{Def}\label{def:balanced}
%We say that $\cH$ admits a \textit{balanced system of cell representations} if for each two-sided cell $\Gamma\in\tsc$ there exists a representation $(\pi_{\Gamma},\cM_{\Gamma})$ defined over an $\sR$-polynomial ring $\sR_\Ga$ (where we could have $\sR_\Ga=\sR$), and a basis $\sB_{\Gamma}$ of $\cM_{\Gamma}$, and an integer $\ba_{\pi_{\Ga}}\geq 0$ such that the following properties hold:
%\begin{enumerate}
%\item[\B{1}\textbf{.}] If $w\notin\Gamma_{\geq_{\mathcal{LR}}}$ then $\pi_{\Gamma}(C_w)=0$. 
%\item[\B{2}\textbf{.}] For all $w\in W$ and all $u,v\in\sB_{\Gamma}$ we have $\sq^{-\ba_{\pi_{\Ga}}}[\pi_{\Gamma}(C_w;\sB_{\Gamma})]_{u,v}\in\sR_{\Ga}^{\leq 0}$. We define the \textit{leading matrices} by
%\begin{center}
%$\displaystyle{\fc_{\pi_\Ga}(w;\sB_{\Ga})=\text{sp}_{|_{\sq^{-1}=0}}\left(\sq^{-\ba_{\pi_\Ga}}\pi_\Ga(C_w;\sB_{\Ga})\right)}\quad\text{for $w\in W$}.$
%\end{center}
%\item[\B{3}\textbf{.}] We have $\fc_{\pi_{\Ga}}(w;\sB_{\Ga})\neq 0$ if and only if $w\in \Ga$. 
%\item[\B{4}\textbf{.}] The leading matrices $\fc_{\pi_\Ga}(w;\sB_{\Ga})$, $w\in \Ga$, are free over $\nZ$.
%\item[\B{5}\textbf{.}] If $\Gamma'\leq_{\cLR}\Gamma$ then $\ba_{\pi_{\Ga'}}\geq \ba_{\pi_{\Gamma}}$. 
%\end{enumerate}
%\end{Def}

The natural numbers $(\ba_{\pi_{\Gamma}})_{\Ga\in \tsc}$ are called the \textit{bounds} of the balanced system of cell representations. The main approach of this paper hinges on the construction of a balanced system of cell representations for the Hecke algebra of type $\tilde{G}_2$ in each parameter regime. 
\medskip

Note that \B{1} above does not depend on the choice of basis. A representation with property \B{1} is called a \textit{cell representation} for the two-sided cell $\Gamma$. It is clear that the representations associated to cells that we introduced in Section \ref{sec:1.2} are cell representations. To see this, let $\Upsilon$ be a right cell, and let $\cH_{\Upsilon}$ be as in Section~\ref{sec:1.2}. If $C_w$ acts nontrivially on $\cH_{\Upsilon}$ then there exist $u,v\in \Upsilon$ such that $C_u\cdot C_w=\sum_{z}h_{u,w,z}C_z$ with $h_{u,w,v}\neq 0$. Thus $v\leq_{\cLR} w$.  

\medskip

We say that a representation $(\pi,\cM)$ is \textit{$\Gamma$-balanced} for the two-sided cell $\Gamma$ if $\cM$ admits a basis such that \B{2}, \B{3} and \B{4} hold. We note that in  \B{2} and \B{3}  it is equivalent to replace $C_w$ by $T_w$, because $C_w=T_w+\sum_{v<w}P_{v,w}T_v$ with $P_{v,w}\in\sq^{-1}\mathbb{Z}[\sq^{-1}]$. However in \B{1} one cannot replace $C_w$ by $T_w$.

\begin{Rem}\label{rem:bounds}
Let $\mathsf{S}$ be an $\sR$-polynomial ring. We define the \textit{degree} of an element $f\in\mathsf{S}$ to be the greatest integer $n\in\mathbb{Z}$ such that $\sq^n$ appears in $f$ with nonzero coefficient (with $\deg(0)=-\infty$). Equivalently, $\deg(f)$ is the greatest integer $n\in\mathbb{Z}$ such that $\sq^{-n}f\in\mathsf{S}^{\leq 0}$ and $\text{sp}_{|_{\sq^{-1}=0}}(\sq^{-n}f)\neq 0$. For example, in the case $\mathsf{S}=\mathsf{R}$ we have $\deg(3\sq^{-1}+\sq^{-2})=-1$ and $\deg(3\sq^{-1}+\sq^2)=2$. Then axioms~$\B{2}$ and~$\B{3}$ can be rephrased as: There exists an integer $\ba_{\pi_{\Ga}}\geq 0$ such that for all $w\in W$,
$$
\max\{\deg[\pi_{\Gamma}(C_w)]_{u,v}\mid u,v\in\sB_{\Ga}\}\leq \ba_{\pi_{\Gamma}}\quad\text{with equality if and only if $w\in \Ga$}.
$$
%In particular the bounds $(\ba_{\pi_{\Ga}})_{\Ga\in\tsc}$, if they exist, are uniquely determined by $(\pi_{\Ga},\cM_{\Ga},\sB_{\Ga})_{\Ga\in\tsc}$. 
\end{Rem}

\begin{Exa} 
\label{exa:balanced-one-dim}
Let $W$ be an affine Weyl group of type $\tilde{G}_2$ with diagram and weight function defined by 
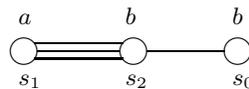
\begin{center}
\begin{picture}(150,32)
\put( 40, 10){\circle{10}}
\put( 44,  7){\line(1,0){33}}
\put( 45,  10){\line(1,0){30.5}}
\put( 44, 13){\line(1,0){33}}
\put( 81, 10){\circle{10}}
\put( 86, 10){\line(1,0){29}}
\put(120, 10){\circle{10}}
\put( 38, 20){$a$}
\put( 78, 20){$b$}
\put(118, 20){$b$}
\put( 38, -3){$s_{1}$}
\put( 78, -3){$s_{2}$}
\put(118,-3){$s_{0}$}
\end{picture}
\end{center}
where $a,b$ are positive integers. Let $I\subseteq S$ be a union of conjugacy classes in $S$. We define the one dimensional representation $\rho_I$ of $W$ by
$$
\rho_{I}(T_s)=\begin{cases}
\sq_s & \mbox{ if $s\in I$},\\
-\sq^{-1}_s &\mbox{otherwise.}
\end{cases}
$$
With this notation $\rho_{\emptyset}$ is the \textit{sign representation} and $\rho_{S}$ is the \textit{trivial representation}. It is easy to see that (1) $\rho_{\emptyset}$ is $\Ga_e$-balanced where $\Ga_e$ is the two-sided cell that contains precisely the identity and (2) 
$\rho_{S}$ cannot be balanced for any two-sided cell $\Ga$. 

\medskip

Consider the representation $\rho_{I}$ where $I=\{s_0,s_2\}$ .  We will see in Section \ref{sec:KL-cells-G2} that $\Ga_5:=\{s_0s_2s_0\}$ is a two-sided cell in~$W$ when $a/b>2$. For  $w\in W$ we have $\rho_I(w)=\sq^{b\ell_2(w)}(-\sq)^{-a\ell_1(w)}$ where $\ell_2(w)$ is the number of $s_2$ and $s_0$ generators in any reduced expression of $w$ and $\ell_1(w)$ is the number of $s_1$ generators. 
Saying that the representation $\rho_I$ is $\Ga_5$-balanced for $a/b>2$ means that
$b\ell_{2}(w)-a\ell_1(w)\leq 3b$ for all $w$ and that there is equality if and only if $w=s_0s_2s_0$.
This can be done by studying reduced expressions in $W$, and we will see another proof using Kazhdan-Lusztig theory in Section~\ref{sec:balanced}.

\medskip

Proceeding as above, one can show that the representation $\rho_{I}$ where $I=\{s_1\}$ is $\Ga_7$-balanced  whenever $a/b<1$.  Once again we will see in Section \ref{sec:KL-cells-G2} that $\Ga_7:=\{s_1\}$ is a two-sided cell in $W$  for this parameter range.

\end{Exa}

\subsection{Computing the $\ba$-function}\label{sec:afunction}

In this section we show that axioms $\B{1}$--$\B{5}$, along with an additional axiom $\B{4}'$ introduced below, are sufficient to show that Lusztig's $\ba$-function is constant on two-sided cells, and moreover we are able to compute the value of the $\ba$-function in terms of the bounds $(\ba_{\pi_{\Gamma}})_{\Gamma\in\tsc}$ from $\B{2}$. 
\medskip

Let $(\pi_\Ga)_{\Ga\in \tsc}$ be a balanced system of cell representations for $\cH$ with bounds $\ba_{\pi_\Ga}$ for all $\Ga\in \tsc$. We have

\begin{align}\label{eq:h}
C_xC_y=\sum_{\Ga\in\tsc }\sum_{z\in\Gamma}h_{x,y,z}C_z.
\end{align}

\begin{Prop} 
\label{prop:bounda}
Let $x,y\in W$ and $w\in\Gamma$ where $\Ga\in \tsc$. We have $\deg(h_{x,y,w})\leq \ba_{\pi_\Ga}$.
\end{Prop}

\begin{proof}
We proceed by downwards induction. Let $\Ga\in \tsc$ and suppose that $\deg(h_{x,y,z})\leq \ba_{\pi_{\Ga'}}$ for all  $z\in \Ga'$ where $\Ga'>_{\cLR} \Ga$. 
 Then applying $\pi_{\Ga}$ to~(\ref{eq:h}) using the fact that $\pi_\Ga$ is a cell representation gives 
\begin{align}\label{eq:XX}
\pi_{\Ga}(C_xC_y)=\sum_{z\in \Ga}h_{x,y,z}\pi_{\Ga}(C_z)+\underset{\Ga'>_{\cLR}\Ga}{\sum_{\Ga'\in \tsc,}}\sum_{z\in \Ga'}h_{x,y,z}\pi_{\Ga}(C_z).
\end{align}
By $\B{2}$ the degree of the matrix coefficients of $\pi_{\Ga}(C_xC_y)=\pi_{\Gamma}(C_x)\pi_{\Gamma}(C_y)$ is bounded by $2\ba_{\pi_\Ga}$. By the induction hypothesis and properties of balanced representations the degree of the matrix coefficients for each term in the double  sum on the right is strictly bounded by $\ba_{\pi_{\Ga'}}+\ba_{\pi_\Ga}\leq 2\ba_{\pi_\Ga}$. Indeed the maximal degree that can appear in $\pi_\Ga(C_z)$ is stricly less than $\ba_{\pi_\Ga}$ since~$z\notin \Ga$ and the bounds of the balanced system are decreasing with respect to $\leq_{\cLR}$. We now show that 
$$\deg(h_{x,y,z})\leq \ba_{\pi_\Ga}\text{ for all }z\in \Ga.$$
Let $m=\max \{\deg(h_{x,y,z})\mid z\in \Ga\}$ and let $\cZ=\{z'\in \Ga\mid \deg(h_{x,y,z'})=m\}\neq\emptyset$.  For $z\in \cZ$ define $\tilde{\gamma}_{x,y,z^{-1}}\in\mathbb{Z}$ by $h_{x,y,z}=\sq^m\tilde{\gamma}_{x,y,z^{-1}}+\text{lower terms}$. By $\B{3}$ we have $\pi_{\Gamma}(C_z)=\sq^{\ba_{\pi_{\Gamma}}}\fc_{\pi_{\Gamma},z}+\text{lower terms}$, with $\fc_{\pi_{\Gamma},z}\neq 0$ the leading matrix. Then the right hand side of~(\ref{eq:XX}) is of the form
$$
\sq^{m+\ba_{\pi_{\Gamma}}}\sum_{z\in \cZ}\tilde{\gamma}_{x,y,z^{-1}}\fc_{\pi_{\Gamma},z}+\text{lower terms},
$$
and by $\B{4}$ the expression in the sum (that is, the coefficient of $\sq^{m+\ba_{\pi_{\Gamma}}}$) is nonzero. By comparing with the lefthand side in (\ref{eq:XX}) it follows that $m+\ba_{\pi_\Ga}\leq 2\ba_{\pi_\Ga}$ that is $m\leq \ba_{\pi_\Ga}$ as required.
 \end{proof}
 
\begin{Cor}\label{cor:tildegamma}
Let $\Gamma\in\tsc$. The subset $\mathcal{J}_{\Ga}$ of $\cM_{\dim(\pi_\Ga)}(\sR_\Ga)$ generated by $\{\fc_{\pi_{\Ga},w}\mid w\in \Ga\}$ is a $\nZ$-subalgebra.
\end{Cor}
\begin{proof}
Let $\Ga$ be the two-sided cell an let $x,y\in \Ga$. Applying $\pi_\Ga$ to $C_xC_y$, using \B{1}, and multiplying by $\sq^{-2\ba_{\pi_{\Ga}}}$ we get 
$$
[\sq^{-\ba_{\pi_{\Ga}}}\pi_{\Ga}(C_x)][\sq^{-\ba_{\pi_{\Ga}}}\pi_{\Ga}(C_y)]=\sum_{z\in \Ga}[\sq^{-\ba_{\pi_{\Ga}}}h_{x,y,z}][\sq^{-\ba_{\pi_{\Ga}}}\pi_{\Ga}(C_z)]+\underset{\Ga'>_{\cLR}\Ga}{\sum_{\Ga'\in \tsc,}}\sum_{z\in \Ga'}[\sq^{-\ba_{\pi_{\Ga}}}h_{x,y,z}][\sq^{-\ba_{\pi_{\Ga}}}\pi_{\Ga}(C_z)].
$$
Specialising at $\sq^{-1}=0$ will annihilate all the terms in the double sum. Indeed for $z\in\Gamma'$ with $\Gamma'>_{\cLR}\Gamma$ we have $\deg(h_{x,y,z})\leq \ba_{\pi_{\Ga'}}\leq \ba_{\pi_{\Ga}}$ and the maximal degree that can appear in $\pi_{\Ga}(C_z)$ is strictly less that $\ba_{\pi_\Ga}$. Thus we obtain 
$$
\fc_{\pi_\Ga,x}\fc_{\pi_\Ga,y}=\sum_{z\in \Ga} \tilde{\ga}_{x,y,z^{-1}}\fc_{\pi_\Ga,z}
$$
where $\tilde{\ga}_{x,y,z^{-1}}\in \nZ$ is the coefficient of degree $\ba_{\pi_\Ga}$ of $h_{x,y,z}$. 
\end{proof}

We introduce the following additional axiom, where $\tilde{\ga}_{x,y,z^{-1}}\in \nZ$ is the coefficient of degree $\ba_{\pi_\Ga}$ of $h_{x,y,z}$. 
\ben
\item[$\B{4}'$\textbf{.}] Let $\Ga\in \tsc$. For each $z\in \Ga$, there exists $(x,y)\in \Ga^{2}$ such that $\tilde{\ga}_{x,y,z^{-1}}\neq 0$. 
\een
We can now show that if all axioms $\B{1}$--$\B{5}$ and $\B{4}'$ are satisfied, then we can compute Lusztig's $\ba$-function in terms of the bounds $\ba_{\pi_{\Gamma}}$. 
\begin{Th}\label{thm:afn}
Suppose that $\B{1}$--$\B{5}$ and $\B{4}'$ are satisfied. Then $\ba(w)=\ba_{\pi_\Ga}$ for all $w\in \Ga$.
\end{Th}
\begin{proof}
According to Proposition \ref{prop:bounda}, we only need to show that $\ba_{\pi_\Ga}\geq \ba(w)$.  To do so it is enough to find $x,y\in W$ such that $\deg(h_{x,y,w})=\ba_{\pi_\Ga}$ or equivalently, to find $x,y\in W$ such that $\tilde{\ga}_{x,y,w}\neq 0$. Hence the result using $\B{4}'.$
\end{proof}

\begin{Cor}\label{cor:J}
Assuming $\B{1}$--$\B{5}$ and $\B{4}'$, the ring $\mathcal{J}_{\Ga}$ is isomorphic to Lusztig's asymptotic algebra associated to $\Ga$. 
\end{Cor}

\begin{proof}
The elements $\tilde{\ga}_{x,y,z^{-1}}$ are the coefficients of $h_{x,y,z}$ of degree $\ba_{\pi_\Ga}$, and are the structure constants of $\mathcal{J}_{\Ga}$ with respect to the basis $\{\fc_{\pi_{\Ga},w}\mid w\in \Ga\}$. Indeed once we know that $\ba_{\pi_\Ga}=\ba(\Ga)$ we know that the structure constants of $\mathcal{J}_\Ga$ are $\ga_{x,y,z^{-1}}$. 
\end{proof}

We note that our construction above parallels Geck's construction from \cite[\S1.5]{Geck:11book}. Another construction of the asymptotic algebra has been obtained by Koenig and Xi~\cite{KX} in the case that $\cH$ is affine cellular.

%%%%%%%%%%%%%%%%%%%%%%%%%%%%%%%%%%%
%%%%%%%%%%%%%%%%%%%%%%%%%%%%%%%%%%%
%%%%%%%%%%%%%%%%%%%%%%%%%%%%%%%%%%%
%%%%%%%%%%%%%%%%%%%%%%%%%%%%%%%%%%%
%%%%%%%%%%%%%%%%%%%%%%%%%%%%%%%%%%%
%%%%%%%%%%%%%%%%%%%%%%%%%%%%%%%%%%%
%%%%%%%%%%%%%%%%%%%%%%%%%%%%%%%%%%%
%%%%%%%%%%%%%%%%%%%%%%%%%%%%%%%%%%%
%%%%%%%%%%%%%%%%%%%%%%%%%%%%%%%%%%%
%%%%%%%%%%%%%%%%%%%%%%%%%%%%%%%%%%%
%%%%%%%%%%%%%%%%%%%%%%%%%%%%%%%%%%%

\section{Affine Weyl groups, affine Hecke algebras, and alcove walks}\label{sec:2}

We begin this section by recalling basic definitions and terminology concerning affine Weyl groups. While we are primarily interested in $\tilde{G}_2$ in this paper, some of our results apply in arbitrary type, and in any case the general language turns out to be more appropriate for the formulation of our results and their proofs. Next we recall the Bernstein-Lusztig basis of the affine Hecke algebra, and its combinatorial interpretation using alcove walks, following~\cite{Ram:06}. Finally we present a combinatorial formula for the Weyl character that will be used in Section~\ref{sec:4}. 

\subsection{Root systems, Weyl groups, and affine Weyl groups}\label{sec:root}

Let $\Phi$ be a reduced, irreducible, finite, crystallographic root system with simple roots $\alpha_1,\ldots,\alpha_n$ in an $n$-dimensional real vector space $V$ with inner product $\langle\cdot,\cdot\rangle$. Let $\Phi^+$ be the set of positive roots relative to the simple roots $\alpha_1,\ldots,\alpha_n$. Let $W_0$ be the \textit{Weyl group}; the subgroup of $GL(V)$ generated by the reflections $s_{\alpha}$, $\alpha\in \Phi$, where
$$
s_{\alpha}\lambda=\lambda-\langle\lambda,\alpha\rangle\alpha^{\vee}\quad\text{with}\quad \alpha^{\vee}=2\alpha/\langle\alpha,\alpha\rangle.
$$
The group $W_0$ is a finite Coxeter group with distinguished generators $s_1,\ldots,s_n$, where $s_i=s_{\alpha_i}$. Let $\sw_0$ be the longest element of~$W_0$.
\medskip

Let $\mathcal{F}_0$ denote the union of the hyperplanes $H_{\alpha}=\{x\in V\mid \langle x,\alpha\rangle=0\}$ with $\alpha\in \Phi$. The closures of the open connected components of $V\backslash\mathcal{F}_0$ are geometric cones, called \textit{Weyl chambers}. The \textit{fundamental Weyl chamber} is given by 
$$
C_0=\{x\in V\mid  \langle x,\alpha\rangle\geq 0\text{ for all $\alpha\in\Phi^+$}\}.
$$
The Weyl group $W_0$ acts simply transitively on the set of Weyl chambers, and we sometimes use this action to identify the set of Weyl chambers with $W_0$ via $w\leftrightarrow wC_0$.

\medskip

The \textit{dual root system} is $\Phi^{\vee}=\{\alpha^{\vee}\mid \alpha\in \Phi\}$ and the \textit{coroot lattice} of $\Phi$ is 
$
Q=\mathbb{Z}\textrm{-span of }\Phi^{\vee}.
$ The \textit{fundamental coweights} of $\Phi$ are the vectors $\omega_1,\ldots,\omega_n$ where $\langle\omega_i,\alpha_j\rangle=\delta_{ij}$. The coweight lattice is $P=\mathbb{Z}\omega_1+\cdots+\mathbb{Z}\omega_n$, and the cone of \textit{dominant coweights} is $P^+=P\cap C_0=\mathbb{N}\omega_1+\cdots+\mathbb{N}\omega_n$. Note that $Q\subseteq P$. 
\medskip

The Weyl group $W_0$ acts on $Q$ and the \textit{affine Weyl group} is
$
W=Q\rtimes W_0
$
where we identify $\lambda\in Q$ with the translation $t_{\lambda}(x)=x+\lambda$. We have the following standard facts:
\ben
\item $W$ is generated by the orthogonal reflections $s_{\alpha,k}$ in the affine hyperplanes 
$
H_{\alpha,k}=\{x\in V\mid\langle x,\alpha\rangle=k\}
$ with $\alpha\in\Phi$ and $k\in\mathbb{Z}$. Explicitly, $s_{\alpha,k}(x)=x-(\langle x,\alpha\rangle-k)\alpha^{\vee}$, so that $s_{\alpha,k}=t_{k\alpha^{\vee}}s_{\alpha}$.
\item The affine Weyl group $W$ is a Coxeter group with generating set $S=\{s_0,s_1,\ldots,s_n\}$, where $s_0=t_{\varphi^{\vee}}s_{\varphi}$, with $\varphi$ the highest root of $\Phi$.  
\een

\medskip

Each hyperplane $H_{\alpha,k}$ with $\alpha\in \Phi^+$ and $k\in\mathbb{Z}$ divides $V$ into two half spaces, denoted
$$
H_{\alpha,k}^+=\{x\in V\mid \langle x,\alpha\rangle\geq k\}\quad\text{and}\quad H_{\alpha,k}^-=\{x\in V\mid\langle x,\alpha\rangle\leq k\}.
$$
This ``orientation'' of the hyperplanes is called the \textit{periodic orientation}, since it is invariant under translation by~$\lambda\in Q$.

\medskip

If $w\in W$ we define the \textit{final direction} $\theta(w)\in W_0$ and the \textit{translation weight} $\mathrm{wt}(w)\in Q$ by the equation $w=t_{\mathrm{wt}(w)}\theta(w)$. Here we use the fact that each element $w\in W$ can be written uniquely as $w=t_{\lambda}v$ with $v\in W_0$ and $\lambda\in Q$.

\medskip

Let $\mathcal{F}$ denote the union of the hyperplanes $H_{\alpha,k}$ with $\alpha\in \Phi$ and $k\in \mathbb{Z}$. The closures of the open connected components of $V\backslash\mathcal{F}$ are called \textit{alcoves}. The \textit{fundamental alcove} is given by 
$$
A_0=\{x\in V\mid 0\leq \langle x,\alpha\rangle\leq 1\text{ for all $\alpha\in\Phi^+$}\}.
$$
The hyperplanes bounding $A_0$ are called the \textit{walls} of $A_0$. Explicitly these walls are $H_{\alpha_i,0}$ with $i=1,\ldots,n$ and $H_{\varphi,1}$. We say that a \textit{face} of $A_0$ (that is, a codimension~$1$ facet) has type $s_i$ for $i=1,\ldots,n$ if it lies on the wall $H_{\alpha_i,0}$ and of type $s_0$ if it lies on the wall $H_{\varphi,1}$. 

\medskip

The affine Weyl group $W$ acts simply transitively on the set of alcoves, and we use this action to identify the set of alcoves with $W$ via $w\leftrightarrow wA_0$. Moreover, we use the action of $W$ to transfer the notions of walls, faces, and types of faces to arbitrary alcoves. Alcoves $A$ and $A'$ are called \textit{$s$-adjacent}, written $A\sim_s A'$, if $A\neq A'$ and $A$ and $A'$ share a common type $s$ face. Under the identification of alcoves with elements of $W$, the alcoves $w$ and $ws$ are $s$-adjacent. 

\medskip

For any sequence $\vec{w}=(s_{i_{1}},s_{i_2},\ldots,s_{i_{\ell}})$ of elements of $S$ we have
$$
e\sim_{s_{i_1}}s_{i_1}\sim_{s_{i_2}} s_{i_1}s_{i_2}\sim_{s_{i_3}}\cdots\sim_{s_{i_{\ell}}}s_{i_{1}}s_{i_2}\cdots s_{i_\ell}.
$$
In this way, sequences $\vec{w}$ of elements of $S$ determine \textit{alcove walks} of \textit{type} $\vec{w}$ starting at the fundamental alcove~$e=A_0$. We will typically abuse notation and refer to alcove walks of type $\vec{w}=s_{i_1}s_{i_2}\cdots s_{i_{\ell}}$ rather than $\vec{w}=(s_{i_{1}},s_{i_2},\ldots,s_{i_{\ell}})$. Thus ``the alcove walk of type $\vec{w}=s_{i_1}s_{i_2}\cdots s_{i_{\ell}}$'' is the sequence $(v_0,v_1,\ldots,v_{\ell})$ of alcoves, where $v_0=e$ and $v_k=s_{i_1}\cdots s_{i_k}$ for $k=1,\ldots,\ell$.

\medskip

We are, of course, primarily interested in the case where $\Phi$ is a root system of type $G_2$. We outline this example below.

\begin{Exa}
\label{rootG2}
Let $\Phi$ be a root system of type $G_2$ with simple roots $\alpha_1$ and $\alpha_2$. We have $P=Q$, and the dual root system is
$$
\Phi^\vee:=\pm \{\al_1^\vee,\al_2^\vee,\al_1^\vee+\al_2^\vee,\al_1^\vee+2\al_2^\vee,\al_1^\vee+3\al_2^\vee, 2\al_1^\vee+3\al_2^\vee\}.
$$
The fundamental alcove is shaded in Figure~\ref{rootsystem}, and the periodic orientation on some hyperplanes is shown. 

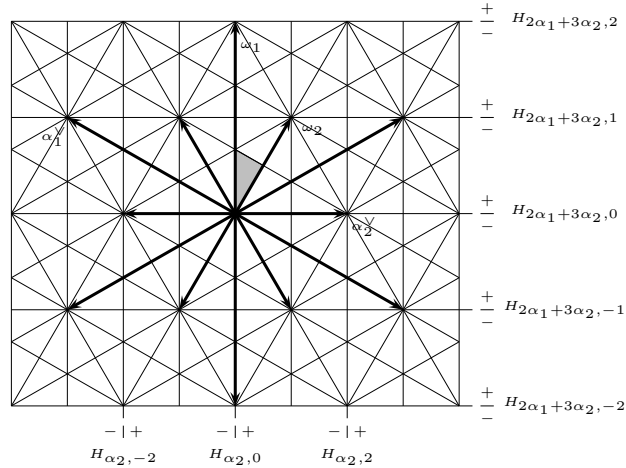
\begin{figure}[H]
\begin{center}
\psset{unit=.85cm}
\begin{pspicture}(-4.5,-4)(4.5,3.3)

\psset{linewidth=0.05mm}

\pspolygon[fillstyle=solid,fillcolor=lightgray](0,0)(0,1)(.433,.75)

\psline[linewidth=.4mm]{->}(0,0)(1.732,0)
\psline[linewidth=.4mm]{->}(0,0)(2.598,1.5)
\psline[linewidth=.4mm]{->}(0,0)(.866,1.5)
\psline[linewidth=.4mm]{->}(0,0)(0,3)
\psline[linewidth=.4mm]{->}(0,0)(-.866,1.5)
\psline[linewidth=.4mm]{->}(0,0)(-2.598,1.5)

\psline[linewidth=.4mm]{->}(0,0)(-1.732,0)
\psline[linewidth=.4mm]{->}(0,0)(-2.598,-1.5)
\psline[linewidth=.4mm]{->}(0,0)(-.866,-1.5)
\psline[linewidth=.4mm]{->}(0,0)(0,-3)
\psline[linewidth=.4mm]{->}(0,0)(.866,-1.5)
\psline[linewidth=.4mm]{->}(0,0)(2.598,-1.5)

\rput(2,-.2){{\tiny $\al^\vee_2$}}
\rput(-2.8,1.2){{\tiny $\al^\vee_1$}}

\rput(.25,2.6){{\tiny $\om_1$}}
\rput(1.2,1.3){{\tiny $\om_2$}}

\rput(5.1,0){{\tiny $H_{2\al_1+3\al_2,0}$}}
\rput(5.1,1.5){{\tiny $H_{2\al_1+3\al_2,1}$}}
\rput(5.1,3){{\tiny $H_{2\al_1+3\al_2,2}$}}
\rput(5.1,-1.5){{\tiny $H_{2\al_1+3\al_2,-1}$}}
\rput(5.1,-3){{\tiny $H_{2\al_1+3\al_2,-2}$}}

\rput(0,-3.8){{\tiny $H_{\al_2,0}$}}

\rput(-1.732,-3.8){{\tiny $H_{\al_2,-2}$}}

\rput(1.732,-3.8){{\tiny $H_{\al_2,2}$}}

\psline[linestyle=dashed](3.464,3)(4,3)
\rput(3.9,3.2){{\tiny $+$}}
\rput(3.9,2.8){{\tiny $-$}}
\psline[linestyle=dashed](3.464,1.5)(4,1.5)
\rput(3.9,1.7){{\tiny $+$}}
\rput(3.9,1.3){{\tiny $-$}}
\psline[linestyle=dashed](3.464,0)(4,0)
\rput(3.9,0.2){{\tiny $+$}}
\rput(3.9,-0.2){{\tiny $-$}}
\psline[linestyle=dashed](3.464,-1.5)(4,-1.5)
\rput(3.9,-1.3){{\tiny $+$}}
\rput(3.9,-1.7){{\tiny $-$}}
\psline[linestyle=dashed](3.464,-3)(4,-3)
\rput(3.9,-2.8){{\tiny $+$}}
\rput(3.9,-3.2){{\tiny $-$}}

\psline(3.464,3)(-3.464,3)
\psline(3.464,-3)(-3.464,-3)
\psline(3.464,3)(3.464,-3)
\psline(-3.464,3)(-3.464,-3)

\psline(3.464,1.5)(-3.464,1.5)
\psline(3.464,0)(-3.464,0)
\psline(3.464,-1.5)(-3.464,-1.5)

\psline(2.598,3)(2.598,-3)
\psline(1.732,3)(1.732,-3)
\psline(.866,3)(.866,-3)
\psline(0,3)(0,-3)
\psline(-2.598,3)(-2.598,-3)
\psline(-1.732,3)(-1.732,-3)
\psline(-.866,3)(-.866,-3)

\psline[linestyle=dashed](0,-3)(0,-3.5)
\rput(-0.2,-3.4){{\tiny $-$}}
\rput(0.2,-3.4){{\tiny $+$}}
\psline[linestyle=dashed](-1.732,-3)(-1.732,-3.5)
\rput(-1.932,-3.4){{\tiny $-$}}
\rput(-1.532,-3.4){{\tiny $+$}}
\psline[linestyle=dashed](1.732,-3)(1.732,-3.5)
\rput(1.532,-3.4){{\tiny $-$}}
\rput(1.932,-3.4){{\tiny $+$}}

\psline(1.732,3)(3.464,2)
\psline(0,3)(3.464,1)
\psline(-1.732,3)(3.464,0)
\psline(-3.464,3)(3.464,-1)
\psline(-3.464,2)(3.464,-2)
\psline(-3.464,1)(3.464,-3)
\psline(-3.464,0)(1.732,-3)
\psline(-3.464,-1)(0,-3)
\psline(-3.464,-2)(-1.732,-3)

\psline(-1.732,3)(-3.464,2)
\psline(0,3)(-3.464,1)
\psline(1.732,3)(-3.464,0)
\psline(3.464,3)(-3.464,-1)
\psline(3.464,2)(-3.464,-2)
\psline(3.464,1)(-3.464,-3)
\psline(3.464,0)(-1.732,-3)
\psline(3.464,-1)(0,-3)
\psline(3.464,-2)(1.732,-3)

\psline(-3.464,0)(-1.732,-3)
\psline(-3.464,3)(0,-3)
\psline(-1.732,3)(1.732,-3)
\psline(0,3)(3.464,-3)
\psline(1.732,3)(3.464,0)

\psline(3.464,0)(1.732,-3)
\psline(3.464,3)(0,-3)
\psline(1.732,3)(-1.732,-3)
\psline(0,3)(-3.464,-3)
\psline(-1.732,3)(-3.464,0)

\end{pspicture}
\end{center}
\caption{The root system of type $G_2$}
\label{rootsystem}
\end{figure}

\end{Exa}

\subsection{Alcove walks and the Bernstein-Lusztig basis of $\cH$}\label{sec:2.2}

Let $W$ be an affine Weyl group as in the previous section. Let $L$ be a weight function on $W$. The standard basis of $\cH$ is well adapted to the Coxeter structure of the affine Weyl group. We now describe another basis of $\cH$, due to Bernstein and Lusztig, that is well adapted to the semi-direct product structure of $W$. Our approach here follows Ram's alcove walk model~\cite{Ram:06}. 
\medskip

Let $\vec w=s_{i_1}s_{i_2}\cdots s_{i_{\ell}}$ be an expression for $w\in W$, and let~$v\in W$. A \textit{positively folded alcove walk of type~$\vec w$ starting at $v$} is a sequence $p=(v_0,v_1,\ldots,v_{\ell})$ with $v_0,\ldots,v_{\ell}\in W$ such that
\begin{enumerate}
\item $v_0=v$,
\item $v_k\in\{v_{k-1},v_{k-1}s_{i_k}\}$ for each $k=1,\ldots,\ell$, and
\item if $v_{k-1}=v_k$ then $v_{k-1}$ is on the positive side of the hyperplane separating $v_{k-1}$ and $v_{k-1}s_{i_k}$. 
\end{enumerate}
The \textit{end} of $p$ is $\mathrm{end}(p)=v_{\ell}$. Let 
$$
\mathcal{P}(\vec{w},v)=\{\text{all positively folded alcove walks of type $\vec{w}$ starting at $v$}\}.
$$

Less formally, a \textit{positively folded alcove walk of type~$\vec w$ starting at $v$} is a sequence of steps from alcove to alcove in $W$, starting at $v$, and made up of the symbols (where the $k$th step has $s=s_{i_k}$ for $k=1,\ldots,\ell$):

\begin{center}
\begin{pspicture}(-6,-1)(6,1)
\psset{unit=.5cm}
\psline(-8.5,-1)(-8.5,1)
\psline{->}(-9,0)(-8,0)
\rput(-9.5,1){{ $-$}}
\rput(-9.5,.2){{ $x$}}
\rput(-7.5,.2){{$xs$}}
\rput(-7.9,1){{ $+$}}
\rput(-8.5,-1.5){{ \text{(positive $s$-crossing)}}}
%\rput(-7.5,-2){{\tiny \text{1}}}

%
%
\psline(-2,-1)(-2,1)
\psline(-1.5,0)(-2,0)
\psline{<-}(-1.5,-.15)(-2,-.15)
\rput(-3,1){{ $-$}}
\rput(-3,.2){{ $xs$}}
\rput(-1,.2){{ $x$}}
\rput(-1.4,1){{ $+$}}
\rput(-2,-1.5){{ \text{(positive $s$-fold)}}}
%\rput(-2.5,-2){{\tiny \text{$Q_i$}}}

\psline(4.5,-1)(4.5,1)
\psline{->}(5,0)(4,0)
\rput(5,1){{ $+$}}
\rput(5.5,.2){{$x$}}
\rput(3.5,.2){{ $xs$}}
\rput(3.6,1){{ $-$}}
\rput(4.5,-1.5){{ \text{(negative $s$-crossing)}}}
%\rput(7.5,-2){{\tiny \text{1}}}
\end{pspicture}
\end{center}

If $p$ has no folds we say that $p$ is \textit{straight}.

\medskip

If $p$ is a positively folded alcove walk we define, for each $s\in S$, 
\begin{align*}
f_s(p)&=\#\textrm{(positive $s$-folds in $p$)}\quad\text{and}\quad \mathcal{Q}(p)=\prod_{s\in S}(\sq_{s}-\sq_{s}^{-1})^{f_s(p)}.
\end{align*}

Let $v\in W$ and choose any expression $v=s_{i_1}\cdots s_{i_{\ell}}$ (not necessarily reduced). Consider the associated straight alcove walk $(v_0,v_1\ldots,v_{\ell})$, where $v_0=1$ and $v_k=s_{i_1}\cdots s_{i_k}$. Let $\eps_1,\ldots,\eps_{\ell}$ be defined using the periodic orientation on hyperplanes as follows:
$$
\eps_k=\begin{cases}
+1&\text{if $v_{k-1}\,\,{^-}\hspace{-0.1cm}\mid^+\,\, v_k$\hspace{0.37cm} (that is, a positive crossing)}\\
-1&\text{if $v_{k}\,\,\hspace{0.34cm}{^-}\hspace{-0.1cm}\mid^+\,\,  v_{k-1}$ (that is, a negative crossing).} 
\end{cases}
$$
It turns out that the element 
$$X_{v}=T^{\eps_{1}}_{s_{i_{1}}}\ldots T^{\eps_\ell}_{s_{i_{\ell}}}$$
of $\cH$ does not depend on the particular expression $v=s_{i_{1}}\cdots s_{i_\ell}$ we have chosen (see \cite{Goe:07}). If $\lambda\in Q$ we write 
$$
X^{\lambda}=X_{t_{\lambda}}.
$$
It follows from the above definitions that
$$
X_v=X_{t_{\mathrm{wt}(v)}\theta(v)}=X^{\mathrm{wt}(p)}X_{\theta(v)}=X^{\mathrm{wt}(v)}T_{\theta(v)^{-1}}^{-1}
$$
(the second equality follows since $t_{\mathrm{wt}(v)}$ is on the positive side of every hyperplane through $\mathrm{wt}(v)$, and the third equality follows since $X_u=T_{u^{-1}}^{-1}$ for all $u\in W_0$). Moreover since $X_v=T_v+\text{(lower terms)}$ the set $\{X_v\mid v\in W\}$ is a basis of $\cH$, called the \textit{Bernstein-Lusztig basis}. 

\medskip

Let $\sR[Q]$ be the free $\sR$-module with basis $\{X^\la\mid \la\in Q\}$. We have a natural action of $W_0$ given by $w X^\la =X^{w\la}$. We set 
$$\sR[Q]^{W_0}=\{p\in \sR[Q]\mid w\cdot p=p \text{ for all $w\in W_0$}\}.$$
It is a well-known result that the centre of $\cH$ is $\cZ(\cH)=\sR[Q]^{W_0}$.

\medskip

The combinatorics of positively folded alcove walks encodes the change of basis from the standard basis $(T_w)_{w\in W}$ of $\cH$ to the Bernstein-Lusztig basis $(X_v)_{v\in W}$. This is seen by taking $u=e$ in the following proposition.

\begin{Prop}\label{prop:basischange}(c.f. \cite[Theorem~3.3]{Ram:06}) Let $w,u\in W$, and let $\vec{w}$ be any reduced expression for $w$. Then
\begin{align*}
X_uT_w&=\sum_{p\in\mathcal{P}(\vec{w},u)}\mathcal{Q}(p)X_{\mathrm{end}(p)}.
\end{align*}
\end{Prop}

\begin{proof}
Suppose that $\ell(ws)=\ell(w)+1$. Then 
\begin{align*}
X_uT_{ws}&=X_uT_wT_s=\sum_{p\in\mathcal{P}(\vec{w},u)}\mathcal{Q}(p)X_{\mathrm{end}(p)}T_s.
\end{align*}
Now, using the formula $T_s=T_s^{-1}+(\sq_s-\sq_s^{-1})$ in the second case below, we have
$$
X_{\mathrm{end}(p)}T_s=\begin{cases}
X_{\mathrm{end}(p\cdot \epsilon_s^+)}&\text{if $\mathrm{end}(p)\hphantom{s}\,\,{^-}\hspace{-0.1cm}\mid^+\,\, \mathrm{end}(p)s$}\\
X_{\mathrm{end}(p\cdot \epsilon_s^-)}+(\sq_s-\sq_s^{-1})X_{\mathrm{end}(p\cdot f_s^+)}&\text{if $\mathrm{end}(p)s\,\,{^-}\hspace{-0.1cm}\mid^+\,\, \mathrm{end}(p)$}
\end{cases}
$$
where $p\cdot \epsilon_s^+$, $p\cdot \epsilon_s^-$, and $p\cdot f_s^+$ denote, respectively, the path $p$ followed by a positive $s$-crossing, a negative $s$-crossing, and a positive $s$-fold. The result follows by induction.
\end{proof}

\begin{Exa}
\label{presentationG2}
Let $(W,S)$ be the affine Weyl group of type $\tG_{2}$ with diagram and weight function as in Example~\ref{exa:balanced-one-dim}. Write $\sq_1=\sq^{L(s_1)}$ and $\sq_2=\sq^{L(s_2)}=\sq^{L(s_0)}$. The coroot system $\Phi^\vee$ is as in Example \ref{rootG2}. Writing $X_1=X^{\al_1^\vee}$ and $X_2=X^{\al_2^\vee}$, the Hecke algebra $\cH$ asscociated to $W$ has generators $T_1=T_{s_1}$, $T_2=T_{s_2}$, $X_1$ and $X_2$ with relations
\begin{align*}
T_1^2&=1+(\sq_1-\sq_1^{-1})T_1& T_1X_1&=X_1^{-1}T_1+(\sq_1-\sq_1^{-1})(1+X_1)\\
T_2^2&=1+(\sq_2-\sq_2^{-1})T_2 & T_2X_2&=X_2^{-1}T_2+(\sq_2-\sq_2^{-1})(1+X_2)\\
(T_1T_2)^3&=(T_2T_1)^3 & T_2X_1&=X_1X_2^3T_2^{-1}-(\sq_2-\sq_2^{-1})X_1X_2(1+X_2)\\
 X_1X_2&=X_2X_1 &
 T_1X_2&=X_1X_2T_1^{-1}.
\end{align*}
\end{Exa}

\subsection{A formula for the Weyl character}\label{sec:2.3}

In this subsection we use the Hecke algebra as a tool to establish a combinatorial formula for the Weyl character $s_{\lambda}(X)$. It is sufficient for this purpose to consider the Hecke algebra $\mathcal{H}$ with weight function $L=\ell$ (that is, the equal parameter case). Let
$$
\mathbf{1}_0=\sum_{w\in W_0}\sq^{\ell(w)}T_w.
$$
We have $T_w\mathbf{1}_0=\mathbf{1}_0T_w=\sq^{\ell(w)}\mathbf{1}_0$ for all $w\in W_0$. For dominant $\lambda$, the \textit{Macdonald spherical function} is the unique element $P_{\lambda}(X,\sq^{-1})$ of $\sR[Q]$ such that 
$$
P_{\lambda}(X,\sq^{-1})\mathbf{1}_0=\sq^{-2\ell(\sw_0)}\mathbf{1}_0X^{\lambda}\mathbf{1}_0.
$$
The well known explicit formula for $P_{\lambda}(X,\sq^{-1})$, due to Macdonald (see \cite{Mac:71}, and also \cite{Ram:03} for a proof in the Hecke algebra context) is 
$$
P_{\lambda}(X,\sq^{-1})=\sum_{w\in W_0}w\bigg(X^{\lambda}\prod_{\alpha\in\Phi^+}\frac{1-\sq^{-2}X^{-\alpha^{\vee}}}{1-X^{-\alpha^{\vee}}}\bigg),
$$
from which we see that $P_{\lambda}(X,\sq^{-1})\in\sR[Q]^{W_0}$ and that on specialising $\sq^{-1}=0$ we have $P_{\lambda}(X,0)=\ss_{\lambda}(X)$.
\medskip

Let $w,u\in W$ and let $\vec{w}$ be any reduced expression for $w$. Let
\begin{align}\label{eq:PP}
\mathbb{P}(\vec{w},u)=\{p\in\mathcal{P}(\vec{w},u)\mid f(p)=\ell(\sw_0)\}.
\end{align}
The following theorem is well known, however we sketch the proof for completeness.

\begin{Th}\label{thm:Schur}
If $\lambda\in Q\cap P^+$ then
$$
\ss_{\lambda}(X)=\sum_{p\in\mathbb{P}(\vec{\sw}_0\cdot \vec{t}_{\lambda},e)}X^{\mathrm{wt}(p)},
$$
\end{Th}

\begin{proof}
Let $\cH$ be the Hecke algebra with $L=\ell$. Since $T_uT_{t_{\lambda}}=T_{ut_{\lambda}}$ for all $u\in W_0$ we have, by Proposition~\ref{prop:basischange},
\begin{align*}
P_{\lambda}(X,\sq^{-1})\mathbf{1}_0&=\sq^{-2\ell(\sw_0)}\sum_{u\in W_0}\sq^{\ell(u)}T_uT_{t_{\lambda}}\mathbf{1}_0=\sq^{-2\ell(\sw_0)}\sum_{u\in W_0}\sum_{p\in\mathcal{P}(\vec{u}\cdot\vec{t}_{\lambda},1)}\sq^{\ell(u)}(\sq-\sq^{-1})^{f(p)}X_{\mathrm{end}(p)}\mathbf{1}_0.
\end{align*}
Since $X_{\mathrm{end}(p)}\mathbf{1}_0=X^{\mathrm{wt}(p)}T_{\theta(p)^{-1}}^{-1}\mathbf{1}_0=\sq^{-\ell(\theta(p))}X^{\mathrm{wt}(p)}\mathbf{1}_0$ it follows that
\begin{align*}
P_{\lambda}(X,\sq^{-1})&=\sum_{u\in W_0}\sum_{p\in\mathcal{P}(\vec{u}\cdot\vec{t}_{\lambda},1)}(\sq^{-1})^{2\ell(\sw_0)-\ell(u)-f(p)+\ell(\theta(p))}(1-\sq^{-2})^{f(p)}X^{\mathrm{wt}(p)}.
\end{align*}
For each positively folded alcove walk $p\in\mathcal{P}(\vec{u}\cdot\vec{t}_{\lambda},e)$, let $p_0,\ldots,p_{f(p)}$ be the \textit{partial folding sequence} of $p$, where $p_j$ is the positively folded alcove walk that agrees with $p$ up to (and including) the $j$th fold of $p$, and is straight thereafter. It is simple to see (either using the technique of Lemma~\ref{lem:tech1} in this paper, or see \cite{Ram:06}) that $\ell(\theta(p_{j+1}))<\ell(\theta(p_j))$ for all $j=0,\ldots,f(p)-1$. Thus $\ell(\theta(p_j))-\ell(\theta(p_{j+1}))-1\geq 0$, and it follows by summing that 
$
\ell(\theta(p_0))-\ell(\theta(p_{f(p)}))-f(p)\geq 0,
$
and hence 
$$
f(p)\leq \ell(u)-\ell(\theta(p)),
$$
with equality if and only if $\ell(\theta(p_j))-\ell(\theta(p_{j+1}))-1= 0$ for each $j=0,\ldots,f(p)-1$. Thus the exponent of $\sq^{-1}$ in the above formula for $P_{\lambda}(X,\sq^{-1})$ is
\begin{align*}
2\ell(\sw_0)-\ell(u)-f(p)+\ell(\theta(p))&\geq 2(\ell(\sw_0)-\ell(u)+\ell(\theta(p)))=2(\ell(\sw_0u^{-1})+\ell(\theta(p)))\geq 0,
\end{align*}
with equality if and only if $f(p)=\ell(u)-\ell(\theta(p))$, $\ell(\sw_0u^{-1})=0$, and $\ell(\theta(p))=0$. Thus equality occurs if and only if $u=\sw_0$, $\theta(p)=e$, and $f(p)=\ell(\sw_0)$. Therefore, upon specialising at $\sq^{-1}=0$ only the terms with $u=\sw_0$ and $f(p)=\ell(\sw_0)$ survive, hence the result.
\end{proof}

%%%%%%%%%%%%%%%%%%%%%%%%%%%%%%%%%%%
%%%%%%%%%%%%%%%%%%%%%%%%%%%%%%%%%%%
%%%%%%%%%%%%%%%%%%%%%%%%%%%%%%%%%%%
%%%%%%%%%%%%%%%%%%%%%%%%%%%%%%%%%%%
%%%%%%%%%%%%%%%%%%%%%%%%%%%%%%%%%%%
%%%%%%%%%%%%%%%%%%%%%%%%%%%%%%%%%%%
%%%%%%%%%%%%%%%%%%%%%%%%%%%%%%%%%%%
%%%%%%%%%%%%%%%%%%%%%%%%%%%%%%%%%%%
%%%%%%%%%%%%%%%%%%%%%%%%%%%%%%%%%%%

\section{Kazhdan-Lusztig cells in type $\tilde{G}_2$}\label{sec:KL-cells-G2}

In this section we recall the decomposition of $\tilde{G}_2$ into right cells and two-sided cells for all choices of parameters $(a,b)\in\mathbb{N}^2$ from~\cite{guilhot4}.  We also recall some ``cell factorisation'' properties for the infinite two-sided cells from~\cite{GM:12}.

\subsection{The partition of $\tilde{G}_2$ into cells}\label{sec:cells}

Let $W$ be an affine Weyl group of type $\tilde{G}_2$ with diagram and weight function $L(s_1)=a$ and $L(s_2)=L(s_0)=b$ as in Example~\ref{exa:balanced-one-dim}. The partition of $W$ into two-sided cells depends only on the ratio $r=a/b$ of the parameters, and it turns out that there are precisely $7$ distinct regimes. We recall these decompositions in the diagrams below where
\bem
\item $w$ and $w'$ are in the same two-sided cell if and only if they have the same colour;
\item $w$ and $w'$ are in the same right cell if and only if they have the same colour and lie in a common connected~component;
\item the graphs represent the two-sided order on two-sided cells for all regimes from $r>2$ on the left to $r<1$ on the right.
\eem

\psset{unit=.45cm}
\begin{figure}[H]
\begin{subfigure}{.33\textwidth}

\begin{center}
\begin{pspicture}(-4.33,-4)(4.33,6.5)

\psset{linewidth=.05mm}

% c_0 lowest two-sided cell

\pspolygon[fillstyle=solid,fillcolor=red!5!yellow!42!](0.866,4.5)(0.866,6)(1.732,6)
\pspolygon[fillstyle=solid,fillcolor=red!5!yellow!42!](0.866,1.5)(3.46,6)(4.33,6)(4.33,3.5)
\pspolygon[fillstyle=solid,fillcolor=red!5!yellow!42!](2.598,1.5)(4.33,1.5)(4.33,2.5)
\pspolygon[fillstyle=solid,fillcolor=red!5!yellow!42!](1.732,0)(4.33,0)(4.33,-1.5)
\pspolygon[fillstyle=solid,fillcolor=red!5!yellow!42!](2.598,-1.5)(4.33,-2.5)(4.33,-3)(3.464,-3)
\pspolygon[fillstyle=solid,fillcolor=red!5!yellow!42!](0.866,-1.5)(0.866,-3)(1.732,-3)
\pspolygon[fillstyle=solid,fillcolor=red!5!yellow!42!](0,0)(0,-3)(-1.732,-3)
\pspolygon[fillstyle=solid,fillcolor=red!5!yellow!42!](-2.598,-1.5)(-3.464,-3)(-4.33,-3)(-4.33,-2.5)
\pspolygon[fillstyle=solid,fillcolor=red!5!yellow!42!](-1.732,0)(-4.33,0)(-4.33,-1.5)
\pspolygon[fillstyle=solid,fillcolor=red!5!yellow!42!](-2.598,1.5)(-4.33,1.5)(-4.33,2.5)
\pspolygon[fillstyle=solid,fillcolor=red!5!yellow!42!](-0.866,1.5)(-3.46,6)(-4.33,6)(-4.33,3.5)
\pspolygon[fillstyle=solid,fillcolor=red!5!yellow!42!](0,3)(0,6)(-1.732,6)

% c_1
\pspolygon[fillstyle=solid,fillcolor=blue!70!black!80!](0,0)(0,-3)(0.866,-3)(0.866,-1.5)
\pspolygon[fillstyle=solid,fillcolor=blue!70!black!80!](1.732,0)(4.33,-1.5)(4.33,-2.5)(2.598,-1.5)
\pspolygon[fillstyle=solid,fillcolor=blue!70!black!80!](0.866,1.5)(2.598,1.5)(4.33,2.5)(4.33,3.5)
\pspolygon[fillstyle=solid,fillcolor=blue!70!black!80!](0,3)(0,6)(0.866,6)(0.866,4.5)
\pspolygon[fillstyle=solid,fillcolor=blue!70!black!80!](-0.866,1.5)(-2.598,1.5)(-4.33,2.5)(-4.33,3.5)
\pspolygon[fillstyle=solid,fillcolor=blue!70!black!80!](-1.732,0)(-4.33,-1.5)(-4.33,-2.5)(-2.598,-1.5)

%c_2
\pspolygon[fillstyle=solid,fillcolor=green!80!black!70!](0.433,2.25)(0,3)(1.732,6)(3.46,6)(1.732,3)
\pspolygon[fillstyle=solid,fillcolor=green!80!black!70!](0,1.5)(0,3)(-1.732,6)(-3.46,6)(-0.866,1.5)
\pspolygon[fillstyle=solid,fillcolor=green!80!black!70!](-0.433,0.75)(-0.866,1.5)(-4.33,1.5)(-4.33,0)(-1.732,0)
\pspolygon[fillstyle=solid,fillcolor=green!80!black!70!](-0.866,0)(-0.866,-1.5)(-1.732,-3)(-3.464,-3)(-1.732,0)
\pspolygon[fillstyle=solid,fillcolor=green!80!black!70!](0.866,0)(0.866,-1.5)(1.732,-3)(3.464,-3)(1.732,0)
\pspolygon[fillstyle=solid,fillcolor=green!80!black!70!](0.433,0.75)(0.866,1.5)(4.33,1.5)(4.33,0)(1.732,0)

\pspolygon[fillstyle=solid,fillcolor=yellow!10!red!80!](0,0)(-0.433,0.75)(-1.732,0)(-0.866,0)(-0.866,-1.5)
\pspolygon[fillstyle=solid,fillcolor=yellow!10!red!80!](0,0)(0.433,0.75)(1.732,0)(0.866,0)(0.866,-1.5)
\pspolygon[fillstyle=solid,fillcolor=yellow!10!red!80!](0,1.5)(0,3)(0.433,2.25)(1.732,3)(0.866,1.5)

%%orange finite
\pspolygon[fillstyle=solid,fillcolor=orange](0,1)(0,1.5)(0.866,1.5)(0.433,0.75)
\pspolygon[fillstyle=solid,fillcolor=orange](0,1)(-.866,1.5)(0,0)

\pspolygon[fillstyle=solid,fillcolor=ForestGreen](0,1)(-.866,1.5)(0,1.5)

%A_0
\pspolygon[fillstyle=solid, fillcolor=black](0,0)(0,1)(0.433,0.75)

%%%%%%%%%%%%
%%%%%%%%%%%%
%%%%%%%%%%%%
%%%%%%%%%%%%
%%%HYPERPLANS TYPE G2
%%%%%%%%%%%%
%%%%%%%%%%%%
%%%%%%%%%%%%
%%%%%%%%%%%%
%%%%%%%%%%%%

% droite horizontale 
\psline(-4.33,6)(4.33,6)
\psline(-4.33,4.5)(4.33,4.5)
\psline(-4.33,3)(4.33,3)
\psline(-4.33,1.5)(4.33,1.5)
\psline(-4.33,0)(4.33,0)
\psline(-4.33,-1.5)(4.33,-1.5)
\psline(-4.33,-3)(4.33,-3)

% droite verticale
\psline(-4.33,-3)(-4.33,6)
\psline(-3.464,-3)(-3.464,6)
\psline(-2.598,-3)(-2.598,6)
\psline(-1.732,-3)(-1.732,6)
\psline(-.866,-3)(-.866,6)
\psline(0,-3)(0,6)
\psline(.866,-3)(.866,6)
\psline(1.732,-3)(1.732,6)
\psline(2.598,-3)(2.598,6)
\psline(3.464,-3)(3.464,6)
\psline(4.33,-3)(4.33,6)

%droite oblique
\psline(-4.33,5.5)(-3.464,6)
\psline(-4.33,4.5)(-1.732,6)
\psline(-4.33,3.5)(0,6)
\psline(-4.33,2.5)(1.732,6)
\psline(-4.33,1.5)(3.464,6)
\psline(-4.33,.5)(4.33,5.5)
\psline(-4.33,-.5)(4.33,4.5)
\psline(-4.33,-1.5)(4.33,3.5)
\psline(-4.33,-2.5)(4.33,2.5)
\psline(-3.464,-3)(4.33,1.5)
\psline(-1.732,-3)(4.33,.5)
\psline(0,-3)(4.33,-.5)
\psline(1.732,-3)(4.33,-1.5)
\psline(3.464,-3)(4.33,-2.5)

%droite oblique
\psline(4.33,5.5)(3.464,6)
\psline(4.33,4.5)(1.732,6)
\psline(4.33,3.5)(0,6)
\psline(4.33,2.5)(-1.732,6)
\psline(4.33,1.5)(-3.464,6)
\psline(4.33,.5)(-4.33,5.5)
\psline(4.33,-.5)(-4.33,4.5)
\psline(4.33,-1.5)(-4.33,3.5)
\psline(4.33,-2.5)(-4.33,2.5)
\psline(3.464,-3)(-4.33,1.5)
\psline(1.732,-3)(-4.33,.5)
\psline(0,-3)(-4.33,-.5)
\psline(-1.732,-3)(-4.33,-1.5)
\psline(-3.464,-3)(-4.33,-2.5)

%droite oblique
\psline(-4.33,-1.5)(-3.464,-3)
\psline(-4.33,1.5)(-1.732,-3)
\psline(-4.33,4.5)(0,-3)
\psline(-3.464,6)(1.732,-3)
\psline(-1.732,6)(3.464,-3)
\psline(0,6)(4.33,-1.5)
\psline(1.732,6)(4.33,1.5)
\psline(3.464,6)(4.33,4.5)

%droite oblique
\psline(4.33,-1.5)(3.464,-3)
\psline(4.33,1.5)(1.732,-3)
\psline(4.33,4.5)(0,-3)
\psline(3.464,6)(-1.732,-3)
\psline(1.732,6)(-3.464,-3)
\psline(0,6)(-4.33,-1.5)
\psline(-1.732,6)(-4.33,1.5)
\psline(-3.464,6)(-4.33,4.5)

%\pspolygon[fillstyle=solid, fillcolor=white,linestyle=none](-14.43,-4.6)(14.43,-4.6)(14.43,-3)(-14.43,-3)
\rput(0,-3.5){{\footnotesize $r>2$}}
\end{pspicture}

\end{center}
\end{subfigure}
\begin{subfigure}{.33\textwidth}

%%%%%%%%%%%%
%%%%%%%%%%%%
%
%	r=2
%
%%%%%%%%%%%%
%%%%%%%%%%%%

\begin{center}
\begin{pspicture}(-4.33,-4)(4.33,6.5)

\psset{linewidth=.05mm}

% c_0 lowest two-sided cell

\pspolygon[fillstyle=solid,fillcolor=red!5!yellow!42!](0.866,4.5)(0.866,6)(1.732,6)
\pspolygon[fillstyle=solid,fillcolor=red!5!yellow!42!](0.866,1.5)(3.46,6)(4.33,6)(4.33,3.5)
\pspolygon[fillstyle=solid,fillcolor=red!5!yellow!42!](2.598,1.5)(4.33,1.5)(4.33,2.5)
\pspolygon[fillstyle=solid,fillcolor=red!5!yellow!42!](1.732,0)(4.33,0)(4.33,-1.5)
\pspolygon[fillstyle=solid,fillcolor=red!5!yellow!42!](2.598,-1.5)(4.33,-2.5)(4.33,-3)(3.464,-3)
\pspolygon[fillstyle=solid,fillcolor=red!5!yellow!42!](0.866,-1.5)(0.866,-3)(1.732,-3)
\pspolygon[fillstyle=solid,fillcolor=red!5!yellow!42!](0,0)(0,-3)(-1.732,-3)
\pspolygon[fillstyle=solid,fillcolor=red!5!yellow!42!](-2.598,-1.5)(-3.464,-3)(-4.33,-3)(-4.33,-2.5)
\pspolygon[fillstyle=solid,fillcolor=red!5!yellow!42!](-1.732,0)(-4.33,0)(-4.33,-1.5)
\pspolygon[fillstyle=solid,fillcolor=red!5!yellow!42!](-2.598,1.5)(-4.33,1.5)(-4.33,2.5)
\pspolygon[fillstyle=solid,fillcolor=red!5!yellow!42!](-0.866,1.5)(-3.46,6)(-4.33,6)(-4.33,3.5)
\pspolygon[fillstyle=solid,fillcolor=red!5!yellow!42!](0,3)(0,6)(-1.732,6)

% c_1
\pspolygon[fillstyle=solid,fillcolor=blue!70!black!80!](0,0)(0,-3)(0.866,-3)(0.866,-1.5)
\pspolygon[fillstyle=solid,fillcolor=blue!70!black!80!](1.732,0)(4.33,-1.5)(4.33,-2.5)(2.598,-1.5)
\pspolygon[fillstyle=solid,fillcolor=blue!70!black!80!](0.866,1.5)(2.598,1.5)(4.33,2.5)(4.33,3.5)
\pspolygon[fillstyle=solid,fillcolor=blue!70!black!80!](0,3)(0,6)(0.866,6)(0.866,4.5)
\pspolygon[fillstyle=solid,fillcolor=blue!70!black!80!](-0.866,1.5)(-2.598,1.5)(-4.33,2.5)(-4.33,3.5)
\pspolygon[fillstyle=solid,fillcolor=blue!70!black!80!](-1.732,0)(-4.33,-1.5)(-4.33,-2.5)(-2.598,-1.5)

%c_2
\pspolygon[fillstyle=solid,fillcolor=green!80!black!70!](0.433,2.25)(0,3)(1.732,6)(3.46,6)(1.732,3)
\pspolygon[fillstyle=solid,fillcolor=green!80!black!70!](0,1)(0,3)(-1.732,6)(-3.46,6)(-0.866,1.5)(0,1)
\pspolygon[fillstyle=solid,fillcolor=green!80!black!70!](-0.433,0.75)(-0.866,1.5)(-4.33,1.5)(-4.33,0)(-1.732,0)
\pspolygon[fillstyle=solid,fillcolor=green!80!black!70!](-0.866,0)(-0.866,-1.5)(-1.732,-3)(-3.464,-3)(-1.732,0)
\pspolygon[fillstyle=solid,fillcolor=green!80!black!70!](0.866,0)(0.866,-1.5)(1.732,-3)(3.464,-3)(1.732,0)
\pspolygon[fillstyle=solid,fillcolor=green!80!black!70!](0.433,0.75)(0.866,1.5)(4.33,1.5)(4.33,0)(1.732,0)

\pspolygon[fillstyle=solid,fillcolor=yellow!10!red!80!](0,0)(-0.433,0.75)(-1.732,0)(-0.866,0)(-0.866,-1.5)
\pspolygon[fillstyle=solid,fillcolor=yellow!10!red!80!](0,0)(0.433,0.75)(1.732,0)(0.866,0)(0.866,-1.5)
\pspolygon[fillstyle=solid,fillcolor=yellow!10!red!80!](0,1.5)(0,3)(0.433,2.25)(1.732,3)(0.866,1.5)

%%orange finite
\pspolygon[fillstyle=solid,fillcolor=orange](0,1)(0,1.5)(0.866,1.5)(0.433,0.75)
\pspolygon[fillstyle=solid,fillcolor=orange](0,1)(-.866,1.5)(0,0)

%\pspolygon[fillstyle=solid,fillcolor=ForestGreen](0,1)(-.866,1.5)(0,1.5)

%A_0
\pspolygon[fillstyle=solid, fillcolor=black](0,0)(0,1)(0.433,0.75)

%%%%%%%%%%%%
%%%%%%%%%%%%
%%%%%%%%%%%%
%%%%%%%%%%%%
%%%HYPERPLANS TYPE G2
%%%%%%%%%%%%
%%%%%%%%%%%%
%%%%%%%%%%%%
%%%%%%%%%%%%
%%%%%%%%%%%%

% droite horizontale 
\psline(-4.33,6)(4.33,6)
\psline(-4.33,4.5)(4.33,4.5)
\psline(-4.33,3)(4.33,3)
\psline(-4.33,1.5)(4.33,1.5)
\psline(-4.33,0)(4.33,0)
\psline(-4.33,-1.5)(4.33,-1.5)
\psline(-4.33,-3)(4.33,-3)

% droite verticale
\psline(-4.33,-3)(-4.33,6)
\psline(-3.464,-3)(-3.464,6)
\psline(-2.598,-3)(-2.598,6)
\psline(-1.732,-3)(-1.732,6)
\psline(-.866,-3)(-.866,6)
\psline(0,-3)(0,6)
\psline(.866,-3)(.866,6)
\psline(1.732,-3)(1.732,6)
\psline(2.598,-3)(2.598,6)
\psline(3.464,-3)(3.464,6)
\psline(4.33,-3)(4.33,6)

%droite oblique
\psline(-4.33,5.5)(-3.464,6)
\psline(-4.33,4.5)(-1.732,6)
\psline(-4.33,3.5)(0,6)
\psline(-4.33,2.5)(1.732,6)
\psline(-4.33,1.5)(3.464,6)
\psline(-4.33,.5)(4.33,5.5)
\psline(-4.33,-.5)(4.33,4.5)
\psline(-4.33,-1.5)(4.33,3.5)
\psline(-4.33,-2.5)(4.33,2.5)
\psline(-3.464,-3)(4.33,1.5)
\psline(-1.732,-3)(4.33,.5)
\psline(0,-3)(4.33,-.5)
\psline(1.732,-3)(4.33,-1.5)
\psline(3.464,-3)(4.33,-2.5)

%droite oblique
\psline(4.33,5.5)(3.464,6)
\psline(4.33,4.5)(1.732,6)
\psline(4.33,3.5)(0,6)
\psline(4.33,2.5)(-1.732,6)
\psline(4.33,1.5)(-3.464,6)
\psline(4.33,.5)(-4.33,5.5)
\psline(4.33,-.5)(-4.33,4.5)
\psline(4.33,-1.5)(-4.33,3.5)
\psline(4.33,-2.5)(-4.33,2.5)
\psline(3.464,-3)(-4.33,1.5)
\psline(1.732,-3)(-4.33,.5)
\psline(0,-3)(-4.33,-.5)
\psline(-1.732,-3)(-4.33,-1.5)
\psline(-3.464,-3)(-4.33,-2.5)

%droite oblique
\psline(-4.33,-1.5)(-3.464,-3)
\psline(-4.33,1.5)(-1.732,-3)
\psline(-4.33,4.5)(0,-3)
\psline(-3.464,6)(1.732,-3)
\psline(-1.732,6)(3.464,-3)
\psline(0,6)(4.33,-1.5)
\psline(1.732,6)(4.33,1.5)
\psline(3.464,6)(4.33,4.5)

%droite oblique
\psline(4.33,-1.5)(3.464,-3)
\psline(4.33,1.5)(1.732,-3)
\psline(4.33,4.5)(0,-3)
\psline(3.464,6)(-1.732,-3)
\psline(1.732,6)(-3.464,-3)
\psline(0,6)(-4.33,-1.5)
\psline(-1.732,6)(-4.33,1.5)
\psline(-3.464,6)(-4.33,4.5)

%\pspolygon[fillstyle=solid, fillcolor=white,linestyle=none](-14.43,-4.6)(14.43,-4.6)(14.43,-3)(-14.43,-3)
\rput(0,-3.5){{\footnotesize $r=2$}}

\end{pspicture}

\end{center}

\end{subfigure}
\begin{subfigure}{.33\textwidth}

%%%%%%%%%%%%
%%%%%%%%%%%%
%
%	2>r>3/2
%
%%%%%%%%%%%%
%%%%%%%%%%%%

\begin{center}
\begin{pspicture}(-4.33,-4)(4.33,6.5)

\psset{linewidth=.05mm}

% c_0 lowest two-sided cell

\pspolygon[fillstyle=solid,fillcolor=red!5!yellow!42!](0.866,4.5)(0.866,6)(1.732,6)
\pspolygon[fillstyle=solid,fillcolor=red!5!yellow!42!](0.866,1.5)(3.46,6)(4.33,6)(4.33,3.5)
\pspolygon[fillstyle=solid,fillcolor=red!5!yellow!42!](2.598,1.5)(4.33,1.5)(4.33,2.5)
\pspolygon[fillstyle=solid,fillcolor=red!5!yellow!42!](1.732,0)(4.33,0)(4.33,-1.5)
\pspolygon[fillstyle=solid,fillcolor=red!5!yellow!42!](2.598,-1.5)(4.33,-2.5)(4.33,-3)(3.464,-3)
\pspolygon[fillstyle=solid,fillcolor=red!5!yellow!42!](0.866,-1.5)(0.866,-3)(1.732,-3)
\pspolygon[fillstyle=solid,fillcolor=red!5!yellow!42!](0,0)(0,-3)(-1.732,-3)
\pspolygon[fillstyle=solid,fillcolor=red!5!yellow!42!](-2.598,-1.5)(-3.464,-3)(-4.33,-3)(-4.33,-2.5)
\pspolygon[fillstyle=solid,fillcolor=red!5!yellow!42!](-1.732,0)(-4.33,0)(-4.33,-1.5)
\pspolygon[fillstyle=solid,fillcolor=red!5!yellow!42!](-2.598,1.5)(-4.33,1.5)(-4.33,2.5)
\pspolygon[fillstyle=solid,fillcolor=red!5!yellow!42!](-0.866,1.5)(-3.46,6)(-4.33,6)(-4.33,3.5)
\pspolygon[fillstyle=solid,fillcolor=red!5!yellow!42!](0,3)(0,6)(-1.732,6)

% c_1
\pspolygon[fillstyle=solid,fillcolor=blue!70!black!80!](0,0)(0,-3)(0.866,-3)(0.866,-1.5)
\pspolygon[fillstyle=solid,fillcolor=blue!70!black!80!](1.732,0)(4.33,-1.5)(4.33,-2.5)(2.598,-1.5)
\pspolygon[fillstyle=solid,fillcolor=blue!70!black!80!](0.866,1.5)(2.598,1.5)(4.33,2.5)(4.33,3.5)
\pspolygon[fillstyle=solid,fillcolor=blue!70!black!80!](0,3)(0,6)(0.866,6)(0.866,4.5)
\pspolygon[fillstyle=solid,fillcolor=blue!70!black!80!](-0.866,1.5)(-2.598,1.5)(-4.33,2.5)(-4.33,3.5)
\pspolygon[fillstyle=solid,fillcolor=blue!70!black!80!](-1.732,0)(-4.33,-1.5)(-4.33,-2.5)(-2.598,-1.5)

%c_2
\pspolygon[fillstyle=solid,fillcolor=green!80!black!70!](0.866,2.5)(.866,4.5)(1.732,6)(3.464,6)(1.732,3)
\pspolygon[fillstyle=solid,fillcolor=green!80!black!70!](0,1)(0,3)(-1.732,6)(-3.46,6)(-0.866,1.5)(0,1)
\pspolygon[fillstyle=solid,fillcolor=green!80!black!70!](-0.866,0.5)(-2.598,1.5)(-4.33,1.5)(-4.33,0)(-1.732,0)
\pspolygon[fillstyle=solid,fillcolor=green!80!black!70!](0.866,0.5)(2.598,1.5)(4.33,1.5)(4.33,0)(1.732,0)
\pspolygon[fillstyle=solid,fillcolor=green!80!black!70!](-.866,-.5)(-.866,-1.5)(-1.732,-3)(-3.464,-3)(-2.598,-1.5)
\pspolygon[fillstyle=solid,fillcolor=green!80!black!70!](.866,-.5)(.866,-1.5)(1.732,-3)(3.464,-3)(2.598,-1.5)

%turquoise

\pspolygon[fillstyle=solid,fillcolor=Turquoise](0,3)(.866,4.5)(.866,2.5)(.433,2.25)
\pspolygon[fillstyle=solid,fillcolor=Turquoise](0.433,.75)(.866,1.5)(2.598,1.5)(.866,.5)
\pspolygon[fillstyle=solid,fillcolor=Turquoise](-0.433,.75)(-.866,1.5)(-2.598,1.5)(-.866,.5)
\pspolygon[fillstyle=solid,fillcolor=Turquoise](.866,0)(.866,-.5)(2.598,-1.5)(1.732,0)
\pspolygon[fillstyle=solid,fillcolor=Turquoise](-.866,0)(-.866,-.5)(-2.598,-1.5)(-1.732,0)

\pspolygon[fillstyle=solid,fillcolor=yellow!10!red!80!](0,0)(-0.433,0.75)(-1.732,0)(-0.866,0)(-0.866,-1.5)
\pspolygon[fillstyle=solid,fillcolor=yellow!10!red!80!](0,0)(0.433,0.75)(1.732,0)(0.866,0)(0.866,-1.5)
\pspolygon[fillstyle=solid,fillcolor=yellow!10!red!80!](0,1.5)(0,3)(0.433,2.25)(1.732,3)(0.866,1.5)

%%orange finite
\pspolygon[fillstyle=solid,fillcolor=orange](0,1)(0,1.5)(0.866,1.5)(0.433,0.75)
\pspolygon[fillstyle=solid,fillcolor=orange](0,1)(-.866,1.5)(0,0)

%\pspolygon[fillstyle=solid,fillcolor=ForestGreen](0,1)(-.866,1.5)(0,1.5)

%A_0
\pspolygon[fillstyle=solid, fillcolor=black](0,0)(0,1)(0.433,0.75)

%%%%%%%%%%%%
%%%%%%%%%%%%
%%%%%%%%%%%%
%%%%%%%%%%%%
%%%HYPERPLANS TYPE G2
%%%%%%%%%%%%
%%%%%%%%%%%%
%%%%%%%%%%%%
%%%%%%%%%%%%
%%%%%%%%%%%%

% droite horizontale 
\psline(-4.33,6)(4.33,6)
\psline(-4.33,4.5)(4.33,4.5)
\psline(-4.33,3)(4.33,3)
\psline(-4.33,1.5)(4.33,1.5)
\psline(-4.33,0)(4.33,0)
\psline(-4.33,-1.5)(4.33,-1.5)
\psline(-4.33,-3)(4.33,-3)

% droite verticale
\psline(-4.33,-3)(-4.33,6)
\psline(-3.464,-3)(-3.464,6)
\psline(-2.598,-3)(-2.598,6)
\psline(-1.732,-3)(-1.732,6)
\psline(-.866,-3)(-.866,6)
\psline(0,-3)(0,6)
\psline(.866,-3)(.866,6)
\psline(1.732,-3)(1.732,6)
\psline(2.598,-3)(2.598,6)
\psline(3.464,-3)(3.464,6)
\psline(4.33,-3)(4.33,6)

%droite oblique
\psline(-4.33,5.5)(-3.464,6)
\psline(-4.33,4.5)(-1.732,6)
\psline(-4.33,3.5)(0,6)
\psline(-4.33,2.5)(1.732,6)
\psline(-4.33,1.5)(3.464,6)
\psline(-4.33,.5)(4.33,5.5)
\psline(-4.33,-.5)(4.33,4.5)
\psline(-4.33,-1.5)(4.33,3.5)
\psline(-4.33,-2.5)(4.33,2.5)
\psline(-3.464,-3)(4.33,1.5)
\psline(-1.732,-3)(4.33,.5)
\psline(0,-3)(4.33,-.5)
\psline(1.732,-3)(4.33,-1.5)
\psline(3.464,-3)(4.33,-2.5)

%droite oblique
\psline(4.33,5.5)(3.464,6)
\psline(4.33,4.5)(1.732,6)
\psline(4.33,3.5)(0,6)
\psline(4.33,2.5)(-1.732,6)
\psline(4.33,1.5)(-3.464,6)
\psline(4.33,.5)(-4.33,5.5)
\psline(4.33,-.5)(-4.33,4.5)
\psline(4.33,-1.5)(-4.33,3.5)
\psline(4.33,-2.5)(-4.33,2.5)
\psline(3.464,-3)(-4.33,1.5)
\psline(1.732,-3)(-4.33,.5)
\psline(0,-3)(-4.33,-.5)
\psline(-1.732,-3)(-4.33,-1.5)
\psline(-3.464,-3)(-4.33,-2.5)

%droite oblique
\psline(-4.33,-1.5)(-3.464,-3)
\psline(-4.33,1.5)(-1.732,-3)
\psline(-4.33,4.5)(0,-3)
\psline(-3.464,6)(1.732,-3)
\psline(-1.732,6)(3.464,-3)
\psline(0,6)(4.33,-1.5)
\psline(1.732,6)(4.33,1.5)
\psline(3.464,6)(4.33,4.5)

%droite oblique
\psline(4.33,-1.5)(3.464,-3)
\psline(4.33,1.5)(1.732,-3)
\psline(4.33,4.5)(0,-3)
\psline(3.464,6)(-1.732,-3)
\psline(1.732,6)(-3.464,-3)
\psline(0,6)(-4.33,-1.5)
\psline(-1.732,6)(-4.33,1.5)
\psline(-3.464,6)(-4.33,4.5)

%\pspolygon[fillstyle=solid, fillcolor=white,linestyle=none](-14.43,-4.6)(14.43,-4.6)(14.43,-3)(-14.43,-3)
\rput(0,-3.5){{\footnotesize $2>r>3/2$}}

\end{pspicture}

\end{center}

\end{subfigure}
\begin{subfigure}{.33\textwidth}

%%%%%%%%%%%%
%%%%%%%%%%%%
%
%	r=3/2
%
%%%%%%%%%%%%
%%%%%%%%%%%%

\begin{center}
\begin{pspicture}(-4.33,-4)(4.33,6.5)

\psset{linewidth=.05mm}

% c_0 lowest two-sided cell

\pspolygon[fillstyle=solid,fillcolor=red!5!yellow!42!](0.866,4.5)(0.866,6)(1.732,6)
\pspolygon[fillstyle=solid,fillcolor=red!5!yellow!42!](0.866,1.5)(3.46,6)(4.33,6)(4.33,3.5)
\pspolygon[fillstyle=solid,fillcolor=red!5!yellow!42!](2.598,1.5)(4.33,1.5)(4.33,2.5)
\pspolygon[fillstyle=solid,fillcolor=red!5!yellow!42!](1.732,0)(4.33,0)(4.33,-1.5)
\pspolygon[fillstyle=solid,fillcolor=red!5!yellow!42!](2.598,-1.5)(4.33,-2.5)(4.33,-3)(3.464,-3)
\pspolygon[fillstyle=solid,fillcolor=red!5!yellow!42!](0.866,-1.5)(0.866,-3)(1.732,-3)
\pspolygon[fillstyle=solid,fillcolor=red!5!yellow!42!](0,0)(0,-3)(-1.732,-3)
\pspolygon[fillstyle=solid,fillcolor=red!5!yellow!42!](-2.598,-1.5)(-3.464,-3)(-4.33,-3)(-4.33,-2.5)
\pspolygon[fillstyle=solid,fillcolor=red!5!yellow!42!](-1.732,0)(-4.33,0)(-4.33,-1.5)
\pspolygon[fillstyle=solid,fillcolor=red!5!yellow!42!](-2.598,1.5)(-4.33,1.5)(-4.33,2.5)
\pspolygon[fillstyle=solid,fillcolor=red!5!yellow!42!](-0.866,1.5)(-3.46,6)(-4.33,6)(-4.33,3.5)
\pspolygon[fillstyle=solid,fillcolor=red!5!yellow!42!](0,3)(0,6)(-1.732,6)

% c_1
\pspolygon[fillstyle=solid,fillcolor=blue!70!black!80!](0,0)(0,-3)(0.866,-3)(0.866,-1.5)
\pspolygon[fillstyle=solid,fillcolor=blue!70!black!80!](.866,0)(1.732,0)(4.33,-1.5)(4.33,-2.5)(.866,-.5)
\pspolygon[fillstyle=solid,fillcolor=blue!70!black!80!](-.866,0)(-1.732,0)(-4.33,-1.5)(-4.33,-2.5)(-.866,-.5)
\pspolygon[fillstyle=solid,fillcolor=blue!70!black!80!](.433,.75)(0.866,1.5)(4.33,3.5)(4.33,2.5)(.866,.5)
\pspolygon[fillstyle=solid,fillcolor=blue!70!black!80!](-.433,.75)(-0.866,1.5)(-4.33,3.5)(-4.33,2.5)(-.866,.5)
\pspolygon[fillstyle=solid,fillcolor=blue!70!black!80!](0,3)(0,6)(.866,6)(.866,2.5)(.433,2.25)

%c_2
\pspolygon[fillstyle=solid,fillcolor=green!80!black!70!](0.866,2.5)(.866,4.5)(1.732,6)(3.464,6)(1.732,3)
\pspolygon[fillstyle=solid,fillcolor=green!80!black!70!](0,1)(0,3)(-1.732,6)(-3.46,6)(-0.866,1.5)(0,1)
\pspolygon[fillstyle=solid,fillcolor=green!80!black!70!](-0.866,0.5)(-2.598,1.5)(-4.33,1.5)(-4.33,0)(-1.732,0)
\pspolygon[fillstyle=solid,fillcolor=green!80!black!70!](0.866,0.5)(2.598,1.5)(4.33,1.5)(4.33,0)(1.732,0)
\pspolygon[fillstyle=solid,fillcolor=green!80!black!70!](-.866,-.5)(-.866,-1.5)(-1.732,-3)(-3.464,-3)(-2.598,-1.5)
\pspolygon[fillstyle=solid,fillcolor=green!80!black!70!](.866,-.5)(.866,-1.5)(1.732,-3)(3.464,-3)(2.598,-1.5)

%turquoise

\pspolygon[fillstyle=solid,fillcolor=yellow!10!red!80!](0,0)(-0.433,0.75)(-1.732,0)(-0.866,0)(-0.866,-1.5)
\pspolygon[fillstyle=solid,fillcolor=yellow!10!red!80!](0,0)(0.433,0.75)(1.732,0)(0.866,0)(0.866,-1.5)
\pspolygon[fillstyle=solid,fillcolor=yellow!10!red!80!](0,1.5)(0,3)(0.433,2.25)(1.732,3)(0.866,1.5)

%%orange finite
\pspolygon[fillstyle=solid,fillcolor=orange](0,1)(0,1.5)(0.866,1.5)(0.433,0.75)
\pspolygon[fillstyle=solid,fillcolor=orange](0,1)(-.866,1.5)(0,0)

%\pspolygon[fillstyle=solid,fillcolor=ForestGreen](0,1)(-.866,1.5)(0,1.5)

%A_0
\pspolygon[fillstyle=solid, fillcolor=black](0,0)(0,1)(0.433,0.75)

%%%%%%%%%%%%
%%%%%%%%%%%%
%%%%%%%%%%%%
%%%%%%%%%%%%
%%%HYPERPLANS TYPE G2
%%%%%%%%%%%%
%%%%%%%%%%%%
%%%%%%%%%%%%
%%%%%%%%%%%%
%%%%%%%%%%%%

% droite horizontale 
\psline(-4.33,6)(4.33,6)
\psline(-4.33,4.5)(4.33,4.5)
\psline(-4.33,3)(4.33,3)
\psline(-4.33,1.5)(4.33,1.5)
\psline(-4.33,0)(4.33,0)
\psline(-4.33,-1.5)(4.33,-1.5)
\psline(-4.33,-3)(4.33,-3)

% droite verticale
\psline(-4.33,-3)(-4.33,6)
\psline(-3.464,-3)(-3.464,6)
\psline(-2.598,-3)(-2.598,6)
\psline(-1.732,-3)(-1.732,6)
\psline(-.866,-3)(-.866,6)
\psline(0,-3)(0,6)
\psline(.866,-3)(.866,6)
\psline(1.732,-3)(1.732,6)
\psline(2.598,-3)(2.598,6)
\psline(3.464,-3)(3.464,6)
\psline(4.33,-3)(4.33,6)

%droite oblique
\psline(-4.33,5.5)(-3.464,6)
\psline(-4.33,4.5)(-1.732,6)
\psline(-4.33,3.5)(0,6)
\psline(-4.33,2.5)(1.732,6)
\psline(-4.33,1.5)(3.464,6)
\psline(-4.33,.5)(4.33,5.5)
\psline(-4.33,-.5)(4.33,4.5)
\psline(-4.33,-1.5)(4.33,3.5)
\psline(-4.33,-2.5)(4.33,2.5)
\psline(-3.464,-3)(4.33,1.5)
\psline(-1.732,-3)(4.33,.5)
\psline(0,-3)(4.33,-.5)
\psline(1.732,-3)(4.33,-1.5)
\psline(3.464,-3)(4.33,-2.5)

%droite oblique
\psline(4.33,5.5)(3.464,6)
\psline(4.33,4.5)(1.732,6)
\psline(4.33,3.5)(0,6)
\psline(4.33,2.5)(-1.732,6)
\psline(4.33,1.5)(-3.464,6)
\psline(4.33,.5)(-4.33,5.5)
\psline(4.33,-.5)(-4.33,4.5)
\psline(4.33,-1.5)(-4.33,3.5)
\psline(4.33,-2.5)(-4.33,2.5)
\psline(3.464,-3)(-4.33,1.5)
\psline(1.732,-3)(-4.33,.5)
\psline(0,-3)(-4.33,-.5)
\psline(-1.732,-3)(-4.33,-1.5)
\psline(-3.464,-3)(-4.33,-2.5)

%droite oblique
\psline(-4.33,-1.5)(-3.464,-3)
\psline(-4.33,1.5)(-1.732,-3)
\psline(-4.33,4.5)(0,-3)
\psline(-3.464,6)(1.732,-3)
\psline(-1.732,6)(3.464,-3)
\psline(0,6)(4.33,-1.5)
\psline(1.732,6)(4.33,1.5)
\psline(3.464,6)(4.33,4.5)

%droite oblique
\psline(4.33,-1.5)(3.464,-3)
\psline(4.33,1.5)(1.732,-3)
\psline(4.33,4.5)(0,-3)
\psline(3.464,6)(-1.732,-3)
\psline(1.732,6)(-3.464,-3)
\psline(0,6)(-4.33,-1.5)
\psline(-1.732,6)(-4.33,1.5)
\psline(-3.464,6)(-4.33,4.5)

%\pspolygon[fillstyle=solid, fillcolor=white,linestyle=none](-14.43,-4.6)(14.43,-4.6)(14.43,-3)(-14.43,-3)

\rput(0,-3.5){{\footnotesize $r=3/2$}}

\end{pspicture}

\end{center}

\end{subfigure}
\begin{subfigure}{.33\textwidth}

%%%%%%%%%%%%
%%%%%%%%%%%%
%
%	32>r>1
%
%%%%%%%%%%%%
%%%%%%%%%%%%

\begin{center}
\begin{pspicture}(-4.33,-4)(4.33,6.5)

\psset{linewidth=.05mm}

% c_0 lowest two-sided cell

\pspolygon[fillstyle=solid,fillcolor=red!5!yellow!42!](0.866,4.5)(0.866,6)(1.732,6)
\pspolygon[fillstyle=solid,fillcolor=red!5!yellow!42!](0.866,1.5)(3.46,6)(4.33,6)(4.33,3.5)
\pspolygon[fillstyle=solid,fillcolor=red!5!yellow!42!](2.598,1.5)(4.33,1.5)(4.33,2.5)
\pspolygon[fillstyle=solid,fillcolor=red!5!yellow!42!](1.732,0)(4.33,0)(4.33,-1.5)
\pspolygon[fillstyle=solid,fillcolor=red!5!yellow!42!](2.598,-1.5)(4.33,-2.5)(4.33,-3)(3.464,-3)
\pspolygon[fillstyle=solid,fillcolor=red!5!yellow!42!](0.866,-1.5)(0.866,-3)(1.732,-3)
\pspolygon[fillstyle=solid,fillcolor=red!5!yellow!42!](0,0)(0,-3)(-1.732,-3)
\pspolygon[fillstyle=solid,fillcolor=red!5!yellow!42!](-2.598,-1.5)(-3.464,-3)(-4.33,-3)(-4.33,-2.5)
\pspolygon[fillstyle=solid,fillcolor=red!5!yellow!42!](-1.732,0)(-4.33,0)(-4.33,-1.5)
\pspolygon[fillstyle=solid,fillcolor=red!5!yellow!42!](-2.598,1.5)(-4.33,1.5)(-4.33,2.5)
\pspolygon[fillstyle=solid,fillcolor=red!5!yellow!42!](-0.866,1.5)(-3.46,6)(-4.33,6)(-4.33,3.5)
\pspolygon[fillstyle=solid,fillcolor=red!5!yellow!42!](0,3)(0,6)(-1.732,6)

% c_1
\pspolygon[fillstyle=solid,fillcolor=blue!70!black!80!](0,-1)(0,-3)(0.866,-3)(0.866,-1.5)(.433,-.75)
\pspolygon[fillstyle=solid,fillcolor=blue!70!black!80!](.866,0)(1.732,0)(4.33,-1.5)(4.33,-2.5)(.866,-.5)
\pspolygon[fillstyle=solid,fillcolor=blue!70!black!80!](-.866,0)(-1.732,0)(-4.33,-1.5)(-4.33,-2.5)(-.866,-.5)
\pspolygon[fillstyle=solid,fillcolor=blue!70!black!80!](.433,.75)(0.866,1.5)(4.33,3.5)(4.33,2.5)(.866,.5)
\pspolygon[fillstyle=solid,fillcolor=blue!70!black!80!](-.433,.75)(-0.866,1.5)(-4.33,3.5)(-4.33,2.5)(-.866,.5)
\pspolygon[fillstyle=solid,fillcolor=blue!70!black!80!](0,3)(0,6)(.866,6)(.866,2.5)(.433,2.25)

%c_2
\pspolygon[fillstyle=solid,fillcolor=green!80!black!70!](0.866,2.5)(.866,4.5)(1.732,6)(3.464,6)(1.732,3)
\pspolygon[fillstyle=solid,fillcolor=green!80!black!70!](0,1)(0,3)(-1.732,6)(-3.46,6)(-0.866,1.5)(0,1)
\pspolygon[fillstyle=solid,fillcolor=green!80!black!70!](-0.866,0.5)(-2.598,1.5)(-4.33,1.5)(-4.33,0)(-1.732,0)
\pspolygon[fillstyle=solid,fillcolor=green!80!black!70!](0.866,0.5)(2.598,1.5)(4.33,1.5)(4.33,0)(1.732,0)
\pspolygon[fillstyle=solid,fillcolor=green!80!black!70!](-.866,-.5)(-.866,-1.5)(-1.732,-3)(-3.464,-3)(-2.598,-1.5)
\pspolygon[fillstyle=solid,fillcolor=green!80!black!70!](.866,-.5)(.866,-1.5)(1.732,-3)(3.464,-3)(2.598,-1.5)

%purple
\pspolygon[fillstyle=solid,fillcolor=purple](0,0)(0,-1)(.433,-.75)

\pspolygon[fillstyle=solid,fillcolor=yellow!10!red!80!](0,0)(-0.433,0.75)(-1.732,0)(-0.866,0)(-0.866,-1.5)
\pspolygon[fillstyle=solid,fillcolor=yellow!10!red!80!](0,0)(0.433,0.75)(1.732,0)(0.866,0)(0.866,-1.5)
\pspolygon[fillstyle=solid,fillcolor=yellow!10!red!80!](0,1.5)(0,3)(0.433,2.25)(1.732,3)(0.866,1.5)

%%orange finite
\pspolygon[fillstyle=solid,fillcolor=orange](0,1)(0,1.5)(0.866,1.5)(0.433,0.75)
\pspolygon[fillstyle=solid,fillcolor=orange](0,1)(-.866,1.5)(0,0)

%\pspolygon[fillstyle=solid,fillcolor=ForestGreen](0,1)(-.866,1.5)(0,1.5)

%A_0
\pspolygon[fillstyle=solid, fillcolor=black](0,0)(0,1)(0.433,0.75)

%%%%%%%%%%%%
%%%%%%%%%%%%
%%%%%%%%%%%%
%%%%%%%%%%%%
%%%HYPERPLANS TYPE G2
%%%%%%%%%%%%
%%%%%%%%%%%%
%%%%%%%%%%%%
%%%%%%%%%%%%
%%%%%%%%%%%%

% droite horizontale 
\psline(-4.33,6)(4.33,6)
\psline(-4.33,4.5)(4.33,4.5)
\psline(-4.33,3)(4.33,3)
\psline(-4.33,1.5)(4.33,1.5)
\psline(-4.33,0)(4.33,0)
\psline(-4.33,-1.5)(4.33,-1.5)
\psline(-4.33,-3)(4.33,-3)

% droite verticale
\psline(-4.33,-3)(-4.33,6)
\psline(-3.464,-3)(-3.464,6)
\psline(-2.598,-3)(-2.598,6)
\psline(-1.732,-3)(-1.732,6)
\psline(-.866,-3)(-.866,6)
\psline(0,-3)(0,6)
\psline(.866,-3)(.866,6)
\psline(1.732,-3)(1.732,6)
\psline(2.598,-3)(2.598,6)
\psline(3.464,-3)(3.464,6)
\psline(4.33,-3)(4.33,6)

%droite oblique
\psline(-4.33,5.5)(-3.464,6)
\psline(-4.33,4.5)(-1.732,6)
\psline(-4.33,3.5)(0,6)
\psline(-4.33,2.5)(1.732,6)
\psline(-4.33,1.5)(3.464,6)
\psline(-4.33,.5)(4.33,5.5)
\psline(-4.33,-.5)(4.33,4.5)
\psline(-4.33,-1.5)(4.33,3.5)
\psline(-4.33,-2.5)(4.33,2.5)
\psline(-3.464,-3)(4.33,1.5)
\psline(-1.732,-3)(4.33,.5)
\psline(0,-3)(4.33,-.5)
\psline(1.732,-3)(4.33,-1.5)
\psline(3.464,-3)(4.33,-2.5)

%droite oblique
\psline(4.33,5.5)(3.464,6)
\psline(4.33,4.5)(1.732,6)
\psline(4.33,3.5)(0,6)
\psline(4.33,2.5)(-1.732,6)
\psline(4.33,1.5)(-3.464,6)
\psline(4.33,.5)(-4.33,5.5)
\psline(4.33,-.5)(-4.33,4.5)
\psline(4.33,-1.5)(-4.33,3.5)
\psline(4.33,-2.5)(-4.33,2.5)
\psline(3.464,-3)(-4.33,1.5)
\psline(1.732,-3)(-4.33,.5)
\psline(0,-3)(-4.33,-.5)
\psline(-1.732,-3)(-4.33,-1.5)
\psline(-3.464,-3)(-4.33,-2.5)

%droite oblique
\psline(-4.33,-1.5)(-3.464,-3)
\psline(-4.33,1.5)(-1.732,-3)
\psline(-4.33,4.5)(0,-3)
\psline(-3.464,6)(1.732,-3)
\psline(-1.732,6)(3.464,-3)
\psline(0,6)(4.33,-1.5)
\psline(1.732,6)(4.33,1.5)
\psline(3.464,6)(4.33,4.5)

%droite oblique
\psline(4.33,-1.5)(3.464,-3)
\psline(4.33,1.5)(1.732,-3)
\psline(4.33,4.5)(0,-3)
\psline(3.464,6)(-1.732,-3)
\psline(1.732,6)(-3.464,-3)
\psline(0,6)(-4.33,-1.5)
\psline(-1.732,6)(-4.33,1.5)
\psline(-3.464,6)(-4.33,4.5)

%\pspolygon[fillstyle=solid, fillcolor=white,linestyle=none](-14.43,-4.6)(14.43,-4.6)(14.43,-3)(-14.43,-3)

\rput(0,-3.5){{\footnotesize $3/2>r>1$}}

\end{pspicture}

\end{center}

\end{subfigure}
\begin{subfigure}{.33\textwidth}

%%%%%%%%%%%%
%%%%%%%%%%%%
%
%	r=1
%
%%%%%%%%%%%%
%%%%%%%%%%%%

\begin{center}
\begin{pspicture}(-4.33,-4)(4.33,6.5)

\psset{linewidth=.05mm}

% c_0 lowest two-sided cell

\pspolygon[fillstyle=solid,fillcolor=red!5!yellow!42!](0.866,4.5)(0.866,6)(1.732,6)
\pspolygon[fillstyle=solid,fillcolor=red!5!yellow!42!](0.866,1.5)(3.46,6)(4.33,6)(4.33,3.5)
\pspolygon[fillstyle=solid,fillcolor=red!5!yellow!42!](2.598,1.5)(4.33,1.5)(4.33,2.5)
\pspolygon[fillstyle=solid,fillcolor=red!5!yellow!42!](1.732,0)(4.33,0)(4.33,-1.5)
\pspolygon[fillstyle=solid,fillcolor=red!5!yellow!42!](2.598,-1.5)(4.33,-2.5)(4.33,-3)(3.464,-3)
\pspolygon[fillstyle=solid,fillcolor=red!5!yellow!42!](0.866,-1.5)(0.866,-3)(1.732,-3)
\pspolygon[fillstyle=solid,fillcolor=red!5!yellow!42!](0,0)(0,-3)(-1.732,-3)
\pspolygon[fillstyle=solid,fillcolor=red!5!yellow!42!](-2.598,-1.5)(-3.464,-3)(-4.33,-3)(-4.33,-2.5)
\pspolygon[fillstyle=solid,fillcolor=red!5!yellow!42!](-1.732,0)(-4.33,0)(-4.33,-1.5)
\pspolygon[fillstyle=solid,fillcolor=red!5!yellow!42!](-2.598,1.5)(-4.33,1.5)(-4.33,2.5)
\pspolygon[fillstyle=solid,fillcolor=red!5!yellow!42!](-0.866,1.5)(-3.46,6)(-4.33,6)(-4.33,3.5)
\pspolygon[fillstyle=solid,fillcolor=red!5!yellow!42!](0,3)(0,6)(-1.732,6)

% c_1
\pspolygon[fillstyle=solid,fillcolor=blue!70!black!80!](0,-1)(0,-3)(0.866,-3)(0.866,-1.5)(.433,-.75)
\pspolygon[fillstyle=solid,fillcolor=blue!70!black!80!](.866,0)(1.732,0)(4.33,-1.5)(4.33,-2.5)(.866,-.5)
\pspolygon[fillstyle=solid,fillcolor=blue!70!black!80!](-.866,0)(-1.732,0)(-4.33,-1.5)(-4.33,-2.5)(-.866,-.5)
\pspolygon[fillstyle=solid,fillcolor=blue!70!black!80!](.433,.75)(0.866,1.5)(4.33,3.5)(4.33,2.5)(.866,.5)
\pspolygon[fillstyle=solid,fillcolor=blue!70!black!80!](-.433,.75)(-0.866,1.5)(-4.33,3.5)(-4.33,2.5)(-.866,.5)
\pspolygon[fillstyle=solid,fillcolor=blue!70!black!80!](0,3)(0,6)(.866,6)(.866,2.5)(.433,2.25)

%c_2
\pspolygon[fillstyle=solid,fillcolor=green!80!black!70!](0.866,2.5)(.866,4.5)(1.732,6)(3.464,6)(1.732,3)
\pspolygon[fillstyle=solid,fillcolor=green!80!black!70!](0,1)(0,3)(-1.732,6)(-3.46,6)(-0.866,1.5)(0,1)
\pspolygon[fillstyle=solid,fillcolor=green!80!black!70!](-0.866,0.5)(-2.598,1.5)(-4.33,1.5)(-4.33,0)(-1.732,0)
\pspolygon[fillstyle=solid,fillcolor=green!80!black!70!](0.866,0.5)(2.598,1.5)(4.33,1.5)(4.33,0)(1.732,0)
\pspolygon[fillstyle=solid,fillcolor=green!80!black!70!](-.866,-.5)(-.866,-1.5)(-1.732,-3)(-3.464,-3)(-2.598,-1.5)
\pspolygon[fillstyle=solid,fillcolor=green!80!black!70!](.866,-.5)(.866,-1.5)(1.732,-3)(3.464,-3)(2.598,-1.5)

%purple

\pspolygon[fillstyle=solid,fillcolor=yellow!10!red!80!](0,0)(0,1)(-.866,1.5)(-.433,.75)(-1.732,0)(-0.866,0)(-0.866,-1.5)
\pspolygon[fillstyle=solid,fillcolor=yellow!10!red!80!](0,0)(0.433,0.75)(1.732,0)(0.866,0)(0.866,-1.5)(.433,-.75)(0,-1)
\pspolygon[fillstyle=solid,fillcolor=yellow!10!red!80!](0,1)(0,3)(0.433,2.25)(1.732,3)(0.433,.75)

%%or
%\pspolygon[fillstyle=solid,fillcolor=ForestGreen](0,1)(-.866,1.5)(0,1.5)

%A_0
\pspolygon[fillstyle=solid, fillcolor=black](0,0)(0,1)(0.433,0.75)

%%%%%%%%%%%%
%%%%%%%%%%%%
%%%%%%%%%%%%
%%%%%%%%%%%%
%%%HYPERPLANS TYPE G2
%%%%%%%%%%%%
%%%%%%%%%%%%
%%%%%%%%%%%%
%%%%%%%%%%%%
%%%%%%%%%%%%

% droite horizontale 
\psline(-4.33,6)(4.33,6)
\psline(-4.33,4.5)(4.33,4.5)
\psline(-4.33,3)(4.33,3)
\psline(-4.33,1.5)(4.33,1.5)
\psline(-4.33,0)(4.33,0)
\psline(-4.33,-1.5)(4.33,-1.5)
\psline(-4.33,-3)(4.33,-3)

% droite verticale
\psline(-4.33,-3)(-4.33,6)
\psline(-3.464,-3)(-3.464,6)
\psline(-2.598,-3)(-2.598,6)
\psline(-1.732,-3)(-1.732,6)
\psline(-.866,-3)(-.866,6)
\psline(0,-3)(0,6)
\psline(.866,-3)(.866,6)
\psline(1.732,-3)(1.732,6)
\psline(2.598,-3)(2.598,6)
\psline(3.464,-3)(3.464,6)
\psline(4.33,-3)(4.33,6)

%droite oblique
\psline(-4.33,5.5)(-3.464,6)
\psline(-4.33,4.5)(-1.732,6)
\psline(-4.33,3.5)(0,6)
\psline(-4.33,2.5)(1.732,6)
\psline(-4.33,1.5)(3.464,6)
\psline(-4.33,.5)(4.33,5.5)
\psline(-4.33,-.5)(4.33,4.5)
\psline(-4.33,-1.5)(4.33,3.5)
\psline(-4.33,-2.5)(4.33,2.5)
\psline(-3.464,-3)(4.33,1.5)
\psline(-1.732,-3)(4.33,.5)
\psline(0,-3)(4.33,-.5)
\psline(1.732,-3)(4.33,-1.5)
\psline(3.464,-3)(4.33,-2.5)

%droite oblique
\psline(4.33,5.5)(3.464,6)
\psline(4.33,4.5)(1.732,6)
\psline(4.33,3.5)(0,6)
\psline(4.33,2.5)(-1.732,6)
\psline(4.33,1.5)(-3.464,6)
\psline(4.33,.5)(-4.33,5.5)
\psline(4.33,-.5)(-4.33,4.5)
\psline(4.33,-1.5)(-4.33,3.5)
\psline(4.33,-2.5)(-4.33,2.5)
\psline(3.464,-3)(-4.33,1.5)
\psline(1.732,-3)(-4.33,.5)
\psline(0,-3)(-4.33,-.5)
\psline(-1.732,-3)(-4.33,-1.5)
\psline(-3.464,-3)(-4.33,-2.5)

%droite oblique
\psline(-4.33,-1.5)(-3.464,-3)
\psline(-4.33,1.5)(-1.732,-3)
\psline(-4.33,4.5)(0,-3)
\psline(-3.464,6)(1.732,-3)
\psline(-1.732,6)(3.464,-3)
\psline(0,6)(4.33,-1.5)
\psline(1.732,6)(4.33,1.5)
\psline(3.464,6)(4.33,4.5)

%droite oblique
\psline(4.33,-1.5)(3.464,-3)
\psline(4.33,1.5)(1.732,-3)
\psline(4.33,4.5)(0,-3)
\psline(3.464,6)(-1.732,-3)
\psline(1.732,6)(-3.464,-3)
\psline(0,6)(-4.33,-1.5)
\psline(-1.732,6)(-4.33,1.5)
\psline(-3.464,6)(-4.33,4.5)

%\pspolygon[fillstyle=solid, fillcolor=white,linestyle=none](-14.43,-4.6)(14.43,-4.6)(14.43,-3)(-14.43,-3)

\rput(0,-3.5){{\footnotesize $r=1$}}

\end{pspicture}

\end{center}

\end{subfigure}
\end{figure}

\begin{figure} 
\ContinuedFloat
\begin{subfigure}{.33\textwidth}

%%%%%%%%%%%%
%%%%%%%%%%%%
%
%	r<1
%
%%%%%%%%%%%%
%%%%%%%%%%%%

\begin{center}
\begin{pspicture}(-4.33,-4)(4.33,6.5)

\psset{linewidth=.05mm}

% c_0 lowest two-sided cell

\pspolygon[fillstyle=solid,fillcolor=red!5!yellow!42!](0.866,4.5)(0.866,6)(1.732,6)
\pspolygon[fillstyle=solid,fillcolor=red!5!yellow!42!](0.866,1.5)(3.46,6)(4.33,6)(4.33,3.5)
\pspolygon[fillstyle=solid,fillcolor=red!5!yellow!42!](2.598,1.5)(4.33,1.5)(4.33,2.5)
\pspolygon[fillstyle=solid,fillcolor=red!5!yellow!42!](1.732,0)(4.33,0)(4.33,-1.5)
\pspolygon[fillstyle=solid,fillcolor=red!5!yellow!42!](2.598,-1.5)(4.33,-2.5)(4.33,-3)(3.464,-3)
\pspolygon[fillstyle=solid,fillcolor=red!5!yellow!42!](0.866,-1.5)(0.866,-3)(1.732,-3)
\pspolygon[fillstyle=solid,fillcolor=red!5!yellow!42!](0,0)(0,-3)(-1.732,-3)
\pspolygon[fillstyle=solid,fillcolor=red!5!yellow!42!](-2.598,-1.5)(-3.464,-3)(-4.33,-3)(-4.33,-2.5)
\pspolygon[fillstyle=solid,fillcolor=red!5!yellow!42!](-1.732,0)(-4.33,0)(-4.33,-1.5)
\pspolygon[fillstyle=solid,fillcolor=red!5!yellow!42!](-2.598,1.5)(-4.33,1.5)(-4.33,2.5)
\pspolygon[fillstyle=solid,fillcolor=red!5!yellow!42!](-0.866,1.5)(-3.46,6)(-4.33,6)(-4.33,3.5)
\pspolygon[fillstyle=solid,fillcolor=red!5!yellow!42!](0,3)(0,6)(-1.732,6)

% c_1
\pspolygon[fillstyle=solid,fillcolor=blue!70!black!80!](0,-1)(0,-3)(0.866,-3)(0.866,-1.5)(.433,-.75)
\pspolygon[fillstyle=solid,fillcolor=blue!70!black!80!](.866,0)(1.732,0)(4.33,-1.5)(4.33,-2.5)(.866,-.5)
\pspolygon[fillstyle=solid,fillcolor=blue!70!black!80!](-.866,0)(-1.732,0)(-4.33,-1.5)(-4.33,-2.5)(-.866,-.5)
\pspolygon[fillstyle=solid,fillcolor=blue!70!black!80!](.433,.75)(0.866,1.5)(4.33,3.5)(4.33,2.5)(.866,.5)
\pspolygon[fillstyle=solid,fillcolor=blue!70!black!80!](-.433,.75)(-0.866,1.5)(-4.33,3.5)(-4.33,2.5)(-.866,.5)
\pspolygon[fillstyle=solid,fillcolor=blue!70!black!80!](0,3)(0,6)(.866,6)(.866,2.5)(.433,2.25)

%c_2
\pspolygon[fillstyle=solid,fillcolor=green!80!black!70!](0.866,2.5)(.866,4.5)(1.732,6)(3.464,6)(1.732,3)
\pspolygon[fillstyle=solid,fillcolor=green!80!black!70!](0,1)(0,3)(-1.732,6)(-3.46,6)(-0.866,1.5)(0,1)
\pspolygon[fillstyle=solid,fillcolor=green!80!black!70!](-0.866,0.5)(-2.598,1.5)(-4.33,1.5)(-4.33,0)(-1.732,0)
\pspolygon[fillstyle=solid,fillcolor=green!80!black!70!](0.866,0.5)(2.598,1.5)(4.33,1.5)(4.33,0)(1.732,0)
\pspolygon[fillstyle=solid,fillcolor=green!80!black!70!](-.866,-.5)(-.866,-1.5)(-1.732,-3)(-3.464,-3)(-2.598,-1.5)
\pspolygon[fillstyle=solid,fillcolor=green!80!black!70!](.866,-.5)(.866,-1.5)(1.732,-3)(3.464,-3)(2.598,-1.5)

%purple

\pspolygon[fillstyle=solid,fillcolor=purple](0,0)(.433,.75)(.866,.5)

\pspolygon[fillstyle=solid,fillcolor=yellow!10!red!80!](0,0)(0,1)(-.866,1.5)(-.433,.75)(-1.732,0)(-0.866,0)(-0.866,-.5)
\pspolygon[fillstyle=solid,fillcolor=yellow!10!red!80!](0,0)(0.866,.5)(1.732,0)(0.866,0)(0.866,-1.5)(.433,-.75)(0,-1)
\pspolygon[fillstyle=solid,fillcolor=yellow!10!red!80!](0,1)(0,3)(0.433,2.25)(.866,2.5)(0.866,1.5)(.433,.75)

%%or
\pspolygon[fillstyle=solid,fillcolor=orange](.866,1.5)(.866,2.5)(1.732,3)
\pspolygon[fillstyle=solid,fillcolor=orange](0,0)(-.866,-1.5)(-.866,-.5)

%A_0
\pspolygon[fillstyle=solid, fillcolor=black](0,0)(0,1)(0.433,0.75)

%%%%%%%%%%%%
%%%%%%%%%%%%
%%%%%%%%%%%%
%%%%%%%%%%%%
%%%HYPERPLANS TYPE G2
%%%%%%%%%%%%
%%%%%%%%%%%%
%%%%%%%%%%%%
%%%%%%%%%%%%
%%%%%%%%%%%%

% droite horizontale 
\psline(-4.33,6)(4.33,6)
\psline(-4.33,4.5)(4.33,4.5)
\psline(-4.33,3)(4.33,3)
\psline(-4.33,1.5)(4.33,1.5)
\psline(-4.33,0)(4.33,0)
\psline(-4.33,-1.5)(4.33,-1.5)
\psline(-4.33,-3)(4.33,-3)

% droite verticale
\psline(-4.33,-3)(-4.33,6)
\psline(-3.464,-3)(-3.464,6)
\psline(-2.598,-3)(-2.598,6)
\psline(-1.732,-3)(-1.732,6)
\psline(-.866,-3)(-.866,6)
\psline(0,-3)(0,6)
\psline(.866,-3)(.866,6)
\psline(1.732,-3)(1.732,6)
\psline(2.598,-3)(2.598,6)
\psline(3.464,-3)(3.464,6)
\psline(4.33,-3)(4.33,6)

%droite oblique
\psline(-4.33,5.5)(-3.464,6)
\psline(-4.33,4.5)(-1.732,6)
\psline(-4.33,3.5)(0,6)
\psline(-4.33,2.5)(1.732,6)
\psline(-4.33,1.5)(3.464,6)
\psline(-4.33,.5)(4.33,5.5)
\psline(-4.33,-.5)(4.33,4.5)
\psline(-4.33,-1.5)(4.33,3.5)
\psline(-4.33,-2.5)(4.33,2.5)
\psline(-3.464,-3)(4.33,1.5)
\psline(-1.732,-3)(4.33,.5)
\psline(0,-3)(4.33,-.5)
\psline(1.732,-3)(4.33,-1.5)
\psline(3.464,-3)(4.33,-2.5)

%droite oblique
\psline(4.33,5.5)(3.464,6)
\psline(4.33,4.5)(1.732,6)
\psline(4.33,3.5)(0,6)
\psline(4.33,2.5)(-1.732,6)
\psline(4.33,1.5)(-3.464,6)
\psline(4.33,.5)(-4.33,5.5)
\psline(4.33,-.5)(-4.33,4.5)
\psline(4.33,-1.5)(-4.33,3.5)
\psline(4.33,-2.5)(-4.33,2.5)
\psline(3.464,-3)(-4.33,1.5)
\psline(1.732,-3)(-4.33,.5)
\psline(0,-3)(-4.33,-.5)
\psline(-1.732,-3)(-4.33,-1.5)
\psline(-3.464,-3)(-4.33,-2.5)

%droite oblique
\psline(-4.33,-1.5)(-3.464,-3)
\psline(-4.33,1.5)(-1.732,-3)
\psline(-4.33,4.5)(0,-3)
\psline(-3.464,6)(1.732,-3)
\psline(-1.732,6)(3.464,-3)
\psline(0,6)(4.33,-1.5)
\psline(1.732,6)(4.33,1.5)
\psline(3.464,6)(4.33,4.5)

%droite oblique
\psline(4.33,-1.5)(3.464,-3)
\psline(4.33,1.5)(1.732,-3)
\psline(4.33,4.5)(0,-3)
\psline(3.464,6)(-1.732,-3)
\psline(1.732,6)(-3.464,-3)
\psline(0,6)(-4.33,-1.5)
\psline(-1.732,6)(-4.33,1.5)
\psline(-3.464,6)(-4.33,4.5)

%\pspolygon[fillstyle=solid, fillcolor=white,linestyle=none](-14.43,-4.6)(14.43,-4.6)(14.43,-3)(-14.43,-3)
\rput(0,-3.5){{\footnotesize $r<1$}}

\end{pspicture}

\end{center}

\end{subfigure}
\begin{subfigure}{.66\textwidth}
\begin{center}
\pgfmathsetmacro{\spv}{0.5}
\pgfmathsetmacro{\sph}{1.5}
\begin{tikzpicture}

%%%%%%%%%%%%%%%%%%%%%%%%%%%%%%%%%%%%%%%%%%%%
%%%%%%%%%%%%%%%%%%%%%%%%%%%%%%%%%%%%%%%%%%%%
%%%%%%%%%%%%%%			1			%%%%%%%%%%%%%%%%%%%%%%%%%
%%%%%%%%%%%%%%%%%%%%%%%%%%%%%%%%%%%%%%%%%%%%
%%%%%%%%%%%%%%%%%%%%%%%%%%%%%%%%%%%%%%%%%%%%

\node[label={[inner sep=.04cm]left:\scalebox{.45}{$\Ga_e$}},circle,draw=black, fill=black, inner sep=0pt,minimum size=3pt] (ae) at (0,0) {};
\node[label={[inner sep=.04cm]left:\scalebox{.45}{$\Ga_4$}},circle,draw=black, fill=ora, inner sep=0pt,minimum size=3pt] (a4) at (0,-1*\spv) {};
\node[label={[inner sep=.04cm]left:\scalebox{.45}{$\Ga_5$}},circle,draw=black, fill=fgr, inner sep=0pt,minimum size=3pt] (a5)  at (-.4,-2*\spv) {};
\node[label={[inner sep=.04cm]right:\scalebox{.45}{$\Ga_3$}},circle,draw=black, fill=rd, inner sep=0pt,minimum size=3pt] (a3)  at (0.4,-2*\spv) {};
\node[label={[inner sep=.04cm]left:\scalebox{.45}{$\Ga_1$}},circle,draw=black, fill=gr, inner sep=0pt,minimum size=3pt] (a1) at (0,-3*\spv) {};
\node[label={[inner sep=.04cm]left:\scalebox{.45}{$\Ga_2$}},circle,draw=black, fill=bl, inner sep=0pt,minimum size=3pt] (a2) at (0,-4*\spv) {};
\node[label={[inner sep=.04cm]left:\scalebox{.45}{$\Ga_0$}},circle,draw=black, fill=ye, inner sep=0pt,minimum size=3pt] (a0)  at (0,-5*\spv) {};

\draw (a0)  to  (a2);
\draw (a2)  to  (a1);
\draw (a1)  to  (a3);
\draw (a1)  to  (a5);
\draw (a3)  to  (a4);
\draw (a4)  to  (a5);
\draw (a4)  to  (ae);

%%%%%%%%%%%%%%%%%%%%%%%%%%%%%%%%%%%%%%%%%%%%
%%%%%%%%%%%%%%%%%%%%%%%%%%%%%%%%%%%%%%%%%%%%
%%%%%%%%%%%%%%			2			%%%%%%%%%%%%%%%%%%%%%%%%%
%%%%%%%%%%%%%%%%%%%%%%%%%%%%%%%%%%%%%%%%%%%%
%%%%%%%%%%%%%%%%%%%%%%%%%%%%%%%%%%%%%%%%%%%%

\node[label={[inner sep=.04cm]left:\scalebox{.45}{$\Ga_e$}},circle,draw=black, fill=black, inner sep=0pt,minimum size=3pt] (ae) at (\sph,0) {};
\node[label={[inner sep=.04cm]left:\scalebox{.45}{$\Ga_4$}},circle,draw=black, fill=ora, inner sep=0pt,minimum size=3pt] (a4) at (\sph,-1*\spv) {};
%\node[label={[inner sep=.04cm]left:\scalebox{.45}{$\Ga_5$}},circle,draw=black, fill=fgr, inner sep=0pt,minimum size=3pt] (a5)  at (\sph-.4,-2*\spv) {};
\node[label={[inner sep=.04cm]left:\scalebox{.45}{$\Ga_3$}},circle,draw=black, fill=rd, inner sep=0pt,minimum size=3pt] (a3)  at (\sph,-2*\spv) {};
\node[label={[inner sep=.04cm]left:\scalebox{.45}{$\Ga_1$}},circle,draw=black, fill=gr, inner sep=0pt,minimum size=3pt] (a1) at (\sph,-3*\spv) {};
\node[label={[inner sep=.04cm]left:\scalebox{.45}{$\Ga_2$}},circle,draw=black, fill=bl, inner sep=0pt,minimum size=3pt] (a2) at (\sph,-4*\spv) {};
\node[label={[inner sep=.04cm]left:\scalebox{.45}{$\Ga_0$}},circle,draw=black, fill=ye, inner sep=0pt,minimum size=3pt] (a0)  at (\sph,-5*\spv) {};

\draw (a0)  to  (a2);
\draw (a2)  to  (a1);
\draw (a1)  to  (a3);
\draw (a3)  to  (a4);
\draw (a4)  to  (ae);

%%%%%%%%%%%%%%%%%%%%%%%%%%%%%%%%%%%%%%%%%%%%
%%%%%%%%%%%%%%%%%%%%%%%%%%%%%%%%%%%%%%%%%%%%
%%%%%%%%%%%%%%			3			%%%%%%%%%%%%%%%%%%%%%%%%%
%%%%%%%%%%%%%%%%%%%%%%%%%%%%%%%%%%%%%%%%%%%%
%%%%%%%%%%%%%%%%%%%%%%%%%%%%%%%%%%%%%%%%%%%%

\node[label={[inner sep=.04cm]left:\scalebox{.45}{$\Ga_e$}},circle,draw=black, fill=black, inner sep=0pt,minimum size=3pt] (ae) at (2*\sph,0) {};
\node[label={[inner sep=.04cm]left:\scalebox{.45}{$\Ga_4$}},circle,draw=black, fill=ora, inner sep=0pt,minimum size=3pt] (a4) at (2*\sph,-1*\spv) {};
\node[label={[inner sep=.04cm]left:\scalebox{.45}{$\Ga_3$}},circle,draw=black, fill=rd, inner sep=0pt,minimum size=3pt] (a3)  at (2*\sph,-2*\spv) {};
\node[label={[inner sep=.04cm]left:\scalebox{.45}{$\Ga_6$}},circle,draw=black, fill=tur, inner sep=0pt,minimum size=3pt] (at)  at (2*\sph,-3*\spv) {};
\node[label={[inner sep=.04cm]right:\scalebox{.45}{$\Ga_1$}},circle,draw=black, fill=gr, inner sep=0pt,minimum size=3pt] (a1) at (2*\sph+.4,-4*\spv) {};
\node[label={[inner sep=.04cm]left:\scalebox{.45}{$\Ga_2$}},circle,draw=black, fill=bl, inner sep=0pt,minimum size=3pt] (a2) at (2*\sph-.4,-4*\spv) {};
\node[label={[inner sep=.04cm]left:\scalebox{.45}{$\Ga_0$}},circle,draw=black, fill=ye, inner sep=0pt,minimum size=3pt] (a0)  at (2*\sph,-5*\spv) {};

\draw (a0)  to  (a2);
\draw (a0)  to  (a1);
\draw (a1)  to  (at);
\draw (a2)  to  (at);
\draw (at)  to  (a3);
\draw (a3)  to  (a4);
\draw (a4)  to  (ae);

%%%%%%%%%%%%%%%%%%%%%%%%%%%%%%%%%%%%%%%%%%%%
%%%%%%%%%%%%%%%%%%%%%%%%%%%%%%%%%%%%%%%%%%%%
%%%%%%%%%%%%%%			4			%%%%%%%%%%%%%%%%%%%%%%%%%
%%%%%%%%%%%%%%%%%%%%%%%%%%%%%%%%%%%%%%%%%%%%
%%%%%%%%%%%%%%%%%%%%%%%%%%%%%%%%%%%%%%%%%%%%

\node[label={[inner sep=.04cm]left:\scalebox{.45}{$\Ga_e$}},circle,draw=black, fill=black, inner sep=0pt,minimum size=3pt] (ae) at (3*\sph,0) {};
\node[label={[inner sep=.04cm]left:\scalebox{.45}{$\Ga_4$}},circle,draw=black, fill=ora, inner sep=0pt,minimum size=3pt] (a4) at (3*\sph,-1*\spv) {};
\node[label={[inner sep=.04cm]left:\scalebox{.45}{$\Ga_3$}},circle,draw=black, fill=rd, inner sep=0pt,minimum size=3pt] (a3)  at (3*\sph,-2*\spv) {};
\node[label={[inner sep=.04cm]left:\scalebox{.45}{$\Ga_2$}},circle,draw=black, fill=bl, inner sep=0pt,minimum size=3pt] (a2) at (3*\sph,-3*\spv) {};
\node[label={[inner sep=.04cm]left:\scalebox{.45}{$\Ga_1$}},circle,draw=black, fill=gr, inner sep=0pt,minimum size=3pt] (a1) at (3*\sph,-4*\spv) {};
\node[label={[inner sep=.04cm]left:\scalebox{.45}{$\Ga_0$}},circle,draw=black, fill=ye, inner sep=0pt,minimum size=3pt] (a0)  at (3*\sph,-5*\spv) {};

\draw (a0)  to  (a1);
\draw (a1)  to  (a2);
\draw (a2)  to  (a3);
\draw (a3)  to  (a4);
\draw (a4)  to  (ae);

%%%%%%%%%%%%%%%%%%%%%%%%%%%%%%%%%%%%%%%%%%%%
%%%%%%%%%%%%%%%%%%%%%%%%%%%%%%%%%%%%%%%%%%%%
%%%%%%%%%%%%%%			5			%%%%%%%%%%%%%%%%%%%%%%%%%
%%%%%%%%%%%%%%%%%%%%%%%%%%%%%%%%%%%%%%%%%%%%
%%%%%%%%%%%%%%%%%%%%%%%%%%%%%%%%%%%%%%%%%%%%

\node[label={[inner sep=.04cm]left:\scalebox{.45}{$\Ga_e$}},circle,draw=black, fill=black, inner sep=0pt,minimum size=3pt] (ae) at (4*\sph,0) {};
\node[label={[inner sep=.04cm]left:\scalebox{.45}{$\Ga_4$}},circle,draw=black, fill=ora, inner sep=0pt,minimum size=3pt] (a4) at (4*\sph,-1*\spv) {};
\node[label={[inner sep=.04cm]left:\scalebox{.45}{$\Ga_3$}},circle,draw=black, fill=rd, inner sep=0pt,minimum size=3pt] (a3)  at (4*\sph,-2*\spv) {};
\node[label={[inner sep=.04cm]left:\scalebox{.45}{$\Ga_7$}},circle,draw=black, fill=pur, inner sep=0pt,minimum size=3pt] (a7)  at (4*\sph,-3*\spv) {};
\node[label={[inner sep=.04cm]left:\scalebox{.45}{$\Ga_2$}},circle,draw=black, fill=bl, inner sep=0pt,minimum size=3pt] (a2) at (4*\sph,-4*\spv) {};
\node[label={[inner sep=.04cm]left:\scalebox{.45}{$\Ga_1$}},circle,draw=black, fill=gr, inner sep=0pt,minimum size=3pt] (a1) at (4*\sph,-5*\spv) {};
\node[label={[inner sep=.04cm]left:\scalebox{.45}{$\Ga_0$}},circle,draw=black, fill=ye, inner sep=0pt,minimum size=3pt] (a0)  at (4*\sph,-6*\spv) {};

\draw (a0)  to  (a1);
\draw (a1)  to  (a2);
\draw (a2)  to  (a7);
\draw (a7)  to  (a3);
\draw (a3)  to  (a4);
\draw (a4)  to  (ae);

%%%%%%%%%%%%%%%%%%%%%%%%%%%%%%%%%%%%%%%%%%%%
%%%%%%%%%%%%%%%%%%%%%%%%%%%%%%%%%%%%%%%%%%%%
%%%%%%%%%%%%%%			6			%%%%%%%%%%%%%%%%%%%%%%%%%
%%%%%%%%%%%%%%%%%%%%%%%%%%%%%%%%%%%%%%%%%%%%
%%%%%%%%%%%%%%%%%%%%%%%%%%%%%%%%%%%%%%%%%%%%

\node[label={[inner sep=.04cm]left:\scalebox{.45}{$\Ga_e$}},circle,draw=black, fill=black, inner sep=0pt,minimum size=3pt] (ae) at (5*\sph,0) {};
\node[label={[inner sep=.04cm]left:\scalebox{.45}{$\Ga_3$}},circle,draw=black, fill=rd, inner sep=0pt,minimum size=3pt] (a3) at (5*\sph,-1*\spv) {};
\node[label={[inner sep=.04cm]left:\scalebox{.45}{$\Ga_2$}},circle,draw=black, fill=bl, inner sep=0pt,minimum size=3pt] (a2) at (5*\sph,-2*\spv) {};
\node[label={[inner sep=.04cm]left:\scalebox{.45}{$\Ga_1$}},circle,draw=black, fill=gr, inner sep=0pt,minimum size=3pt] (a1) at (5*\sph,-3*\spv) {};
\node[label={[inner sep=.04cm]left:\scalebox{.45}{$\Ga_0$}},circle,draw=black, fill=ye, inner sep=0pt,minimum size=3pt] (a0) at (5*\sph,-4*\spv) {};

\draw (a0)  to  (a1);
\draw (a1)  to  (a2);
\draw (a2)  to  (a3);
\draw (a3)  to  (ae);

%%%%%%%%%%%%%%%%%%%%%%%%%%%%%%%%%%%%%%%%%%%%
%%%%%%%%%%%%%%%%%%%%%%%%%%%%%%%%%%%%%%%%%%%%
%%%%%%%%%%%%%%			7			%%%%%%%%%%%%%%%%%%%%%%%%%
%%%%%%%%%%%%%%%%%%%%%%%%%%%%%%%%%%%%%%%%%%%%
%%%%%%%%%%%%%%%%%%%%%%%%%%%%%%%%%%%%%%%%%%%%

\node[label={[inner sep=.04cm]left:\scalebox{.45}{$\Ga_e$}},circle,draw=black, fill=black, inner sep=0pt,minimum size=3pt] (ae) at (6*\sph,0) {};
\node[label={[inner sep=.04cm]left:\scalebox{.45}{$\Ga_7$}},circle,draw=black, fill=pur, inner sep=0pt,minimum size=3pt] (a7) at (6*\sph,-1*\spv) {};
\node[label={[inner sep=.04cm]left:\scalebox{.45}{$\Ga_3$}},circle,draw=black, fill=rd, inner sep=0pt,minimum size=3pt] (a3)  at (6*\sph,-2*\spv) {};
\node[label={[inner sep=.04cm]right:\scalebox{.45}{$\Ga_4$}},circle,draw=black, fill=ora, inner sep=0pt,minimum size=3pt] (a4)  at (6*\sph+.4,-3*\spv) {};
\node[label={[inner sep=.04cm]left:\scalebox{.45}{$\Ga_2$}},circle,draw=black, fill=bl, inner sep=0pt,minimum size=3pt] (a2) at (6*\sph-.4,-3*\spv) {};
\node[label={[inner sep=.04cm]left:\scalebox{.45}{$\Ga_1$}},circle,draw=black, fill=gr, inner sep=0pt,minimum size=3pt] (a1) at (6*\sph,-4*\spv) {};
\node[label={[inner sep=.04cm]left:\scalebox{.45}{$\Ga_0$}},circle,draw=black, fill=ye, inner sep=0pt,minimum size=3pt] (a0)  at (6*\sph,-5*\spv) {};

\draw (a0)  to  (a1);
\draw (a1)  to  (a2);
\draw (a1)  to  (a4);
\draw (a2)  to  (a3);
\draw (a4)  to  (a3);
\draw (a3)  to  (a7);
\draw (a7)  to  (ae);

\end{tikzpicture}
\end{center}
\end{subfigure}
\caption{Partition of $\tG_2$ into Kazhdan-Lusztig cells, $r=a/b$}
\label{partition}
\end{figure}
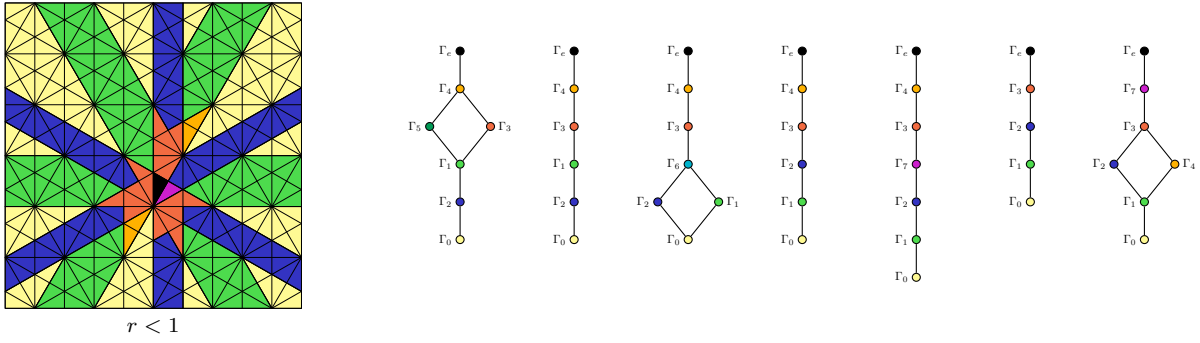

Let $\Ga$ be a two-sided cell for the parameter $r\in \nQ$. We say that $r$ is generic for $\Ga$ if there exists $\eta>0$ such that $\Ga$ is a cell for all parameters $r'\in \nQ$ such that $r-\eta<r'<r+\eta$. By considering the decomposition of $\tG_2$ into cells, it is easy to see that the only pairs $(\Ga,r)$ such that $r$ is non-generic for $\Ga$ are $(\Ga_1,2)$, $(\Ga_2,3/2)$ and $(\Ga_3,1)$.

\subsection{Cell factorisation for the lowest two-sided cell $\Gamma_0$}\label{sec:factor0}

Note that the yellow two-sided cell is constant for all choices of~$r$ (see Figure~\ref{partition}). This cell is called the \textit{lowest two-sided cell}, and is denoted $\Gamma_0$. By direct inspection we have the following representation of elements of $\Gamma_0$ (see Figure~\ref{lowest}):
\bem
\item Each right cell $\Gamma^k\subseteq \Gamma$ ($1\leq k\leq 12$) contains a unique element $\sw_{\Gamma^k}$ of minimal length. 
\item The longest element $\sw_0$ of $G_2$ is a suffix of each $\sw_{\Gamma^k}$, $1\leq k\leq 12$. Let $u_k=\sw_0\sw_{\Gamma^k}^{-1}$ for $1\leq k\leq 12$ (these elements are the inverses of the grey elements on the left). Let $\sB_{\Ga_0}=\{u_k\mid 1\leq k\leq 12\}$ (this ``box'' is shaded in light grey on the right of Figure~\ref{lowest}). 
\item We have
$$
\Gamma_0=\{u^{-1}\sw_0t_{\lambda}v\mid u,v\in\sB_{\Ga_0},\,\lambda\in P^+\}.
$$
\eem
Moreover, each $w\in\Gamma_0$ has a unique expression in the form $w=u^{-1}\sw_0t_{\lambda}v$ with $u,v\in \sB_{\Ga_0}$ and $\lambda\in P^+$, and this expression is reduced (that is, $\ell(w)=\ell(u^{-1})+\ell(\sw_0)+\ell(t_{\lambda})+\ell(v)$). This expression is called the \textit{cell factorisation} of $w\in\Gamma_0$. It should be understood in the following way: The element $u^{-1}$ indicates in which connected component (right cell) of $\Gamma_0$ the alcove $w$ lies. The element $\lambda$ indicates in which translate of the box $u^{-1}\sw_0\sB_{\Ga_0}$ the alcove $w$ lies. The element $v$ indicates location of $w$ in the box $u^{-1}\sw_0t_{\lambda}\sB_{\Ga_0}$. 

\medskip

We will often write $\sB_0$ in place of $\sB_{\Ga_0}$. Note that the translates of $\sB_{0}$ cover $W$. An analogue of the above cell factorisation applies to the lowest two-sided cell in arbitrary type, see \cite[Proposition 4.3]{Shil2} and \cite[Proposition~3.1]{Blas}.

\psset{unit=.6cm}

\begin{figure}[H]
\begin{subfigure}[b]{0.5\textwidth}
\begin{center}
\begin{pspicture}(-4.33,-3.3)(4.33,6.5)
\psset{linewidth=.05mm}

\pspolygon[fillstyle=solid, fillcolor=gray](0,0)(0,1)(0.433,0.75)

\pspolygon[fillstyle=solid, fillcolor=gray](0,1)(0.433,0.75)(.866,1.5)
\pspolygon[fillstyle=solid, fillcolor=gray](0,1)(-0.433,0.75)(-.866,1.5)

\pspolygon[fillstyle=solid, fillcolor=gray](1.732,1.5)(2.598,1.5)(1.732,1)
\pspolygon[fillstyle=solid, fillcolor=gray](-1.732,1.5)(-2.598,1.5)(-1.732,1)

\pspolygon[fillstyle=solid, fillcolor=gray](0,3)(0,2)(.433,2.25)

\pspolygon[fillstyle=solid, fillcolor=gray](.866,4.5)(.866,3.5)(.434,3.75)

\pspolygon[fillstyle=solid, fillcolor=gray](.866,.0)(.866,.5)(1.732,0)
\pspolygon[fillstyle=solid, fillcolor=gray](-.866,.0)(-.866,.5)(-1.732,0)

\pspolygon[fillstyle=solid, fillcolor=gray](2.165,-.75)(1.732,-1)(2.598,-1.5)
\pspolygon[fillstyle=solid, fillcolor=gray](-2.165,-.75)(-1.732,-1)(-2.598,-1.5)

\pspolygon[fillstyle=solid, fillcolor=gray](.866,-.5)(.866,-1.5)(.433,-.75)
% c_0 lowest two-sided cell

\pspolygon[fillstyle=solid,fillcolor=red!5!yellow!42!](0.866,4.5)(0.866,6)(1.732,6)
\pspolygon[fillstyle=solid,fillcolor=red!5!yellow!42!](0.866,1.5)(3.46,6)(4.33,6)(4.33,3.5)
\pspolygon[fillstyle=solid,fillcolor=red!5!yellow!42!](2.598,1.5)(4.33,1.5)(4.33,2.5)
\pspolygon[fillstyle=solid,fillcolor=red!5!yellow!42!](1.732,0)(4.33,0)(4.33,-1.5)
\pspolygon[fillstyle=solid,fillcolor=red!5!yellow!42!](2.598,-1.5)(4.33,-2.5)(4.33,-3)(3.464,-3)
\pspolygon[fillstyle=solid,fillcolor=red!5!yellow!42!](0.866,-1.5)(0.866,-3)(1.732,-3)
\pspolygon[fillstyle=solid,fillcolor=red!5!yellow!42!](0,0)(0,-3)(-1.732,-3)
\pspolygon[fillstyle=solid,fillcolor=red!5!yellow!42!](-2.598,-1.5)(-3.464,-3)(-4.33,-3)(-4.33,-2.5)
\pspolygon[fillstyle=solid,fillcolor=red!5!yellow!42!](-1.732,0)(-4.33,0)(-4.33,-1.5)
\pspolygon[fillstyle=solid,fillcolor=red!5!yellow!42!](-2.598,1.5)(-4.33,1.5)(-4.33,2.5)
\pspolygon[fillstyle=solid,fillcolor=red!5!yellow!42!](-0.866,1.5)(-3.46,6)(-4.33,6)(-4.33,3.5)
\pspolygon[fillstyle=solid,fillcolor=red!5!yellow!42!](0,3)(0,6)(-1.732,6)

%%%%%%%%%%%%
%%%%%%%%%%%%
%%%%%%%%%%%%
%%%%%%%%%%%%
%%%HYPERPLANS TYPE G2
%%%%%%%%%%%%
%%%%%%%%%%%%
%%%%%%%%%%%%
%%%%%%%%%%%%
%%%%%%%%%%%%

% droite horizontale 
\psline(-4.33,6)(4.33,6)
\psline(-4.33,4.5)(4.33,4.5)
\psline(-4.33,3)(4.33,3)
\psline(-4.33,1.5)(4.33,1.5)
\psline(-4.33,0)(4.33,0)
\psline(-4.33,-1.5)(4.33,-1.5)
\psline(-4.33,-3)(4.33,-3)

% droite verticale
\psline(-4.33,-3)(-4.33,6)
\psline(-3.464,-3)(-3.464,6)
\psline(-2.598,-3)(-2.598,6)
\psline(-1.732,-3)(-1.732,6)
\psline(-.866,-3)(-.866,6)
\psline(0,-3)(0,6)
\psline(.866,-3)(.866,6)
\psline(1.732,-3)(1.732,6)
\psline(2.598,-3)(2.598,6)
\psline(3.464,-3)(3.464,6)
\psline(4.33,-3)(4.33,6)

%droite oblique
\psline(-4.33,5.5)(-3.464,6)
\psline(-4.33,4.5)(-1.732,6)
\psline(-4.33,3.5)(0,6)
\psline(-4.33,2.5)(1.732,6)
\psline(-4.33,1.5)(3.464,6)
\psline(-4.33,.5)(4.33,5.5)
\psline(-4.33,-.5)(4.33,4.5)
\psline(-4.33,-1.5)(4.33,3.5)
\psline(-4.33,-2.5)(4.33,2.5)
\psline(-3.464,-3)(4.33,1.5)
\psline(-1.732,-3)(4.33,.5)
\psline(0,-3)(4.33,-.5)
\psline(1.732,-3)(4.33,-1.5)
\psline(3.464,-3)(4.33,-2.5)

%droite oblique
\psline(4.33,5.5)(3.464,6)
\psline(4.33,4.5)(1.732,6)
\psline(4.33,3.5)(0,6)
\psline(4.33,2.5)(-1.732,6)
\psline(4.33,1.5)(-3.464,6)
\psline(4.33,.5)(-4.33,5.5)
\psline(4.33,-.5)(-4.33,4.5)
\psline(4.33,-1.5)(-4.33,3.5)
\psline(4.33,-2.5)(-4.33,2.5)
\psline(3.464,-3)(-4.33,1.5)
\psline(1.732,-3)(-4.33,.5)
\psline(0,-3)(-4.33,-.5)
\psline(-1.732,-3)(-4.33,-1.5)
\psline(-3.464,-3)(-4.33,-2.5)

%droite oblique
\psline(-4.33,-1.5)(-3.464,-3)
\psline(-4.33,1.5)(-1.732,-3)
\psline(-4.33,4.5)(0,-3)
\psline(-3.464,6)(1.732,-3)
\psline(-1.732,6)(3.464,-3)
\psline(0,6)(4.33,-1.5)
\psline(1.732,6)(4.33,1.5)
\psline(3.464,6)(4.33,4.5)

%droite oblique
\psline(4.33,-1.5)(3.464,-3)
\psline(4.33,1.5)(1.732,-3)
\psline(4.33,4.5)(0,-3)
\psline(3.464,6)(-1.732,-3)
\psline(1.732,6)(-3.464,-3)
\psline(0,6)(-4.33,-1.5)
\psline(-1.732,6)(-4.33,1.5)
\psline(-3.464,6)(-4.33,4.5)
\end{pspicture}
\end{center}
\end{subfigure}
\begin{subfigure}[b]{0.5\textwidth}
\begin{center}
\begin{pspicture}(0,-1)(4.5,8)
\psset{linewidth=.05mm}

\pspolygon[fillstyle=solid,fillcolor=lightgray](0,0)(0,3)(.866,4.5)(.866,1.5)
\pspolygon[fillstyle=solid,fillcolor=gray](0,0)(0,1)(.433,.75)
\pspolygon[fillstyle=solid,fillcolor=gray](0.866,1.5)(0.866,2.5)(1.3,2.25)
\pspolygon[fillstyle=solid,fillcolor=gray](0,3)(0,4)(.433,3.75)
\pspolygon[fillstyle=solid,fillcolor=gray](0,6)(0,7)(.433,6.75)
\pspolygon[fillstyle=solid,fillcolor=gray](0.866,4.5)(0.866,5.5)(1.3,5.25)
\pspolygon[fillstyle=solid,fillcolor=gray](1.732,3)(1.732,4)(2.165,3.75)
\pspolygon[fillstyle=solid,fillcolor=gray](1.732,6)(1.732,7)(2.165,6.75)
\pspolygon[fillstyle=solid,fillcolor=gray](2.6,4.5)(2.6,5.5)(3.033,5.25)
\pspolygon[fillstyle=solid,fillcolor=gray](3.464,6)(3.464,7)(3.897,6.75)
%\pspolygon[fillstyle=solid,fillcolor=red!5!yellow!42!](.866,1.5)(.866,4.5)(1.732,6)(1.732,3)
%\pspolygon[fillstyle=solid,fillcolor=red!5!yellow!42!](0,3)(0,6)(.866,7.5)(.866,4.5)

%\pspolygon[fillstyle=solid,fillcolor=yellow](0,3)(0,4)(.433,3.75)
%\pspolygon[fillstyle=solid,fillcolor=yellow](.866,1.5)(.866,2.5)(1.3,2.25)

\psline[linewidth=.3mm](0,0)(0,7.5)
\psline[linewidth=.3mm](0.866,1.5)(0.866,7.5)
\psline[linewidth=.3mm](1.732,3)(1.732,7.5)
\psline[linewidth=.3mm](2.6,4.5)(2.6,7.5)
\psline[linewidth=.3mm](3.464,6)(3.464,7.5)
\psline[linewidth=.3mm](0,0)(4.33,7.5)
\psline[linewidth=.3mm](0,3)(2.6,7.5)
\psline[linewidth=.3mm](0,6)(0.866,7.5)

%\pspolygon[linewidth=.3mm](0,0)(0,3)(.866,4.5)(.866,1.5)
%\pspolygon[linewidth=.3mm](.866,1.5)(.866,4.5)(1.732,6)(1.732,3)
%\pspolygon[linewidth=.3mm](0,3)(0,6)(.866,7.5)(.866,4.5)

\psline(0,0)(0,7.5)
\psline(0,0)(4.33,7.5)

\psline(0,1.5)(.866,1.5)
\psline(0,3)(1.732,3)
\psline(0,4.5)(2.598,4.5)
\psline(0,6)(3.464,6)
\psline(0,7.5)(4.33,7.5)

\psline(.866,1.5)(.866,7.5)
\psline(1.732,3)(1.732,7.5)
\psline(2.598,4.5)(2.598,7.5)
\psline(3.464,6)(3.464,7.5)

\psline(0,1)(.433,.75)
\psline(0,2)(.866,1.5)
\psline(0,3)(1.3,2.25)
\psline(0,4)(1.732,3)
\psline(0,5)(2.165,3.75)
\psline(0,6)(2.598,4.5)
\psline(0,7)(3.031,5.25)
\psline(.866,7.5)(3.464,6)
\psline(2.598,7.5)(3.897,6.75)

\psline(0,3)(.866,1.5)
\psline(0,6)(1.732,3)
\psline(2.598,4.5)(.866,7.5)
\psline(3.464,6)(2.598,7.5)

\psline(0,1)(.866,1.5)
\psline(0,2)(1.732,3)
\psline(0,3)(2.598,4.5)
\psline(0,4)(3.464,6)
\psline(0,5)(4.33,7.5)
\psline(0,6)(2.598,7.5)
\psline(0,7)(.866,7.5)

\psline(0,3)(2.598,7.5)
\psline(0,6)(.866,7.5)

\end{pspicture}
\end{center}
\end{subfigure}
\caption{The lowest two-sided cell $\Ga_0$}
\label{lowest}
\end{figure}
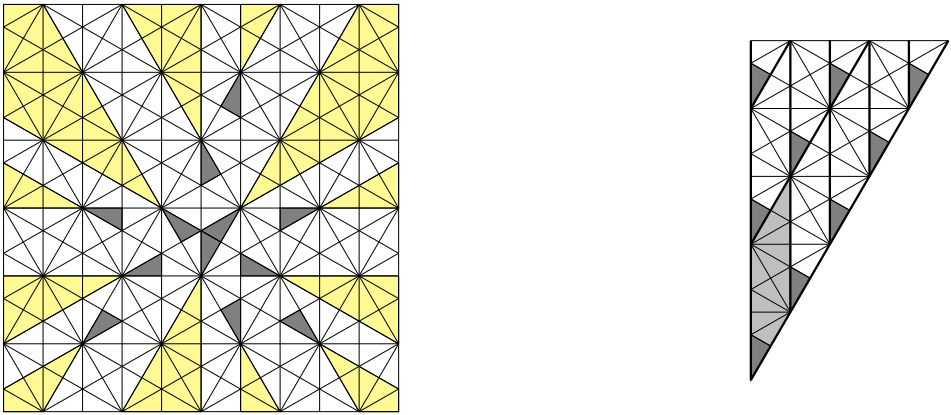
\medskip

We use the third property to define functions $\su,\sv: \Ga_0\lra \sB_{0}$ and $\tau: \Ga_0\ra \{t_{\lambda}\mid \lambda\in P^+\}$ by the equation
$w=\su(w)^{-1}\sw_0 \tau(w)\sv(w)$. We will usually write $\su_w,\sv_w$ and $\tau_w$ in place of $\su(w),\sv(w)$ and $\tau(w)$. Thus the cell factorisation of $w\in \Gamma_0$ is the expression
$$w=\su^{-1}_w\sw_0\tau_w\sv_w.$$

\subsection{Cell factorisation for the cells $\Gamma_1$ and $\Gamma_2$ with generic parameters}\label{sec:factor1}

Note that for each value of $r=a/b$ there are precisely three infinite two-sided cells (including the lowest two-sided cell~$\Gamma_0$). With reference to Figure~\ref{partition}, let $\Gamma_1$ be the green cell, and let $\Gamma_2$ be the blue cell. Note that the two-sided cells $\Gamma_1$ and $\Gamma_2$ are dependent on the choice of~$r$. 
It turns out that for most parameters $(a,b)$ the infinite two-sided cells $\Gamma_1$ and $\Gamma_2$ admit analogous cell factorisations to~$\Gamma_0$. Recall from above that:
 \smallskip

\begin{quotation}
\noindent {\bf Convention:} When speaking about the cell $\Ga_i$ in the ``non-generic case'', we will mean either the cell $\Ga_1$ in the case $r_1=a/b=2$ or the cell $\Ga_2$ in the case $r_2=a/b=3/2$. All other parameter values are generic for these cells.
\end{quotation}

\smallskip

With this convention, if $\Gamma\in\{\Gamma_1,\Gamma_2\}$ and $r$ is generic for $\Gamma$ then we have the following cell factorisation properties: let $\Ga^1,\ldots,\Ga^6$ be the right cells contained in $\Ga$. Then
\bem
\item Each right cell  $\Ga^k$ contains a unique element $\sw_{\Gamma^k}$ of minimal length.
\item There exists a unique element $\sw_\Ga\in\Ga$ of maximal length subject to the conditions that $\sw_{\Ga}$ lies in a finite parabolic subgroup of $W$ and $\sw_\Ga$ is a suffix of each $\sw_{\Ga^k}$ for all $1\leq k\leq 6$. The element $\sw_\Ga$ is called the \textit{generating element} of $\Ga$. We set $u_k=\sw_\Ga \sw_{\Ga^k}^{-1}$ for all $k$ and $\sB_{\Ga}=\{u_k\mid 1\leq k\leq 6\}$.
\item There exists $\st_{\Gamma}\in W$ such that 
$$
\Gamma=\{u^{-1}\sw_{\Gamma}\st_{\Gamma}^nv\mid u,v\in\sB_{\Gamma},n\in\mathbb{N}\},
$$
and moreover each $w\in\Gamma$ has a unique expression in the form $w=u^{-1}\sw_{\Gamma}\st_{\Gamma}^nv$ with $u,v\in\sB_{\Gamma}$ and $n\in\mathbb{N}$. 
\eem
\medskip

We will list the explicit cell factorisations for generic parameters below. Here, and elsewhere, we will employ the shorthand notation $s_{i_1}\cdots s_{i_k}\to i_1\cdots i_k$. For example $012$ is used to denote $s_0s_1s_2$. In particular, note that $1=s_1$ is not the identity; we will denote the identity of $W$ by~$e$.
\medskip

Explicitly, in each case the elements $\sw_{\Gamma}$ and $\st_{\Gamma}$, and the ``box'' $\sB_{\Gamma}$ are as follows. For $\Gamma_1$ there are 2 distinct generic regimes, given by $r>2$, and $r<2$ (see Figure~\ref{fig:gamma1}). We have
\begin{align*}
\sw_{{\Ga_1}}&=\begin{cases}
01&\mbox{ if $r>2$}\\
020&\mbox{ if $r<2$}
\end{cases}\qquad\st_{\Gamma_1}=\begin{cases}
210&\mbox{ if $r>2$}\\
120&\mbox{ if $r<2$}
\end{cases}\qquad\text{and}\qquad \sB_{{\Ga_1}}=\begin{cases}
\{e,2,20,21,212,2120\} &\mbox{ if $r>2$}\\
\{e,1,12,121,1212,12120\} &\mbox{ if $r<2$.}
\end{cases}
\end{align*}
Note that the translates of $\sB_{\Gamma_1}$ by $\st_{\Gamma_1}$ tessellate a ``strip'' in $W$.
\vspace*{-3mm}
\psset{unit=.4cm}
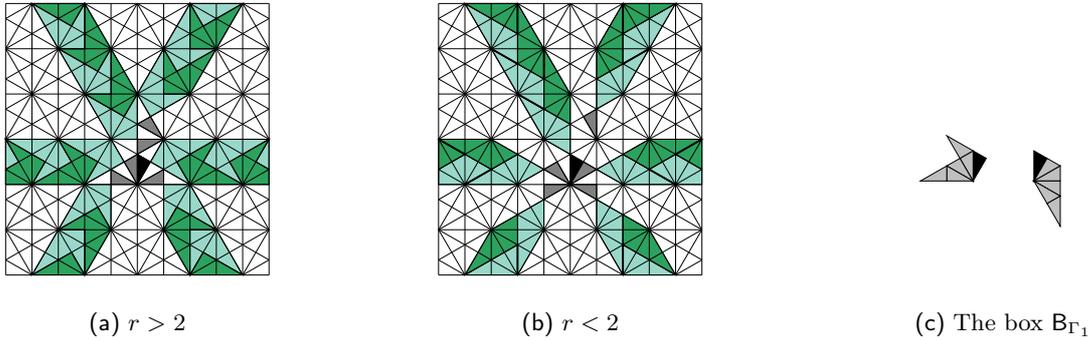
\begin{figure}[H]
\begin{subfigure}[b]{0.33\textwidth}
\begin{center}
\begin{pspicture}(-4.33,-3)(4.33,6)
\psset{linewidth=.05mm}

\pspolygon[fillstyle=solid, fillcolor=black](0,0)(0,1)(0.433,0.75)

\pspolygon[fillstyle=solid,fillcolor=green1](0.433,2.25)(0,3)(.866,4.5)(.866,3)(1.732,3)
\pspolygon[fillstyle=solid,fillcolor=green1](1.3,3.75)(.866,4.5)(1.732,6)(1.732,4.5)(2.598,4.5)
\pspolygon[fillstyle=solid,fillcolor=green1](2.165,5.25)(1.732,6)(3.464,6)

\pspolygon[fillstyle=solid,fillcolor=green2](0.866,3)(.866,4.5)(1.3,3.75)(2.598,4.5)(1.732,3)
\pspolygon[fillstyle=solid,fillcolor=green2](1.732,4.5)(1.732,6)(2.166,5.25)(3.464,6)(2.598,4.5)

\pspolygon[fillstyle=solid,fillcolor=green2](-0.433,2.25)(-0,3)(-.866,4.5)(-.866,3)(-1.732,3)
\pspolygon[fillstyle=solid,fillcolor=green2](-1.3,3.75)(-.866,4.5)(-1.732,6)(-1.732,4.5)(-2.598,4.5)
\pspolygon[fillstyle=solid,fillcolor=green2](-2.165,5.25)(-1.732,6)(-3.464,6)

\pspolygon[fillstyle=solid,fillcolor=green1](-0.866,3)(-.866,4.5)(-1.3,3.75)(-2.598,4.5)(-1.732,3)
\pspolygon[fillstyle=solid,fillcolor=green1](-1.732,4.5)(-1.732,6)(-2.166,5.25)(-3.464,6)(-2.598,4.5)
\pspolygon[fillstyle=solid,fillcolor=green1](0,1.5)(0,3)(-.433,2.25)(-1.732,3)(-.866,1.5)

%\pspolygon[fillstyle=solid,fillcolor=green!80!black!70!](-0.866,0)(-0.866,-1.5)(-1.732,-3)(-3.464,-3)(-1.732,0)
%\pspolygon[fillstyle=solid,fillcolor=green!80!black!70!](0.866,0)(0.866,-1.5)(1.732,-3)(3.464,-3)(1.732,0)

\pspolygon[fillstyle=solid,fillcolor=green1](0.433,0.75)(0.866,1.5)(2.598,1.5)(1.3,.75)(1.732,0)
\pspolygon[fillstyle=solid,fillcolor=green1](2.166,0.75)(2.598,1.5)(4.33,1.5)(3.032,.75)(3.464,0)
\pspolygon[fillstyle=solid,fillcolor=green1](3.897,.75)(4.33,1.5)(4.33,.5)

\pspolygon[fillstyle=solid,fillcolor=green2](1.3,.75)(2.598,1.5)(2.166,.75)(3.464,0)(1.732,0)
\pspolygon[fillstyle=solid,fillcolor=green2](3.032,.75)(4.33,1.5)(3.897,.75)(4.33,.5)(4.33,0)(3.464,0)

\pspolygon[fillstyle=solid,fillcolor=green1](-0.433,0.75)(-0.866,1.5)(-2.598,1.5)(-1.3,.75)(-1.732,0)
\pspolygon[fillstyle=solid,fillcolor=green1](-2.166,0.75)(-2.598,1.5)(-4.33,1.5)(-3.032,.75)(-3.464,0)
\pspolygon[fillstyle=solid,fillcolor=green1](-3.897,.75)(-4.33,1.5)(-4.33,.5)

\pspolygon[fillstyle=solid,fillcolor=green2](-1.3,.75)(-2.598,1.5)(-2.166,.75)(-3.464,0)(-1.732,0)
\pspolygon[fillstyle=solid,fillcolor=green2](-3.032,.75)(-4.33,1.5)(-3.897,.75)(-4.33,.5)(-4.33,0)(-3.464,0)

\pspolygon[fillstyle=solid,fillcolor=green1](-.866,0)(-1.732,0)(-2.598,-1.5)(-1.3,-.75)(-.866,-1.5)
\pspolygon[fillstyle=solid,fillcolor=green2](-2.598,-1.5)(-1.3,-.75)(-.866,-1.5)(-1.732,-3)(-1.732,-1.5)
\pspolygon[fillstyle=solid,fillcolor=green1](-2.598,-1.5)(-1.732,-1.5)(-1.732,-3)(-2.165,-2.25)(-3.464,-3)
\pspolygon[fillstyle=solid,fillcolor=green2](-1.732,-3)(-2.165,-2.25)(-3.464,-3)

\pspolygon[fillstyle=solid,fillcolor=green1](.866,0)(1.732,0)(2.598,-1.5)(1.3,-.75)(.866,-1.5)
\pspolygon[fillstyle=solid,fillcolor=green2](2.598,-1.5)(1.3,-.75)(.866,-1.5)(1.732,-3)(1.732,-1.5)
\pspolygon[fillstyle=solid,fillcolor=green1](2.598,-1.5)(1.732,-1.5)(1.732,-3)(2.165,-2.25)(3.464,-3)
\pspolygon[fillstyle=solid,fillcolor=green2](1.732,-3)(2.165,-2.25)(3.464,-3)

\pspolygon[fillstyle=solid,fillcolor=gray](0,0)(.866,0)(.866,.5)
\pspolygon[fillstyle=solid,fillcolor=gray](0,0)(-.866,0)(-.866,.5)
\pspolygon[fillstyle=solid,fillcolor=gray](0,0)(0,1)(-.433,.74)
\pspolygon[fillstyle=solid,fillcolor=gray](.433,2.25)(0,2)(.866,1.5)
\pspolygon[fillstyle=solid,fillcolor=gray](0,1.5)(0,1)(.866,1.5)

%%%%%%%%%%%%
%%%%%%%%%%%%
%%%%%%%%%%%%
%%%%%%%%%%%%
%%%HYPERPLANS TYPE G2
%%%%%%%%%%%%
%%%%%%%%%%%%
%%%%%%%%%%%%
%%%%%%%%%%%%
%%%%%%%%%%%%

% droite horizontale 
\psline(-4.33,6)(4.33,6)
\psline(-4.33,4.5)(4.33,4.5)
\psline(-4.33,3)(4.33,3)
\psline(-4.33,1.5)(4.33,1.5)
\psline(-4.33,0)(4.33,0)
\psline(-4.33,-1.5)(4.33,-1.5)
\psline(-4.33,-3)(4.33,-3)

% droite verticale
\psline(-4.33,-3)(-4.33,6)
\psline(-3.464,-3)(-3.464,6)
\psline(-2.598,-3)(-2.598,6)
\psline(-1.732,-3)(-1.732,6)
\psline(-.866,-3)(-.866,6)
\psline(0,-3)(0,6)
\psline(.866,-3)(.866,6)
\psline(1.732,-3)(1.732,6)
\psline(2.598,-3)(2.598,6)
\psline(3.464,-3)(3.464,6)
\psline(4.33,-3)(4.33,6)

%droite oblique
\psline(-4.33,5.5)(-3.464,6)
\psline(-4.33,4.5)(-1.732,6)
\psline(-4.33,3.5)(0,6)
\psline(-4.33,2.5)(1.732,6)
\psline(-4.33,1.5)(3.464,6)
\psline(-4.33,.5)(4.33,5.5)
\psline(-4.33,-.5)(4.33,4.5)
\psline(-4.33,-1.5)(4.33,3.5)
\psline(-4.33,-2.5)(4.33,2.5)
\psline(-3.464,-3)(4.33,1.5)
\psline(-1.732,-3)(4.33,.5)
\psline(0,-3)(4.33,-.5)
\psline(1.732,-3)(4.33,-1.5)
\psline(3.464,-3)(4.33,-2.5)

%droite oblique
\psline(4.33,5.5)(3.464,6)
\psline(4.33,4.5)(1.732,6)
\psline(4.33,3.5)(0,6)
\psline(4.33,2.5)(-1.732,6)
\psline(4.33,1.5)(-3.464,6)
\psline(4.33,.5)(-4.33,5.5)
\psline(4.33,-.5)(-4.33,4.5)
\psline(4.33,-1.5)(-4.33,3.5)
\psline(4.33,-2.5)(-4.33,2.5)
\psline(3.464,-3)(-4.33,1.5)
\psline(1.732,-3)(-4.33,.5)
\psline(0,-3)(-4.33,-.5)
\psline(-1.732,-3)(-4.33,-1.5)
\psline(-3.464,-3)(-4.33,-2.5)

%droite oblique
\psline(-4.33,-1.5)(-3.464,-3)
\psline(-4.33,1.5)(-1.732,-3)
\psline(-4.33,4.5)(0,-3)
\psline(-3.464,6)(1.732,-3)
\psline(-1.732,6)(3.464,-3)
\psline(0,6)(4.33,-1.5)
\psline(1.732,6)(4.33,1.5)
\psline(3.464,6)(4.33,4.5)

%droite oblique
\psline(4.33,-1.5)(3.464,-3)
\psline(4.33,1.5)(1.732,-3)
\psline(4.33,4.5)(0,-3)
\psline(3.464,6)(-1.732,-3)
\psline(1.732,6)(-3.464,-3)
\psline(0,6)(-4.33,-1.5)
\psline(-1.732,6)(-4.33,1.5)
\psline(-3.464,6)(-4.33,4.5)
\end{pspicture}
\end{center}
\vspace*{-1mm}
\caption{$r>2$}
\end{subfigure}
\begin{subfigure}[b]{0.33\textwidth}
\begin{center}
\begin{pspicture}(-4.33,-3)(4.33,6)
\psset{linewidth=.05mm}

\pspolygon[fillstyle=solid, fillcolor=black](0,0)(0,1)(0.433,0.75)

\pspolygon[fillstyle=solid,fillcolor=green1](0.866,2.5)(0.866,3.5)(2.598,4.5)(1.732,3)
\pspolygon[fillstyle=solid,fillcolor=green2](0.866,3.5)(1.732,4)(1.732,6)(0.866,4.5)
\pspolygon[fillstyle=solid,fillcolor=green1](1.732,4)(1.732,5)(3.464,6)(2.598,4.5)
\pspolygon[fillstyle=solid,fillcolor=green2](1.732,5)(1.732,6)(3.464,6)

\psline[linewidth=.25mm](1.732,6)(1.732,5)
\psline[linewidth=.25mm](1.732,4)(2.598,4.5)

\pspolygon[fillstyle=solid,fillcolor=green1](0,1)(0,2)(-1.732,3)(-.866,1.5)
\pspolygon[fillstyle=solid,fillcolor=green2](0,2)(-.866,2.5)(-.866,4.5)(0,3)
\pspolygon[fillstyle=solid,fillcolor=green1](-.866,2.5)(-.866,3.5)(-2.598,4.5)(-1.732,3)
\pspolygon[fillstyle=solid,fillcolor=green2](-.866,3.5)(-1.732,4)(-1.732,6)(-.866,4.5)
\pspolygon[fillstyle=solid,fillcolor=green1](-1.732,4)(-1.732,5)(-3.464,6)(-2.598,4.5)
\pspolygon[fillstyle=solid,fillcolor=green2](-1.732,5)(-1.732,6)(-3.464,6)

\psline[linewidth=.25mm](-1.732,6)(-1.732,5)
\psline[linewidth=.25mm](-.866,4.5)(-.866,3.5)
\psline[linewidth=.25mm](-2.598,4.5)(-1.732,4)
\psline[linewidth=.25mm](-1.732,3)(-.866,2.5)

\pspolygon[fillstyle=solid,fillcolor=green1](-.866,0.5)(-1.732,1)(-3.464,0)(-1.732,0)
\pspolygon[fillstyle=solid,fillcolor=green2](-1.732,1)(-2.598,1.5)(-4.33,1.5)(-2.598,.5)
\pspolygon[fillstyle=solid,fillcolor=green1](-3.464,0)(-2.598,.5)(-3.464,1)(-4.33,.5)(-4.33,0)
\pspolygon[fillstyle=solid,fillcolor=green2](-3.464,1)(-4.33,1.5)(-4.33,.5)

\psline[linewidth=.25mm](-3.464,1)(-4.33,1.5)
\psline[linewidth=.25mm](-2.598,.5)(-3.464,0)

%%%%%%%%%%%%%%%%%%

\pspolygon[fillstyle=solid,fillcolor=green1](.866,0.5)(1.732,1)(3.464,0)(1.732,0)
\pspolygon[fillstyle=solid,fillcolor=green2](1.732,1)(2.598,1.5)(4.33,1.5)(2.598,.5)
\pspolygon[fillstyle=solid,fillcolor=green1](3.464,0)(2.598,.5)(3.464,1)(4.33,.5)(4.33,0)
\pspolygon[fillstyle=solid,fillcolor=green2](3.464,1)(4.33,1.5)(4.33,.5)

\psline[linewidth=.25mm](3.464,1)(4.33,1.5)
\psline[linewidth=.25mm](2.598,.5)(3.464,0)

%%%%%%%%%%%%%%%%%%

\pspolygon[fillstyle=solid,fillcolor=green1](-.866,-.5)(-.866,-1.5)(-1.732,-3)(-1.732,-1)
\pspolygon[fillstyle=solid,fillcolor=green2](-1.732,-1)(-2.598,-1.5)(-3.464,-3)(-1.732,-2)
\pspolygon[fillstyle=solid,fillcolor=green1](-1.732,-2)(-3.464,-3)(-1.732,-3)

\psline[linewidth=.25mm](-1.732,-2)(-1.732,-3)

%%%%%%%%%%%%%%%%%%

\pspolygon[fillstyle=solid,fillcolor=green1](.866,-.5)(.866,-1.5)(1.732,-3)(1.732,-1)
\pspolygon[fillstyle=solid,fillcolor=green2](1.732,-1)(2.598,-1.5)(3.464,-3)(1.732,-2)
\pspolygon[fillstyle=solid,fillcolor=green1](1.732,-2)(3.464,-3)(1.732,-3)

\psline[linewidth=.25mm](1.732,-2)(1.732,-3)

\pspolygon[fillstyle=solid,fillcolor=gray](0,0)(.433,.75)(.866,.5)
\pspolygon[fillstyle=solid,fillcolor=gray](0,0)(-.433,.75)(-.866,.5)
\pspolygon[fillstyle=solid,fillcolor=gray](.866,1.5)(.866,2.5)(.433,2.25)
\pspolygon[fillstyle=solid,fillcolor=gray](0,0)(.866,0)(.866,-.5)
\pspolygon[fillstyle=solid,fillcolor=gray](0,0)(-.866,0)(-.866,-.5)

%%%%%%%%%%%%
%%%%%%%%%%%%
%%%%%%%%%%%%
%%%%%%%%%%%%
%%%HYPERPLANS TYPE G2
%%%%%%%%%%%%
%%%%%%%%%%%%
%%%%%%%%%%%%
%%%%%%%%%%%%
%%%%%%%%%%%%

% droite horizontale 
\psline(-4.33,6)(4.33,6)
\psline(-4.33,4.5)(4.33,4.5)
\psline(-4.33,3)(4.33,3)
\psline(-4.33,1.5)(4.33,1.5)
\psline(-4.33,0)(4.33,0)
\psline(-4.33,-1.5)(4.33,-1.5)
\psline(-4.33,-3)(4.33,-3)

% droite verticale
\psline(-4.33,-3)(-4.33,6)
\psline(-3.464,-3)(-3.464,6)
\psline(-2.598,-3)(-2.598,6)
\psline(-1.732,-3)(-1.732,6)
\psline(-.866,-3)(-.866,6)
\psline(0,-3)(0,6)
\psline(.866,-3)(.866,6)
\psline(1.732,-3)(1.732,6)
\psline(2.598,-3)(2.598,6)
\psline(3.464,-3)(3.464,6)
\psline(4.33,-3)(4.33,6)

%droite oblique
\psline(-4.33,5.5)(-3.464,6)
\psline(-4.33,4.5)(-1.732,6)
\psline(-4.33,3.5)(0,6)
\psline(-4.33,2.5)(1.732,6)
\psline(-4.33,1.5)(3.464,6)
\psline(-4.33,.5)(4.33,5.5)
\psline(-4.33,-.5)(4.33,4.5)
\psline(-4.33,-1.5)(4.33,3.5)
\psline(-4.33,-2.5)(4.33,2.5)
\psline(-3.464,-3)(4.33,1.5)
\psline(-1.732,-3)(4.33,.5)
\psline(0,-3)(4.33,-.5)
\psline(1.732,-3)(4.33,-1.5)
\psline(3.464,-3)(4.33,-2.5)

%droite oblique
\psline(4.33,5.5)(3.464,6)
\psline(4.33,4.5)(1.732,6)
\psline(4.33,3.5)(0,6)
\psline(4.33,2.5)(-1.732,6)
\psline(4.33,1.5)(-3.464,6)
\psline(4.33,.5)(-4.33,5.5)
\psline(4.33,-.5)(-4.33,4.5)
\psline(4.33,-1.5)(-4.33,3.5)
\psline(4.33,-2.5)(-4.33,2.5)
\psline(3.464,-3)(-4.33,1.5)
\psline(1.732,-3)(-4.33,.5)
\psline(0,-3)(-4.33,-.5)
\psline(-1.732,-3)(-4.33,-1.5)
\psline(-3.464,-3)(-4.33,-2.5)

%droite oblique
\psline(-4.33,-1.5)(-3.464,-3)
\psline(-4.33,1.5)(-1.732,-3)
\psline(-4.33,4.5)(0,-3)
\psline(-3.464,6)(1.732,-3)
\psline(-1.732,6)(3.464,-3)
\psline(0,6)(4.33,-1.5)
\psline(1.732,6)(4.33,1.5)
\psline(3.464,6)(4.33,4.5)

%droite oblique
\psline(4.33,-1.5)(3.464,-3)
\psline(4.33,1.5)(1.732,-3)
\psline(4.33,4.5)(0,-3)
\psline(3.464,6)(-1.732,-3)
\psline(1.732,6)(-3.464,-3)
\psline(0,6)(-4.33,-1.5)
\psline(-1.732,6)(-4.33,1.5)
\psline(-3.464,6)(-4.33,4.5)
\end{pspicture}
\end{center}
\vspace*{-1mm}
\caption{$r<2$}
\end{subfigure}
\begin{subfigure}[b]{0.33\textwidth}
\centering

\begin{pspicture}(-4.33,-4)(4.33,6.5)
\psset{linewidth=.1mm}

\pspolygon[fillstyle=solid,fillcolor=black](-1,0)(-1,1)(-.567,.75)
\pspolygon[fillstyle=solid,fillcolor=lightgray](-1,0)(-1,1)(-1.433,.75)
\pspolygon[fillstyle=solid,fillcolor=lightgray](-1,1)(-1.433,.75)(-1.866,1.5)
\pspolygon[fillstyle=solid,fillcolor=lightgray](-1,0)(-1.433,.75)(-1.866,.5)
\pspolygon[fillstyle=solid,fillcolor=lightgray](-1,0)(-1.866,.5)(-1.866,0)
\pspolygon[fillstyle=solid,fillcolor=lightgray](-2.732,0)(-1.866,.5)(-1.866,0)

\pspolygon[fillstyle=solid,fillcolor=black](1,0)(1,1)(1.433,.75)
\pspolygon[fillstyle=solid,fillcolor=lightgray](1,0)(1.433,.75)(1.866,.5)
\pspolygon[fillstyle=solid,fillcolor=lightgray](1,0)(1.866,.5)(1.866,0)
\pspolygon[fillstyle=solid,fillcolor=lightgray](1,0)(1.866,-.5)(1.866,0)
\pspolygon[fillstyle=solid,fillcolor=lightgray](1,0)(1.866,-.5)(1.433,-.75)
\pspolygon[fillstyle=solid,fillcolor=lightgray](1.866,-.5)(1.433,-.75)(1.866,-1.5)

\end{pspicture}

\vspace*{-1mm}
\caption{The box $\sB_{\Ga_1}$}
\label{translate-b1}
\end{subfigure}
\caption{The green cell $\Gamma_1$ in generic regimes ($r\neq 2$)}\label{fig:gamma1}
\end{figure}

For $\Gamma_2$ there are 2 distinct generic regimes, given by $r>3/2$, and $r<3/2$ (see Figure~\ref{fig:gamma2}). We have
\begin{align*}
\sw_{{\Ga_2}}&=\begin{cases}
12121&\mbox{ if $r>3/2$}\\
01&\mbox{ if $r<3/2$}
\end{cases}\quad\st_{\Gamma_2}=\begin{cases}
02121&\mbox{ if $r>3/2$}\\
21210&\mbox{ if $r<3/2$}
\end{cases}
\quad\text{and}\quad \sB_{{\Ga_2}}=\begin{cases}
\{e,0,02,021,0212,02120\} &\mbox{ if $r>3/2$}\\
\{e,2,21,212,2121,2120\} &\mbox{ if $r<3/2$.}
\end{cases}
\end{align*}
Note that the translates of $\sB_{\Gamma_2}$ by $\st_{\Gamma_2}$ tessellate a ``strip'' in $W$.
\vspace*{-3mm}
\psset{unit=.4cm}
\begin{figure}[H]
\begin{subfigure}[b]{0.33\textwidth}
\begin{center}
\begin{pspicture}(-4.33,-3)(4.33,6)
\psset{linewidth=.05mm}

\pspolygon[fillstyle=solid, fillcolor=black](0,0)(0,1)(0.433,0.75)

% c_0 lowest two-sided cell
%
%\pspolygon[fillstyle=solid,fillcolor=red!5!yellow!42!](0.866,4.5)(0.866,6)(1.732,6)
%\pspolygon[fillstyle=solid,fillcolor=red!5!yellow!42!](0.866,1.5)(3.46,6)(4.33,6)(4.33,3.5)
%\pspolygon[fillstyle=solid,fillcolor=red!5!yellow!42!](2.598,1.5)(4.33,1.5)(4.33,2.5)
%\pspolygon[fillstyle=solid,fillcolor=red!5!yellow!42!](1.732,0)(4.33,0)(4.33,-1.5)
%\pspolygon[fillstyle=solid,fillcolor=red!5!yellow!42!](2.598,-1.5)(4.33,-2.5)(4.33,-3)(3.464,-3)
%\pspolygon[fillstyle=solid,fillcolor=red!5!yellow!42!](0.866,-1.5)(0.866,-3)(1.732,-3)
%\pspolygon[fillstyle=solid,fillcolor=red!5!yellow!42!](0,0)(0,-3)(-1.732,-3)
%\pspolygon[fillstyle=solid,fillcolor=red!5!yellow!42!](-2.598,-1.5)(-3.464,-3)(-4.33,-3)(-4.33,-2.5)
%\pspolygon[fillstyle=solid,fillcolor=red!5!yellow!42!](-1.732,0)(-4.33,0)(-4.33,-1.5)
%\pspolygon[fillstyle=solid,fillcolor=red!5!yellow!42!](-2.598,1.5)(-4.33,1.5)(-4.33,2.5)
%\pspolygon[fillstyle=solid,fillcolor=red!5!yellow!42!](-0.866,1.5)(-3.46,6)(-4.33,6)(-4.33,3.5)
%\pspolygon[fillstyle=solid,fillcolor=red!5!yellow!42!](0,3)(0,6)(-1.732,6)

\pspolygon[fillstyle=solid,fillcolor=blue1](0,0)(0,-3)(.866,-1.5)
\pspolygon[fillstyle=solid,fillcolor=blue2](0,-3)(.866,-1.5)(.866,-3)

\pspolygon[fillstyle=solid,fillcolor=blue1](0,3)(0,6)(.866,4.5)
\pspolygon[fillstyle=solid,fillcolor=blue2](0,6)(.866,4.5)(.866,6)

\pspolygon[fillstyle=solid,fillcolor=blue1](.866,1.5)(2.598,1.5)(3.464,3)
\pspolygon[fillstyle=solid,fillcolor=blue2](2.598,1.5)(3.464,3)(4.33,3)(4.33,2.5)
\pspolygon[fillstyle=solid,fillcolor=blue1](3.464,3)(4.33,3.5)(4.33,3)

\pspolygon[fillstyle=solid,fillcolor=blue1](-.866,1.5)(-2.598,1.5)(-3.464,3)
\pspolygon[fillstyle=solid,fillcolor=blue2](-2.598,1.5)(-3.464,3)(-4.33,3)(-4.33,2.5)
\pspolygon[fillstyle=solid,fillcolor=blue1](-3.464,3)(-4.33,3.5)(-4.33,3)

%\pspolygon[fillstyle=solid,fillcolor=blue!70!black!80!](-.866,1.5)(-2.598,1.5)(-4.33,2.5)(-4.33,3.5)

\pspolygon[fillstyle=solid,fillcolor=blue1](1.732,0)(4.33,-1.5)(2.598,-1.5)
\pspolygon[fillstyle=solid,fillcolor=blue2](4.33,-2.5)(4.33,-1.5)(2.598,-1.5)

\pspolygon[fillstyle=solid,fillcolor=blue1](-1.732,0)(-4.33,-1.5)(-2.598,-1.5)
\pspolygon[fillstyle=solid,fillcolor=blue2](-4.33,-2.5)(-4.33,-1.5)(-2.598,-1.5)

\pspolygon[fillstyle=solid,fillcolor=gray](.433,.75)(0,1)(.866,1.5)
\pspolygon[fillstyle=solid,fillcolor=gray](.866,0)(.866,0.5)(1.732,0)
\pspolygon[fillstyle=solid,fillcolor=gray](-.866,0)(-.866,0.5)(-1.732,0)
\pspolygon[fillstyle=solid,fillcolor=gray](0,2)(0,3)(.433,2.25)
\pspolygon[fillstyle=solid,fillcolor=gray](0,1)(-.866,1.5)(-.433,.75)
%%%%%%%%%%%%
%%%%%%%%%%%%
%%%%%%%%%%%%
%%%%%%%%%%%%
%%%HYPERPLANS TYPE G2
%%%%%%%%%%%%
%%%%%%%%%%%%
%%%%%%%%%%%%
%%%%%%%%%%%%
%%%%%%%%%%%%

% droite horizontale 
\psline(-4.33,6)(4.33,6)
\psline(-4.33,4.5)(4.33,4.5)
\psline(-4.33,3)(4.33,3)
\psline(-4.33,1.5)(4.33,1.5)
\psline(-4.33,0)(4.33,0)
\psline(-4.33,-1.5)(4.33,-1.5)
\psline(-4.33,-3)(4.33,-3)

% droite verticale
\psline(-4.33,-3)(-4.33,6)
\psline(-3.464,-3)(-3.464,6)
\psline(-2.598,-3)(-2.598,6)
\psline(-1.732,-3)(-1.732,6)
\psline(-.866,-3)(-.866,6)
\psline(0,-3)(0,6)
\psline(.866,-3)(.866,6)
\psline(1.732,-3)(1.732,6)
\psline(2.598,-3)(2.598,6)
\psline(3.464,-3)(3.464,6)
\psline(4.33,-3)(4.33,6)

%droite oblique
\psline(-4.33,5.5)(-3.464,6)
\psline(-4.33,4.5)(-1.732,6)
\psline(-4.33,3.5)(0,6)
\psline(-4.33,2.5)(1.732,6)
\psline(-4.33,1.5)(3.464,6)
\psline(-4.33,.5)(4.33,5.5)
\psline(-4.33,-.5)(4.33,4.5)
\psline(-4.33,-1.5)(4.33,3.5)
\psline(-4.33,-2.5)(4.33,2.5)
\psline(-3.464,-3)(4.33,1.5)
\psline(-1.732,-3)(4.33,.5)
\psline(0,-3)(4.33,-.5)
\psline(1.732,-3)(4.33,-1.5)
\psline(3.464,-3)(4.33,-2.5)

%droite oblique
\psline(4.33,5.5)(3.464,6)
\psline(4.33,4.5)(1.732,6)
\psline(4.33,3.5)(0,6)
\psline(4.33,2.5)(-1.732,6)
\psline(4.33,1.5)(-3.464,6)
\psline(4.33,.5)(-4.33,5.5)
\psline(4.33,-.5)(-4.33,4.5)
\psline(4.33,-1.5)(-4.33,3.5)
\psline(4.33,-2.5)(-4.33,2.5)
\psline(3.464,-3)(-4.33,1.5)
\psline(1.732,-3)(-4.33,.5)
\psline(0,-3)(-4.33,-.5)
\psline(-1.732,-3)(-4.33,-1.5)
\psline(-3.464,-3)(-4.33,-2.5)

%droite oblique
\psline(-4.33,-1.5)(-3.464,-3)
\psline(-4.33,1.5)(-1.732,-3)
\psline(-4.33,4.5)(0,-3)
\psline(-3.464,6)(1.732,-3)
\psline(-1.732,6)(3.464,-3)
\psline(0,6)(4.33,-1.5)
\psline(1.732,6)(4.33,1.5)
\psline(3.464,6)(4.33,4.5)

%droite oblique
\psline(4.33,-1.5)(3.464,-3)
\psline(4.33,1.5)(1.732,-3)
\psline(4.33,4.5)(0,-3)
\psline(3.464,6)(-1.732,-3)
\psline(1.732,6)(-3.464,-3)
\psline(0,6)(-4.33,-1.5)
\psline(-1.732,6)(-4.33,1.5)
\psline(-3.464,6)(-4.33,4.5)
\end{pspicture}
\end{center}
\vspace*{-1mm}
\caption{$a/b>3/2$}
\end{subfigure}
\begin{subfigure}[b]{0.33\textwidth}
\begin{center}
\begin{pspicture}(-4.33,-3)(4.33,6)
\psset{linewidth=.05mm}

\pspolygon[fillstyle=solid, fillcolor=black](0,0)(0,1)(0.433,0.75)

%% c_0 lowest two-sided cell
%
%\pspolygon[fillstyle=solid,fillcolor=red!5!yellow!42!](0.866,4.5)(0.866,6)(1.732,6)
%\pspolygon[fillstyle=solid,fillcolor=red!5!yellow!42!](0.866,1.5)(3.46,6)(4.33,6)(4.33,3.5)
%\pspolygon[fillstyle=solid,fillcolor=red!5!yellow!42!](2.598,1.5)(4.33,1.5)(4.33,2.5)
%\pspolygon[fillstyle=solid,fillcolor=red!5!yellow!42!](1.732,0)(4.33,0)(4.33,-1.5)
%\pspolygon[fillstyle=solid,fillcolor=red!5!yellow!42!](2.598,-1.5)(4.33,-2.5)(4.33,-3)(3.464,-3)
%\pspolygon[fillstyle=solid,fillcolor=red!5!yellow!42!](0.866,-1.5)(0.866,-3)(1.732,-3)
%\pspolygon[fillstyle=solid,fillcolor=red!5!yellow!42!](0,0)(0,-3)(-1.732,-3)
%\pspolygon[fillstyle=solid,fillcolor=red!5!yellow!42!](-2.598,-1.5)(-3.464,-3)(-4.33,-3)(-4.33,-2.5)
%\pspolygon[fillstyle=solid,fillcolor=red!5!yellow!42!](-1.732,0)(-4.33,0)(-4.33,-1.5)
%\pspolygon[fillstyle=solid,fillcolor=red!5!yellow!42!](-2.598,1.5)(-4.33,1.5)(-4.33,2.5)
%\pspolygon[fillstyle=solid,fillcolor=red!5!yellow!42!](-0.866,1.5)(-3.46,6)(-4.33,6)(-4.33,3.5)
%\pspolygon[fillstyle=solid,fillcolor=red!5!yellow!42!](0,3)(0,6)(-1.732,6)

\pspolygon[fillstyle=solid,fillcolor=blue1](0.433,-.75)(0,-1)(0,-3)(.433,-2.25)(.866,-2.5)(.866,-1.5)
\pspolygon[fillstyle=solid,fillcolor=blue2](0,-3)(.433,-2.25)(.866,-2.5)(.866,-3)

\pspolygon[fillstyle=solid,fillcolor=blue1](0.433,2.25)(0,3)(0,4)(.433,3.75)(.866,4.5)(.866,2.5)
\pspolygon[fillstyle=solid,fillcolor=blue2](0.433,3.75)(0,4)(0,6)(.433,5.25)(.866,5.5)(.866,4.5)
\pspolygon[fillstyle=solid,fillcolor=blue1](0.433,5.25)(0,6)(.866,6)(.866,5.5)

\pspolygon[fillstyle=solid,fillcolor=blue1](.433,.75)(.866,1.5)(1.732,2)(1.732,1.5)(2.598,1.5)(.866,.5)
\pspolygon[fillstyle=solid,fillcolor=blue2](1.732,1.5)(1.732,2)(3.464,3)(3.031,2.25)(3.464,2)(2.598,1.5)
\pspolygon[fillstyle=solid,fillcolor=blue1](3.031,2.25)(3.464,3)(4.33,3.5)(4.33,2.5)(3.464,2)

\pspolygon[fillstyle=solid,fillcolor=blue1](-.433,.75)(-.866,1.5)(-1.732,2)(-1.732,1.5)(-2.598,1.5)(-.866,.5)
\pspolygon[fillstyle=solid,fillcolor=blue2](-1.732,1.5)(-1.732,2)(-3.464,3)(-3.031,2.25)(-3.464,2)(-2.598,1.5)
\pspolygon[fillstyle=solid,fillcolor=blue1](-3.031,2.25)(-3.464,3)(-4.33,3.5)(-4.33,2.5)(-3.464,2)

\pspolygon[fillstyle=solid,fillcolor=blue1](.866,0)(1.732,0)(2.598,-.5)(2.165,-.75)(2.598,-1.5)(.866,-.5)
\pspolygon[fillstyle=solid,fillcolor=blue2](2.165,-.75)(2.598,-.5)(4.33,-1.5)(3.464,-1.5)(3.464,-2)(2.598,-1.5)
\pspolygon[fillstyle=solid,fillcolor=blue1](3.464,-2)(3.464,-1.5)(4.33,-1.5)(4.33,-2.5)

\pspolygon[fillstyle=solid,fillcolor=blue1](-.866,0)(-1.732,0)(-2.598,-.5)(-2.165,-.75)(-2.598,-1.5)(-.866,-.5)
\pspolygon[fillstyle=solid,fillcolor=blue2](-2.165,-.75)(-2.598,-.5)(-4.33,-1.5)(-3.464,-1.5)(-3.464,-2)(-2.598,-1.5)
\pspolygon[fillstyle=solid,fillcolor=blue1](-3.464,-2)(-3.464,-1.5)(-4.33,-1.5)(-4.33,-2.5)

\pspolygon[fillstyle=solid,fillcolor=gray](0,0)(0,1)(-0.433,0.75)
\pspolygon[fillstyle=solid,fillcolor=gray](0,0)(-.866,0)(-.866,.5)
\pspolygon[fillstyle=solid,fillcolor=gray](0,0)(.866,0)(.866,.5)
\pspolygon[fillstyle=solid,fillcolor=gray](0,0)(.433,-.75)(.866,-.5)
\pspolygon[fillstyle=solid,fillcolor=gray](.866,1.5)(.433,2.25)(0,2)

%%%%%%%%%%%%
%%%%%%%%%%%%
%%%%%%%%%%%%
%%%%%%%%%%%%
%%%HYPERPLANS TYPE G2
%%%%%%%%%%%%
%%%%%%%%%%%%
%%%%%%%%%%%%
%%%%%%%%%%%%
%%%%%%%%%%%%

% droite horizontale 
\psline(-4.33,6)(4.33,6)
\psline(-4.33,4.5)(4.33,4.5)
\psline(-4.33,3)(4.33,3)
\psline(-4.33,1.5)(4.33,1.5)
\psline(-4.33,0)(4.33,0)
\psline(-4.33,-1.5)(4.33,-1.5)
\psline(-4.33,-3)(4.33,-3)

% droite verticale
\psline(-4.33,-3)(-4.33,6)
\psline(-3.464,-3)(-3.464,6)
\psline(-2.598,-3)(-2.598,6)
\psline(-1.732,-3)(-1.732,6)
\psline(-.866,-3)(-.866,6)
\psline(0,-3)(0,6)
\psline(.866,-3)(.866,6)
\psline(1.732,-3)(1.732,6)
\psline(2.598,-3)(2.598,6)
\psline(3.464,-3)(3.464,6)
\psline(4.33,-3)(4.33,6)

%droite oblique
\psline(-4.33,5.5)(-3.464,6)
\psline(-4.33,4.5)(-1.732,6)
\psline(-4.33,3.5)(0,6)
\psline(-4.33,2.5)(1.732,6)
\psline(-4.33,1.5)(3.464,6)
\psline(-4.33,.5)(4.33,5.5)
\psline(-4.33,-.5)(4.33,4.5)
\psline(-4.33,-1.5)(4.33,3.5)
\psline(-4.33,-2.5)(4.33,2.5)
\psline(-3.464,-3)(4.33,1.5)
\psline(-1.732,-3)(4.33,.5)
\psline(0,-3)(4.33,-.5)
\psline(1.732,-3)(4.33,-1.5)
\psline(3.464,-3)(4.33,-2.5)

%droite oblique
\psline(4.33,5.5)(3.464,6)
\psline(4.33,4.5)(1.732,6)
\psline(4.33,3.5)(0,6)
\psline(4.33,2.5)(-1.732,6)
\psline(4.33,1.5)(-3.464,6)
\psline(4.33,.5)(-4.33,5.5)
\psline(4.33,-.5)(-4.33,4.5)
\psline(4.33,-1.5)(-4.33,3.5)
\psline(4.33,-2.5)(-4.33,2.5)
\psline(3.464,-3)(-4.33,1.5)
\psline(1.732,-3)(-4.33,.5)
\psline(0,-3)(-4.33,-.5)
\psline(-1.732,-3)(-4.33,-1.5)
\psline(-3.464,-3)(-4.33,-2.5)

%droite oblique
\psline(-4.33,-1.5)(-3.464,-3)
\psline(-4.33,1.5)(-1.732,-3)
\psline(-4.33,4.5)(0,-3)
\psline(-3.464,6)(1.732,-3)
\psline(-1.732,6)(3.464,-3)
\psline(0,6)(4.33,-1.5)
\psline(1.732,6)(4.33,1.5)
\psline(3.464,6)(4.33,4.5)

%droite oblique
\psline(4.33,-1.5)(3.464,-3)
\psline(4.33,1.5)(1.732,-3)
\psline(4.33,4.5)(0,-3)
\psline(3.464,6)(-1.732,-3)
\psline(1.732,6)(-3.464,-3)
\psline(0,6)(-4.33,-1.5)
\psline(-1.732,6)(-4.33,1.5)
\psline(-3.464,6)(-4.33,4.5)
\end{pspicture}
\end{center}
\vspace*{-1mm}
\caption{$a/b<3/2$}
\end{subfigure}
\begin{subfigure}[b]{0.33\textwidth}
\begin{center}

\begin{pspicture}(-4.33,-4)(4.33,6)
\psset{linewidth=.1mm}

\pspolygon[fillstyle=solid,fillcolor=black](-2,0)(-2,1)(-1.567,.75)
\pspolygon[fillstyle=solid,fillcolor=lightgray](-2,1)(-1.567,.75)(-1.134,1.5)
\pspolygon[fillstyle=solid,fillcolor=lightgray](-2,1)(-1.134,1.5)(-2,1.5)
\pspolygon[fillstyle=solid,fillcolor=lightgray](-1.134,1.5)(-2,1.5)(-2,2)
\pspolygon[fillstyle=solid,fillcolor=lightgray](-1.134,1.5)(-2,2)(-1.567,2.25)
\pspolygon[fillstyle=solid,fillcolor=lightgray](-2,2)(-1.567,2.25)(-2,3)

\pspolygon[fillstyle=solid,fillcolor=black](1,0)(1,1)(1.433,.75)
\pspolygon[fillstyle=solid,fillcolor=lightgray](1,0)(1,1)(.567,.75)
\pspolygon[fillstyle=solid,fillcolor=lightgray](1,0)(.567,.75)(.134,.5)
\pspolygon[fillstyle=solid,fillcolor=lightgray](1,0)(.134,.5)(.134,0)
\pspolygon[fillstyle=solid,fillcolor=lightgray](1,0)(.134,-.5)(.134,0)
\pspolygon[fillstyle=solid,fillcolor=lightgray](-.732,0)(.134,.5)(.134,0)

\end{pspicture}
\end{center}
\vspace*{-1mm}
\caption{The box $\sB_{\Ga_2}$}
\label{translate-b2}
\end{subfigure}
\caption{The blue cell $\Gamma_2$ in generic regimes $(r\neq 3/2)$}\label{fig:gamma2}
\end{figure}
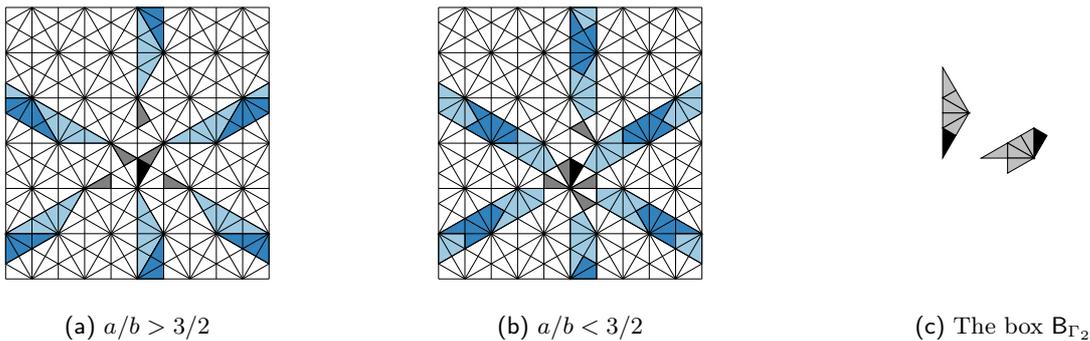
\medskip

We will typically write $\sw_i$, $\st_i$, and $\sB_i$ in place of $\sw_{\Gamma_i}$, $\st_{\Ga_i}$, and $\sB_{\Gamma_i}$ (for $i=1,2$).

\medskip

For $i\in\{1,2\}$ we use the third property of cell factorisation to define functions $\su,\sv: \Ga_i\lra \sB_{i}$ and $\tau: \Ga_i\ra \{\st_i^n\mid n\in\mathbb{N}\}$ by the equation
$w=\su(w)^{-1}\sw_{i} \tau(w)\sv(w)$. We will usually write $\su_w,\sv_w$ and $\tau_w$ in place of $\su(w),\sv(w)$ and $\tau(w)$. Thus the cell factorisation of $w\in \Gamma_i$ is the expression
$$w=\mathsf{u}^{-1}_w\sw_{i} \tau_w\mathsf{v}_w.$$

\begin{Rem}
\label{rem:cell-fact-finite} 
It is possible to have similar decompositions for most finite cells $\Ga$ when $r$ is generic for $\Ga$:
\bem
\item For $\Gamma_3$ there are 2 distinct generic regimes, given by $r>1$ and~$r<1$. When $r>1$, if we set $\sw_{\Ga_3}:=1$, $\st_{\Ga_3}:=21$ and $\sB_{\Ga_3}:=\{e,2,20\}$ then we have $
\Gamma_3=\{u^{-1}\sw_{\Ga_3}\st_{\Ga_3}^kv\mid u,v\in\sB_{\Ga_3},k\in\{0,1\}\}$.
\item For $\Gamma_4$ there are 2 distinct generic regimes, given by $r>1$ and $r<1$. When $r<1$ if we set $\sw_{\Ga_4}:=21212$ and $\sB_{\Ga_4}:=\{e,0\}$ then we have $
\Gamma_4=\{u^{-1}\sw_{\Ga_4}v\mid u,v\in\sB_{\Ga_4}\}$.
\item For $\Gamma_6$ there is only one regime given by $2>r>3/2$. If we set $\sw_{\Ga_6}:=10$ and $\sB_{\Ga_6}:=\{e,2,21,212,2120\}$ then we have  $
\Gamma_6=\{u^{-1}\sw_{\Ga_6}v\mid u,v\in\sB_{\Ga_3}\}$.
\item All other finite cells in generic parameters contain a unique element $\sw_\Ga$ and we set $\sB_{\Ga}=\{e\}$.
\eem
It is possible to have a similar description for the cells $\Ga_3$ when $r<1$ and $\Ga_4$ when $r>1$, however the notation becomes more technical due to the fact that the graph automorphism of the parabolic subgroup $W_{\{0,2\}}$ is involved in these cases.

\medskip

We will typically write $\sw_i$ and $\sB_i$ in place of $\sw_{\Gamma_i}$ and $\sB_{\Gamma_i}$ and $\st_3$ in place of $\st_{\Gamma_3}$. As above, when there is a cell factorisation for the finite cell $\Ga$, we obtain functions $\su,\sv$ on $\Ga$. For the two-sided cells $\Ga_3$ when $r>1$, we also have a function $\tau_3:\Ga_3\lra \{\st_3^k\mid k=0,1\}$.

\end{Rem}
\begin{Rem}
\label{rem:useful}
Let $w,w'\in \Ga_i$ where $\Ga_i$ is such that there is a cell factorisation. Then we have 
$$w\sim_{\cL} w'\eq \sv_{w}=\sv_{w'}\quand w\sim_{\cR} w'\eq \su_{w}=\su_{w'}.$$
Furthermore we note that $\tau(w^{-1})=\tau(w)$. Indeed if $w\in \Ga_i$ where $i=0,1,2,3$ then $w^{-1}=\sv^{-1}\tau_w^{-1}\sw_i\su=\sv^{-1}\sw_i\tau_w\su$.
\end{Rem}

%%%%%%%%%%%%%%%%%%%%%%%%%%%%%%%%%%%%%%%%%%%%%%%%%
%%%%%%%%%%%%%%%%%%%%%%%%%%%%%%%%%%%%%%%%%%%%%%%%%
%%%%%%%%%%%%%%%%%%%%%%%%%%%%%%%%%%%%%%%%%%%%%%%%%
%%%%%%%%%%%%%%%%%%%%%%%%%%%%%%%%%%%%%%%%%%%%%%%%%
%%%%%%%%%%%%%%%%%%%%%%%%%%%%%%%%%%%%%%%%%%%%%%%%%
%%%%%%%%%%%%%%%%%%%%%%%%%%%%%%%%%%%%%%%%%%%%%%%%%
%%%%%%%%%%%%%%%%%%%%%%%%%%%%%%%%%%%%%%%%%%%%%%%%%
%%%%%%%%%%%%%%%%%%%%%%%%%%%%%%%%%%%%%%%%%%%%%%%%%
%%%%%%%%%%%%%%%%%%%%%%%%%%%%%%%%%%%%%%%%%%%%%%%%%
%%%%%%%%%%%%%%%%%%%%%%%%%%%%%%%%%%%%%%%%%%%%%%%%%
%%%%%%%%%%%%%%%%%%%%%%%%%%%%%%%%%%%%%%%%%%%%%%%%%
%%%%%%%%%%%%%%%%%%%%%%%%%%%%%%%%%%%%%%%%%%%%%%%%%
%%%%%%%%%%%%%%%%%%%%%%%%%%%%%%%%%%%%%%%%%%%%%%%%%
%%%%%%%%%%%%%%%%%%%%%%%%%%%%%%%%%%%%%%%%%%%%%%%%%

\subsection{Cell factorisation for the cells $\Gamma_1$ and $\Gamma_2$ with non-generic parameters}\label{sec:equal}
Let $r_1=2$ and $r_2=3/2$.  The behaviour of the cell $\Ga_1$ when $r=r_1$ is similar to the behaviour of the cell $\Ga_2$ when $r=r_2$. The two-sided cell $\Ga_i$ is the union of $\Ga^+_{i}$, the two-sided cell in the generic case $a/b>r_i$ and $\Ga^-_{i}$, the two-sided cell in the generic case $a/b<r_i$ (in line with the \textit{semicontinuity conjecture} of Bonnaf\'e~\cite{Bon:09}).  More precisely we have
\begin{align*}
\Ga_1^+\backslash \Ga_1^-&=\{u^{-1}\sw_1^+v\mid u,v\in \sB_1^+\cap s_1\sB_1^-\},&\Ga_1^-\backslash \Ga_1^+&=\{\sw_1^-\},\\
\Ga_2^-\backslash \Ga_2^+&=\{u^{-1}\sw_1^+v\mid u,v\in \sB_2^-\cap s_0\sB_2^+\},& \Ga_2^+\backslash \Ga_2^-&=\{\sw_1^+\}.
\end{align*}
Furthermore, each right cell $\Upsilon\subset\Ga_i$ is either 
\bem
\item equal to a right cell in the case $a/b>r_i$, in which case we say $\Upsilon$ is of positive type;
\item equal to a right cell in the case $a/b<r_i$, in which case we say $\Upsilon$ is of negative type.
\eem

\begin{Def}
Let $w\in \Ga_i$. We say that $w$ is of type $(\eps_1,\eps_2)$ where $\eps_k=\pm$ if $w$ belongs to a  right cell of type $\eps_1$ and $w^{-1}$ belongs to a right cell of type $\eps_2$.
\end{Def}

It is immediate from the definition that if $w$ is of type $(\eps_1,\eps_2)$ then $w^{-1}$ will be of type $(\eps_2,\eps_1)$.
We represent the types of the elements of $\Ga_i$ in Figure \ref{equalp}: the dark blue, light blue, light red, dark red alcoves are respectively of type  $(-,-)$, $(-,+)$, $(+,-)$ and $(+,+)$. 

\medskip

\psset{unit=.5cm}
\begin{figure}[H]
\begin{subfigure}{.5\textwidth}
\begin{center}
\begin{pspicture}(-6.33,-3)(4.33,6.5)
\psset{linewidth=.05mm}

\pspolygon[fillstyle=solid, fillcolor=black](0,0)(0,1)(0.433,0.75)

\pspolygon[fillstyle=solid,fillcolor=pp](.433,.75)(.866,1.5)(2.598,1.5)(.866,.5)
\pspolygon[fillstyle=solid,fillcolor=pm](.866,.5)(1.3,.75)(1.732,0)
\pspolygon[fillstyle=solid,fillcolor=pp](1.3,.75)(1.732,1)(3.464,0)(1.732,0)
\pspolygon[fillstyle=solid,fillcolor=pm](1.732,1)(2.598,1.5)(2.165,.75)
\pspolygon[fillstyle=solid,fillcolor=pp](2.598,1.5)(2.165,.75)(2.598,.5)(4.33,1.5)
\pspolygon[fillstyle=solid,fillcolor=pm](2.598,.5)(3.031,.75)(3.464,0)
\pspolygon[fillstyle=solid,fillcolor=pp](3.464,0)(3.031,.75)(3.464,1)(4.33,.5)(4.33,0)
\pspolygon[fillstyle=solid,fillcolor=pm](3.464,1)(4.33,1.5)(3.897,.75)
\pspolygon[fillstyle=solid,fillcolor=pp](3.897,.75)(4.33,1.5)(4.33,.5)

\pspolygon[fillstyle=solid,fillcolor=pp](-.433,.75)(-.866,1.5)(-2.598,1.5)(-.866,.5)
\pspolygon[fillstyle=solid,fillcolor=pm](-.866,.5)(-1.3,.75)(-1.732,0)
\pspolygon[fillstyle=solid,fillcolor=pp](-1.3,.75)(-1.732,1)(-3.464,0)(-1.732,0)
\pspolygon[fillstyle=solid,fillcolor=pm](-1.732,1)(-2.598,1.5)(-2.165,.75)
\pspolygon[fillstyle=solid,fillcolor=pp](-2.598,1.5)(-2.165,.75)(-2.598,.5)(-4.33,1.5)
\pspolygon[fillstyle=solid,fillcolor=pm](-2.598,.5)(-3.031,.75)(-3.464,0)
\pspolygon[fillstyle=solid,fillcolor=pp](-3.464,0)(-3.031,.75)(-3.464,1)(-4.33,.5)(-4.33,0)
\pspolygon[fillstyle=solid,fillcolor=pm](-3.464,1)(-4.33,1.5)(-3.897,.75)
\pspolygon[fillstyle=solid,fillcolor=pp](-3.897,.75)(-4.33,1.5)(-4.33,.5)

\pspolygon[fillstyle=solid,fillcolor=pp](.866,0)(.866,-.5)(2.598,-1.5)(1.732,0)
\pspolygon[fillstyle=solid,fillcolor=pm](.866,-.5)(.866,-1.5)(1.3,-.75)
\pspolygon[fillstyle=solid,fillcolor=pp](1.3,-.75)(.866,-1.5)(1.732,-3)(1.732,-1)
\pspolygon[fillstyle=solid,fillcolor=pm](1.732,-1)(1.732,-1.5)(2.598,-1.5)
\pspolygon[fillstyle=solid,fillcolor=pp](1.732,-1.5)(1.732,-2)(3.464,-3)(2.598,-1.5)
\pspolygon[fillstyle=solid,fillcolor=pm](1.732,-2)(1.732,-3)(2.165,-2.25)
\pspolygon[fillstyle=solid,fillcolor=pp](2.165,-2.25)(1.732,-3)(2.598,-3)(2.598,-2.5)
\pspolygon[fillstyle=solid,fillcolor=pm](2.598,-2.5)(2.598,-3)(3.464,-3)

\pspolygon[fillstyle=solid,fillcolor=pp](-.866,0)(-.866,-.5)(-2.598,-1.5)(-1.732,0)
\pspolygon[fillstyle=solid,fillcolor=pm](-.866,-.5)(-.866,-1.5)(-1.3,-.75)
\pspolygon[fillstyle=solid,fillcolor=pp](-1.3,-.75)(-.866,-1.5)(-1.732,-3)(-1.732,-1)
\pspolygon[fillstyle=solid,fillcolor=pm](-1.732,-1)(-1.732,-1.5)(-2.598,-1.5)
\pspolygon[fillstyle=solid,fillcolor=pp](-1.732,-1.5)(-1.732,-2)(-3.464,-3)(-2.598,-1.5)
\pspolygon[fillstyle=solid,fillcolor=pm](-1.732,-2)(-1.732,-3)(-2.165,-2.25)
\pspolygon[fillstyle=solid,fillcolor=pp](-2.165,-2.25)(-1.732,-3)(-2.598,-3)(-2.598,-2.5)
\pspolygon[fillstyle=solid,fillcolor=pm](-2.598,-2.5)(-2.598,-3)(-3.464,-3)

\pspolygon[fillstyle=solid,fillcolor=pp](.433,2.25)(0,3)(.866,4.5)(.866,2.5)
\pspolygon[fillstyle=solid,fillcolor=pm](.866,2.5)(.866,3)(1.732,3)
\pspolygon[fillstyle=solid,fillcolor=pp](.866,3)(.866,3.5)(2.598,4.5)(1.732,3)
\pspolygon[fillstyle=solid,fillcolor=pm](.866,3.5)(.866,4.5)(1.3,3.75)
\pspolygon[fillstyle=solid,fillcolor=pp](1.3,3.75)(.866,4.5)(1.732,6)(1.732,4)
\pspolygon[fillstyle=solid,fillcolor=pm](1.732,4)(1.732,4.5)(2.598,4.5)
\pspolygon[fillstyle=solid,fillcolor=pp](1.732,4.5)(1.732,5)(3.464,6)(2.598,4.5)
\pspolygon[fillstyle=solid,fillcolor=pm](1.732,5)(1.732,6)(2.165,5.25)
\pspolygon[fillstyle=solid,fillcolor=pp](2.165,5.25)(1.732,6)(2.598,6)(2.598,5.5)
\pspolygon[fillstyle=solid,fillcolor=pm](2.598,5.5)(2.598,6)(3.464,6)

\pspolygon[fillstyle=solid,fillcolor=mm](0,1)(0,1.5)(-.866,1.5)
\pspolygon[fillstyle=solid,fillcolor=mp](0,1.5)(-.866,1.5)(-1.732,3)(0,2)
\pspolygon[fillstyle=solid,fillcolor=mm](0,2)(0,3)(-.433,2.25)
\pspolygon[fillstyle=solid,fillcolor=mp](-.433,2.25)(-.866,2.5)(-.866,4.5)(0,3)
\pspolygon[fillstyle=solid,fillcolor=mm](-.866,2.5)(-.866,3)(-1.732,3)
\pspolygon[fillstyle=solid,fillcolor=mp](-.866,3)(-1.732,3)(-2.598,4.5)(-.866,3.5)
\pspolygon[fillstyle=solid,fillcolor=mm](-.866,3.5)(-.866,4.5)(-1.3,3.75)
\pspolygon[fillstyle=solid,fillcolor=mm](-1.732,5)(-1.732,6)(-2.165,5.25)
\pspolygon[fillstyle=solid,fillcolor=mp](-2.165,5.25)(-2.598,5.5)(-2.598,6)(-1.732,6)

\pspolygon[fillstyle=solid,fillcolor=mp](-1.3,3.75)(-1.732,4)(-1.732,6)(-.866,4.5)

\pspolygon[fillstyle=solid,fillcolor=mm](-1.732,4)(-1.732,4.5)(-2.598,4.5)
\pspolygon[fillstyle=solid,fillcolor=mp](-1.732,4.5)(-2.598,4.5)(-3.464,6)(-1.732,5)

\pspolygon[fillstyle=solid,fillcolor=mm](-2.598,5.5)(-2.598,6)(-3.464,6)
%

%%%%%%%%%%%%%
%%%%%%%%%%%%
%%%%%%%%%%%%
%%%%%%%%%%%%
%%%HYPERPLANS TYPE G2
%%%%%%%%%%%%
%%%%%%%%%%%%
%%%%%%%%%%%%
%%%%%%%%%%%%
%%%%%%%%%%%%

% droite horizontale 
\psline(-4.33,6)(4.33,6)
\psline(-4.33,4.5)(4.33,4.5)
\psline(-4.33,3)(4.33,3)
\psline(-4.33,1.5)(4.33,1.5)
\psline(-4.33,0)(4.33,0)
\psline(-4.33,-1.5)(4.33,-1.5)
\psline(-4.33,-3)(4.33,-3)

% droite verticale
\psline(-4.33,-3)(-4.33,6)
\psline(-3.464,-3)(-3.464,6)
\psline(-2.598,-3)(-2.598,6)
\psline(-1.732,-3)(-1.732,6)
\psline(-.866,-3)(-.866,6)
\psline(0,-3)(0,6)
\psline(.866,-3)(.866,6)
\psline(1.732,-3)(1.732,6)
\psline(2.598,-3)(2.598,6)
\psline(3.464,-3)(3.464,6)
\psline(4.33,-3)(4.33,6)

%droite oblique
\psline(-4.33,5.5)(-3.464,6)
\psline(-4.33,4.5)(-1.732,6)
\psline(-4.33,3.5)(0,6)
\psline(-4.33,2.5)(1.732,6)
\psline(-4.33,1.5)(3.464,6)
\psline(-4.33,.5)(4.33,5.5)
\psline(-4.33,-.5)(4.33,4.5)
\psline(-4.33,-1.5)(4.33,3.5)
\psline(-4.33,-2.5)(4.33,2.5)
\psline(-3.464,-3)(4.33,1.5)
\psline(-1.732,-3)(4.33,.5)
\psline(0,-3)(4.33,-.5)
\psline(1.732,-3)(4.33,-1.5)
\psline(3.464,-3)(4.33,-2.5)

%droite oblique
\psline(4.33,5.5)(3.464,6)
\psline(4.33,4.5)(1.732,6)
\psline(4.33,3.5)(0,6)
\psline(4.33,2.5)(-1.732,6)
\psline(4.33,1.5)(-3.464,6)
\psline(4.33,.5)(-4.33,5.5)
\psline(4.33,-.5)(-4.33,4.5)
\psline(4.33,-1.5)(-4.33,3.5)
\psline(4.33,-2.5)(-4.33,2.5)
\psline(3.464,-3)(-4.33,1.5)
\psline(1.732,-3)(-4.33,.5)
\psline(0,-3)(-4.33,-.5)
\psline(-1.732,-3)(-4.33,-1.5)
\psline(-3.464,-3)(-4.33,-2.5)

%droite oblique
\psline(-4.33,-1.5)(-3.464,-3)
\psline(-4.33,1.5)(-1.732,-3)
\psline(-4.33,4.5)(0,-3)
\psline(-3.464,6)(1.732,-3)
\psline(-1.732,6)(3.464,-3)
\psline(0,6)(4.33,-1.5)
\psline(1.732,6)(4.33,1.5)
\psline(3.464,6)(4.33,4.5)

%droite oblique
\psline(4.33,-1.5)(3.464,-3)
\psline(4.33,1.5)(1.732,-3)
\psline(4.33,4.5)(0,-3)
\psline(3.464,6)(-1.732,-3)
\psline(1.732,6)(-3.464,-3)
\psline(0,6)(-4.33,-1.5)
\psline(-1.732,6)(-4.33,1.5)
\psline(-3.464,6)(-4.33,4.5)
\end{pspicture}
\end{center}
\end{subfigure}
\begin{subfigure}{.5\textwidth}
\begin{center}
\begin{pspicture}(-4.33,-3)(6.33,6.5)
\psset{linewidth=.05mm}

\pspolygon[fillstyle=solid, fillcolor=black](0,0)(0,1)(0.433,0.75)

\pspolygon[fillstyle=solid,fillcolor=mm](.433,.75)(.866,1.5)(1.732,1.5)(2.598,1.5)(.866,.5)
\pspolygon[fillstyle=solid,fillcolor=mp](.866,1.5)(1.732,1.5)(1.732,2)
\pspolygon[fillstyle=solid,fillcolor=mp](3.464,3)(4.33,3)(4.33,3.5)
\pspolygon[fillstyle=solid,fillcolor=mm](1.732,1.5)(1.732,2)(3.464,3)(2.598,1.5)
\pspolygon[fillstyle=solid,fillcolor=mp](2.598,1.5)(3.464,2)(3.032,2.25)
\pspolygon[fillstyle=solid,fillcolor=mm](3.031,2.25)(3.464,3)(4.33,3)(4.33,2.5)(3.464,2)

\pspolygon[fillstyle=solid,fillcolor=mm](-.433,.75)(-.866,1.5)(-1.732,1.5)(-2.598,1.5)(-.866,.5)
\pspolygon[fillstyle=solid,fillcolor=mp](-.866,1.5)(-1.732,1.5)(-1.732,2)
\pspolygon[fillstyle=solid,fillcolor=mp](-3.464,3)(-4.33,3)(-4.33,3.5)
\pspolygon[fillstyle=solid,fillcolor=mm](-1.732,1.5)(-1.732,2)(-3.464,3)(-2.598,1.5)
\pspolygon[fillstyle=solid,fillcolor=mp](-2.598,1.5)(-3.464,2)(-3.032,2.25)
\pspolygon[fillstyle=solid,fillcolor=mm](-3.031,2.25)(-3.464,3)(-4.33,3)(-4.33,2.5)(-3.464,2)

\pspolygon[fillstyle=solid,fillcolor=mm](.866,0)(1.732,0)(2.598,-1.5)(.866,-.5)
\pspolygon[fillstyle=solid,fillcolor=mp](1.732,0)(2.598,-.5)(2.165,-.75)
\pspolygon[fillstyle=solid,fillcolor=mp](2.598,-1.5)(3.464,-1.5)(3.464,-2)
\pspolygon[fillstyle=solid,fillcolor=mm](2.165,-.75)(2.598,-.5)(4.33,-1.5)(2.598,-1.5)
\pspolygon[fillstyle=solid,fillcolor=mm](3.464,-1.5)(3.464,-2)(4.33,-2.5)(4.33,-1.5)
%\pspolygon[fillstyle=solid,fillcolor=mm](3.031,2.25)(3.464,3)(4.33,3)(4.33,2.5)(3.464,2)

\pspolygon[fillstyle=solid,fillcolor=mm](-.866,0)(-1.732,0)(-2.598,-1.5)(-.866,-.5)
\pspolygon[fillstyle=solid,fillcolor=mp](-1.732,0)(-2.598,-.5)(-2.165,-.75)
\pspolygon[fillstyle=solid,fillcolor=mp](-2.598,-1.5)(-3.464,-1.5)(-3.464,-2)
\pspolygon[fillstyle=solid,fillcolor=mm](-2.165,-.75)(-2.598,-.5)(-4.33,-1.5)(-2.598,-1.5)
\pspolygon[fillstyle=solid,fillcolor=mm](-3.464,-1.5)(-3.464,-2)(-4.33,-2.5)(-4.33,-1.5)

\pspolygon[fillstyle=solid,fillcolor=mm](.433,2.25)(0,3)(.866,4.5)(.866,2.5)
\pspolygon[fillstyle=solid,fillcolor=mp](0,3)(0,4)(.433,3.75)
\pspolygon[fillstyle=solid,fillcolor=mm](0,4)(0,6)(.866,4.5)(.433,3.75)
\pspolygon[fillstyle=solid,fillcolor=mp](.866,4.5)(.866,5.5)(.433,5.25)
\pspolygon[fillstyle=solid,fillcolor=mm](.433,5.25)(0,6)(.866,6)(.866,5.5)

\pspolygon[fillstyle=solid,fillcolor=pp](0,0)(0,-1)(.433,-.75)
\pspolygon[fillstyle=solid,fillcolor=pm](0,-1)(0,-3)(.866,-1.5)(.433,-.75)
\pspolygon[fillstyle=solid,fillcolor=pp](.866,-1.5)(.866,-2.5)(.433,-2.25)
\pspolygon[fillstyle=solid,fillcolor=pm](.433,-2.25)(0,-3)(.866,-3)(.866,-2.5)

%%%%%%%%%%%%%
%%%%%%%%%%%%
%%%%%%%%%%%%
%%%%%%%%%%%%
%%%HYPERPLANS TYPE G2
%%%%%%%%%%%%
%%%%%%%%%%%%
%%%%%%%%%%%%
%%%%%%%%%%%%
%%%%%%%%%%%%

% droite horizontale 
\psline(-4.33,6)(4.33,6)
\psline(-4.33,4.5)(4.33,4.5)
\psline(-4.33,3)(4.33,3)
\psline(-4.33,1.5)(4.33,1.5)
\psline(-4.33,0)(4.33,0)
\psline(-4.33,-1.5)(4.33,-1.5)
\psline(-4.33,-3)(4.33,-3)

% droite verticale
\psline(-4.33,-3)(-4.33,6)
\psline(-3.464,-3)(-3.464,6)
\psline(-2.598,-3)(-2.598,6)
\psline(-1.732,-3)(-1.732,6)
\psline(-.866,-3)(-.866,6)
\psline(0,-3)(0,6)
\psline(.866,-3)(.866,6)
\psline(1.732,-3)(1.732,6)
\psline(2.598,-3)(2.598,6)
\psline(3.464,-3)(3.464,6)
\psline(4.33,-3)(4.33,6)

%droite oblique
\psline(-4.33,5.5)(-3.464,6)
\psline(-4.33,4.5)(-1.732,6)
\psline(-4.33,3.5)(0,6)
\psline(-4.33,2.5)(1.732,6)
\psline(-4.33,1.5)(3.464,6)
\psline(-4.33,.5)(4.33,5.5)
\psline(-4.33,-.5)(4.33,4.5)
\psline(-4.33,-1.5)(4.33,3.5)
\psline(-4.33,-2.5)(4.33,2.5)
\psline(-3.464,-3)(4.33,1.5)
\psline(-1.732,-3)(4.33,.5)
\psline(0,-3)(4.33,-.5)
\psline(1.732,-3)(4.33,-1.5)
\psline(3.464,-3)(4.33,-2.5)

%droite oblique
\psline(4.33,5.5)(3.464,6)
\psline(4.33,4.5)(1.732,6)
\psline(4.33,3.5)(0,6)
\psline(4.33,2.5)(-1.732,6)
\psline(4.33,1.5)(-3.464,6)
\psline(4.33,.5)(-4.33,5.5)
\psline(4.33,-.5)(-4.33,4.5)
\psline(4.33,-1.5)(-4.33,3.5)
\psline(4.33,-2.5)(-4.33,2.5)
\psline(3.464,-3)(-4.33,1.5)
\psline(1.732,-3)(-4.33,.5)
\psline(0,-3)(-4.33,-.5)
\psline(-1.732,-3)(-4.33,-1.5)
\psline(-3.464,-3)(-4.33,-2.5)

%droite oblique
\psline(-4.33,-1.5)(-3.464,-3)
\psline(-4.33,1.5)(-1.732,-3)
\psline(-4.33,4.5)(0,-3)
\psline(-3.464,6)(1.732,-3)
\psline(-1.732,6)(3.464,-3)
\psline(0,6)(4.33,-1.5)
\psline(1.732,6)(4.33,1.5)
\psline(3.464,6)(4.33,4.5)

%droite oblique
\psline(4.33,-1.5)(3.464,-3)
\psline(4.33,1.5)(1.732,-3)
\psline(4.33,4.5)(0,-3)
\psline(3.464,6)(-1.732,-3)
\psline(1.732,6)(-3.464,-3)
\psline(0,6)(-4.33,-1.5)
\psline(-1.732,6)(-4.33,1.5)
\psline(-3.464,6)(-4.33,4.5)
\end{pspicture}
\end{center}
\end{subfigure}
\caption{$(\eps_1,\eps_2)$-type in $\Ga_i$.}
\label{equalp}
\end{figure}

We denote by $\sw_i^\eps$, $\sB_i^\eps$, and $\sP_i^\eps$ the data associated to $\Ga^\eps_i$ where $\eps=\pm$.  We define $\su_{\eps},\sv_\eps: \Ga_i^\eps\lra \sB^\eps_{i}$ and $\tau^\eps: \Ga^\eps_i\ra \{\st_{i,\eps}^n\mid n\in \nN\}$ by the equation
$w=\su_\eps(w)^{-1}\sw^\eps_i \tau_\eps(w)\sv_\eps(w)$. 
In the case where $i=1$, we extend the definition of $\su_\eps,\sv_\eps$ and $\tau_\eps$ by setting 
for all $u,v\in \sB_1^+\cap s_1\sB_1^-$
$$
\begin{array}{llllll}
\su_+(\sw_1^-):=s_2s_0 & \su_-(u^{-1}\sw_1^+v)=s_1u,\\
\sv_+(\sw_1^-):=s_2s_0& \su_-(u^{-1}\sw_1^+v)=s_1v,\\
\tau_+(\sw_1^-):=-1& \tau_-(u^{-1}\sw_1^+v)=-1.\\
\end{array}
$$
Similarly when $i=2$, we extend the definition of $\su_\eps,\sv_\eps$ and $\tau_\eps$ by setting 
for all $u,v\in \sB_2^-\cap s_0\sB_2^+$
$$
\begin{array}{llllll}
\su_-(\sw_2^+):=s_2s_1s_2s_1 & \su_+(u^{-1}\sw_2^-v)=s_0u,\\
\sv_-(\sw_2^+):=s_2s_1s_2s_1& \su_+(u^{-1}\sw_2^-v)=s_0v,\\
\tau_-(\sw_2^+):=-1& \tau_+(u^{-1}\sw_2^-v)=-1.\\
\end{array}
$$
These definitions are coherent since we have 
\bem
\item for all $w\in \Ga_i$ and $\eps=\pm$ we have $w=\su^{-1}_\eps(w)\sw_i^\eps\tau_\eps(w)\sv_\eps(w)$;
\item for all $w,w'\in \Ga_i$, $w\sim_{\cL} w'$ if and only if $\sv_\eps(w)=\sv_\eps(w')$;
\item for all $w,w'\in \Ga_i$, $w\sim_{\cR} w'$ if and only if $\su_\eps(w)=\su_\eps(w')$.
\eem
The relation between those two expressions when $i=1$ are as follows
\bem
\item if $w$ is of type $(+,+)$ then $\su_-(w)=s_1\su_+(w)$, $\sv_-(w)=s_1\sv_+(w)$ and $\tau_-(w)=\tau_+(w)-1$;
\item if $w$ is of type $(-,-)$ then $\su_+(w)=s_0s_2=\sv_+(w)$ and $\tau_+(w)=\tau_-(w)-1$;
\item if $w$ is of type $(+,-)$ then $\su_-(w)=s_1\su_+(w)$, $\sv_-(w)=s_0s_2$ and $\tau_-(w)=\tau_+(w)$;
\item if $w$ is of type $(-,+)$ then $\su_-(w)=s_0s_2$, $\sv_-(w)=s_1\sv_+(w)$ and $\tau_-(w)=\tau_+(w)$.
\eem
There are similar formulas for $i=2$.

%%%%%%%%%%%%%%%%%%%%%%%%%%%%%%%%%%%
%%%%%%%%%%%%%%%%%%%%%%%%%%%%%%%%%%%
%%%%%%%%%%%%%%%%%%%%%%%%%%%%%%%%%%%
%%%%%%%%%%%%%%%%%%%%%%%%%%%%%%%%%%%
%%%%%%%%%%%%%%%%%%%%%%%%%%%%%%%%%%%
%%%%%%%%%%%%%%%%%%%%%%%%%%%%%%%%%%%
%%%%%%%%%%%%%%%%%%%%%%%%%%%%%%%%%%%
%%%%%%%%%%%%%%%%%%%%%%%%%%%%%%%%%%%
%%%%%%%%%%%%%%%%%%%%%%%%%%%%%%%%%%%

\section{Cell representations in type $G_2$}\label{sec:balanced}

In this section we prove that each finite cell admits a finite dimensional representation satisfying $\B{1}$--$\B{4}$ and $\B{4}'$. Moreover, we show that each infinite cell admits a finite dimensional representation satisfying $\B{1}$.   
\medskip
  
We will use the following notation. We write $E_{i,j}$ for the square matrix with $1$ in the $(i,j)$ place, and zeros elsewhere (the dimension of the matrix will be clear from context). For $i,j\in\mathbb{Z}$ we write $\mu_{i,j}=\sq^{ia-jb}+\sq^{-ia+jb}$.

%%%%%%%%%%%%%%%%%%%%%%%%%%%%%%%%%%%
%%%%%%%%%%%%%%%%%%%%%%%%%%%%%%%%%%%
%%%%%%%%%%%%%%%%%%%%%%%%%%%%%%%%%%%
%%%%%%%%%%%%%%%%%%%%%%%%%%%%%%%%%%%
%%%%%%%%%%%%%%%%%%%%%%%%%%%%%%%%%%%
%%%%%%%%%%%%%%%%%%%%%%%%%%%%%%%%%%%
%%%%%%%%%%%%%%%%%%%%%%%%%%%%%%%%%%%
%%%%%%%%%%%%%%%%%%%%%%%%%%%%%%%%%%%
%%%%%%%%%%%%%%%%%%%%%%%%%%%%%%%%%%%

\subsection{Finite cells}

Let $\Ga$ be a finite two-sided cell and let $\Up$ be a right cell lying in $\Ga$. By Table~\ref{partition}, $\Gamma$ intersects a dihedral parabolic subgroup~$W_I$, and we set 
$$
\tilde{\ba}_{\Gamma}=\ba_I(z)\quad\text{for any $z\in\Gamma\cap W_I$}
$$
(here $\ba_I$ is Lusztig's $\ba$-function on~$W_I$). It is easily verified, using Table~\ref{dihedral-cell}, that this is well defined. 

\medskip

We write $\rho\sim \Up$ to indicate that $\rho$ is the cell module over $\sR$ associated to $\Up$ equipped with the natural Kazhdan-Lusztig basis as in Section~\ref{sec:1.2}. From the data in Figure~\ref{partition}, we see that $\Upsilon_{\geq_\cLR}$ and $\Upsilon_{>_\cLR}$ are also finite subsets of~$W$.

\begin{Th}\label{thm:finitecellbalanced} Let $\Gamma$ be a finite two-sided cell. If $(\Gamma,r)\neq (\Gamma_3,1)$ let $\Up$ be any right cell contained in $\Ga$ and let $\rho\sim \Up$. If $(\Gamma,r)=(\Gamma_3,1)$ let $\rho$ be the direct sum of the cell representations for each of the right cells contained in $\Gamma$. Then $\rho$ satisfies $\B{1}$--$\B{4}$ and $\B{4}'$ with $\ba_{\rho}=\tba_{\Gamma}$. Thus $\rho$ is $\Ga$-balanced over $\sR$.
\end{Th}

\begin{proof}
We have already noted in Section~\ref{sec:def-balanced} that $\rho$ satisfies $\B{1}$. To check $\B{2}$, note that the set $\Upsilon_{\geq_\cLR}$ is finite, and hence it is clear that there exists $M\geq 0$ such that $\deg([\rho(C_{w})]_{i,j})\leq M$ for all $w\in \Upsilon_{\geq_{\cLR}}$. Since $\rho$ satisfies $\B{1}$ we have $\rho(C_w)=0$ if $w\notin \Upsilon_{\geq_{\cLR}}$, so $\B{2}$ holds.
\medskip

We now verify $\B{3}$, $\B{4}$ and $\B{4}'$. Since $\rho(C_w)=0$ if $w\notin \Upsilon_{\geq_{\cLR}}$ it is sufficient to look at the matrices $\rho(C_{w})$ where $w$ lies in the finite set $\Upsilon_{\geq_{\cLR}}$.

\medskip

We start by treating the~$1$-dimensional cells. There are 4 such two-sided cells, $\Ga_e=\{e\}$ (in all parameter regimes) and the cells $\Gamma_5=\{s_0s_2s_0\}$ (for $r>2$), $\Gamma_7=\{s_1s_2s_1s_2s_1\}$ (for $3/2>r>1$) and $\Gamma_7=\{s_1\}$ (for $r<1$). The associated cell modules are $\rho_I$ where $I=\emptyset,\{s_1\}$ or $\{s_0,s_2\}$ (see Example~\ref{exa:balanced-one-dim}). We now verify $\B{3}$ for each of these cells. Then $\B{4}$ and $\B{4}'$ are obvious since there is only one leading matrix, and it is just a nonzero element of $\sR$. 
\bem
\item We have $\rho_{\emptyset}\sim \Ga_e$ and since $\max\deg(\rho_\emptyset(\Ga_e))=0=\tba_{\Gamma_e}$ the result is clear. 
\item When $r>2$, we have $\rho_{I}\sim \Ga_5$ where $I=\{0,2\}$.  We have ${\Ga_5}_{\leq_{\cLR}}=\Ga_5\cup \Ga_4\cup \Ga_e$ and by direct calculation $\max\deg(\rho(\Ga_5))=3b$ and $\max \deg(\rho(\Ga_4))=2b$. This shows that $\ba_\rho=3b=\tba_{\Gamma_5}$ and hence $\B{3}$. 
\item When $3/2>r>1$, we have $\rho_{I}\sim \Ga_7$ where $I=\{1\}$. We have ${\Ga_7}_{\leq_{\cLR}}=\Ga_7\cup \Ga_3\cup \Ga_4\cup \Ga_e$ and 
\begin{center}
$\max\deg(\rho(\Ga_7))=3a-2b,\quad
\max \deg(\rho(\Ga_3))=2a-b,\quad\text{and}\quad 
 \max \deg(\rho(\Ga_4))=-b.$
 \end{center}
 This shows that  $\ba_\rho=3a-2b=\tba_{\Ga_7}$ and hence $\B{3}$ holds.
 \item When $r<1$, we have $\rho_{I}\sim \Ga_7$ where $I=\{1\}$. We have ${\Ga_7}_{\leq_{\cLR}}=\Ga_7\cup \Ga_e$ and $
\max\deg(\rho(\Ga_7))=a$. 
 This shows that  $\ba_\rho=a=\tba_{\Ga_7}$ and hence $\B{3}$ holds. 
 \eem

We now consider the remaining finite cells. Consider $\Ga_6$, which occurs in the regime $2>r>3/2$ only. Let $\rho\sim \Up$ where $\Up$ is a right cell included in $\Ga_6$. Thus $\rho$ is a $5$-dimensional representation with basis indexed by the elements of $\Upsilon$. To be concrete we will take $\Upsilon=\{s_1s_0,s_1s_0s_2,s_1s_0s_2s_1,s_1s_0s_2s_1s_2,s_1s_0s_2s_1s_2s_0\}$, however it turns out that the representations for the right cells are pairwise isomorphic. Then the matrices of $T_{s_1}$, $T_{s_2}$, and $T_{s_0}$ are, respectively,
$$
\begin{psmallmatrix}
\sq^a&\mu_{1,1}&0&1&-\mu_{2,3}\\
0&-\sq^{-a}&0&0&0\\
0&1&\sq^a&\mu_{1,1}&0\\
0&0&0&-\sq^{-a}&0\\
0&0&0&0&-\sq^{-a}
\end{psmallmatrix},\,
\begin{psmallmatrix}
-\sq^{-b}&0&0&0&0\\
1&\sq^b&0&0&0\\
0&0&-\sq^{-b}&0&0\\
0&0&1&\sq^b&1\\
0&0&0&0&-\sq^{-b}
\end{psmallmatrix},\,
\begin{psmallmatrix}
\sq^b&1&-\mu_{1,2}&0&0\\
0&-\sq^{-b}&0&0&0\\
0&0&-\sq^{-b}&0&0\\
0&0&0&-\sq^{-b}&0\\
0&0&0&1&\sq^b
\end{psmallmatrix}.
$$
We have ${\Ga_6}_{\leq_{\cLR}}=\Ga_6\cup \Ga_3\cup \Ga_4\cup \Ga_e$ and we check by direct computation that  
\begin{align*}
\max\deg(\rho(\Ga_6))=a+b,\quad
\max \deg(\rho(\Ga_3))=a,\quad\text{and}\quad
\max \deg(\rho(\Ga_4))=b.
\end{align*}
This shows that $\ba_\rho=a+b=\tba_{\Ga_6}$ and $\B{3}$ holds. To verify $\B{4}$ requires further computation. Recall that any $w\in \Ga_6$ can be written in a unique way in the form $u^{-1}s_1s_0v$ where $u,v\in \sB_6$, see Remark~\ref{rem:cell-fact-finite}. Again by direct computation we see that
$$
\fc_{\rho,w}=E_{s_1s_0u,s_1s_0v}\quad\text{if $w=u^{-1}s_1s_0v$ with $u,v\in\sB_6$}
$$
(recall that the rows and columns of the matrices for $\rho(T_w)$ are indexed by the elements of $\Upsilon=\{s_1s_0v\mid v\in \sB_6\}$). Thus $\B{4}$ holds. To verify $\B{4}'$ we note that if $w=u^{-1}s_1s_0v\in\Gamma_6$ with $u,v\in\sB_6$ then writing $x=u^{-1}s_1s_0\in\Gamma_6$ and $y=s_1s_0v\in\Gamma_6$ we have
$$
\fc_{\rho,x}\fc_{\rho_y}=E_{s_1s_0u,s_1s_0}E_{s_1s_0,s_1s_0v}=E_{s_1s_0u,s_1s_0v}=\fc_{\rho,w}.
$$

\medskip

Consider $\Gamma_4$, which occurs in for all $r\neq 1$. The matrices for $\rho_{\Up}(T_j)$ are easily computed, and we find
\begin{align*}
\rho_{\Up}(T_{s_0})&=\begin{psmallmatrix}\sq^b&1\\
0&-\sq^{-b}\end{psmallmatrix}&
\rho_{\Up}(T_{s_1})&=
\begin{psmallmatrix}
-\sq^{-a}&0\\
0&-\sq^{-a}
\end{psmallmatrix}&
\rho_{\Up}(T_{s_2})&=
\begin{psmallmatrix}
-\sq^{-b}&0\\
1&\sq^b
\end{psmallmatrix}
\end{align*} 
If $r>1$ then $\Gamma_{4\leq_{\cLR}}=\Gamma_4\cup\Gamma_e$, and by direct computation we see that $\max\deg(\rho(\Gamma_4))=b$, hence \B{3} holds. We compute 
$$
\fc_{\rho,s_0}=E_{1,1},\quad \fc_{\rho,s_0s_2}=E_{1,2},\quad \fc_{\rho,s_2}=E_{2,2},\quad \fc_{\rho,s_2s_0}=E_{2,1},
$$
from which \B{4} and $\B{4}'$ follow. If $r<1$ then $\Gamma_{4\leq_{\cLR}}=\Gamma_4\cup\Gamma_7\cup\Gamma_3\cup\Gamma_e$. By direct calculation we have $\max\deg(\rho(\Gamma_4))=3b-2a$, $\max\deg(\rho(\Gamma_7))=-a$, and $\max\deg(\rho(\Gamma_3))=2b-a$, and hence \B{3} holds. We have
$$
\fc_{\rho,s_0s_2s_1s_2s_1s_2}=E_{1,2},\quad \fc_{\rho,s_0s_2s_1s_2s_1s_2s_0}=E_{1,1},\quad \fc_{\rho,s_2s_1s_2s_1s_2}=E_{2,2},\quad \fc_{\rho,s_2s_1s_2s_1s_2s_0}=E_{2,1},
$$
and hence \B{4} and $\B{4}'$ hold.
\medskip

We are left with the red cells $\Ga_3$. When $r>1$ all the representations afforded by the right cells are isomorphic and the matrices of $T_{s_1},T_{s_2}$ and $T_{s_0}$ are given by 
$$
\begin{psmallmatrix}
\sq^{a}&\mu_{1,1}&0&0&1&0\\
0&-\sq^{-a}&0&0&0&0\\
0&0&-\sq^{-a}&0&0&0\\
0&1&0&\sq^a&\mu_{1,1}&0\\
0&0&0&0&-\sq^{-a}&0\\
0&0&0&0&0&-\sq^{-a}
\end{psmallmatrix},
\begin{psmallmatrix}
-\sq^{-b}&0&0&0&0&0\\
1&\sq^{b}&1&0&0&0\\
0&0&-\sq^{-b}&0&0&0\\
0&0&0&-\sq^{-b}&0&0\\
0&0&0&1&\sq^{b}&1\\
0&0&0&0&0&-\sq^{-b}
\end{psmallmatrix},
\begin{psmallmatrix}
-\sq^{-b}&0&0&0&0&0\\
0&-\sq^{-b}&0&0&0&0\\
0&1&\sq^b&0&0&0\\
0&0&0&-\sq^{-b}&0&0\\
0&0&0&0&-\sq^{-b}&0\\
0&0&0&0&1&\sq^b
\end{psmallmatrix}.
$$
A direct check shows that $\deg([\rho(T_w)]_{i,j})$ is bounded by $\tba_{\Ga_3}=a$ and that $\B{3}$ holds. Moreover,
$$
\{\fc_{\rho,w}\mid w\in \Ga_3\}=\{E_{1+i,1+j}+E_{4+i,4+j}, E_{1+i,4+j}+E_{4+i,1+j}\mid 0\leq i,j\leq 2\},
$$
from which $\B{4}$ and $\B{4}'$ follow. The case $r<1$ can be treated similarly. 
\medskip

The case $(\Gamma,r)=(\Gamma_3,1)$ is slightly different since the right cells contained in $\Gamma$ do not give rise to isomorphic cell representations (there are two right cells with $8$ elements, and one with $7$). However in this case it turns out, by calculation, that the direct sum of these representations is bounded by $\tba_{\Ga_3}=1$ and $\B{3}$, $\B{4}$ and $\B{4}'$ hold. Explicit matrices for all finite cells can be found on the authors' webpage, and are provided below. 
 \end{proof}

\begin{Rem}\label{rem:6dim}
When $\Ga=\Ga_3$ and $r>1$, it is possible to use the cell factorisation described in Remark~\ref{rem:cell-fact-finite} to construct a 3 dimensional balanced representation over a quotient of an $\sR$-polynomial ring (this is a slight generalisation of our definition of balanced representations). The construction is based on the induction process introduced by Geck in~\cite{geck}. Recall that $\sB_3=\{e,s_2,s_2s_0\}$ and $\st_3=s_2s_1$. For all $x\in \sB_3\cup \{\st_3\}$, there exist $\sh_x\in \cH$ such that $C_{s_1}\sh_x\equiv C_{s_1x}\mod \cH_{\Ga_3}$.   Then $\cH_{\Ga_3}=\sg \sh^\flat_{u}C_{s_1}\sh^k_{\st_3}\sh_v\mid u,v\in \sB_3, k\in \{0,1\}\sd_{\sR}$ where $\flat$ denotes the anti-involution defined \cite[\S 3.4]{bible}. This allows us to define a 3 dimensional representation $\rho$ over $\sR[\eps]/ (\eps^2-1)$ with basis $\{\se_{s_1v}\mid v\in \sB_3\}$ by setting 
$$\se_{s_1v}\cdot T_w=\sum_{v\in \sB_3,k\in \{0,1\}} \la^{k,v'}_{v,w} \eps^k\se_{s_1v'}\text{ whenever }C_{s_1v}\cdot T_w\equiv \sum_{v'\in \sB_3,k\in \{0,1\}} \la^{k,v'}_{v,w} C_{s_1}\sh^k_{\st_3}\sh_{v'}\mod \cH_{\Ga_3}.$$
We obtain the following matrices for $T_{s_1}$, $T_{s_2}$ and $T_{s_0}$: 
$$\begin{psmallmatrix}
\sq^a&\eps+\mu_{1,1}&0\\
0&-\sq^{-a}&0\\
0&0&-\sq^{-a}
\end{psmallmatrix},
\begin{psmallmatrix}
-\sq^{-b}&0&0\\
1&\sq^b&1\\
0&0&-\sq^{-b}
\end{psmallmatrix}
\quand 
\begin{psmallmatrix}
-\sq^{-b}&0&0\\
0&-\sq^{-b}&0\\
0&1&\sq^b
\end{psmallmatrix}.
$$
Then it can be checked that $\rho$ is $\Ga_3$-balanced. More precisely we have $\fc_{\rho,w}=\eps^k E_{s_1u,s_1v}$ if $w=u^{-1}s_1\st^k_3v$.
\end{Rem}

 It is useful for later results to understand the decomposition of cell modules of finite cells into irreducible components. We summarise this in the following proposition.

\begin{Prop}\label{prop:decompositions}
Let $\Gamma$ be a finite two-sided cell and let $\Up$ be a right cell in $\Gamma$. Let $\rho_{\Up}\sim\Up$.
\begin{enumerate}
\item If $\Gamma\neq \Gamma_3$ then the representations $\rho_{\Up}$ are irreducible and pairwise isomorphic. 
\item If $\Gamma=\Gamma_3$ and $r\neq 1$ then the representations $\rho_{\Up}$ are pairwise isomorphic and decompose into a direct sum $\rho_{\Up}=\rho_3^+\oplus\rho_3^-$ where $\rho_3^{\pm}$ are irreducible $3$ dimensional representations with $\rho_3^+\not\cong\rho_3^-$.  
\item Suppose that $\Gamma=\Gamma_3$ and $r=1$. Let $\Up_1$, $\Up_2$, and $\Up_3$ be the right cells containing $s_1$, $s_0$, and $s_2$ (respectively). Then 
$$
\rho_{\Up_1}\cong \rho_3^+\oplus\rho_3^-\oplus\rho_3'\quad\text{and}\quad \rho_{\Up_2}\cong\rho_{\Up_3}\cong\rho_3^+\oplus\rho_3^-\oplus\rho_3''
$$
where $\rho_3^{+}$, $\rho_3^-$, $\rho_3'$, and $\rho_3''$ are pairwise non-isomorphic irreducible representations of dimension $3,3,1$, and $2$. 
\end{enumerate}
\end{Prop}

\begin{proof}
Statement (1) follows by direct calculation, and we omit the details. 
\medskip

Suppose that $\Gamma=\Gamma_3$ and $r\neq 1$. Again we verify that each right cell gives rise to an isomorphic representation by direct calculation. Let us discuss the decomposition into irreducible components. If $r>1$ then the cell $\Gamma_3$ admits a cell factorisation, and it follows from Remark~\ref{rem:6dim} that $\rho$ decomposes as $\rho_3^+\oplus\rho_3^-$, where the matrices for $\rho_3^{\eps}(T_j)$ are as in Remark~\ref{rem:6dim} (with $\eps$ now considered to be $\pm1$, and so these representations are over~$\sR$). If $r<1$ then we compute directly that $\rho_{\Up}\cong \rho_3^+\oplus \rho_3^-$ with matrices 
\begin{align*}
\rho_3^{\eps}(T_0)=\begin{psmallmatrix}
-\sq^{-b}&0&0\\
1&\sq^b&0\\
0&0&-\sq^{-b}
\end{psmallmatrix},\quad
\rho_3^{\eps}(T_1)=\begin{psmallmatrix}
-\sq^{-a}&0&0\\
0&-\sq^{-a}&0\\
1&0&\sq^a
\end{psmallmatrix}\quad\text{and}\quad
\rho_3^{\eps}(T_2)=
\begin{psmallmatrix}
\sq^b&1&\eps+\mu_{1,1}\\
0&-\sq^{-b}&0\\
0&0&-\sq^{-b}
\end{psmallmatrix}
\end{align*}
In each case it is easy to see that $\rho_3^{\eps}$ is irreducible, and that $\rho_3^+\not\cong\rho_3^-$. Hence~(2). 
\medskip

Finally, consider $\Gamma=\Gamma_3$ with $r=1$. In this case the result follows from \cite[(3.13.1)]{Lus4}. Indeed, if $\rho_{\Up_1}$ is constructed using the basis of residues 
$(\mathbf{e}_1,\mathbf{e}_2,\mathbf{e}_3,\mathbf{e}_4,\mathbf{e}_5,\mathbf{e}_6,\mathbf{e}_7)=(C_{1},C_{12},C_{120},C_{121},C_{1212},C_{12120},C_{12121})$ 
then the submodules giving the claimed decomposition are $\langle \mathbf{e}_1+2\mathbf{e}_4+\mathbf{e}_7,\mathbf{e}_2+\mathbf{e}_5,\mathbf{e}_3+\mathbf{e}_6\rangle$, $\langle \mathbf{e}_1-\mathbf{e}_7,\mathbf{e}_2-\mathbf{e}_5,\mathbf{e}_3-\mathbf{e}_6\rangle$, and $\langle \mathbf{e}_1-\mathbf{e}_4+\mathbf{e}_7\rangle$. If $\rho_{\Up_2}$ is constructed using the basis of residues 
$$(\mathbf{e}_1,\mathbf{e}_2,\mathbf{e}_3,\mathbf{e}_4,\mathbf{e}_5,\mathbf{e}_6,\mathbf{e}_7,\mathbf{e}_8)=(C_{0},C_{02},C_{021},C_{0212},C_{02120},C_{02121},C_{021212},C_{0212120})
$$ 
then the submodules are $\langle \textbf{e}_3+\textbf{e}_6,\textbf{e}_2+2\textbf{e}_4+\textbf{e}_7,\textbf{e}_1+2\textbf{e}_5+\textbf{e}_8\rangle$, $\langle \textbf{e}_3-\textbf{e}_6,\textbf{e}_2-\textbf{e}_7,\textbf{e}_1-\textbf{e}_8\rangle$, and $\langle\textbf{e}_2-\textbf{e}_4+\textbf{e}_7,\textbf{e}_1-\textbf{e}_5+\textbf{e}_8\rangle$. The same submodule structure works for $\rho_{\Up_3}$ using the basis of residues $(C_{20},C_2,C_{21},C_{212}, C_{0212}, C_{2121}, C_{21212}, C_{212120})$. 
\end{proof}

\begin{Rem}\label{rem:satisfyB}
We note the following for later use. In the case $\Gamma=\Gamma_3$ the representations $\rho_3^+$, $\rho_3^-$, $\rho_3'$, and $\rho_3''$, equipped with the bases from the above proposition, satisfy \B{1} and \B{2} (in the respective parameter regimes). It is clear that \B{1} holds (because if $\pi$ is semisimple and $\pi(C_w)=0$ then $\pi'(C_w)=0$ for all submodules). To see that \B{2} holds we note that the change of basis matrix that converts the cell representation into block form is independent of~$\sq$.
\end{Rem}

%%%%%%%%%%%%%%%%%%%%%%%%%%%%%%%%%%%%%%%%%
%%%%%%%%%%%%%%%%%%%%%%%%%%%%%%%%%%%%%%%%%
%%%%%%%%%%%%%%%%%%%%%%%%%%%%%%%%%%%%%%%%%
%%%%%%%%%%%%%%%%%%%%%%%%%%%%%%%%%%%%%%%%%
%%%%%%%%%%%%%%%%%%%%%%%%%%%%%%%%%%%%%%%%%
%%%%%%%%%%%%%%%%%%%%%%%%%%%%%%%%%%%%%%%%%
%%%%%%%%%%%%%%%%%%%%%%%%%%%%%%%%%%%%%%%%%
%%%%%%%%%%%%%%%%%%%%%%%%%%%%%%%%%%%%%%%%%
%%%%%%%%%%%%%%%%%%%%%%%%%%%%%%%%%%%%%%%%%
%%%%%%%%%%%%%%%%%%%%%%%%%%%%%%%%%%%%%%%%%

\subsection{The principal series representation $\pi_0$}\label{sec:pi0const}

We now associate a representation $\pi_0$ to the lowest two-sided cell $\Gamma_0$. It is convenient to set this section up in arbitrary type, and so $\cH$ is an affine Hecke algebra of rank~$n$. Recall that $\sR[Q]$ denotes the subalgebra of $\cH$ spanned by the elements $\{X^{\lambda}\mid \lambda\in Q\}$. We use this large commutative subalgebra to construct finite dimensional representations of $\cH$ as follows. Let $\zeta_1,\ldots,\zeta_n$ be commuting indeterminants, and let $M_0$ be the $1$-dimensional right $\sR[Q]$-module over the ring $\sR[\zeta_1,\ldots,\zeta_n,\zeta_1^{-1},\ldots,\zeta_n^{-1}]$, with generator $\xi_0$ and $\sR[Q]$-action given by linearly extending 
$$
\xi_0\cdot X^{\mu}=\xi_0\,\zeta^{\mu}\quad\text{where $\zeta^{\mu}=\zeta_1^{k_1}\cdots \zeta_n^{k_n}$ if $\mu= k_1\alpha_1^{\vee}+\cdots+k_n\alpha_n^{\vee}\in Q$}.
$$ 
Now let $(\pi_0,\mathcal{M}_0)$ be the induced right $\cH$-module. That is,
$$\mathcal{M}_0=\mathrm{Ind}_{\sR[Q]}^{\cH}(M_0)=M_0\otimes_{\sR[Q]}\cH.$$ 
Since $\{X^{\mu}T_{u^{-1}}^{-1}\mid \mu\in Q,u\in W_0\}$ is a basis of $\cH$, and since $\xi_0\otimes X^{\mu}=(\xi_0\otimes 1)\zeta^{\mu}$, we see that $\{\xi_0\otimes X_u\mid u\in W_0\}$ is a basis of $\mathcal{M}_0$. Thus $\mathcal{M}_0$ is a $|W_0|$-dimensional right $\cH$-module, called the \textit{principal series representation} with \textit{central character}~$\zeta=(\zeta_1,\ldots,\zeta_n)$. 
\medskip

Since $\Gamma_0$ is the lowest two-sided cell, the representation $\pi_0$ trivially satisfies $\B{1}$ with respect to $\Gamma=\Gamma_0$

\subsection{The induced representations $\pi_1$ and $\pi_2$}\label{sec:piiconst}

For each $i\in\{1,2\}$ let $\cH_i$ be the subalgebra of $\cH$ generated by $T_i,X_1,X_2$ (where $X_j=X^{\alpha_j^{\vee}}$). Let $\zeta$ be an indeterminant, and for each $i\in\{1,2\}$ let $M_i$ be the $1$-dimensional right $\cH_i$-module over the ring $\sR[\zeta,\zeta^{-1}]$ with generator $\xi_i$ and $\cH_i$-action given by 
\begin{align*}
\xi_1\cdot T_1&=\xi_1(-\sq^{-a})& \xi_1\cdot X_1&=\xi_1\,\sq^{-2a}& \xi_1\cdot X_2&=\xi_1\,(-\sq^a\zeta)\\
\xi_2\cdot T_2&=\xi_2(-\sq^{-b})&\xi_2\cdot X_1&=\xi_2\,(-\sq^{3b}\zeta)& \xi_2\cdot X_2&=\xi_2\,\sq^{-2b}
\end{align*}
One checks directly using the formulae in Example~\ref{presentationG2} that these are representations. 
\medskip

For $i\in\{1,2\}$, let $(\pi_i,\mathcal{M}_i)$ be the induced right $\cH$-module. Thus $\mathcal{M}_i=M_i\otimes_{\cH_i}\cH$. For $i\in\{1,2\}$ let $W_i=\langle s_i\rangle$ and let $W_0^i$ denote the set of minimal length coset representatives for cosets in $W_{i}\backslash W_0$. Note that the module $\mathcal{M}_i$ has basis $\{\xi_i\otimes X_v\mid v\in W_0^i\}$ for $i=1,2$.

\begin{Th}\label{thm:piiB1}
Let $i\in\{1,2\}$. The representation $\pi_i$ satisfies $\B{1}$ with respect to $\Gamma=\Gamma_i$. 
\end{Th}

\begin{proof}
We need to show that $\pi_i(C_w)=0$ for all $w\in\Gamma$ with $\Gamma\not\leq_{\cLR}\Gamma_i$. The set of such $\Gamma$ is determined by the Hasse diagrams in Figure~\ref{partition}. It suffices to show that $\pi_i(C_{\sw_j})=0$ whenever $\Gamma_j\not\leq_{\cLR}\Gamma_i$ (here $j\in\{0,1,2\}$), plus in the regime $r<1$ we need to show that $\pi_2(C_w)=0$ for all $w$ in the finite cell $\Gamma_4$. For example, in the parameter regime $2>r>3/2$ we need to check that $\pi_1(C_{\sw_0})=\pi_1(C_{\sw_2})=0$ and $\pi_2(C_{\sw_0})=\pi_2(C_{\sw_1})=0$. 
\medskip

In the cases that $\sw_i$ is the longest element of some dihedral parabolic subgroup $W_J$ we have the formula
$$
C_{\sw_i}=\sq^{-L(\sw_i)}\sum_{w\in W_J}\sq^{L(w)}T_w.
$$
The only case required when $\sw_i$ is not the longest element of a dihedral parabolic subgroup is $\sw_2$ in the parameter regime $2>r>3/2$. In this case
\begin{align*}
C_{\sw_2}=& (\sq^{-3a-2b}-\sq^{-3a}+\sq^{-3a+2b})T_e
+(\sq^{-3a-b}-\sq^{-3a+b})T_{s_2}\\
&+(-\sq^{-2a-2b}+\sq^{-2a}+\sq^{-2a+2b})T_{s_1}+
(\sq^{-2a-b}+\sq^{-2a+b})(T_{s_2s_1}+T_{s_1s_2})\\
&+\sq^{-3a}T_{s_2s_1s_2}+
(\sq^{-a-b}+\sq^{-a+b})T_{s_1s_2s_1}+
+\sq^{-a}(T_{s_2s_1s_2s_1}+T_{s_1s_2s_1s_2})
+T_{s_1s_2s_1s_2s_1}.
\end{align*}
For the case $r<1$, to show that $\pi_2(C_w)=0$ for $w\in\Gamma_4$ it is sufficient to show that $\pi_2(C_{s_2s_1s_2s_1s_2})=0$. The formula for $C_{s_2s_1s_2s_1s_2}$ in the $T_w$ basis is as in the $C_{\sw_2}$ formula above with the roles of $s_1$ and $s_2$ interchanged. The result now follows by direct computation. 
\end{proof}

%%%%%%%%%%%%%%%%%%%%%%%%%%%%%%%%%%%
%%%%%%%%%%%%%%%%%%%%%%%%%%%%%%%%%%%
%%%%%%%%%%%%%%%%%%%%%%%%%%%%%%%%%%%
%%%%%%%%%%%%%%%%%%%%%%%%%%%%%%%%%%%
%%%%%%%%%%%%%%%%%%%%%%%%%%%%%%%%%%%
%%%%%%%%%%%%%%%%%%%%%%%%%%%%%%%%%%%
%%%%%%%%%%%%%%%%%%%%%%%%%%%%%%%%%%%
%%%%%%%%%%%%%%%%%%%%%%%%%%%%%%%%%%%
%%%%%%%%%%%%%%%%%%%%%%%%%%%%%%%%%%%

\section{The lowest two-sided cell $\Gamma_0$}\label{sec:4}

In this section we show that the principal series representation $\pi_0$, equipped with certain natural bases, satisfies $\B{2}$--$\B{4}$ and $\B{4}'$ for the cell $\Ga_0$, with bound $\ba_{\pi_0}=L(\sw_0)$. It is convenient to work more generally than $\tilde{G}_2$. However since ultimately we are interested in $\tilde{G}_2$, and in this case $Q=P$, we will sometimes assume this setting (however we note that by slight modifications, in particular to the definition of $\sB_0$, the analysis below applies to all extended affine Weyl groups). 
\medskip

We first show that the degree of the matrix coefficients of $\pi_0(T_w)$ are bounded by $L(\sw_0)$ for all $w\in W$ (verifying $\B{2}$), and then we determine explicitly the set of $w\in W$ for which this bound is attained: it turns out to be precisely the lowest two-sided cell $\Ga_0$ (hence $\B{3}$). Finally we will compute the leading matrices $\fc_{\pi_0,w}$ in terms of Schur functions, verifying $\B{4}$ and $\B{4}'$.

\subsection{Path formula for the principal series representation $\pi_0$}\label{sec:4.1}

Let $\sB$ be any fundamental domain for the action of the group $Q$ of translations on the set of alcoves (for example, both $W_0$ and $\sB_{0}$ are fundamental domains). Thus any $w\in W$ can be written uniquely as $w=t_{\mu} u$ for some $u\in \sB$, and we set $\mathrm{wt}_{\sB}(w)=\mu$ and $\theta_{\sB}(w)=u$. If $p$ is a positively folded alcove walk we write 
$$\mathrm{wt}_{\sB}(p)=\mathrm{wt}_{\sB}(\mathrm{end}(p))\quad\text{and}\quad \theta_{\sB}(p)=\theta_{\sB}(\mathrm{end}(p)).
$$ 
The following theorem generalises the formula presented in \cite[Theorem~5.16]{Par:11}.

\begin{Th}\label{thm:pi0}
Let $\sB$ be a fundamental domain for $Q$. The set $\{\xi_0\otimes X_{u}\mid u\in\sB\}$ is a basis for $\mathcal{M}_0$, and with respect to this basis the matrix entries of $\pi_0(T_w)$, $w\in W$, are given by 
$$
[\pi_0(T_w)]_{u,v}=\sum_{\{p\in \mathcal{P}_{\sB}(\vec{w},u)\mid\theta_{\sB}(p)=v\}}\mathcal{Q}(p)\zeta^{\mathrm{wt}_{\sB}(p)}
$$
where $\vec{w}$ is any reduced expression for $w$.
\end{Th}

\begin{proof} 
Since $W_0$ is a fundamental domain for $Q$, each $u\in\sB$ can be written as $b=t_{\mu_u}u'$ for some $\mu_u\in Q$ and some $u'\in W_0$. Then
$
\xi_0\otimes X_u=\xi_0\otimes X^{\mu_u}X_{u'}=(\xi_0\otimes X_{u'})\zeta^{\mu_u}.
$
The first claim follows since $\{\xi_0\otimes X_{u'}\mid u'\in W_0\}$ is clearly a basis of $\mathcal{M}_0$. 
\medskip

Let $\vec{w}$ be any reduced expression for $w$. Using Proposition~\ref{prop:basischange} we have
\begin{align*}
(\xi_0\otimes X_u)\cdot T_w=\sum_{p\in\mathcal{P}(\vec{w},u)}(\xi_0\otimes X^{\mathrm{wt}_{\sB}(p)}X_{\theta_{\sB}(p)})\mathcal{Q}(p)=\sum_{v\in \sB}\bigg(\sum_{\{p\in \mathcal{P}(\vec{w},u)\mid\theta_{\sB}(p)=v\}}(\xi_0\otimes X_{v})\mathcal{Q}(p)\zeta^{\mathrm{wt}_{\sB}(p)}\bigg),
\end{align*}
hence the result.
\end{proof}

\subsection{Leading matrices for $\pi_0$}\label{sec:4.2}

We begin with some definitions in preparation for the following lemma. Let $u,w\in W$, let $\vec{w}$ be a reduced expression, and let $p\in\mathcal{P}(\vec{w},u)$. The \textit{partial foldings} of $p$ are the positively folded alcove walks $p_0,p_1,\ldots,p_{f(p)}$, where $p_j$ is the positively folded alcove walk of type $\vec{w}$ starting at $u$ that agrees with $p$ up to (and including) the $j$th fold of $p$, and is straight thereafter. Thus $p_0$ is the straight path of type $\vec{w}$ starting at $u$, and $p_{f(p)}=p$. The \textit{pivots} of $p$ are the alcoves $u_0,\ldots,u_{f(p)}$ in which the folds occur, with $u_0=u$. More formally, if the folds of $p$ occur at positions $k_1<\ldots<k_{f(p)}$ in the reduced expression $\vec{w}=r_1\cdots r_{\ell}$ (with $r_j\in S$) then the pivots of $p$ are the alcoves $u_0=u$, $u_1=ur_1\cdots r_{k_1-1}$, and $u_{j+1}=u_jr_{k_j+1}\cdots r_{k_{j+1}-1}$ for $j=1,\ldots,f(p)-1$.
\medskip

The following lemma applies in arbitrary affine type, with the minor assumption $L(s_0)\leq L(s_n)$ required for type $\tilde{C}_n$ in part 4 of the lemma (where $\tilde{C}_1=\tilde{A}_1$). If $L(s_0)>L(s_n)$ then one may, of course, apply the diagram automorphism of $\tilde{C}_n$ an then apply the lemma below. 

\begin{Lem}\label{lem:tech1}
Let $u,w\in W$ and let $\mathrm{wt}(u)=\mu$. Let $v\in W_0$ be such that $uw\in t_{\mu}vC_0$. Let $p\in\mathcal{P}(\vec{w},u)$, and suppose that the folds of $p$ occur on the hyperplanes $H_{\beta_1,k_1},\ldots, H_{\beta_{f(p)},k_{f(p)}}$, where $\beta_1,\ldots,\beta_{f(p)}\in\Phi^+$. Let $v_0=v$, and let $v_{j+1}=s_{\beta_{j+1}}v_j$ for $j=0,1,\ldots,f(p)-1$. Then, with the above assumption in type $\tilde{C}_n$, we have
\begin{enumerate}
\item $\ell(v_{j+1})<\ell(v_j)$ for $j=0,1,\ldots,f(p)-1$.
\item $f(p)\leq \ell(v)-\ell(v_{f(p)})$ with equality if and only if $\ell(v_{j+1})=\ell(v_j)-1$ for all $j=0,1,\ldots,f(p)-1$.
\item If $f(p)=\ell(\sw_0)$ then $v=\sw_0$, $v_{f(p)}=e$, and $\beta_1$ and $\beta_{f(p)}$ are simple roots. 
\item We have $\deg(\cQ(p))\leq L(\sw_0)$ with equality if and only if $f(p)=\ell(\sw_0)$ (and, in the case $\tilde{C}_n$ with $L(s_0)<L(s_n)$, no folds occur on hyperplanes $H_{\beta,k}$ with $s_{\beta,k}$ conjugate to $s_0$). 
\end{enumerate}
\end{Lem}

\begin{proof}
1) We may assume that $\mu=0$ (if not, translate the entire proof by $t_{-\mu}$, and then translate back at the end). Thus $u\in W_0$. Let $p\in\mathcal{P}(\vec{w},u)$, and let $f=f(p)$. Let $p_0,\ldots,p_{f}$ be the partial foldings of $p$. Let $p_0^{\infty}$ be an ``infinite continuation'' of $p_0$ such that each finite segment of $p_0^{\infty}$ is reduced, and $p_0^{\infty}$ moves into the ``interior'' of the Weyl chamber $vC_0$ (that is, away from all walls). More formally, $p_0^{\infty}$ can be constructed by first extending $p_0$ to $y=t_{\mathrm{wt}(uw)}v$ (the longest element of $uwW_0\cap vC_0$) and then appending infinitely many copies of a fixed reduced expression for $t_{\rho}$, where $\rho=\omega_1+\cdots+\omega_n$ (or any other choice of strictly dominant coweight). Verifying that any finite segment of the resulting infinite path $p_0^{\infty}$ is reduced is a straightforward exercise in computing separating hyperplanes.

\medskip

Let $p_1^{\infty},\ldots,p_{f}^{\infty}$ be the infinite extensions of $p_1,\ldots,p_{f}$ induced from $p_0^{\infty}$. In other words, $p_1^{\infty},\ldots,p_{f}^{\infty}$ are generated by successively performing the folds of $p$ to $p_0^{\infty}$. The hyperplane $H_{\beta_{j+1},k_{j+1}}$ separates the pivot $u_j$ from all alcoves of $p_{j}^{\infty}$ occurring after $u_j$, and $u_j$ is on the positive side of this hyperplane. Thus the linear hyperplane $H_{\beta_{j+1},0}$ separates the identity alcove $e$ from all alcoves sufficiently far along $p_{j}^{\infty}$ (this is because the former is on the positive side of this hyperplane, and the latter are on the negative side). It is clear that all alcoves sufficiently far along $p_j^{\infty}$ lie in $v_jC_0$ (here it is important that $\rho$ is strictly dominant). Thus $H_{\beta_{j+1},0}$ separates the Weyl chamber $C_0$ from the Weyl chamber $v_{j}C_0$. By the strong exchange condition $s_{\beta_{j+1}}v_j$ is obtained from a reduced expression of $v_j$ by deleting a generator, and thus $\ell(s_{\beta_{j+1}}v_j)<\ell(v_j)$. Therefore $\ell(v_{j+1})<\ell(v_j)$ for all $j=0,1,\ldots,f-1$.
\medskip

2) By the above we have $\ell(v_{j+1})-\ell(v_j)+1\leq 0$ for all $j=0,\ldots,f-1$, and hence
$$
0\geq \sum_{j=0}^{f-1}\left(\ell(v_{j+1})-\ell(v_j)+1\right)=\ell(v_f)-\ell(v)+f(p),
$$ 
with equality if and only if $\ell(v_{j+1})=\ell(v_j)-1$ for all $j=0,\ldots,f-1$.
\medskip

3) If $f(p)=\ell(\sw_0)$ then by 2) we have $v=\sw_0$ and $v_{f}=e$. Applying the equality $\ell(v_{j+1})=\ell(v_j)-1$ in the cases $j=0$ and $j=f-1$ gives $\ell(s_{\beta_1}\sw_0)=\ell(\sw_0)-1$ and $\ell(e)=\ell(s_{\beta_{f}})-1$, which forces $\beta_1$ and $\beta_{f}$ to be simple roots.  
\medskip

4) The conditions $\ell(v_{j+1})<\ell(v_j)$ and $v_{j+1}=s_{\beta_{j+1}}v_j$ imply, by the strong exchange condition, that $v_{j+1}$ is obtained from a reduced expression of $v_j$ by deleting a single generator. Moreover, by the proof of the strong exchange condition this deleted generator is conjugate to $s_{\beta_{j+1}}$, and thus $L(v_{j+1})\leq L(v_j)-L(s_{\beta_{j+1}})$. It follows that 
\begin{align*}
\deg(\cQ(p))=\sum_{j=0}^{f-1}L(s_{\beta_j,k_j})\leq \sum_{j=0}^{f-1}L(s_{\beta_j})\leq \sum_{j=0}^{f(p)-1}\left(L(v_j)-L(v_{j+1})\right)=L(v)-L(\theta(p))
\end{align*}
(in fact, in all types other than $\tilde{C}_n$ the first inequality is an equality since $s_{\beta_j}$ and $s_{\beta_j,k_j}$ are conjugate, while in the case $\tilde{C}_n$ if $s_{\beta_j,k_j}$ is not conjugate to $s_{\beta_j}$ then necessarily $s_{\beta_j,k_j}$ is conjugate to $s_0$ and $s_{\beta_j}$ is conjugate to $s_n$, and hence $L(s_{\beta_j,k_j})=L(s_0)\leq L(s_n)=L(s_{\beta_j})$ by assumption). Hence $\deg(\cQ(p))\leq L(\sw_0)$, and the condition for equality is clear.
\end{proof}

\begin{Cor}
The representation $\pi_0$, equipped with any basis of the form $\{\xi_0\otimes X_u\mid u\in\sB\}$ with $\sB$ a fundamental domain for the action of $Q$ on $W$, satisfies $\B{2}$ with $\ba_{\pi_0}=L(\sw_0)$. 
\end{Cor}

\begin{proof}
This follows immediately from Theorem~\ref{thm:pi0} and Lemma~\ref{lem:tech1}. 
\end{proof}

\begin{Rem}\label{rem:tech1}
Note that part 3) of Lemma~\ref{lem:tech1} says that if $p\in\mathcal{P}(\vec{w},u)$ with $f(p)=\ell(\sw_0)$ then the first and last folds of $p$ occur on simple root directions. Here we mean `simple direction' when $p$ is drawn, as usual, in `folded form'. One can also draw $p$ in `unfolded form' by drawing the unfolded path $p_0$ and marking the positions on this path where the folds of $p$ occur. We may then ask if $f(p)=\ell(\sw_0)$ forces the first and last folds in the unfolded form to also be on simple root directions. Indeed this is the case. The first fold is on the same hyperplane in both the folded and unfolded forms. We note that in the notation of Lemma~\ref{lem:tech1} the last fold in unfolded form occurs on a hyperplane whose linear root is $s_{\beta_{f(p)-1}}\cdots s_{\beta_1}\beta_{f(p)}=s_{\beta_{f(p)}}v_{f(p)}\beta_{f(p)}=s_{\beta_{f(p)}}\beta_{f(p)}=-\beta_{f(p)}$, which is a negative simple root.
\end{Rem}

\begin{Cor}\label{cor:tech1}
Let $p$ be a positively folded alcove walk of reduced type $\vec{w}$ starting at $u\in W$. If $f(p)=\ell(\sw_0)$ then the straight path from $u$ to $uw$ of type $\vec{w}$ crosses at least one hyperplane of each direction. 
\end{Cor}

\begin{proof}
In the notation of the lemma, we see that the set of hyperplanes on which the infinite extensions $p_j^{\infty}$ make negative crossings has strictly decreasing cardinality as $j$ increases. It follows that if $f(p)=\ell(\sw_0)$ then $p_0$ crosses at least one hyperplane of each of the $\ell(\sw_0)$ directions. 
\end{proof}

The main result of this section is the following. Recall that 
$
\Gamma_0=\{u^{-1}\sw_0t_{\lambda}v\mid u,v\in\sB_0,\,\lambda\in P^+\},
$
and for $w\in\Gamma_0$ we define $\su_w,\sv_w\in\sB_0$ and $\tau_w\in P^+$ by $w=\su_w^{-1}\sw_0\tau_w\sv_w$.

\begin{Th}\label{thm:mainlowest}
The representation $\pi_0$, equipped with the basis $\{\xi_0\otimes X_u\mid u\in\sB_0\}$, satisfies $\B{3}$, $\B{4}$ and $\B{4}'$. Moreover,
$$
\fc_{\pi_0,w}=\ss_{\tau_w}(\zeta)E_{\su_w,\sv_w}\quad\text{for all $w\in\Gamma_0$}.
$$
\end{Th}

\begin{proof}
Suppose that $w\in W$ is such that $[\pi_0(T_w)]_{u,v}$ has degree $L(\sw_0)$ for some $u,v\in \sB_0$. Thus by Theorem~\ref{thm:pi0} we see that for every reduced expression $\vec{w}$ there exists a path $p\in\mathcal{P}(\vec{w},u)$ such that $\deg(\cQ(p))=L(\sw_0)$ and $f(p)=\ell(\sw_0)$. By Corollary~\ref{cor:tech1} the straight path from $u$ to $uw$ of type $\vec{w}$ crosses every hyperplane direction. It follows that $uw$ lies in the anti-dominant sector based at $0$. To see this, recall that there are no simple directions available in $\sB_0$ (as $Q=P$), and thus if all hyperplane directions are crossed then the hyperplanes $H_{\alpha_i}$ are crossed for each $1\leq i\leq n$. Thus we may choose a reduced expression for $\vec{w}$ such that the straight path from $u$ to $uw$ of type $\vec{w}$ passes through the alcoves $1$ and $\sw_0$. It follows that $w$ admits a reduced expression of the form $\vec{w}=\vec{u}^{-1}\cdot\vec{\sw}_0\cdot \vec{t}_{\lambda}\cdot \vec{v}$ for some $\lambda\in P^+$ and $v\in\sB_0$, and hence $w\in \Gamma_0$. 
\medskip

We now consider the converse. Let $w\in\Gamma_0$ and write $\vec{w}=\vec{\su}_w\cdot \vec{\sw}_0\cdot \vec{\tau}_w\cdot \vec{\sv}_w$. If there exists $p\in\mathcal{P}(\vec{w},\su_w)$ with $f(p)=\ell(\sw_0)$ then $p$ has no folds in the initial $\vec{\su}_w^{-1}$ part (since the first fold must be on a simple direction by Lemma~\ref{lem:tech1}). Thus in the notation of~(\ref{eq:PP}) we have $\mathbb{P}(\vec{w},\su_w)=\mathbb{P}(\vec{\sw_0}\cdot\vec{\tau}_w\cdot\vec{\sv}_w,e)$. Moreover there are no folds in the final $\vec{\sv}_w$ part (from Lemma~\ref{lem:tech1} and Remark~\ref{rem:tech1}) and thus
$$
\{p\in\mathbb{P}(\vec{\sw}_0\cdot \vec{\tau}_w\cdot\vec{\sv}_w,e)\mid\theta_{\sB_0}(p)=\sv_w\}=\mathbb{P}(\vec{\sw}_0\cdot\vec{\tau}_w\cdot\vec{\sv}_w,e).
$$
Finally, there is a bijection from $\mathbb{P}(\vec{\sw_0}\cdot\vec{\tau}_w\cdot\vec{\sv}_w,e)$ to $\mathbb{P}(\vec{\sw_0}\cdot\vec{\tau}_w,e)$ by simply removing the final $\vec{\sv}_w$ part, and it follows from  Theorem~\ref{thm:pi0}, Theorem~\ref{thm:Schur}, and the above observations, that
\begin{align*}
[\fc_{\pi_0,w}]_{\su_w,\sv_w}=\sum_{p\in\mathbb{P}(\vec{\sw}_0\cdot \vec{\tau}_w\cdot\vec{\sv}_w,e)}\zeta^{\mathrm{wt}_{\sB_0}(p)}=\sum_{p\in\mathbb{P}(\vec{\sw}_0\cdot \vec{\tau}_w,e)}\zeta^{\mathrm{wt}(p)}=\ss_{\tau_w}(\zeta).
\end{align*}
From this formula it follows, in particular, that $\mathbb{P}(\vec{w},\su_w)\neq\emptyset$, and hence $\B{3}$ holds. Moreover, if $\su_w\neq u$ then we get by the first paragraph of the proof that $f(p)<L(\sw_0)$ for all $p\in\mathcal{P}(\vec{w},u)$ and hence $[\fc_{\pi_0,w}]_{u,v}=0$. If $u=\su_w$ and $v\neq \sv_w$ then by an observation above we have $\{p\in\mathcal{P}(\vec{\sw}_0\cdot\vec{\tau}_w\cdot\vec{\sv}_w,e)\mid\theta_{\sB_0}(p)=v\}=\emptyset$, and so again $[\fc_{\pi_0,w}]_{u,v}=0$. This proves that $\fc_{\pi_0,w}=\ss_{\tau_w}(\zeta)E_{\su_w,\sv_w}$ for all $w\in\Gamma_0$. 
\medskip

We also see that $\B{4}$ holds, because the set of matrices $\{\ss_{\lambda}(\zeta)E_{u,v}\mid\lambda\in P^+,u,v\in \sB_0\}$ is free over $\nZ$ (using linear independence of the Schur characters).

\medskip

Finally, to check $\B{4}'$, let $w\in\Gamma_0$, and let $x=\su_w^{-1}\sw_0\in\Gamma_0$ and $y=\sw_0\tau_w\sv_w\in\Gamma_0$. Then
$$
\fc_{\pi_0,x}\fc_{\pi_0,y}=\ss_0(\zeta)\ss_{\tau_w}(\zeta)E_{\su_w,e}E_{e,\sv_w}=\ss_{\tau_w}(\zeta)E_{\su_w,\sv_w}=\fc_{\pi_0,w},
$$
completing the proof.
\end{proof}

We note that the above theorem recovers a result of Xie~\cite[Corollary~5.4]{Xie:17}.

%%%%%%%%%%%%%%%%%%%%%%%%%%%%%%%%%%%
%%%%%%%%%%%%%%%%%%%%%%%%%%%%%%%%%%%
%%%%%%%%%%%%%%%%%%%%%%%%%%%%%%%%%%%
%%%%%%%%%%%%%%%%%%%%%%%%%%%%%%%%%%%
%%%%%%%%%%%%%%%%%%%%%%%%%%%%%%%%%%%
%%%%%%%%%%%%%%%%%%%%%%%%%%%%%%%%%%%
%%%%%%%%%%%%%%%%%%%%%%%%%%%%%%%%%%%

\section{The infinite cells $\Gamma_1$ and $\Gamma_2$}\label{sec:5}

In this section we carry out an analogue of the work of Section~\ref{sec:4} for the other infinite cells $\Gamma_i$ with $i=1,2$. We begin by introducing and developing a combinatorial model of ``$\alpha_i$-folded alcove walks''. We then show that this model encodes the matrix coefficients of $\pi_i(T_w)$, and we use  this model to prove that our representations are balanced for the cells $\Gamma_1$ and $\Gamma_2$, compute the bounds for the degree of matrix coefficients in each parameter regime, and compute the leading matrices in terms of Schur functions of type $A_1$. This section is necessarily more involved that the previous section, since we need to pay careful attention to the non-generic parameter regimes.

\subsection{$\alpha_i$-folded alcove walks}\label{sec:5.1}

The following definitions apply to any affine Coxeter group. Let $\alpha_i$ be a fixed simple root, and let $$\mathcal{U}_i=\{x\in V\mid 0\leq \langle x,\alpha_i\rangle\leq 1\}$$ be the region between the hyperplanes $H_{\alpha_i,0}$ and $H_{\alpha_i,1}$. Let $w\in W$ and write $\vec{w}=s_{i_1}\cdots s_{i_{\ell}}$. An \textit{$\alpha_i$-folded alcove walk of type $\vec{w}$ starting at $v\in \mathcal{U}_i$} is a sequence $p=(v_0,v_1,\ldots,v_{\ell})$ with $v_0,\ldots,v_{\ell}\in\mathcal{U}_i$ such that
\begin{enumerate}
\item $v_0=v$, and $v_k\in \{v_{k-1},v_{k-1}s_{i_k}\}$ for each $k=1,\ldots,\ell$, and
\item if $v_{k-1}=v_k$ then either:
\begin{enumerate}
\item $v_{k-1}s_{i_k}\notin \mathcal{U}_i$, or
\item $v_{k-1}$ is on the positive side of the hyperplane separating $v_{k-1}$ and $v_{k-1}s_{i_k}$. 
\end{enumerate}
\end{enumerate}

We note that condition 2)(a) can only occur if $v_{k-1}$ and $v_{k-1}s_{i_k}$ are separated by either $H_{\alpha_i,0}$ or $H_{\alpha_i,1}$. The \textit{end} of $p=(v_0,\ldots,v_{\ell})$ is $\mathrm{end}(p)=v_{\ell}$. 
\medskip

Less formally, $\alpha_i$-folded alcove walks are made up of the following symbols, where $x\in \mathcal{U}_i$ and $s\in S$:
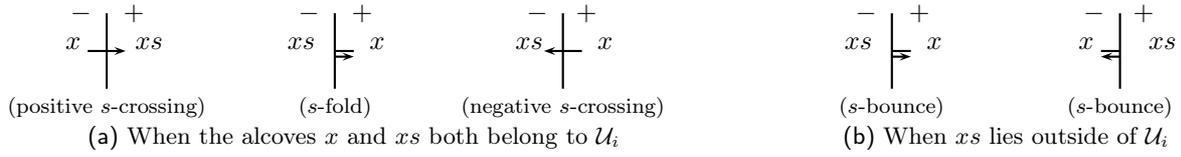
\medskip

\begin{figure}[H]
\begin{subfigure}{.6\textwidth}
\begin{center}
\begin{pspicture}(-5,-1.5)(6,1)
\psset{unit=.5cm}
\psline(-6,-1)(-6,1)
\psline{->}(-6.5,0)(-5.5,0)
\rput(-6.8,1){{ $-$}}
\rput(-7,.2){{ $x$}}
\rput(-5,.2){{ $xs$}}
\rput(-5.4,1){{ $+$}}
\rput(-6,-1.5){\footnotesize{(positive $s$-crossing)}}
\psline(-0,-1)(-0,1)
\psline(.5,0)(-0,0)
\psline{<-}(.5,-.15)(-0,-.15)
\rput(-.8,1){{ $-$}}
\rput(-1,.2){{ $xs$}}
\rput(1,.2){{ $x$}}
\rput(.6,1){{ $+$}}
\rput(0,-1.5){\footnotesize{($s$-fold)}}

\psline(6,-1)(6,1)
\psline{->}(6.5,0)(5.5,0)
\rput(6.5,1){{ $+$}}
\rput(7,.2){{ $x$}}
\rput(5,.2){{ $xs$}}
\rput(5.2,1){{ $-$}}
\rput(6,-1.5){\footnotesize{(negative $s$-crossing)}}
\end{pspicture}
\caption{When the alcoves $x$ and $xs$ both belong to $\mathcal{U}_i$}
\end{center}
\end{subfigure}
\begin{subfigure}{0.4\textwidth}
\begin{center}
\begin{pspicture}(-3,-1.5)(3,1)
\psset{unit=.5cm}
\psline(3,-1)(3,1)
\psline(2.5,0)(3,0)
\psline{<-}(2.5,-.15)(3,-.15)
\rput(3.5,1){{ $+$}}
\rput(4,.2){{ $xs$}}
\rput(2,.2){{ $x$}}
\rput(2.2,1){{ $-$}}
\rput(3,-1.5){\footnotesize{($s$-bounce)}}
\psline(-3,-1)(-3,1)
\psline(-2.5,0)(-3,0)
\psline{<-}(-2.5,-.15)(-3,-.15)
\rput(-3.8,1){{ $-$}}
\rput(-4,.2){{ $xs$}}
\rput(-2,.2){{ $x$}}
\rput(-2.4,1){{ $+$}}
\rput(-3,-1.5){\footnotesize{($s$-bounce)}}
\end{pspicture}
\caption{When $xs$ lies outside of $\cU_{i}$}
\end{center}
\end{subfigure}
\end{figure}

We refer to the two symbols in (b) as ``$s$-bounces'' rather than folds, since they play a different role in the theory. Note that bounces only occur on the hyperplanes $H_{\alpha_i,0}$ and $H_{\alpha_i,1}$. Moreover, note that there are no folds on the walls $H_{\alpha_i,0}$ and $H_{\alpha_i,1}$ -- the only interactions with these walls are bounces. We note that in all cases except for $\tilde{A}_1$ and $\tilde{C}_n$ every $s$-bounce necessarily has $\sq_s=\sq_{s_i}$ (although it is not necessarily true that $s=s_i$). In type $\tilde{A}_1$ and $\tilde{C}_n$ this property holds under the assumption that either $L(s_0)=L(s_n)$, or by modifying the definition of $\cU_i$. In any case, here we are interested in $\tilde{G}_2$, and in this case we have $\sq_{s}=\sq_{s_i}$ for all $s$-bounces. Thus we will typically simply say \textit{bounces}.

\medskip

Let $p$ be an $\alpha_i$-folded alcove walk. Let
$$
f_s(p)=\#(\text{$s$-folds in $p$})\quad\text{and}\quad b(p)=\#(\text{bounces in $p$}).
$$
Define a modified $\sq$-weight for $p$ by 
$$
\mathcal{Q}_i(p)=(-\sq_{s_i}^{-1})^{b(p)}\prod_{s\in S}(\sq_s-\sq_s^{-1})^{f_s(p)}.
$$
Finally, for each $1\leq i\leq n$ define 
$$\theta^i(p)=\psi_i(\theta(p))\quad\text{and}\quad \mathrm{wt}^i(p)=\langle\mathrm{wt}(p),\omega_i\rangle,$$ 
where $\psi_i:W_0\to W^i_0$ is the natural projection map taking $u\in W_0$ to the minimal length representative of $W_{i}u$, and $\omega_1,\ldots,\omega_n$ are the fundamental coweights of $\Phi$. Thus if $\mathrm{wt}(p)=m_1\alpha_1^{\vee}+\cdots+m_n\alpha_n^{\vee}$ then $\mathrm{wt}^i(p)=m_i$. We refer to $\theta^i(p)$ as the \textit{final direction} of $p$, and $\mathrm{wt}^i(p)$ as the \textit{weight} of $p$ (with respect to $\alpha_i$).
\medskip

We now specialise to the case $\tilde{G}_2$. Let
\begin{align*}
\sigma_1=s_{\alpha_1,1}t_{\alpha_1^{\vee}+\alpha_2^{\vee}}=t_{\alpha_1^{\vee}+\alpha_2^{\vee}}s_{1}\quad\text{and}\quad 
\sigma_2=s_{\alpha_2,1}t_{\alpha_1^{\vee}+2\alpha_2^{\vee}}=t_{\alpha_1^{\vee}+2\alpha_2^{\vee}}s_2.
\end{align*}
Observe that for each $i\in\{1,2\}$ the ``glide reflection'' $\sigma_i$ preserves $\cU_i$, and that $W_0^i$ is a fundamental domain for the action of $\langle \sigma_i\rangle$ on $\cU_i$. Let $\sB$ be any other fundamental domain for this action. For $w\in \cU_i$ we define $\mathrm{wt}^i_{\sB}(w)\in\mathbb{Z}$ and $\theta^i_{\sB}(w)\in \sB$ by the equation
$$
w=\sigma_i^{\mathrm{wt}^i_{\sB}(w)}\theta^i_{\sB}(w),
$$
and for $\alpha_i$-folded alcove walks $p$ we define 
$$
\mathrm{wt}_{\sB}^i(p)=\mathrm{wt}_{\sB}^i(\mathrm{end}(p))\quad\text{and}\quad \theta^i_{\sB}(p)=\theta_{\sB}^i(\mathrm{end}(p)).
$$ 
It is easy to see that in the case $\sB=W_0^i$ these definitions agree with those for $\mathrm{wt}^i(p)$ and $\theta^i(p)$ made above.
\medskip

\begin{Exa}
Let $i=1$. Let $\vec{w}=121021210212102120212102120$ (a reduced expression). Figure~\ref{fig:paths} illustrates an $\alpha_1$-folded alcove walk of type $\vec{w}$, with two choices of fundamental domain $\sB$ (the gray shaded regions). The tessellation of $\cU_1$ by $\sB$ is shown. The alcove walk has $2$ folds and $3$ bounces, and $\cQ_1(p)=-\sq^{-3a}(\sq^a-\sq^{-a})(\sq^b-\sq^{-b})$. The weight of $p$ is $4$ with respect to the first fundamental domain, and $2$ with respect to the second fundamental domain.
\psset{linewidth=.1mm,unit=0.9cm}
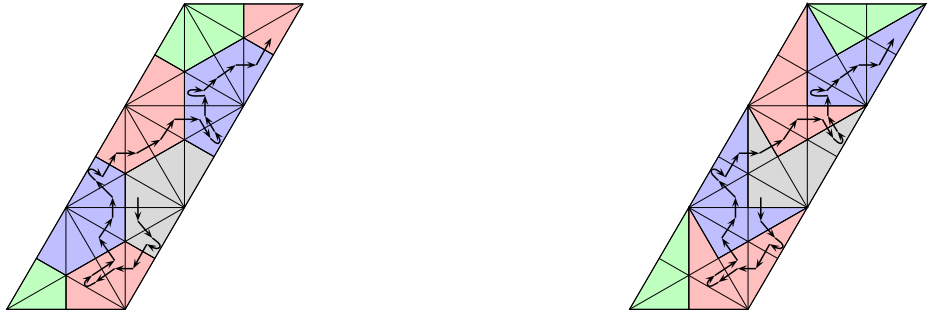
\begin{figure}[H]
\begin{subfigure}{.45\textwidth}
\begin{center}
\begin{pspicture}(-4,-1.5)(2,3.25)
%\pspolygon[fillcolor=lightgray,fillstyle=solid](0,0)(0,1)(.433,.75)
%\pspolygonfillcolor=green!15!,fillstyle=solid](-.866,0)(-.866,1.5)(-.433,.75)(.866,1.5)(0,0)
%\pspolygon[fillcolor=red!15!,fillstyle=solid](-.433,.75)(-.866,1.5)(0,3)(1.732,3)(.866,1.5)
%\pspolygon[fillcolor=blue!15!,fillstyle=solid](-.866,0)(-.866,1.5)(-1.732,0)
%\pspolygon[fillcolor=blue!15!,fillstyle=solid](-.866,-1.5)(-.866,0)(0,0)
%\pspolygon[fillcolor=blue!15!,fillstyle=solid](-.866,1.5)(-.433,.75)(0,1)(0,3)
%\pspolygon[fillcolor=yellow,fillstyle=solid](-.866,0)(-1.732,0)(-.866,-0.5)
%\psline[linewidth=0.3mm](-.866,0)(-.866,1.5)
%\psline[linewidth=0.3mm](-.866,1.5)(-.433,.75)
%\psline[linewidth=0.3mm](-.433,.75)(.866,1.5)
%\psline[linewidth=0.3mm](.866,1.5)(0,0)
%\psline[linewidth=0.3mm](0,0)(-.866,0)
\pspolygon[fillcolor=gray!30!,fillstyle=solid](-0.433,-0.75)(-.866,-.5)(-0.866,0.5)(0,1)(.433,0.75)
%\pspolygon[fillcolor=green!45!,fillstyle=solid](0,0)(-.866,-0.5)(-.433,-0.75)
\pspolygon[fillcolor=red!25!,fillstyle=solid](-1.3,0.75)(-0.866,0.5)(0,1)(0,2)(-.433,2.25)
\pspolygon[fillcolor=blue!25!,fillstyle=solid](0,1)(0,2)(.866,2.5)(1.3,2.25)(.433,0.75)
\pspolygon[fillcolor=green!25!,fillstyle=solid](-.433,2.25)(0,2)(.866,2.5)(.866,3)(0,3)
\pspolygon[fillcolor=red!25!,fillstyle=solid](.866,3)(.866,2.5)(1.3,2.25)(1.732,3)
\pspolygon[fillcolor=blue!25!,fillstyle=solid](-1.3,0.75)(-.866,0.5)(-.866,-0.5)(-1.732,-1)(-2.165,-0.75)
\pspolygon[fillcolor=red!25!,fillstyle=solid](-.866,-0.5)(-1.732,-1)(-1.732,-1.5)(-.866,-1.5)(-.433,-0.75)
\pspolygon[fillcolor=green!25!,fillstyle=solid](-1.732,-1.5)(-1.732,-1)(-2.165,-0.75)(-2.598,-1.5)

\psline(-2.598,-1.5)(0,3)(1.732,3)
\psline(-2.598,-1.5)(-.866,-1.5)(1.732,3)
\psline(-.866,1.5)(.866,1.5)
\psline(-1.732,0)(0,0)
\psline(-2.598,-1.5)(0,0)
\psline(-1.732,0)(.866,1.5)
\psline(-.866,1.5)(1.732,3)

\psline(-.866,-1.5)(-2.165,-.75)
\psline(-0.433,-.75)(-1.732,0)
\psline(0,0)(-1.3,.75)
\psline(.433,.75)(-.866,1.5)
\psline(.866,1.5)(-.433,2.25)
\psline(1.3,2.25)(0,3)

\psline(-1.732,0)(-1.732,-1.5)
\psline(-.866,1.5)(-.866,-1.5)
\psline(0,0)(0,3)
\psline(.866,1.5)(.866,3)

\psline(-.866,-1.5)(-1.732,0)
\psline(0,0)(-.866,1.5)
\psline(.866,1.5)(0,3)

\psline[linewidth=0.2mm]{->}(-.68,0.15)(-0.68,-0.2)
\psline[linewidth=0.2mm]{->}(-0.68,-0.2)(-0.47,-0.45)
\pscurve[linewidth=0.2mm]{->}(-0.47,-0.45)(-.36,-.58)(-0.55,-0.54)
\psline[linewidth=0.2mm]{->}(-0.55,-0.54)(-0.75,-0.9)
\psline[linewidth=0.2mm]{->}(-0.75,-0.9)(-1.02,-0.9)
\psline[linewidth=0.2mm]{->}(-1.02,-0.9)(-1.3,-1.1)
\pscurve[linewidth=0.2mm]{->}(-1.3,-1.1)(-1.45,-1.15)(-1.35,-1)
\psline[linewidth=0.2mm]{->}(-1.35,-1)(-1.02,-0.78)
\psline[linewidth=0.2mm]{->}(-1.02,-0.78)(-1.25,-0.45)
\psline[linewidth=0.2mm]{->}(-1.25,-0.45)(-1.05,-0.15)
\psline[linewidth=0.2mm]{->}(-1.05,-0.15)(-1.05,0.15)
\psline[linewidth=0.2mm]{->}(-1.05,0.15)(-1.3,0.4)
\pscurve[linewidth=0.2mm]{->}(-1.3,0.4)(-1.4,0.55)(-1.2,0.45)
\psline[linewidth=0.2mm]{->}(-1.2,0.45)(-1,0.8)
\psline[linewidth=0.2mm]{->}(-1,0.8)(-0.7,0.8)
\psline[linewidth=0.2mm]{->}(-0.7,0.8)(-0.35,1)
\psline[linewidth=0.2mm]{->}(-0.35,1)(-0.15,1.3)
\psline[linewidth=0.2mm]{->}(-0.15,1.3)(0.22,1.3)
\psline[linewidth=0.2mm]{->}(0.22,1.3)(0.38,1.02)
\pscurve[linewidth=0.2mm]{->}(0.38,1.02)(0.5,0.95)(0.44,1.12)
\psline[linewidth=0.2mm]{->}(0.44,1.12)(0.3,1.35)
\psline[linewidth=0.2mm]{->}(0.3,1.35)(0.3,1.65)
\pscurve[linewidth=0.2mm]{->}(0.3,1.65)(0.05,1.69)(0.28,1.72)
\psline[linewidth=0.2mm]{->}(0.28,1.72)(0.48,1.91)
\psline[linewidth=0.2mm]{->}(0.48,1.91)(0.75,2.1)
\psline[linewidth=0.2mm]{->}(0.75,2.1)(1.05,2.1)
%\pscurve[linewidth=0.2mm]{->}(1.05,2.1)(1.17,2.05)(1.01,2.23)
%\psline[linewidth=0.2mm]{->}(1.01,2.23)(1.22,2.5)
\psline[linewidth=0.2mm]{->}(1.05,2.1)(1.27,2.5)

%\pscircle[fillcolor=red,fillstyle=solid](-1.2,-.5){.08}
%\pscircle[fillcolor=red,fillstyle=solid](-.334,1){.08}
%\pscircle[fillcolor=red,fillstyle=solid](.532,2.5){.08}

%\pscircle[fillcolor=red,fillstyle=solid](.166,1.7){.08}
%\pscircle[fillcolor=red,fillstyle=solid](-.7,.2){.08}
%\pscircle[fillcolor=red,fillstyle=solid](-1.566,-1.3){.08}
\end{pspicture}
\end{center}
\caption{$\sB=W_0^1$, $\mathrm{wt}^1(p)=4$, $\theta^1(p)=21212$}
\end{subfigure}
\begin{subfigure}{.5\textwidth}
\begin{center}
\begin{pspicture}(-4,-1.5)(2,3.25)
%\pspolygon[fillcolor=lightgray,fillstyle=solid](0,0)(0,1)(.433,.75)
%\pspolygonfillcolor=green!15!,fillstyle=solid](-.866,0)(-.866,1.5)(-.433,.75)(.866,1.5)(0,0)
%\pspolygon[fillcolor=red!15!,fillstyle=solid](-.433,.75)(-.866,1.5)(0,3)(1.732,3)(.866,1.5)
%\pspolygon[fillcolor=blue!15!,fillstyle=solid](-.866,0)(-.866,1.5)(-1.732,0)
%\pspolygon[fillcolor=blue!15!,fillstyle=solid](-.866,-1.5)(-.866,0)(0,0)
%\pspolygon[fillcolor=blue!15!,fillstyle=solid](-.866,1.5)(-.433,.75)(0,1)(0,3)
%\pspolygon[fillcolor=yellow,fillstyle=solid](-.866,0)(-1.732,0)(-.866,-0.5)
%\psline[linewidth=0.3mm](-.866,0)(-.866,1.5)
%\psline[linewidth=0.3mm](-.866,1.5)(-.433,.75)
%\psline[linewidth=0.3mm](-.433,.75)(.866,1.5)
%\psline[linewidth=0.3mm](.866,1.5)(0,0)
%\psline[linewidth=0.3mm](0,0)(-.866,0)
\pspolygon[fillcolor=blue!25!,fillstyle=solid](-0.866,1.5)(-1.732,0)(-1.3,-0.75)(0,0)(-0.866,0)
\pspolygon[fillcolor=gray!30!,fillstyle=solid](-0.866,1.5)(-0.866,0)(0,0)(0.866,1.5)(-0.433,0.75)
%\pspolygon[fillcolor=green!45!,fillstyle=solid](0,-0,5)(0,1)(.433,0.75)
\pspolygon[fillcolor=red!25!,fillstyle=solid](0,3)(-0.866,1.5)(-0.433,0.75)(0.866,1.5)(0,1.5)
\pspolygon[fillcolor=blue!25!,fillstyle=solid](0,3)(0,1.5)(0.866,1.5)(1.732,3)(0.433,2.25)
\pspolygon[fillcolor=red!25!,fillstyle=solid](-1.732,0)(-1.732,-1.5)(-.866,-1.5)(0,0)(-1.3,-0.75)
\pspolygon[fillcolor=green!25!,fillstyle=solid](0,3)(1.732,3)(0.433,2.25)
\pspolygon[fillcolor=green!25!,fillstyle=solid](-1.732,0)(-1.732,-1.5)(-2.598,-1.5)
%\pspolygon[fillcolor=green!45!,fillstyle=solid](.866,1.5)(0,1)(.433,0.75)

\psline(-2.598,-1.5)(0,3)(1.732,3)
\psline(-2.598,-1.5)(-.866,-1.5)(1.732,3)
\psline(-.866,1.5)(.866,1.5)
\psline(-1.732,0)(0,0)
\psline(-2.598,-1.5)(0,0)
\psline(-1.732,0)(.866,1.5)
\psline(-.866,1.5)(1.732,3)

\psline(-.866,-1.5)(-2.165,-.75)
\psline(-0.433,-.75)(-1.732,0)
\psline(0,0)(-1.3,.75)
\psline(.433,.75)(-.866,1.5)
\psline(.866,1.5)(-.433,2.25)
\psline(1.3,2.25)(0,3)

\psline(-1.732,0)(-1.732,-1.5)
\psline(-.866,1.5)(-.866,-1.5)
\psline(0,0)(0,3)
\psline(.866,1.5)(.866,3)

\psline(-.866,-1.5)(-1.732,0)
\psline(0,0)(-.866,1.5)
\psline(.866,1.5)(0,3)

\psline[linewidth=0.2mm]{->}(-.68,0.15)(-0.68,-0.2)
\psline[linewidth=0.2mm]{->}(-0.68,-0.2)(-0.47,-0.45)
\pscurve[linewidth=0.2mm]{->}(-0.47,-0.45)(-.36,-.58)(-0.55,-0.54)
\psline[linewidth=0.2mm]{->}(-0.55,-0.54)(-0.75,-0.9)
\psline[linewidth=0.2mm]{->}(-0.75,-0.9)(-1.02,-0.9)
\psline[linewidth=0.2mm]{->}(-1.02,-0.9)(-1.3,-1.1)
\pscurve[linewidth=0.2mm]{->}(-1.3,-1.1)(-1.45,-1.15)(-1.35,-1)
\psline[linewidth=0.2mm]{->}(-1.35,-1)(-1.02,-0.78)
\psline[linewidth=0.2mm]{->}(-1.02,-0.78)(-1.25,-0.45)
\psline[linewidth=0.2mm]{->}(-1.25,-0.45)(-1.05,-0.15)
\psline[linewidth=0.2mm]{->}(-1.05,-0.15)(-1.05,0.15)
\psline[linewidth=0.2mm]{->}(-1.05,0.15)(-1.3,0.4)
\pscurve[linewidth=0.2mm]{->}(-1.3,0.4)(-1.4,0.55)(-1.2,0.45)
\psline[linewidth=0.2mm]{->}(-1.2,0.45)(-1,0.8)
\psline[linewidth=0.2mm]{->}(-1,0.8)(-0.7,0.8)
\psline[linewidth=0.2mm]{->}(-0.7,0.8)(-0.35,1)
\psline[linewidth=0.2mm]{->}(-0.35,1)(-0.15,1.3)
\psline[linewidth=0.2mm]{->}(-0.15,1.3)(0.22,1.3)
\psline[linewidth=0.2mm]{->}(0.22,1.3)(0.38,1.02)
\pscurve[linewidth=0.2mm]{->}(0.38,1.02)(0.5,0.95)(0.44,1.12)
\psline[linewidth=0.2mm]{->}(0.44,1.12)(0.3,1.35)
\psline[linewidth=0.2mm]{->}(0.3,1.35)(0.3,1.65)
\pscurve[linewidth=0.2mm]{->}(0.3,1.65)(0.05,1.69)(0.28,1.72)
\psline[linewidth=0.2mm]{->}(0.28,1.72)(0.48,1.91)
\psline[linewidth=0.2mm]{->}(0.48,1.91)(0.75,2.1)
\psline[linewidth=0.2mm]{->}(0.75,2.1)(1.05,2.1)
%\pscurve[linewidth=0.2mm]{->}(1.05,2.1)(1.17,2.05)(1.01,2.23)
%\psline[linewidth=0.2mm]{->}(1.01,2.23)(1.22,2.5)
\psline[linewidth=0.2mm]{->}(1.05,2.1)(1.27,2.5)

%\pscircle[fillcolor=red,fillstyle=solid](-1.2,-.5){.08}
%\pscircle[fillcolor=red,fillstyle=solid](-.334,1){.08}
%\pscircle[fillcolor=red,fillstyle=solid](.532,2.5){.08}

%\pscircle[fillcolor=red,fillstyle=solid](.166,1.7){.08}
%\pscircle[fillcolor=red,fillstyle=solid](-.7,.2){.08}
%\pscircle[fillcolor=red,fillstyle=solid](-1.566,-1.3){.08}
\end{pspicture}
\end{center}
\caption{$\sB=\{e,0,2,21,212,2121,2120\}$, $\mathrm{wt}^1_{\sB}(p)=2$, $\theta_{\sB}^1(p)=s_0$}
\end{subfigure}
\caption{An $\alpha_1$-folded alcove walk $p$, with two choices of fundamental domain $\sB$}
\label{fig:paths}
\end{figure}
\end{Exa}

We now prove an analogue of Theorem~\ref{thm:pi0}, giving a combinatorial formula for the matrix entries of $\pi_i(T_w)$ in terms of $\alpha_i$-folded alcove walks. We first consider the fundamental domain $W_0^i$, and then deduce the general case in Corollary~\ref{cor:funddom} below.

\begin{Th}\label{thm:pii}
Let $i\in\{1,2\}$ and let $w\in W$. With respect to the basis $\{\xi_i\otimes X_u\mid u\in W_0^i\}$ of $\mathcal{M}_i$, the matrix entries of $\pi_i(T_w)$ are given by
$$
[\pi_i(T_w)]_{u,v}=\sum_{\{p\in\mathcal{P}_{i}(\vec{w},u)\mid \theta^i(p)=v\}}\mathcal{Q}_i(p)\zeta^{\mathrm{wt}^i(p)}
$$
where $\vec{w}$ is any reduced expression for~$w$. 
\end{Th}

\begin{proof}
We will prove the case $i=1$. The case $i=2$ is completely analogous. We first prove the following formula by induction on $\ell(w)$:
\begin{align}\label{eq:projectpath}
(\xi_1\otimes X_u)\cdot T_w=\sum_{p\in\mathcal{P}_1(\vec{w},u)}(\xi_1\otimes X_{\mathrm{end}(p)})\mathcal{Q}_1(p).
\end{align}
Suppose that $\ell(ws)=\ell(w)+1$. Then by the induction hypothesis
\begin{align*}
(\xi_1\otimes X_u)\cdot T_{ws}&=\sum_{p\in\mathcal{P}_1(\vec{w},u)}(\xi_1\otimes X_{\mathrm{end}(p)})T_s\mathcal{Q}_1(p).
\end{align*}
Let $p\in\mathcal{P}_1(\vec{w},u)$. Consider the following cases:
\begin{enumerate}
\item If $\mathrm{end}(p)\,{^-}\hspace{-0.1cm}\mid^+\, \mathrm{end}(p)s$ with $\mathrm{end}(p)s\in\mathcal{U}_1$ then 
$$
(\xi_1\otimes X_{\mathrm{end}(p)})T_s\mathcal{Q}_1(p)=(\xi_1\otimes X_{\mathrm{end}(p\cdot \epsilon^+_s)})\mathcal{Q}_1(p\cdot \epsilon^+_s),
$$ 
where $p\cdot\epsilon_s^+$ denotes the path obtained from $p$ by appending a positive $s$-crossing. 
\item If $\mathrm{end}(p)\,{^+}\hspace{-0.1cm}\mid^-\, \mathrm{end}(p)s$ with $\mathrm{end}(p)s\in\mathcal{U}_1$ then using $T_s=T_s^{-1}-(\sq_s-\sq_s^{-1})$ gives
$$
(\xi_1\otimes X_{\mathrm{end}(p)})T_s\mathcal{Q}_1(p)=(\xi_1\otimes X_{\mathrm{end}(p\cdot \epsilon^-_s)})\mathcal{Q}_1(p\cdot\epsilon_s^-)+(\xi_1\otimes X_{\mathrm{end}(p\cdot f_s)})\mathcal{Q}_1(p\cdot f_s),
$$ 
where $p\cdot f_s$ denotes the path obtained from $p$ by appending an $s$-fold.
\item If $\mathrm{end}(p)\,{^-}\hspace{-0.1cm}\mid^+\, \mathrm{end}(p)s$ with $\mathrm{end}(p)s\notin\mathcal{U}_1$ then necessarily $\mathrm{end}(p)\cap \mathrm{end}(p)s$ is a face of $H_{\alpha_1,1}$ (since the crossing is positive). Then $\mathrm{end}(p)s=s_{\alpha_1,1}\mathrm{end}(p\cdot b_s)$ where $p\cdot b_s$ denotes the path obtained from $p$ by appending an $s$-bounce, and since $s_{\alpha_1,1}=t_{\alpha_1^{\vee}}s_{1}$ and
$
X_{s_{\alpha_1,1}\mathrm{end}(p\cdot b_s)}=X^{\alpha_1^{\vee}}T_{s_1}^{-1}X_{\mathrm{end}(p\cdot b_s)},
$
we have
\begin{align*}
(\xi_1\otimes X_{\mathrm{end}(p)})T_s\mathcal{Q}_1(p)&=(\xi_1\otimes X_{\mathrm{end}(p)s})\mathcal{Q}_1(p)\\
&=(\xi_1\cdot X^{\alpha_1^{\vee}}T_{s_1}^{-1}\otimes X_{\mathrm{end}(p\cdot b_s)})\mathcal{Q}_1(p)\\
&=(\xi_1\otimes X_{\mathrm{end}(p\cdot b_s)})\mathcal{Q}_1(p)(-\sq_1)(\sq_1^{-2})\\
&=(\xi_1\otimes X_{\mathrm{end}(p\cdot b_s)})\mathcal{Q}_1(p\cdot b_s).
\end{align*} 
\item If $\mathrm{end}(p)\,{^+}\hspace{-0.1cm}\mid^-\, \mathrm{end}(p)s$ with $\mathrm{end}(p)s\notin\mathcal{U}_1$ then necessarily $\mathrm{end}(p)\cap \mathrm{end}(p)s$ is a face of $H_{\alpha_1,0}$ (since the crossing is negative). Using the formula $T_s=T_s^{-1}+(\sq_1-\sq_1^{-1})$, and the fact that $\mathrm{end}(p)s=s_1\mathrm{end}(p)=s_1\mathrm{end}(p\cdot b_s)$, we have
\begin{align*}
(\xi_1\otimes X_{\mathrm{end}(p)})T_s\mathcal{Q}_1(p)&=(\xi_1\otimes X_{\mathrm{end}(p)s})\mathcal{Q}_1(p)+(\xi_1\otimes X_{\mathrm{end}(p)})(\sq_1-\sq_1^{-1})\mathcal{Q}_1(p)\\
&=(\xi_1\otimes X_{s_1\mathrm{end}(p\cdot b_s)})\mathcal{Q}_1(p)+(\xi_1\otimes X_{\mathrm{end}(p\cdot b_s)})(\sq_1-\sq_1^{-1})\mathcal{Q}_1(p)\\
&=(\xi_1\cdot T_{s_1}^{-1}\otimes X_{\mathrm{end}(p\cdot b_s)})\mathcal{Q}_1(p)+(\xi_1\otimes X_{\mathrm{end}(p\cdot b_s)})(\sq_1-\sq_1^{-1})\mathcal{Q}_1(p)\\
&=(\xi_1\otimes X_{\mathrm{end}(p\cdot b_s)})(-\sq_1)\mathcal{Q}_1(p)+(\xi_1\otimes X_{\mathrm{end}(p\cdot b_s)})(\sq_1-\sq_1^{-1})\mathcal{Q}_1(p)\\
&=(\xi_1\otimes X_{\mathrm{end}(p\cdot b_s)})\mathcal{Q}_1(p\cdot b_s).
\end{align*}
\end{enumerate}
Equation~(\ref{eq:projectpath}) follows. 

\medskip

Let $p\in\mathcal{P}_1(\vec{w},u)$ and write $\mathrm{end}(p)=t_{\mu}v$ with $\mu\in Q$ and $v\in W_0$. Then $\mu\in H_{\alpha_1,0}\cup H_{\alpha_1,1}$ (since $\mathrm{end}(p)\in\cU_1$). If $\mu\in H_{\alpha_1,0}$ then $\mu=k\alpha_1^{\vee}+2k\alpha_2^{\vee}$ for some $k\in\mathbb{Z}$ and $v\in W^1_0$. Thus 
$$
\xi_1\otimes X_{\mathrm{end}(p)}=\xi_1\cdot X^{\mu}\otimes X_v=(\xi_1\otimes X_v)\zeta^{2k}=(\xi_1\otimes X_{\theta^1(p)})\zeta^{\mathrm{wt}^1(p)}.
$$
If $\mu\in H_{\alpha_1,1}$ then $\mu=k\alpha_1^{\vee}+(2k-1)\alpha_2^{\vee}$ for some $k\in\mathbb{Z}$, and $v\notin W^1_0$. Thus $\theta^1(p)=s_1v$, and hence 
\begin{align*}
\xi_1\otimes X_{\mathrm{end}(p)}&=\xi_1\cdot X^{\mu}\otimes X_v\\
&=(\xi_1\otimes X_{s_1\theta^1(p)})\sq_1^{-1}(-\zeta^{2k-1})\\
&=(\xi_1\cdot T_{s_1}^{-1}\otimes X_{\theta^1(p)})\sq_1^{-1}(-\zeta^{2k-1})\\
&=(\xi_1\otimes X_{\theta^1(p)})\zeta^{2k-1}= (\xi_1\otimes X_{\theta^1(p)})\zeta^{\mathrm{wt}^1(p)},
\end{align*}
and the theorem follows.
\end{proof}

It is convenient to have a version of Theorem~\ref{thm:pii} for other choices of fundamental domain. It is not hard to see that for each $p\in \mathcal{P}_i(\vec{w},u)$ the path $\sigma_i( p)$ obtained by applying $\sigma_i$ to each part of $p$ is again a valid $\alpha_i$-folded alcove walk starting at $\sigma_iu$ (the main point here is that the reflection part of $\sigma_i$ is in the simple root direction $\alpha_i$, and thus sends $\Phi^+\backslash\{\alpha_i\}$ to itself). Moreover, $\mathcal{Q}_i(p)$ and $\theta^i(p)$ are preserved under the application of $\sigma_i$, and a direct calculation shows that $\mathrm{wt}^i(\sigma_i^{k}(p))=k+\mathrm{wt}^i(p)$. 

\begin{Cor}\label{cor:funddom}
Let $w\in W$, $i\in\{1,2\}$, and let $\sB$ be a fundamental domain for the action of $\sigma_i$ on $\cU_i$. Then the matrix entries of $\pi_i(T_w)$ with respect to the basis $\{\xi_i\otimes X_u\mid u\in \sB\}$ are
$$
[\pi_i(T_w)]_{u,v}=\sum_{\{p\in\mathcal{P}_i(\vec{w},u)\mid \theta^i_{\mathsf{B}}(p)=v\}}\mathcal{Q}_i(p)\zeta^{\mathrm{wt}_{\sB}^i(p)},
$$
where $\vec{w}$ is any choice of reduced expression for~$w$.
\end{Cor}

\begin{proof}
We will prove the result for $i=1$, with the case $i=2$ being similar. For each $u\in \sB$ define $k(u)\in\mathbb{Z}$ and $u'\in W_0^1$ by the formula $u=\sigma_1^{k(u)}u'$. A direct calculation, using the formulae $\sigma_1^{2k}=t_{k\alpha_1^{\vee}+2k\alpha_2^{\vee}}$ and $\sigma_1^{2k-1}=t_{k\alpha_1^{\vee}+(2k-1)\alpha_2^{\vee}}s_1$ 
shows that 
\begin{align*}
\xi_1\otimes X_u&=\xi_1\otimes X_{\sigma_1^{k(u)}u'}=(\xi_1\otimes X_{u'})\,\zeta^{k(u)}.
\end{align*}
It follows from Theorem~\ref{thm:pii} (by applying change of basis) that
$$
[\pi_1(T_w)]_{u,v}=\sum_{\{p\in\mathcal{P}_1(\vec{w},u')\mid\theta^1(p)=v'\}}\mathcal{Q}_i(p)\zeta^{\mathrm{wt}^1(p)+k(u)-k(v)}.
$$
By definition we have $\theta^1(p)=v'$ if and only if $\theta^1_{\sB}(p)=v$. Recall that $\sigma_1^{k(u)}(\mathcal{P}_1(\vec{w},u'))=\mathcal{P}_1(\vec{w},u)$ and that for each $p\in\mathcal{P}_1(\vec{w},u')$ the value of $\mathcal{Q}_1(p)$ is preserved under this transformation. Thus
$$
[\pi_1(T_w)]_{u,v}=\sum_{\{p\in\mathcal{P}_1(\vec{w},u)\mid\theta^1_{\sB}(p)=v\}}\mathcal{Q}_i(p)\zeta^{\mathrm{wt}^1(p)-k(v)},
$$
and the result follows since $\mathrm{wt}_{\sB}^1(p)=\mathrm{wt}^1(p)-k(v)$ if $\theta_{\sB}^1(p)=v$.
\end{proof}

\subsection{Folding tables and admissible sequences}\label{sec:5.2}

Our next task is to show that the representations $\pi_1$ and $\pi_2$ satisfy~\B{2}. By our combinatorial formula for the matrix coefficients of $\pi_i(T_w)$ in terms $\al_i$-folded alcove walks it is equivalent to show that $\deg(\cQ_i(p))$ is bounded by some numbers $\ba_{\pi_i}$ for all $\alpha_i$-folded alcove walks~$p$. In this subsection we explain our approach to bounding the degree of $\alpha_i$-folded  alcove walks.
\medskip

Note that every $w\in W$ admits a reduced expression of the form
\begin{align}\label{eq:form}
\vec{w}=\vec{v}\cdot \vec{t}_{\omega_1}^{\,m}\cdot\vec{t}_{\omega_2}^{\,n}\cdot \vec{\sbb}\quad\text{with $v\in W_0$, $m,n\in\mathbb{N}$, and $\sbb\in\sB_0$},
\end{align}
and each walk $p\in\mathcal{P}_i(\vec{w},u)$ with $u\in W_0^i$ and $\vec{w}$ as above can naturally be decomposed as $p=p_0\cdot p^0$ where 
$$
p_0\in\mathcal{P}_i(\vec{v},u)\quad\text{and}\quad p^0\in\mathcal{P}_i(\vec{w}_1,\mathrm{end}(p_0))\quad\text{where}\quad \vec{w}_1=\vec{t}_{\omega_1}^{\,m}\cdot\vec{t}_{\omega_2}^{\,n}\cdot\vec{\sbb}.
$$
 Since $\cQ_i(p)=\cQ_i(p_0)\cQ_i(p^0)$ it is sufficient to bound the degrees of $\cQ_i(p_0)$ and $\cQ_i(p^0)$. The former is straight forward (since $v$ is in the dihedral group $G_2$). Thus the main effort is involved in bounding the degree of $\cQ_i(p^0)$. For this purpose we will fix reduced expressions for $\vec{t}_{\omega_1}$ and $\vec{t}_{\omega_2}$, and construct \textit{folding tables} that record the possible degrees of $\mathcal{Q}_i(p^0)$.
 
\medskip

We now explain the construction of our folding tables, via an analogue of the \textit{admissible sets} of Lenart and Postnikov~\cite{LP,LP2}. Let $v\in W^i_0$ and $x\in W$ with reduced expression $\vec{x}=s_{i_1}\ldots s_{i_n}$ . We denote by $p(\vec{x},v)\in\cP_i(\vec{x},v)$ the unique $\alpha_i$-folded alcove walk of type $\vec{x}$ starting at $v$ with no folds. Of course $p(\vec{x},v)$ may still have bounces, because $\alpha_i$-folded alcove walks are required to say in the strip $\cU_i$. Nonetheless, we refer to $p(\vec{x},v)$ as the \textit{straight walk} of type $\vec{x}$ starting at~$v$. Let
\begin{align*}
\mathcal{I}^-(\vec{x},v)&=\{k\in\{1,\ldots,n\}\mid \text{$p(\vec{x},v)$ makes a negative crossing at the $k$th step}\}\\
\mathcal{I}^+(\vec{x},v)&=\{k\in\{1,\ldots,n\}\mid \text{$p(\vec{x},v)$ makes a positive crossing at the $k$th step}\}\\
\mathcal{I}^{\,*}(\vec{x},v)&=\{k\in\{1,\ldots,n\}\mid \text{$p(\vec{x},v)$ bounces at the $k$th step}\}.
\end{align*}
Note that $\mathcal{I}^-\cup\mathcal{I}^+\cup\mathcal{I}^*=\{1,\ldots,n\}$. We define a function
$$
\varphi_{\vec{x}}^v:\mathcal{I}^-(\vec{x},v)\to W^i_0\times\mathbb{Z}
$$
as follows. For $k\in \mathcal{I}^-(\vec{x},v)$ let $p_k$ be the $\alpha_i$-folded alcove walk obtained from the straight walk $p_0=p(\vec{x},v)$ by folding at the $k$th step (note that after performing this fold one may need to include bounces at places where the folded walk $p_k$ attempts to exit the strip $\cU_i$; also note that this notation differs from the partial foldings defined earlier). Let 
$$
\varphi_{\vec{x}}^v(k)=\text{the unique $(u,n)\in W_0^i\times\mathbb{Z}$ such that $p(\vec{x},\sigma_i^nu)$ and $p_k$ agree after the $k$th step}.
$$
Equivalently, $(u,n)$ is the unique pair such that $\mathrm{end}(p(\vec{x},\sigma_i^nu))=\mathrm{end}(p_k)$, and thus $\sigma_i^nu$ is simply the end of the straight alcove walk $p(\mathrm{rev}(\vec{x}),\mathrm{end}(p_k))$, where $\mathrm{rev}(\vec{x})$ is the expression $\vec{x}$ read backwards.  
\medskip

\begin{Def}[Folding table]
Fix the enumeration $y_1,\ldots,y_6$ of $W_0^i$ with $\ell(y_j)=j-1$ for $j=1,\ldots,6$. For each $(j,k)$ with $1\leq j\leq 6$ and $1\leq k\leq \ell(x)$ define $f_{j,k}(\vec{x})\in\{-,*,1,2,3,4,5,6\}$ by
$$
f_{j,k}(\vec{x})=\begin{cases}
-&\text{if $k\in\mathcal{I}^+(\vec{x},y_j)$}\\
*&\text{if $k\in\mathcal{I}^{\,*}(\vec{x},y_j)$}\\
j'&\text{if $k\in\mathcal{I}^-(\vec{x},y_j)$ and $\varphi_{\vec{x}}^{y_j}=(y_{j'},n)$ for some $n\in\mathbb{Z}$.}
\end{cases}
$$
The \textit{$\alpha_i$-folding table} of $\vec{x}$ is the $6\times \ell(x)$ array $\mathbb{F}(\vec{x})$ with $(j,k)^{th}$ entry equal to $f_{j,k}(\vec{x})$.
\end{Def}

 \begin{Rem}\label{prefix}
If $\vec{y}$ is a prefix of $\vec{y}$ then $\mathbb{F}(\vec{y})$ is the subarray of $\mathbb{F}(\vec{x})$ consisting of the first $\ell(y)$ columns. Also note that of course any other enumeration of $W_0^i$ can be used in the definition.
\end{Rem}

\begin{Exa}\label{ex:foldingtables}
The $\alpha_i$-folding tables of 
\begin{align*}
\vec{t}_{\omega_1}&=0212012121,\quad \vec{t}_{\omega_2}=021212,\quad\text{and each element $\sbb$ in $\sB_0$}
\end{align*}
are shown in Tables~\ref{folding-table} and~\ref{folding-table2a} (resulting from a direct calculation).

\begin{table}[H]
\renewcommand{\arraystretch}{1.2}  
\begin{subfigure}{.6\linewidth}
\centering
$\begin{array}{|c||c|c|c|c|c|c|c|c|c|c||c|}\hline
&0&2&1&2&0&1&2&1&2&1&0\\\hline\hline
1&-&-&-&-&-&-&-&-&-&*&- \\\hline
2&-&-&-&-&-&*&-&-&-&-&-\\\hline
3&-&1&*&-&1&-&-&*&1&2&- \\\hline
4&2&-&*&2&-&1&2&*&-&-&2 \\\hline
5&3&2&1&3&2&*&3&1&2&4&3 \\\hline
6&1&4&2&1&4&3&1&2&4&*&1 \\\hline
\end{array}$
\caption{$\al_1$-folding table of $\vec{t}_{\omega_1}$ and $\vec{\sbb}_0$}
\end{subfigure}
\begin{subfigure}{.4\linewidth}
   \centering
   $\begin{array}{|c||c|c|c|c|c|c|}
   \hline
   & 0 &2 &1 &2 &1 &2 \\\hline
   \hline
1& - &- &- &- &-&- \\\hline
2& - &- &- &- &*&1 \\\hline
3& -&1&*&-&-&-\\\hline
4& 2&- &*&2  &1&3 \\\hline
5& 3&2 &1 &3&* &- \\\hline
6& 1&4 &2 &1&3 &5 \\\hline
\end{array}$
\caption{$\al_1$-folding table of $\vec{t}_{\omega_2}$}
\end{subfigure}
\caption{$\al_1$-folding tables for $\sB_0\cup\{\vec{t}_{\omega_1},\vec{t}_{\omega_2}\}$.}
\label{folding-table}
\end{table}

\begin{table}[H]
\renewcommand{\arraystretch}{1.2}  
\begin{subfigure}{.6\linewidth}
\centering
$\begin{array}{|c||c|c|c|c|c|c|c|c|c|c||c|}\hline
&0&2&1&2&0&1&2&1&2&1&0\\\hline\hline
1&-&-&-&-&-&-&-&-&-&-&- \\\hline
2&-&*&-&-&*&-&-&-&*&1&- \\\hline
3&*&-&-&*&-&1&*&-&-&-&* \\\hline
4&*&1&2&*&1&-&*&2&1&3&* \\\hline
5&1&*&3&1&*&2&1&3&*&-&1 \\\hline
6&2&3&1&2&3&4&2&1&3&5&2 \\\hline
\end{array}$
\caption{$\al_2$-folding table of $\vec{t}_{\omega_1}$ and $\vec{\sbb}_0$}
\end{subfigure}
\begin{subfigure}{.4\linewidth}
   \centering
   $\begin{array}{|c||c|c|c|c|c|c|}
   \hline
   & 0 &2 &1 &2 &1 &2 \\\hline
   \hline
1& - &- &- &- &-&* \\\hline
2& - &* &- &- &-&- \\\hline
3& *&-&-&*&1&2\\\hline
4& *&1 &2&*  &-&- \\\hline
5& 1&* &3 &1&2 &4 \\\hline
6& 2&3 &1 &2&4 &* \\\hline
\end{array}$
\caption{$\al_2$-folding table of $\vec{t}_{\omega_2}$}
\end{subfigure}
\caption{$\al_2$-folding tables of $\sB_0\cup\{\vec{t}_{\omega_1},\vec{t}_{\omega_2}\}$.}
\label{folding-table2a}
\end{table}

For efficiency of presenting the tables, we note that 10 of the 12 elements of $\sB_0$ are prefixes of $\vec{t}_{\omega_1}$, and one of the remaining elements of $\sB_0$ is a prefix of $\vec{t}_{\omega_2}$. Thus the folding tables of these $11$ elements of $\sB_0$ are `contained' in the folding tables $\mathbb{F}(\vec{t}_{\omega_1})$ and $\mathbb{F}(\vec{t}_{\omega_2})$ (see Remark~\ref{prefix}). The final element of $\sB_0$ (namely the longest element $\sB_0$) is $\vec{\mathsf{b}}_0=0212012120$ and thus agrees with $\vec{t}_{\omega_1}$ except in the last step. Thus in the tables we record the folding tables of $\vec{t}_{\omega_1}$ and $\vec{\mathsf{b}}_0$ simultaneously, with the table for $\vec{t}_{\omega_1}$ obtained by deleting the last column, and the table for $\vec{\mathsf{b}}_0$ obtained by deleting the penultimate column.
\end{Exa}

The connection between the $\alpha_i$-folding tables and the degree $\cQ_i(p)$ of an $\alpha_i$-folded alcove walk is understood through the notion of $(\vec{x},v)$-admissible sequences defined below.

\begin{Def} Let $x\in W$ with reduced expression $\vec{x}=s_{i_1}\ldots s_{i_{\ell}}$ and let $v\in W^i_0$.
We say that a sequence $(k_1,\ldots,k_r)$ with $1\leq k_1<k_2<\ldots<k_r\leq \ell$ is \textit{$(\vec{x},v)$-admissible} if, for all  $0\leq j\leq r-1$,
$$
k_{j+1}\in \mathcal{I}^-(\vec{x},\si_i^{n_j}v_j)\quwhere (v_0,n_0)=(v,0)\text{ and }(v_{j},n_{j})=\varphi_{\vec{x}}^{v_{j-1}}(k_{j})\text{ for $j>0$}.
$$
\end{Def}

\begin{Prop} Let $x\in W$ with reduced expression $\vec{x}=s_{i_1}\cdots s_{i_n}$ and let $v\in W_0^i$.
There is a bijection between the set of all $(\vec{x},v)$-admissible sequences and the set $\mathcal{P}_i(\vec{x},v)$. 
\end{Prop}

\begin{proof}
It is clear that if $p\in \cP(\vec{x},v)$ with $v\in W_0^{i}$, and if the folds of $p$ occur at indices $k_1<k_2<\ldots<k_r$,  then $J=(k_1,\ldots,k_r)$ is an $(\vec{x},v)$-admissible sequence. 
\medskip

Consider the converse. If $p=(w_t)_{r=0}^{\ell}$ is an $\alpha_i$-folded alcove walk and $j\leq k$ we write $p[j,k]=(w_r)_{t=j}^k$ (this is the segment of $p$ between the $j^{th}$ and $k^{th}$ steps). Let $J=(k_1,\ldots,k_r)$ be an $(\vec{x},v)$-admissible sequence. Define $(v_0,n_0)=(v,0)$ and let $(v_j,n_j)=\varphi_{\vec{x}}^{v_{j-1}}(k_j)$. Induction shows that the concatenation of paths
$$
p_J=p(\vec{x},v_0)[0,k_1-1]\cdot p(\vec{x},\sigma_i^{n_1}v_1)[k_1,k_2-1]\cdot\cdots\cdot p(\vec{x},\sigma_i^{n_{r}}v_r)[k_r,\ell] 
$$
is an $\alpha_i$-folded alcove walk, and that $J$ is the set of indices where the walk $p_J$ folds. 
\end{proof}
 
 The above proposition encodes how one uses folding tables to compute $\cQ_i(p)$ for all $p\in\cP_i(\vec{w},u)$ with $u\in W_0^i$. Let us explain this in an example. In fact we are mainly interested in $\deg(\cQ_i(p))$, and so we consider this below. Let $\vec{w}=\vec{t}_{\omega_1}^m\cdot \vec{t}_{\omega_2}^n$ where $m,n\in\mathbb{N}$, and let $u\in W_0^i$. Let $\mathcal{T}$ be the table obtained by concatenating the $\alpha_i$-folding tables of $\vec{t}_{\omega_1}$ and $\vec{t}_{\omega_2}$ with $m$ copies of the $\vec{t}_{\omega_1}$ table followed by $n$ copies of the $\vec{t}_{\omega_2}$ table. The elements of $\mathcal{P}_i(\vec{w},u)$ correspond to the excursions through $\mathcal{T}$ with the properties described below. We begin the excursion by entering the table $\mathcal{T}$ at the first cell on row $\ell(u)+1$, and at each step we move to a cell strictly to the right of the current cell according to the following rules. Suppose we are currently at the $N^{th}$ cell of row $r$, and this cell contains the symbol $x\in\{-,*,1,2,3,4,5,6\}$. 
 \begin{enumerate}
\item If $x=-$ then we move to the $(N+1)^{st}$ cell of row $r$. These steps correspond to positive crossings, and have no contribution to $\deg(\cQ_i(p))$. 
 \item If $x=*$ then we move to the $(N+1)^{st}$ cell of row $r$, and we have a contribution of $-L(s_i)$ to $\deg(\cQ_i(p))$. These steps correspond to bounces on either $H_{\alpha_i,0}$ or $H_{\alpha_i,1}$.
 \item If $x=j\in\{1,2,3,4,5,6\}$ then we have two options. 
 \begin{enumerate}
 \item We can move to the $(N+1)^{st}$ cell of row $r$. These steps correspond to negative crossings, with no contribution to $\deg(\cQ_i(p))$.
 \item We can move to the $(N+1)^{st}$ cell of row $j$. These steps correspond to folds, and give a contribution of $L(s_k)$ to $\deg(\cQ_i(p))$, where $k\in\{0,1,2\}$ is the entry in the $N^{th}$ cell of the ``$0$-row'' (the header) of~$\mathcal{T}$. 
  \end{enumerate}
  \end{enumerate}
 In the case that $N$ is the last cell of the table, moving to the $(N+1)^{st}$ cell should be interpreted as exiting the table and completing the excursion. We note that the above process can be regarded as $m$ passes through the $\alpha_i$-folding table of $\vec{t}_{\omega_1}$, followed by $n$ passes through the $\alpha_i$-folding table of $\vec{t}_{\omega_2}$, rather than concatenating the $m+n$ tables into one table.

\begin{Rem}\label{rem:endtable}
In the above explanation, concatenating the folding tables relied on the constituent pieces $\vec{t}_{\omega_1}$ and $\vec{t}_{\omega_2}$ being translations. If $\vec{w}=\vec{w}_1\cdot\vec{w}_2$ is a reduced expression with $w_1$ and $w_2$ not necessarily translations, then one needs to make a correction when combining the individual tables for $\vec{w}_1$ and $\vec{w}_2$ into the table for $\vec{w}_1\cdot \vec{w}_2$. Specifically, one adds an extra column at the end of the $\vec{w}_1$ table with $j^{th}$ entry $\theta^i(y_jw_1)$. This records the ``exit orientation'' of the path, and when concatenating the tables for $\vec{w}_1$ and $\vec{w}_2$, the rows of the $\vec{w}_2$ table are permuted so that they match with the exit column of $\vec{w}_1$. Alternatively, to interpret this process as one pass through $\vec{w}_1$ followed by one pass through $\vec{w}_2$ one should simply take the exit column entry of $\vec{w}_2$ to indicate the row on which to enter the $\vec{w}_1$ table. 
\end{Rem}

\subsection{Bounding the degree of matrix coefficients}

We are now able to establish bounds on the degree of $\cQ_i(p)$ for all $\alpha_i$-folded alcove walks, from which \B{2} will follow.

\begin{Th}\label{thm:mainbounds}
Let $p$ be an $\al_i$-folded alcove walk of reduced type. Then $\deg(\cQ_i(p))\leq \ba_{\pi_i}$ where  
$$\ba_{\pi_1}=\begin{cases}
a+b &\mbox{ if $a\geq 2b$,}\\
3b &\mbox{ if $a\leq 2b$}\\
\end{cases}
\quand 
\ba_{\pi_2}=
\begin{cases}
3a-2b &\mbox{ if $2a\geq 3b$,}\\
 a+b&\mbox{ if $2a\leq 3b$.}\\
\end{cases}
$$
Moreover, if $p\in\cP_i(\vec{w},u)$ with $u\in W_0^i$ is such that $\deg(\cQ_i(p))=\ba_{\pi_i}$ then $uw\in\cU_i$. 
 \end{Th}

\begin{proof}
Using the action of $\sigma_i$ on $\alpha_i$-folded paths we may assume that $p$ starts at $u\in W_0^i$. We note that if $\vec{w}$ and $\vec{w}'$ are two reduced expressions for the same element $w$ and if $\deg(\cQ_i(p))\leq \ba_{\pi_i}$ for all $p\in\mathcal{P}_i(\vec{w},u)$, then Theorem~\ref{thm:pii} implies that $\deg(\cQ_i(p))\leq \ba_{\pi_i}$ for all $p\in \mathcal{P}_i(\vec{w}',u)$. Thus we are free to choose any reduced expression for $w$. We choose a reduced expression for $\vec{w}$ as in (\ref{eq:form}). 
Let $\vec{w}_1=\vec{t}_{\omega_1}^{\,m}\cdot\vec{t}_{\omega_2}^{\,n}\cdot \vec{\sbb}$, and decompose $p\in\mathcal{P}_i(\vec{w},u)$ as $p=p_0\cdot p^0$ where $p_0\in\mathcal{P}_i(\vec{v},u)$ and $p^0\in\mathcal{P}_i(\vec{w}_1,u_0)$, where $u_0=\mathrm{end}(p_0)\in W_0^i$.  The bounds for $\cQ_i(p_0)$  in Table~\ref{tab:bounds1} are elementary (the left hand columns represent the elements of $W_0^i$ in the natural order of increasing length). 
\begin{table}[H]
\renewcommand{\arraystretch}{1.2}  
\begin{subfigure}{.5\linewidth}
\centering
$\begin{array}{|c||c|c|c}\hline
 u_0=\mathrm{end}(p_0)& a\geq b& a<b\\\hline
1&a & 3b-2a \\\hline
2&a & 2b-a \\\hline
3&a & 2b-a \\\hline
4&a & b \\\hline
5&b & b \\\hline
6&0 & 0 \\\hline
\end{array}$\caption{$i=1$}
\end{subfigure}
\begin{subfigure}{.5\linewidth}
   \centering
$\begin{array}{|c||c|c|c}\hline
u_0=\mathrm{end}(p_0) & b\geq a& b<a\\\hline
1&b & 3a-2b \\\hline
2&b & 2a-b \\\hline
3&b & 2a-b \\\hline
4&b & a \\\hline
5&a & a \\\hline
6&0 & 0 \\\hline
\end{array}$\caption{$i=2$}
\end{subfigure}
\caption{Bounds $\deg(\cQ_i(p_0))$ where $p_0\in\mathcal{P}_i(\vec{v},u)$ with $u\in W_0^i$ and $v\in W_0$.}
\label{tab:bounds1}
\end{table}
\medskip

One can now use the folding tables from Example~\ref{ex:foldingtables} to produce bounds for $\deg(\cQ_i(p^0))$. The following observations make this possible. Firstly, all folding tables for $\vec{t}_{\omega_1}$, $\vec{t}_{\omega_2}$, and $\vec{\mathsf{b}}$ with $\mathsf{b}\in\sB_0$ have the property that for $1\leq j\leq 6$, all entries in the $j^{th}$ row are either $-$, $*$, or are strictly smaller than~$j$. This means that with each fold we move to a strictly lower row. Secondly, if one makes a full pass of a table without making any folds (that is, without changing row) then the contribution to $\deg(\cQ_i(p))$ is at most $0$ and since the entry and exit rows are the same this pass can be ignored for the purpose of bounding $\deg(\cQ_i(p))$. Thus we may assume that at least one row change is made on each pass through a table, and therefore, by the above observation, we need only consider $\vec{w}_1=\vec{t}_{\omega_1}^m\cdot \vec{t}_{\omega_2}^n$ with $m+n\leq 6$ and $\vec{w}_1=\vec{t}_{\omega_1}^m\cdot \vec{t}_{\omega_2}^n\cdot \vec{\mathsf{b}}$ with $m+n\leq 5$. This reduces the work to a finite problem. As a third observation, we note that every row in the $\alpha_1$-folding table of $\vec{t}_{\omega_1}$, and every row in the $\alpha_2$-folding table of $\vec{t}_{\omega_2}$, contains a $*$, and thus these tables tend to have a negative influence on $\deg(\cQ_1(p))$ and $\deg(\cQ_2(p))$, respectively. 
\medskip

With the above observations in mind we find the bounds on $\deg(\cQ_i(p))$ for $p\in\cP_i(\vec{w}_1,u_0)$ with $u_0\in W_0^i$ and $\vec{w}_1=\vec{t}_{\omega_1}^{\,m}\cdot\vec{t}_{\omega_2}^{\,n}\cdot \vec{\mathsf{b}}$ listed in Table~\ref{tab:bounds2} below. We have checked these both by hand, and also implemented the process in $\textsf{MAGMA}$~\cite{MAGMA}. Moreover we see that if these bounds are attained then if $i=1$ then $m=0$, and if $i=2$ then $n=0$ (intuitively this is due to the third observation above).

\begin{table}[H]
\renewcommand{\arraystretch}{1.2}  
\begin{subfigure}{.5\linewidth}
\centering
$\begin{array}{|c||c|c|c}\hline
u_0& a\geq 2b& a\leq 2b\\\hline
1&0 & 0 \\\hline
2&0 & \max\{0,-a+b\} \\\hline
3&b & b \\\hline
4&b & 2b \\\hline
5&a & 2b \\\hline
6&a+b & 3b \\\hline
\end{array}$\caption{$i=1$}
\end{subfigure}
\begin{subfigure}{.5\linewidth}
   \centering
$\begin{array}{|c||c|c|c}\hline
u_0 & 2a\geq 3b& 2a\leq 3b\\\hline
1&0 & 0 \\\hline
2&\max\{0,a-3b\} & 0 \\\hline
3&\max\{0,a-2b\} & 0 \\\hline
4&2a-3b & \max\{0,a-b\} \\\hline
5&2a-2b & b \\\hline
6&3a-2b & a+b \\\hline
\end{array}$\caption{$i=2$}
\end{subfigure}
\caption{Bounds $\deg(\cQ_i(p^0))$ where $p_0\in\mathcal{P}_i(\vec{w}_1,u_0)$ with $u_0\in W_0^i$ and $\vec{w}_1=\vec{t}_{\omega_1}^{\,m}\cdot\vec{t}_{\omega_2}^{\,n}\cdot \vec{\mathsf{b}}$.}
\label{tab:bounds2}
\end{table}

\medskip

The bounds $\ba_{\pi_1}$ and $\ba_{\pi_2}$ follow by combining the bounds in Tables~\ref{tab:bounds1} and~\ref{tab:bounds2}.

\medskip

We now analyse paths such that $\deg(\cQ_i(p))=\ba_{\pi_i}$. We claim that in this case $uw\in\cU_i$. We have already shown that $\vec{w}=\vec{v}\cdot \vec{t}_{\omega_{j}}^n\cdot\vec{\mathsf{b}}$ for some $v\in W_0$, $n\in\mathbb{N}$, and $\mathsf{b}\in\sB_0$, where $\{j\}=\{1,2\}\backslash\{i\}$. In combining the bounds in Tables~\ref{tab:bounds1} and~\ref{tab:bounds2} we see that if $\deg(\cQ_i(p))= \ba_{\pi_i}$ then either:
\begin{enumerate}
\item $i=1$, $a\geq 2b$, and $u_0\in\{3,4,5,6\}$, or $a< 2b$ and $u_0\in\{4,5,6\}$, or
\item $i=2$ and $u_0\in\{5,6\}$. 
\end{enumerate}
Consider the case $i=1$ and $a\geq 2b$. If $u_0=6$ (that is $u_6=s_2s_1s_2s_1s_2$) then $\deg(\cQ_1(p_0))=0$, and it follows that the walk $p_0$ is straight with no bounces, and thus $uv=s_2s_1s_2s_1s_2$ (with $u$ and $\vec{v}$ as in Table~\ref{tab:bounds1}). Therefore $uw=s_2s_1s_2s_1s_2t_{\omega_2}^n\sbb$ for some $\sbb\in\sB_0$, and all such elements are obviously in~$\cU_i$. 
\medskip

Suppose now that $u_0=5$. In this case we see that for the bound in Table~\ref{tab:bounds1} to be attained we see, by direct observation, that $(u,\vec{v})=(e,s_2s_1s_2s_1s_2)$, $(s_2,s_1s_2s_1s_2)$, $(s_2s_1,s_2s_1s_2)$, $(s_2s_1s_2,s_1s_2)$, or $(s_2s_1s_2s_1,s_2)$ with the last step of $\vec{v}$ a fold. Thus $uw=s_2s_1s_2s_1s_2t_{\omega_2}^n\sbb$ for some $\sbb\in\sB_0$, and so again $uw\in\cU_i$. 

\medskip

Suppose now that $u_0=4$. Since the bound $\deg(\cQ_1(p_0))=a$ in Table~\ref{tab:bounds1} is attained we see that $(u,\vec{v})=(e,s_2s_1s_2s_1)$, $(s_2,s_1s_2s_1)$, $(s_2s_1,s_2s_1)$, or $(s_2s_1s_2,s_1)$ with the last term of $\vec{v}$ being a fold. Thus $uw=s_2s_1s_2s_1t_{\omega_2}^n\mathsf{b}$ for some $\mathsf{b}\in\sB_0$. However an easy check using the folding table shows that if $n\geq1$ then the maximum bound in $\deg(\cQ_1(p^0))$ is not attained. Moreover, again by the folding tables, we see that $\mathsf{b}$ is such that $uw=s_2s_1s_2s_1\mathsf{b}\in\cU_1$. 
\medskip

The remaining cases are similar.
\end{proof}

\begin{Cor}\label{B2generic}
Let $i\in\{1,2\}$. For generic parameters the representation $\pi_i$, equipped with any basis of the form $\{\xi_i\otimes X_u\mid u\in\sB\}$ with $\sB$ a fundamental domain for the action of $\sigma_i$ on $\cU_i$, satisfies \B{2} with $\ba_{\pi_i}$ as in Theorem~\ref{thm:mainbounds}.
\end{Cor}

\begin{proof}
This is immediate from Corollary~\ref{cor:funddom} and Theorem~\ref{thm:mainbounds}.
\end{proof}

\subsection{Leading matrices for generic parameters}\label{sec:5.3}

In this subsection we assume generic parameters. Thus, by our convention, if $i=1$ then $a\neq 2b$ and if $i=2$ then $2a\neq 3b$. If $p\in\cP_i(\vec{w},u)$ with $\deg(\cQ_i(p))=\ba_{\pi_i}$ then $p$ is called a \textit{maximal path}. In this section we determine all maximal paths, and show that $\pi_i(T_w)$ has a matrix coefficient of maximal degree if and only if $w\in\Gamma_i$, for $i=1,2$. Finally, we compute the leading matrices $\fc_{\pi_i,w}$ in terms of Schur functions of type $A_1$ and deduce that $\B{3}$, $\B{4}$, and $\B{4}'$ hold.

\medskip

To tighten the connection between $\pi_i$ and $\Gamma_i$ it is convenient to work with the following fundamental domains in Corollary~\ref{cor:funddom}. Of course, using the action of $\sigma_i$ on $\cU_i$, the choice of fundamental domain does not change the bounds on $\deg(\cQ_i(p))$. We define
 \begin{align}\label{eq:g}
 g_1=\begin{cases}
s_2s_1s_2& \mbox{ if $a/b>2$}\\
s_2s_1s_2s_1& \mbox{ if $a/b<2$}\\
\end{cases}\quand 
g_2=\begin{cases}
e& \mbox{ if $a/b>3/2$}\\
s_1s_2s_1s_2& \mbox{ if $a/b<3/2$,}\\
\end{cases}
\end{align}
and set $\sB_i'=g_i\sB_i$, where $\sB_i=\sB_{\Gamma_i}$ is as in Section~\ref{sec:factor1}. Then $\sB_i'$ is a fundamental domain for the action of $\si_i$ on $\cU_i$, represented as the green region in Figure~\ref{fundamental-domain}. The blue and red regions are translates of $\sB_i'$ by $\sigma_i$, and the ``base alcove'' $g_i$ of $\sB_i'$ is heavily shaded. We fix an indexing of $\sB_i'$ in Figure~\ref{fundamental-domain} in two cases for later use. Generally we write $\sbb_u=g_i u$ for $u\in\sB_i$, and so $\sB_i'=\{\sbb_u\mid u \in\sB_i\}$. 

\vspace*{-5mm}

\psset{linewidth=.1mm,unit=.6cm}
\begin{figure}[H]
\begin{subfigure}{.24\textwidth}
\begin{center}
\begin{pspicture}(-4,-2)(2,4)

%\pspolygon[fillcolor=lightgray,fillstyle=solid](0,0)(0,1)(.433,.75)
\pspolygon[fillcolor=green!15!,fillstyle=solid](-.866,0)(-.866,1.5)(-.433,.75)(.866,1.5)(0,0)
\pspolygon[fillcolor=red!15!,fillstyle=solid](-.433,.75)(-.866,1.5)(0,3)(0,1.5)(.866,1.5)
\pspolygon[fillcolor=blue!15!,fillstyle=solid](-.866,0)(-.866,1.5)(-1.732,0)(-1.3,-.75)(0,0)

\pspolygon[fillcolor=green!45!,fillstyle=solid](-.866,0)(0,0)(-0.866,0.5)

\psline(-2.598,-1.5)(0,3)(1.732,3)
\psline(-2.598,-1.5)(-.866,-1.5)(1.732,3)
\psline(-.866,1.5)(.866,1.5)
\psline(-1.732,0)(0,0)
\psline(-2.598,-1.5)(0,0)
\psline(-1.732,0)(.866,1.5)
\psline(-.866,1.5)(1.732,3)
\psline(-.866,-1.5)(-2.165,-.75)
\psline(-0.433,-.75)(-1.732,0)
\psline(0,0)(-1.3,.75)
\psline(.433,.75)(-.866,1.5)
\psline(.866,1.5)(-.433,2.25)
\psline(1.3,2.25)(0,3)
\psline(-1.732,0)(-1.732,-1.5)
\psline(-.866,1.5)(-.866,-1.5)
\psline(0,0)(0,3)
\psline(.866,1.5)(.866,3)
\psline(-.866,-1.5)(-1.732,0)
\psline(0,0)(-.866,1.5)
\psline(.866,1.5)(0,3)

\rput(-0.7,0.22){\tiny{$1$}}
\rput(-0.5,0.5){\tiny{$2$}}
\rput(-0.18,0.6){\tiny{$3$}}
\rput(0.12,0.6){\tiny{$4$}}
\rput(0.35,1){\tiny{$5$}}
\rput(-0.7,0.8){\tiny{$6$}}

%\pscircle[fillcolor=red,fillstyle=solid](-1.2,-.5){.08}
%\pscircle[fillcolor=red,fillstyle=solid](-.334,1){.08}
%\pscircle[fillcolor=red,fillstyle=solid](.532,2.5){.08}
%\pscircle[fillcolor=red,fillstyle=solid](.166,1.7){.08}
%\pscircle[fillcolor=red,fillstyle=solid](-.7,.2){.08}
%\pscircle[fillcolor=red,fillstyle=solid](-1.566,-1.3){.08}

\end{pspicture}
\end{center}
\caption{$\sB_1'$ when $a>2b$}
\end{subfigure}
\begin{subfigure}{.24\textwidth}
\begin{center}
\begin{pspicture}(-4,-2)(2,4)

\pspolygon[fillcolor=green!15!,fillstyle=solid](-.866,-.5)(-.866,.5)(.866,1.5)(0,0)
\pspolygon[fillcolor=red!15!,fillstyle=solid](-.866,.5)(-.866,1.5)(-.866,1.5)(0,3)(0,1)
\pspolygon[fillcolor=blue!15!,fillstyle=solid](-1.732,-1)(-1.732,0)(-.866,1.5)(-.866,-.5)

\pspolygon[fillcolor=green!45!,fillstyle=solid](-.866,-.5)(0,0)(-.866,0)

\psline(-2.598,-1.5)(0,3)(1.732,3)
\psline(-2.598,-1.5)(-.866,-1.5)(1.732,3)
\psline(-.866,1.5)(.866,1.5)
\psline(-1.732,0)(0,0)
\psline(-2.598,-1.5)(0,0)
\psline(-1.732,0)(.866,1.5)
\psline(-.866,1.5)(1.732,3)

\psline(-.866,-1.5)(-2.165,-.75)
\psline(-0.433,-.75)(-1.732,0)
\psline(0,0)(-1.3,.75)
\psline(.433,.75)(-.866,1.5)
\psline(.866,1.5)(-.433,2.25)
\psline(1.3,2.25)(0,3)

\psline(-1.732,0)(-1.732,-1.5)
\psline(-.866,1.5)(-.866,-1.5)
\psline(0,0)(0,3)
\psline(.866,1.5)(.866,3)

\psline(-.866,-1.5)(-1.732,0)
\psline(0,0)(-.866,1.5)
\psline(.866,1.5)(0,3)

%\pscircle[fillcolor=red,fillstyle=solid](-1.55,-.7){.08}
%\pscircle[fillcolor=red,fillstyle=solid](-.684,.8){.08}
%\pscircle[fillcolor=red,fillstyle=solid](.182,2.3){.08}
%\pscircle[fillcolor=red,fillstyle=solid](.166,1.3){.08}
%\pscircle[fillcolor=red,fillstyle=solid](-.7,-.2){.08}

\end{pspicture}
\end{center}
\caption{$\sB_1'$ when $a<2b$}
\end{subfigure}
\begin{subfigure}{.24\textwidth}
\begin{center}
\begin{pspicture}(-1.5,-3)(0.5,3)

\pspolygon[fillcolor=blue!15!,fillstyle=solid](0,0)(-.866,-1.5)(0,-3)
\pspolygon[fillcolor=green!15!,fillstyle=solid](0,0)(-.866,1.5)(-.866,-1.5)
\pspolygon[fillcolor=red!15!,fillstyle=solid](0,0)(-.866,1.5)(-0,3)

\pspolygon[fillcolor=green!45!,fillstyle=solid](-.866,-1.5)(-0.433,-.75)(-0.866,-.5)

\psline(0,0)(-.866,0)
\psline(0,0)(0,1.5)
\psline(0,1.5)(-.866,1.5)
\psline(-.866,0)(-.866,1.5)
\psline(0,0)(-.866,.5)
\psline(0,0)(-.866,1.5)
\psline(0,1)(-.866,1.5)
\psline(0,1)(-.866,.5)
%%% rectangle type 2
\psline(0,1.5)(0,3)
\psline(-.866,3)(-.866,1.5)
\psline(0,3)(-.866,3)
\psline(-.866,1.5)(0,2)
\psline(-.866,1.5)(0,3)
\psline(0,2)(-.866,2.5)
\psline(0,3)(-.866,2.5)
%%% rectangle type 3
\psline(0,-1.5)(0,0)
\psline(-.866,0)(-.866,-1.5)
\psline(0,-1.5)(-.866,-1.5)
\psline(-.866,-1.5)(0,-1)
\psline(-.866,-1.5)(0,0)
\psline(0,-1)(-.866,-.5)
\psline(0,0)(-.866,-.5)
\psline(0,-1.5)(0,-3)
\psline(-.866,-1.5)(-.866,-3)
\psline(0,-3)(-.866,-3)
\psline(-.866,-1.5)(0,-3)
\psline(-.866,-1.5)(0,-2)
\psline(0,-2)(-.866,-2.5)
\psline(0,-3)(-.866,-2.5)

\rput(-0.7,-0.8){\tiny{$1$}}
\rput(-0.5,-0.5){\tiny{$2$}}
\rput(-0.7,-0.2){\tiny{$3$}}
\rput(-0.7,0.2){\tiny{$4$}}
\rput(-0.5,0.5){\tiny{$5$}}
\rput(-0.7,0.8){\tiny{$6$}}

%\rput(.2,-.7){{\tiny $\bullet$}}
%\rput(.666,-2.2){{\tiny $\bullet$}}
%\rput(.666,.8){{\tiny $\bullet$}}
%\pscircle[fillcolor=red,fillstyle=solid](.15,.6){.08}
%\pscircle[fillcolor=red,fillstyle=solid](.15,-2.4){.08}
%\pscircle[fillcolor=red,fillstyle=solid](.7,-.9){.08}
%\pscircle[fillcolor=red,fillstyle=solid](.7,2.1){.08}

\end{pspicture}
\end{center}
\caption{$\sB_2'$ when $3a>2b$}
\end{subfigure}
\begin{subfigure}{.24\textwidth}
\begin{center}
\begin{pspicture}(-1,-3)(1,3)
\pspolygon[fillcolor=red!15!,fillstyle=solid](0,1)(.433,.75)(.866,1.5)(.866,2.5)(.433,2.25)(0,3)
\pspolygon[fillcolor=green!15!,fillstyle=solid](0,1)(.433,.75)(.866,1.5)(.866,-.5)(.433,-.75)(0,0)
\pspolygon[fillcolor=blue!15!,fillstyle=solid](0,-2)(.433,-2.25)(.866,-1.5)(.866,-.5)(.433,-.75)(0,0)

\pspolygon[fillcolor=green!45!,fillstyle=solid](.433,-.75)(.866,-.5)(0,0)

\psline(0,0)(.866,0)
\psline(0,0)(0,1.5)
\psline(0,1.5)(.866,1.5)
\psline(.866,0)(.866,1.5)

\psline(0,0)(.866,.5)
\psline(0,0)(.866,1.5)
\psline(0,1)(.866,1.5)

\psline(0,1)(.866,.5)

%%% rectangle type 2

\psline(0,1.5)(0,3)
\psline(.866,3)(.866,1.5)
\psline(0,3)(.866,3)

\psline(.866,1.5)(0,2)
\psline(.866,1.5)(0,3)

\psline(0,2)(.866,2.5)
\psline(0,3)(.866,2.5)

%%% rectangle type 3

\psline(0,-1.5)(0,0)
\psline(.866,0)(.866,-1.5)
\psline(0,-1.5)(.866,-1.5)
\psline(.866,-1.5)(0,-1)
\psline(.866,-1.5)(0,0)

\psline(0,-1)(.866,-.5)
\psline(0,0)(.866,-.5)

\psline(0,-1.5)(0,-3)
\psline(.866,-1.5)(.866,-3)
\psline(0,-3)(.866,-3)
\psline(.866,-1.5)(0,-3)
\psline(.866,-1.5)(0,-2)
\psline(0,-2)(.866,-2.5)
\psline(0,-3)(.866,-2.5)

%\pscircle[fillcolor=red,fillstyle=solid](.5,.5){.08}
%\pscircle[fillcolor=red,fillstyle=solid](.5,-2.5){.08}
%\pscircle[fillcolor=red,fillstyle=solid](.4,-1){.08}
%\pscircle[fillcolor=red,fillstyle=solid](.4,2){.08}

\end{pspicture}
\end{center}
\caption{$\sB_2'$ when $3a<2b$}
\end{subfigure}
\caption{The set $\sB_i'$ and translates by $\si_i$.}
\label{fundamental-domain}
\end{figure}

\vspace*{-.3mm}

\begin{Lem}\label{lem:inGammai}
Let $w\in W$ and $u\in \sB_i$ with $i\in\{1,2\}$. Let $\vec{w}$ be any reduced expression for~$w$. If $\mathcal{P}_i(\vec{w},\sbb_u)$ contains a maximal path then $w=u^{-1}\sw_i\st_i^N v$ for some $u,v\in\sB_i$ and $N\in\mathbb{N}$, and hence $w\in\Gamma_i$. 
\end{Lem}

\begin{proof}
Let $p$ be a maximal path. Thus $\sbb_uw\in\cU_i$ by Theorem~\ref{thm:mainbounds}. Note that the second sentence in the proof of Theorem~\ref{thm:mainbounds} we may choose any reduced expression $\vec{w}$ for $w$. We first claim that there is a minimal length (straight) path from $\sbb_u$ to $\sbb_uw$ passing through the element $\sbb_uu^{-1}\sw_i$ (geometrically this element is the element ``opposite'' the base alcove of $\sB_i'$, and is shaded yellow in Figure~\ref{fig:finiteregions}). If no minimal length path passes through $\sbb_u u^{-1}\sw_i$ then $\sbb_uw$ lies in either the red, green, or blue region in Figure~\ref{fig:finiteregions}.

\psset{linewidth=.1mm,unit=.6cm}
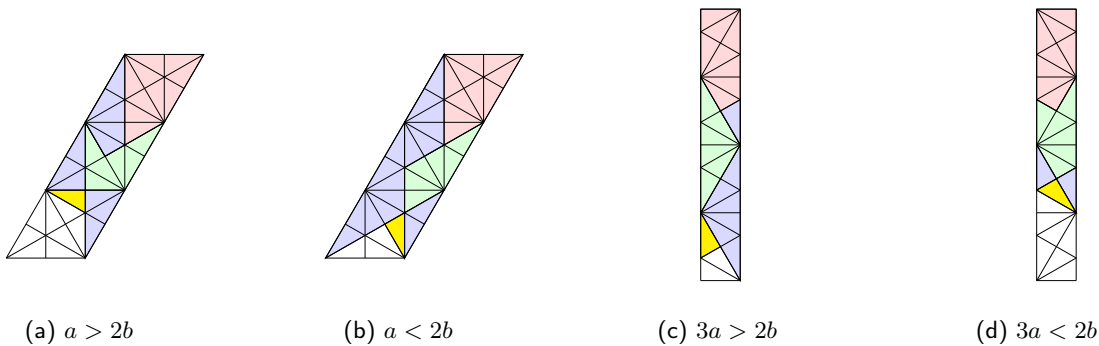
\begin{figure}[H]
\begin{subfigure}{.24\textwidth}
\begin{center}
\begin{pspicture}(-4,-2)(2,4)
%\pspolygon[fillcolor=lightgray,fillstyle=solid](0,0)(0,1)(.433,.75)
\pspolygon[fillcolor=green!15!,fillstyle=solid](-.866,0)(-.866,1.5)(-.433,.75)(.866,1.5)(0,0)
\pspolygon[fillcolor=red!15!,fillstyle=solid](-.433,.75)(-.866,1.5)(0,3)(1.732,3)(.866,1.5)
\pspolygon[fillcolor=blue!15!,fillstyle=solid](-.866,0)(-.866,1.5)(-1.732,0)
\pspolygon[fillcolor=blue!15!,fillstyle=solid](-.866,-1.5)(-.866,0)(0,0)
\pspolygon[fillcolor=blue!15!,fillstyle=solid](-.866,1.5)(-.433,.75)(0,1)(0,3)
\pspolygon[fillcolor=yellow,fillstyle=solid](-.866,0)(-1.732,0)(-.866,-0.5)

\psline(-2.598,-1.5)(0,3)(1.732,3)
\psline(-2.598,-1.5)(-.866,-1.5)(1.732,3)
\psline(-.866,1.5)(.866,1.5)
\psline(-1.732,0)(0,0)
\psline(-2.598,-1.5)(0,0)
\psline(-1.732,0)(.866,1.5)
\psline(-.866,1.5)(1.732,3)

\psline(-.866,-1.5)(-2.165,-.75)
\psline(-0.433,-.75)(-1.732,0)
\psline(0,0)(-1.3,.75)
\psline(.433,.75)(-.866,1.5)
\psline(.866,1.5)(-.433,2.25)
\psline(1.3,2.25)(0,3)

\psline(-1.732,0)(-1.732,-1.5)
\psline(-.866,1.5)(-.866,-1.5)
\psline(0,0)(0,3)
\psline(.866,1.5)(.866,3)

\psline(-.866,-1.5)(-1.732,0)
\psline(0,0)(-.866,1.5)
\psline(.866,1.5)(0,3)

%\pscircle[fillcolor=red,fillstyle=solid](-1.2,-.5){.08}
%\pscircle[fillcolor=red,fillstyle=solid](-.334,1){.08}
%\pscircle[fillcolor=red,fillstyle=solid](.532,2.5){.08}

%\pscircle[fillcolor=red,fillstyle=solid](.166,1.7){.08}
%\pscircle[fillcolor=red,fillstyle=solid](-.7,.2){.08}
%\pscircle[fillcolor=red,fillstyle=solid](-1.566,-1.3){.08}
\end{pspicture}\end{center}
\caption{$a>2b$}
\end{subfigure}
\begin{subfigure}{.24\textwidth}
\begin{center}
\begin{pspicture}(-4,-2)(2,4)

\pspolygon[fillcolor=green!15!,fillstyle=solid](-.866,-.5)(-.866,.5)(.866,1.5)(0,0)
\pspolygon[fillcolor=red!15!,fillstyle=solid](-.433,.75)(-.866,1.5)(0,3)(1.732,3)(.866,1.5)
\pspolygon[fillcolor=blue!15!,fillstyle=solid](-2.598,-1.5)(-1.732,0)(-.866,1.5)(-.866,-.5)
\pspolygon[fillcolor=blue!15!,fillstyle=solid](-.866,-.5)(-.866,-1.5)(0,0)
\pspolygon[fillcolor=blue!15!,fillstyle=solid](-.866,1.5)(-.866,.5)(0,1)(0,3)
\pspolygon[fillcolor=yellow,fillstyle=solid](-.866,-.5)(-.866,-1.5)(-1.3,-0.75)

%\pspolygon[fillcolor=green!45!,fillstyle=solid](-.866,-.5)(0,0)(-.866,0)

\psline(-2.598,-1.5)(0,3)(1.732,3)
\psline(-2.598,-1.5)(-.866,-1.5)(1.732,3)
\psline(-.866,1.5)(.866,1.5)
\psline(-1.732,0)(0,0)
\psline(-2.598,-1.5)(0,0)
\psline(-1.732,0)(.866,1.5)
\psline(-.866,1.5)(1.732,3)

\psline(-.866,-1.5)(-2.165,-.75)
\psline(-0.433,-.75)(-1.732,0)
\psline(0,0)(-1.3,.75)
\psline(.433,.75)(-.866,1.5)
\psline(.866,1.5)(-.433,2.25)
\psline(1.3,2.25)(0,3)

\psline(-1.732,0)(-1.732,-1.5)
\psline(-.866,1.5)(-.866,-1.5)
\psline(0,0)(0,3)
\psline(.866,1.5)(.866,3)

\psline(-.866,-1.5)(-1.732,0)
\psline(0,0)(-.866,1.5)
\psline(.866,1.5)(0,3)

%\pscircle[fillcolor=red,fillstyle=solid](-1.55,-.7){.08}
%\pscircle[fillcolor=red,fillstyle=solid](-.684,.8){.08}
%\pscircle[fillcolor=red,fillstyle=solid](.182,2.3){.08}
%\pscircle[fillcolor=red,fillstyle=solid](.166,1.3){.08}
%\pscircle[fillcolor=red,fillstyle=solid](-.7,-.2){.08}

\end{pspicture}
\end{center}
\caption{$a<2b$}
\end{subfigure}
\begin{subfigure}{.24\textwidth}
\begin{center}
\begin{pspicture}(-1.5,-3)(0.5,3)

\pspolygon[fillcolor=blue!15!,fillstyle=solid](0,0)(-.866,-1.5)(0,-3)
\pspolygon[fillcolor=green!15!,fillstyle=solid](0,0)(-.866,1.5)(-.866,-1.5)
\pspolygon[fillcolor=red!15!,fillstyle=solid](0,0)(-.866,1.5)(-.866,3)(0,3)
\pspolygon[fillcolor=blue!15!,fillstyle=solid](0,0)(0,1)(-0.433,0.75)
\pspolygon[fillcolor=yellow,fillstyle=solid](-.866,-1.5)(-.866,-2.5)(-.433,-2.25)

\psline(0,0)(-.866,0)
\psline(0,0)(0,1.5)
\psline(0,1.5)(-.866,1.5)
\psline(-.866,0)(-.866,1.5)
\psline(0,0)(-.866,.5)
\psline(0,0)(-.866,1.5)
\psline(0,1)(-.866,1.5)
\psline(0,1)(-.866,.5)
%%% rectangle type 2
\psline(0,1.5)(0,3)
\psline(-.866,3)(-.866,1.5)
\psline(0,3)(-.866,3)
\psline(-.866,1.5)(0,2)
\psline(-.866,1.5)(0,3)
\psline(0,2)(-.866,2.5)
\psline(0,3)(-.866,2.5)
%%% rectangle type 3
\psline(0,-1.5)(0,0)
\psline(-.866,0)(-.866,-1.5)
\psline(0,-1.5)(-.866,-1.5)
\psline(-.866,-1.5)(0,-1)
\psline(-.866,-1.5)(0,0)
\psline(0,-1)(-.866,-.5)
\psline(0,0)(-.866,-.5)
\psline(0,-1.5)(0,-3)
\psline(-.866,-1.5)(-.866,-3)
\psline(0,-3)(-.866,-3)
\psline(-.866,-1.5)(0,-3)
\psline(-.866,-1.5)(0,-2)
\psline(0,-2)(-.866,-2.5)
\psline(0,-3)(-.866,-2.5)

%\rput(.2,-.7){{\tiny $\bullet$}}
%\rput(.666,-2.2){{\tiny $\bullet$}}
%\rput(.666,.8){{\tiny $\bullet$}}
%\pscircle[fillcolor=red,fillstyle=solid](.15,.6){.08}
%\pscircle[fillcolor=red,fillstyle=solid](.15,-2.4){.08}
%\pscircle[fillcolor=red,fillstyle=solid](.7,-.9){.08}
%\pscircle[fillcolor=red,fillstyle=solid](.7,2.1){.08}

\end{pspicture}
\end{center}
\caption{$3a>2b$}
\end{subfigure}
\begin{subfigure}{.24\textwidth}
\begin{center}
\begin{pspicture}(-1,-3)(1,3)
\pspolygon[fillcolor=red!15!,fillstyle=solid](0,1)(.433,.75)(.866,1.5)(.866,2.5)(0.866,3)(0,3)
\pspolygon[fillcolor=green!15!,fillstyle=solid](0,1)(.433,.75)(.866,1.5)(.866,-.5)(.433,-.75)(0,0)
\pspolygon[fillcolor=blue!15!,fillstyle=solid](0,0)(0,-1)(.433,-0.75)
\pspolygon[fillcolor=blue!15!,fillstyle=solid](.433,-0.75)(.866,-.5)(.866,-1.5)
\pspolygon[fillcolor=yellow,fillstyle=solid](.433,-.75)(0,-1)(0.866,-1.5)

\psline(0,0)(.866,0)
\psline(0,0)(0,1.5)
\psline(0,1.5)(.866,1.5)
\psline(.866,0)(.866,1.5)

\psline(0,0)(.866,.5)
\psline(0,0)(.866,1.5)
\psline(0,1)(.866,1.5)
\psline(0,1)(.866,.5)
%%% rectangle type 2
\psline(0,1.5)(0,3)
\psline(.866,3)(.866,1.5)
\psline(0,3)(.866,3)
\psline(.866,1.5)(0,2)
\psline(.866,1.5)(0,3)
\psline(0,2)(.866,2.5)
\psline(0,3)(.866,2.5)
%%% rectangle type 3
\psline(0,-1.5)(0,0)
\psline(.866,0)(.866,-1.5)
\psline(0,-1.5)(.866,-1.5)
\psline(.866,-1.5)(0,-1)
\psline(.866,-1.5)(0,0)

\psline(0,-1)(.866,-.5)
\psline(0,0)(.866,-.5)

\psline(0,-1.5)(0,-3)
\psline(.866,-1.5)(.866,-3)
\psline(0,-3)(.866,-3)
\psline(.866,-1.5)(0,-3)
\psline(.866,-1.5)(0,-2)
\psline(0,-2)(.866,-2.5)
\psline(0,-3)(.866,-2.5)

%\pscircle[fillcolor=red,fillstyle=solid](.5,.5){.08}
%\pscircle[fillcolor=red,fillstyle=solid](.5,-2.5){.08}
%\pscircle[fillcolor=red,fillstyle=solid](.4,-1){.08}
%\pscircle[fillcolor=red,fillstyle=solid](.4,2){.08}

\end{pspicture}
\end{center}
\caption{$3a<2b$}
\end{subfigure}
\caption{Configuration for Lemma~\ref{lem:inGammai}}
\label{fig:finiteregions}
\end{figure}

It is clear that if $\sbb_uw$ lies in the red region then $\deg(\cQ_i(p))=0$ since there are no negative crossings in the straight path from $\sbb_u$ to $\sbb_{u}w$. Thus $\sbb_uw$ lies in either the green region (that is, $\sB_i'$) or in the blue region. Hence there are finitely many possibilities for $w$, and quick check shows that for these $w$ there is no path attaining the degree bound $\ba_{\pi_i}$.

\medskip

Thus $w$ admits a reduced expression with $\vec{u}^{-1}\cdot\vec{\sw}_1$ as a prefix. Since $\sbb_u w$ lies in $\cU_i$ it follows that $w$ admits a reduced expression of the form $\vec{w}=\vec{u}^{-1}\cdot \vec{\sw}_i\cdot\vec{\st}_i^N\cdot\vec{v}$ for some $u,v\in\sB_i$ and $N\in\mathbb{N}$, and thus $w\in\Gamma_i$ by the cell factorisation of Section~\ref{sec:factor1}.
\end{proof}

The following Theorem, along with Theorem~\ref{thm:mainbounds} and Lemma~\ref{lem:inGammai}, verifies that $\pi_i$ satisfies $\B{3}$ for generic parameters. Recall that if $w\in \Ga_i$ with generic parameters then $w=\su_w^{-1}\sw_i\st_i^{\tau_w}\sv_w$ with $\su_w,\sv_w\in \sB_i$ and $\tau_w\in\mathbb{N}$ (we sometimes write $\tau_w$ in place of $\st_i^{\tau_w}$ by identifying $\mathbb{N}$ with $\{\st_i^k\mid k\in\mathbb{N}\}$).

\begin{Th}\label{thm:maxpath1}
Let $w\in\Gamma_i$ with reduced expression $\vec{w}=\vec{\su}_w^{-1}\cdot \vec{\sw}_i\cdot\vec{\st}_i^{\,\tau_w}\cdot\vec{\sv}_w$. For generic parameters we have:
\begin{enumerate}
\item There exist precisely $\tau_w+1$ maximal paths in $\cP_i(\vec{w},\sbb_{\su_w})$. 
\item For each $0\leq n\leq \tau_w$ there is a unique maximal path $p\in\mathcal{P}_i(\vec{w},\sbb_{\su_w})$ such that $\mathrm{wt}_{\sB_i'}^i(p)=\tau_w-2n.$
\item For all maximal paths we have $\theta_{\sB_i'}^i(p)=\sv_w$.
\end{enumerate}
\end{Th}

\begin{proof}
Write $u=\su_w$, $v=\sv_w$, and $N=\tau_w$. Let $p$ be maximal. We claim that there are no folds in the initial $\vec{u}^{-1}$ segment. This is easily checked directly in each case. For example, consider $i=1$ and $a>2b$, and suppose that $u=s_2s_1s_2s_0$. Then $\sbb_u=s_0$ is the ``top right'' element of $\sB_1'$. Suppose that the path $p$ of type $\vec{u}^{-1}\cdot \vec{\sw}_1\cdot \vec{\st}_1^N\cdot\vec{v}$ folds in the initial $\vec{u}^{-1}$ part. If this fold occurs on the $4^{th}$ step, then the remainder of the path consists of positive crossings only, and hence has degree $b$. If the fold occurs on the $3^{rd}$ step, then the $4^{th}$, $5^{th}$, and $6^{th}$ steps (the last two coming from $\vec{\sw}_1=s_0s_1$) are forced to be, respectively, a positive crossing, a positive crossing, and a bounce. After this the path consists of positive crossings (and perhaps bounces) and so the degree is bounded by $a-b$. The remaining cases are similar. 
\medskip

Writing $p=p_0\cdot p^0$, with $p_0$ corresponding to the initial $\vec{u}^{-1}$ segment, the previous paragraph shows that $\deg(\cQ_i(p_0))=0$, and that $p^0$ starts at $\mathrm{end}(p_0)=\sbb_uu^{-1}=g_i$ (the ``base'' alcove of $\sB_i'$). Note that $p^0$ has type $\vec{w}_1=\vec{\sw}_i\cdot \vec{\st}_i^N\cdot\vec{v}$. 

\medskip

Consider the case $i=1$ and $a>2b$. We construct the $\alpha_1$-folding tables of the elements $\sw_1$, $\st_1$, and $v\in\sB_1'$ in Table~\ref{folding-table2} below. We construct these tables with respect to the fundamental domain $\sB_i'$ (rather than $W_0^i$), and thus we modify the definition of $\varphi_{\vec{x}}^v(k)$ given in the previous section (and in the notation of that section) to be
$$
\varphi_{\vec{x}}^v(k)=\text{the unique $(u,n)\in \sB_i'\times\mathbb{Z}$ such that $p(\vec{x},\sigma_i^nu)$ and $p_k$ agree after the $k$th step}.
$$
Note that the elements of $\sB_1'$ are the prefixes of $s_2s_1s_2s_0$, along with the element $v'=s_2s_0$, and so it suffices to provide the tables for these two elements of $\sB_1'$. These are given in Table~\ref{folding-table2} below. See Remark~\ref{rem:endtable} for the meaning of the final ``exit columns'' in these tables, and note that we use the indexing of $\sB_1'$ as shown in Figure~\ref{fundamental-domain}. 
\medskip

\begin{table}[H]
\renewcommand{\arraystretch}{1.2}  
\begin{subfigure}{.2\linewidth}
\centering
$\begin{array}{|c||c|c||c|}\hline
&0&1& \\\hline\hline
1&3&6&2 \\\hline
2&-&-&1 \\\hline
3&-&2&6 \\\hline
4&-&*&5 \\\hline
5&4&*&4 \\\hline
6&2&-&3 \\\hline
\end{array}$
\caption{$\sw_1=s_0s_1$}
\end{subfigure}
\begin{subfigure}{.25\linewidth}
   \centering
   $\begin{array}{|c||c|c|c||c|}
   \hline
   & 2 &1 &0 &\\\hline
   \hline
1&-&-&-&1 \\\hline
2&1&5&4&2 \\\hline
3& -&*&-&5\\\hline
4& 3&1&-&6 \\\hline
5& -&-&1&3 \\\hline
6& 5&*&3&4 \\\hline
\end{array}$
\caption{$\st_1=s_2s_1s_0$}
\end{subfigure}
\begin{subfigure}{.26\linewidth}
   \centering
   $\begin{array}{|c||c|c|c|c||c|}
   \hline
   & 2 &1 &2&0 &\\\hline
   \hline
1&-&-&-&-&5 \\\hline
2&1&5&6&1&4 \\\hline
3& -&1&*&-&1\\\hline
4& 3&1&5&3&3 \\\hline
5& -&-&-&-&6\\\hline
6& 5&*&-&5&2 \\\hline
\end{array}$
\caption{$v=s_2s_1s_2s_0$}
\end{subfigure}
\begin{subfigure}{.22\linewidth}
   \centering
   $\begin{array}{|c||c|c||c|}
   \hline
   & 2 &0  &\\\hline
   \hline
1&-&-&6 \\\hline
2&1&4&3 \\\hline
3&-&-&5\\\hline
4&3&-&1 \\\hline
5&-&1&2 \\\hline
6& 5&3&4 \\\hline
\end{array}$
\caption{$v'=s_2s_0$}
\end{subfigure}
\caption{$\al_1$-folding tables for $\vec{w}_1=\vec{\sw}_1\cdot\vec{\st}_1^N\cdot\vec{v}$ with respect to $\sB_1'$, in regime $a>2b$.}
\label{folding-table2}
\end{table}

Note that a path $p^0$ of type $\vec{\sw}_1\cdot\vec{\st}_1^N\cdot\vec{v}$ starting at $g_1$ enters the $\sw_1$ table on row~$1$, and it is then elementary to check that such a path is maximal if and only if it either folds at both places of the $\sw_1$-part, or at both places of the $s_1s_0$ part of $\st_1$ in on of the passes of $\st_1$. That is, 
$$
\mathrm{end}(p^0)=\begin{cases}
\hat{s}_0\hat{s}_1\st_1^{N}v&\text{if both folds in $\vec{\sw}_1$ occur}\\
\sw_1\st_1^{n-1}s_2\hat{s}_1\hat{s}_0\st_1^{N-n}v&\text{if the two folds occur in the $n^{th}$ pass of $\vec{\st}_1$},
\end{cases}
$$
where as usual $\hat{s}_j$ indicates that the term is omitted. In the first case we have $\mathrm{wt}_{\sB_1'}^1(p)=N$ and $\theta_{\sB_1'}^1(p)=v$. In the second case the equality $\sw_1\st_1^{n-1}s_2\st_1^n=e$ for all $n$ shows that $\mathrm{wt}_{\sB_1'}^1(p)=N-2n$ and again $\theta_{\sB_1'}^1(p)=v$. This establishes the theorem in this case. 
\medskip

Now consider the case $i=2$ with $3a>2b$. The $\alpha_2$-folding tables of $\sw_2=s_1s_2s_1s_2s_1$ and $\st_2=s_0s_2s_1s_2s_1$ and $v'\in\sB_2'$ with respect to $\sB_2'$ are given in Table~\ref{folding-table3} (note that every element of $\sB_2'$ is a prefix of $v'=s_0s_2s_1s_2s_0$). 

\begin{table}[H]
\renewcommand{\arraystretch}{1.2}  
\begin{subfigure}{.33\linewidth}
\centering
$\begin{array}{|c||c|c|c|c|c||c|}\hline
&1&2&1&2&1& \\\hline\hline
1&5&3&6&2&4&6 \\\hline
2&6&*&-&-&-&4 \\\hline
3&-&-&-&*&6&5 \\\hline
4&3&6&2&*&-&2 \\\hline
5&-&*&3&6&2&3 \\\hline
6&-&-&-&-&-&1 \\\hline
\end{array}$
\caption{$\sw_2=s_1s_2s_1s_2s_1$}
\end{subfigure}
\begin{subfigure}{.33\linewidth}
   \centering
   $\begin{array}{|c||c|c|c|c|c||c|}
   \hline
   & 0 &2 &1 &2&1&\\\hline
   \hline
1&-&-&-&-&-&1 \\\hline
2&1&*&4&1&5&3 \\\hline
3&*&1&5&*&-&2\\\hline
4& *&-&-&*&1&5 \\\hline
5&-&*&-&-&-&4 \\\hline
6& 5&4&1&5&3&6 \\\hline
\end{array}$
\caption{$\st_2=s_0s_2s_1s_2s_1$}
\end{subfigure}
\begin{subfigure}{.33\linewidth}
   \centering
   $\begin{array}{|c||c|c|c|c|c||c|}
   \hline
   & 0 &2 &1 &2&0&\\\hline
   \hline
1&-&-&-&-&-&6 \\\hline
2&1&*&4&1&*&4 \\\hline
3&*&1&5&*&1&5\\\hline
4& *&-&-&*&-&2 \\\hline
5&-&*&-&-&*&3 \\\hline
6& 5&4&1&5&4&1 \\\hline
\end{array}$
\caption{$v'=s_0s_2s_1s_2s_0$}
\end{subfigure}
\caption{$\al_2$-folding tables for $\vec{\sw_2}$, $\vec{\st}_2$, and $v'\in \sB_2'$ with respect to $\sB_2'$, in regime $3a>2b$.}
\label{folding-table3}
\end{table}

Using these tables it is easy to check that a path $p^0$ of type $\vec{\sw}_2\cdot\vec{\st}_2^N\cdot\vec{v}$ starting at $g_2$ is maximal if and only if one of the following occur (recall we enter the $\vec{\sw}_2$ table on row $1$): 
\bem
\item There are three folds in the $\vec{\sw}_2$ part, at positions $1,3,5$ (and hence no further folds).
\item There is one fold in the $\vec{\sw}_2$ part at position $5$, followed by $2$ folds in the subsequent $\vec{\st}_2$ at positions $3$ and~$5$. 
\item There are three folds distributed over two consecutive $\vec{\st}_2$ cycles, at position $5$ in the pass cycle, and then positions $3$ and $5$ in the next pass. 
\eem
The theorem follows in this case in a similar way to the previous example. The two remaining cases are similar. 
\end{proof}

\begin{Cor}\label{cor:B3B4generic}
Let $w\in \Gamma_i$ with generic parameters. Then
$$
\fc_{\pi_i,w}=\ss_{\tau_w}(\zeta)E_{\su_w,\sv_w},
$$
where $\ss_k(\zeta)$ is the Schur function of type $A_1$. Thus $\pi_i$ satisfies $\B{4}$ and $\B{4}'$. 
\end{Cor}

\begin{proof}
Let $\mathbb{P}_i(\vec{w},uu)=\{p\in\mathcal{P}_i(\vec{w},uu)\mid \deg(\cQ_i(p))=\ba_{\pi_i}\}$ be the set of maximal paths. By Corollary~\ref{cor:funddom} and the definition of $\fc_{\pi_i,w}$ we have
\begin{align*}
[\fc_{\pi_i,w}]_{u,v}&=\mathrm{sp}_{|_{\sq^{-1}=0}}\left(\sq^{-\ba_{\pi_i}}[\pi_i(T_w)]_{u,v}\right)=\sum_{\{p\in\mathbb{P}_i(\vec{w},u)\mid \theta_{\sB_i'}^i(p)=v\}}\zeta^{\mathrm{wt}_{\sB_i'}^i(p)}
\end{align*}
(note that there are either no bounces, or precisely two bounces in maximal paths $p$, and thus $\cQ_i(p)$ is positive, and so $\sq^{-\ba_{\pi_i}}\cQ_i(p)$ specialises to $+1$).  Theorem~\ref{thm:maxpath1} gives $\{\mathbb{P}_i(\vec{w},u)\mid \theta_{\sB_i'}^i(p)=v\}=\emptyset$ unless $u=\su_w$ and $v=\sv_w$, and thus $[\fc_{\pi_i,w}]_{u,v}=0$ unless $u=\su_w$ and $v=\sv_w$. Moreover Theorem~\ref{thm:maxpath1} gives
\begin{align*}
[\fc_{\pi_i,w}]_{\su_w,\sv_w}&=\sum_{n=0}^{\tau_w}\zeta^{2\tau_w-n}=\ss_{\tau_w}(\zeta).
\end{align*}
The verification of $\B{4}$ and $\B{4}'$ follows from $\fc_{\pi_i,w}=\ss_{\tau_w}(\zeta)E_{\su_w,\sv_w}$ in an analogous way to Theorem~\ref{thm:mainlowest}. 
\end{proof}

\subsection{Leading matrices for non-generic parameters}\label{sec:5.4}

In this final subsection we compute the leading matrix coefficients for $\pi_i$ with non-generic parameters $r=r_i$. This if $i=1$ and $r=2$, and if $i=2$ and $r=3/2$. In fact most of the work has been done in the previous sections, and all that remains is to piece together the paths from the generic regimes on either side of the generic parameter. 

\medskip

Recall the notation of Section~\ref{sec:equal}.  For $\eps\in\{\pm\}$ we define $g_i^{\eps}$ as in~(\ref{eq:g}) with $g_i^-$ corresponding to $r<r_i$ and $g_i^+$ corresponding to $r>r_i$. We set $\sB_i'^{\eps}=g_i^{\eps}\sB_i^{\eps}$ and write $\sB_i'^\eps=\{\sbb^\eps_u\mid u\in W_0^i\}$. When working in the case where $i=1$ all the matrices will be written in the basis $\{\xi_1\otimes X_{\sbb^+_u}\mid u\in W_0^1\}$, and for $i=2$ we use the $\eps=-$ basis.

\begin{Th}
\label{equal}
Let $w\in \Ga_1$ and $r=2$ and let $u^{-1}\sw_1^+\st_{1,+}^N v$ be the positive cell factorisation of $w$. We have 
$$\fc_{\pi_1,w}=\begin{cases}
\left(\ss_{N}(\zeta)+\ss_{N-1}(\zeta) \right)E_{u,v}& \mbox{ if $w$ is of type $(+,+)$;}\\
\left(\ss_{N}(\zeta)+\ss_{N+1}(\zeta)\right)E_{u,v}& \mbox{ if $w$ is of type $(-,-)$;}\\
(1+\zeta^{-1})\ss_N(\zeta)E_{u,v} & \mbox{ if $w$ is of type $(+,-)$;}\\
(1+\zeta)\ss_N(\zeta) E_{u,v}& \mbox{ if $w$ is of type $(-,+)$}\\
\end{cases}$$
where by definition we set  $\ss_{-1}(\zeta):=0$. If $i=2$ the corresponding result applies with all signs reversed. Hence in the case $r=r_i$ the representation $\pi_i$ satisfies $\B{1}$--$\B{4}$ and $\B{4}'$ for the cell $\Gamma_i$. 
 \end{Th}

\begin{proof}
The case $i=2$ is completely analogous to the case $i=1$, and so we only present the $i=1$ case. First assume that $w$ is of type $(+,+)$. Then there exists $(u,v)\in \sB^+_1\cap s_1\sB_1^-$ such that 
$$w=u^{-1}\sw_1^+\st_{1,+}^Nv=u^{-1}s_1\sw_1^-\st_{1,-}^{N-1}s_1v.$$
According to Theorem \ref{thm:maxpath1}, we see that there will be two families of maximal paths  starting at $\sbb^+_{u}$, one with endpoints of the form  $\sbb_e^{\pn}\st_{1,\pn}^{N-2r} v$ for all $0\leq r\leq N$ and one with endpoints of the form $\sbb_e^{\mn} \st_{1,\mn}^{N-1-2r} s_1v$ for all $0\leq r\leq N-1$. We have 
$$\begin{array}{lllllll}
\wt_{\sB'^{+}_1}^1(\sbb_e^+\st_{1,-}^{N-2r} v)=N-2r,& & \theta^1_{\sB'^{+}_1}(\sbb_e^+\st_{1,+}^{N-2r} v)=v,\\[.2cm]
\wt_{\sB'^+_1}^1 (\sbb_e^{-} \st_{1,-}^{N-1-2r} s_1v)=N-2r-1&\text{ and }& \theta^1_{\sB'^{+}_1}(\sbb_e^{-} \st_{1,-}^{N-1-2r} s_1v)=v.\\
\end{array}$$
It follows that $\fc_{\pi_1,w}=\left(\ss_{N}(\zeta)+\ss_{N-1}(\zeta) \right)E_{u,v}$ in this case.  

\medskip

Assume that $w$ of type $(-,-)$. Then we have
$$w=s_0s_2\sw_1^+\st_{1,+}^Ns_2s_0=\sw_1^-\st_{1,-}^{N+1}.$$
We see that there will be two families of maximal paths  starting at $\sbb^+_{s_2s_0}=\sbb_e^-\st_{1,-}$, one with endpoints of the form  $\sbb_e^{+}\st_{1,+}^{N-2r} s_2s_0$ for all $0\leq r\leq N$ and one with endpoints of the form $\sbb_e^{-} \st_{1,\mn}^{N+2-2r}$ for all $0\leq r\leq N+1$. We have 
$$\begin{array}{lllllll}
\wt_{\sB'^{+}_1}^1(\sbb_e^+\st_{1,+}^{N-2r} s_2s_0)=N-2r,& & \theta^1_{\sB'^{+}_1}(\sbb_e^+\st_{1,+}^{N-2r} s_2s_0)=s_2s_0,\\[.2cm]
\wt_{\sB'^{+}_1}^1(\sbb_e^-\st_{1,-}^{N+2-2r})=N+1-2r&\text{ and }& \theta^1_{\sB'^{+}_1}(\sbb_e^-  \st_{1,-}^{N+2-2r})=s_2s_0\\
\end{array}$$
since $\sbb_e^-\st_{1,-}^{N+2-2r}=\sbb_{s_2s_0}^+\st_{1,-}^{N+1-2r}$. It follows that $\fc_{\pi_1,w}=\left(\ss_{N}(\zeta)+\ss_{N+1}(\zeta)\right)E_{u,v}$ in this case.

\medskip

Assume that $w$ is of type $(+,-)$.  Then there exists $u\in \sB^+_1\cap s_1\sB_1^-$ such that 
$$w=u^{-1}\sw_1^+\st_{1,+}^Ns_2s_0=u^{-1}s_1\sw_1^-\st_{1,-}^{N}.$$
We see that there will be two families of maximal paths  starting at $\sbb^+_{u}$, one with endpoints of the form  $\sbb_e^{+}\st_{1,+}^{N-2r} s_2s_0$ and one with endpoints of the form $\sbb_e^{-} \st_{1,-}^{N-2r}$ for all $0\leq r\leq N$. We have 
$$\begin{array}{lllllll}
\wt_{\sB'^{+}_1}^1(\sbb_e^+\st_{1,+}^{N-2r} s_2s_0)=N-2r,& & \theta^1_{\sB'^{+}_1}(\sbb_e^+\st_{1,+}^{N-2r} s_2s_0)=s_2s_0,\\[.2cm]
\wt_{\sB'^{+}_1}^1(\sbb_e^-  \st_{1,-}^{N-2r})=N-2r-1&\text{ and }&\theta^1_{\sB'^{+}_1}(\sbb_e^- \st_{1,-}^{N-2r})=s_2s_0.\\
\end{array}$$
It follows that $\fc_{\pi_1,w}=(1+\zeta^{-1})\ss_N(\zeta)$ in this case.  

\medskip

Assume that $w$ is of type $(-,+)$.   Then there exists $v\in \sB^+_1\cap s_1\sB_1^-$ such that 
$$w=s_0s_2\sw_1^+\st_{1,+}^Nv=\sw_1^-\st_{1,-}^{N}s_1v.$$
We see that there will be two families of maximal paths  starting at $\sbb^+_{s_2s_0}=\sbb^{-}_{e}\st_{1,-}$, one with endpoints of the form  $\sbb_e^+\st_{1,\pn}^{N-2r} v$ for all $0\leq r\leq N$ and one with endpoints of the form $\sbb^-_{e}\st_{1,-}^{N+1-2r}s_1v$ for all $0\leq r\leq N$. But we have 
$$\begin{array}{lllllll}
\wt_{\sB'^{+}_1}^1(\sbb_e^+\st_{1,+}^{N-2r}u)=N-2r,& & \theta^1_{\sB'^{+}_1}(\sbb_e^+\st_{1,+}^{N-2r}u)=u,\\[.2cm]
\wt_{\sB'^{+}_1}^1(\sbb^-_{e}\st_{1,-}^{N+1-2r}s_1v)=N+1-2r&\text{ and }& \theta^1_{\sB'^{+}_1}(\sbb^-_{e}\st_{1,-}^{N+1-2r}s_1v)=v.\\
\end{array}$$
It follows that $\fc_{\pi_1,w}=(1+\zeta)\ss_N(\zeta)$ in this case.  
\medskip

We have already seen that $\B{1}$ and $\B{2}$ hold, and $\B{3}$ follows from the above. Axioms $\B{4}$ and $\B{4}'$ also follow easily. For example, consider the case $i=1$. Let $\sB=\sB_{\Ga_1^+}=\{e,2,20,21,212,2120\}$ and $u_0=20$. If
$
\sum_{w\in \Gamma_1}a_w\fc_{\pi_1,w}=0
$
then considering the matrix entries (and writing $\sw=\sw_1^+$ and $\st=\st_{1,+}$ to ease notation) gives the equations
\begin{align*}
\sum_{k\geq 0}\sum_{u,v\in\sB\backslash\{u_0\}}a_{u^{-1}\sw\st^k v}(\ss_k(\zeta)+\ss_{k-1}(\zeta))&=0&
\sum_{k\geq -1}a_{u_0^{-1}\sw\st^{k}u_0}(\ss_{k}(\zeta)+\ss_{k+1}(\zeta))&=0\\
\sum_{k\geq 0}\sum_{v\in \sB\backslash\{u_0\}}a_{u_0^{-1}\sw\st^k v}(1+\zeta)\ss_k(\zeta)&=0&
\sum_{k\geq 0}\sum_{u\in \sB\backslash\{u_0\}}a_{u^{-1}\sw\st^k u_0}(1+\zeta^{-1})\ss_k(\zeta)&=0
\end{align*}
and $\B{4}$ follows. To verify $\B{4}'$ we define elements $d_u\in\Gamma_1$ by $d_{u_0}=\sw_1^-$ and $d_{u}=u^{-1}\sw u$ for $u\in\sB\backslash\{u_0\}$ (these elements turn out to be the Duflo infvolutions; see Theorem~\ref{th:set-D}). Note that $\fc_{\pi_i,d_u}=E_{u,u}$ (if $u\neq u_0$ then $d_u$ is of type $(+,+)$, and $d_{u_0}$ is of type $(-,-)$). Then for $w\in\Gamma_1$ we have
\begin{align*}
\fc_{\pi_1,d_{\su_w}}\fc_{\pi_1,w}=E_{\su_w,\su_w}\fc_{\pi_1,w}=\fc_{\pi_1,w},
\end{align*}
and hence $\B{4}'$. 
\end{proof}

\begin{Rem}
Note that the formulae in Theorem~\ref{equal} show how the two leading matrices from the generic regimes on either side of the parameter $r=r_1$ combine to give the leading matrix at $r=r_1$. This suggests an approach to understanding the semicontinuity conjecture of Bonnaf\'e~\cite{Bon:09}. 
\end{Rem}

We define the following sets which are the sets of non-zero leading matrix coefficients of the elements $w$ of $(\eps,\eps')$-type
\begin{align*}
\cB_{\eps,\eps}=\{\ss_N(\zeta)+\ss_{N-1}(\zeta)\mid N\geq 0\},\quad\text{and}\quad
\cB_{\eps,-\eps}=\{(1+\zeta^\eps)\ss_N(\zeta)\mid N\geq 0\}.
\end{align*} 
We will write $\ss_N^{\eps,\eps'}\in \cB_{\eps,\eps'}$ to denote the element corresponding to $N$ in $\cB_{\eps,\eps'}$.

\medskip

The following proposition will be useful at a later stage. 

\begin{Prop}
\label{prop:Weyl-like}
Let $\eps_1,\eps_2,\eps_3\in\{-,+\}$ and $k,\ell\in \nN$. 
\begin{enumerate}
\item We have $\ss_k^{(\eps_1,\eps_2)}\cdot \ss_\ell^{(\eps_2,\eps_3)}=\sum_{m\in\mathbb{N}}\mu_{k,\ell}^m(\eps_1,\eps_2,\eps_3)\ss_m^{(\eps_1,\eps_3)}$ for some integers $\mu_{k,\ell}^m(\eps_1,\eps_2,\eps_3)$. 
\item We have $\mu_{k,\ell}^m(\eps_1,\eps_2,\eps_3)=\mu_{m,k}^{\ell}(\eps_3,\eps_1,\eps_2)$ for all $k,\ell,m\in\mathbb{N}$.
\item We have $\mu_{k,\ell}^0(\eps_1,\eps_2,\eps_1)\neq 0$ if and only if $\eps_1=\eps_2$ and $k=\ell$.
\end{enumerate}
\end{Prop}

\begin{proof}
By obvious symmetry and commutativity it is sufficient to check the cases $(\eps_1,\eps_2,\eps_3)=(+,+,+)$, $(+,+,-)$, and $(+,-,+)$. We first recall that the Schur functions $\ss_{\lambda}=\ss_{\lambda}(\zeta)$ form an orthonormal basis with respect to the Hall inner product $\langle\cdot,\cdot\rangle$, and in type $A_1$ they are self adjoint with respect to this inner product. Therefore if $\ss_k\ss_{\ell}=\sum_mc_{k,\ell}^m\ss_m$ we have
$
c_{k,\ell}^m=\langle \ss_k\ss_{\ell},\ss_m\rangle=\langle \ss_m\ss_k,\ss_{\ell}\rangle=\langle \ss_{\ell}\ss_m,\ss_k\rangle,
$
and thus $c_{k,\ell}^m=c_{m,k}^{\ell}=c_{\ell,m}^k$. Furthermore, if $\ell\leq k$ we have $\ss_k\ss_{\ell}=\sum_{j=0}^{\ell}\ss_{k-\ell+2j}$, and thus $c_{k,\ell}^0=\delta_{k,\ell}$ (if $k<\ell$ then interchange the roles of $k$ and $\ell$).

\medskip

Consider the case $(\eps_1,\eps_2,\eps_3)=(+,+,+)$. Using the formula for Schur functions of type $A_1$ we compute $\ss_k^{(+,+)}=\ss_k(\zeta)+\ss_{k-1}(\zeta)=\ss_{2k}(\zeta^{1/2})$, where we introduce a new formal indeterminant $\zeta^{1/2}$ with $(\zeta^{1/2})^2=\zeta$. It follows that $\ss_k^{(+,+)}\ss_{\ell}^{(+,+)}=\ss_k(\zeta^{1/2})\ss_{\ell}(\zeta^{1/2})$ can be expressed as a linear combination of $\ss_m(\zeta^{1/2})$, and that the coefficients in this expansion are $\mu_{k,\ell}^m(+,+,+)=c_{2k,2\ell}^{2m}$. Thus (1) and (2) hold, and the ``if'' part of (3).
\medskip

Consider the case $(\eps_1,\eps_2,\eps_3)=(+,+,-)$. Then
\begin{align*}
\ss_k^{(+,+)}\ss_{\ell}^{(+,-)}&=(1+\zeta)(\ss_k+\ss_{k-1})\ss_{\ell}=\sum_m(c_{k,\ell}^m+c_{k-1,\ell}^m)(1+\zeta)\ss_m,
\end{align*}
and so $\mu_{k,\ell}^m(+,+,-)=c_{k,\ell}^m+c_{k-1,\ell}^m$. Similarly,
$
\mu_{m,k}^{\ell}(-,+,+)=c_{m,k}^{\ell}+c_{m,k-1}^{\ell},
$ and thus $\mu_{m,k}^{\ell}(-,+,+)=\mu_{k,\ell}^m(+,+,-)$.  
\medskip

Consider the case $(\eps_1,\eps_2,\eps_3)=(+,-,+)$. Then 
$$
\ss_k^{(+,-)}\ss_{\ell}^{(-,+)}=(2\ss_0+\ss_1)\ss_k\ss_{\ell}=\sum_{m}c_{k,\ell}^m(2\ss_0+\ss_1)\ss_m=\sum_m(c_{k,\ell}^m+c_{k,\ell}^{m-1})(\ss_m+\ss_{m-1})=\sum_m(c_{k,\ell}^m+c_{k,\ell}^{m-1})\ss_m^{(+,+)},
$$
and so $\mu_{k,\ell}^m(+,-,+)=c_{k,\ell}^m+c_{k,\ell}^{m-1}$ (here $c_{k,\ell}^{-1}=0$ by definition). An easy calculation gives $\mu_{m,k}^{\ell}(+,+,-)=c_{m,k}^{\ell}+c_{m-1,k}^{\ell}$, and hence (2). The `only if' part of (3) also follows.
\end{proof}

Thus, finally we have:

\begin{Th}\label{thm:balanced}
For each choice of parameters there exists a balanced system of cell representations.
% $(\pi_{\Gamma})_{\Gamma\in\tsc}$. 
\end{Th}

\begin{proof}
This follows from Theorem~\ref{thm:finitecellbalanced}, Theorem~\ref{thm:mainlowest}, Corollary~\ref{B2generic}, Corollary~\ref{cor:B3B4generic} and Theorem~\ref{equal}. Property $\B{5}$ is checked directly from our formulae for $\ba_{\pi_{\Ga}}$. 
\end{proof}

We can now explicitly compute Lusztig's $\ba$-function for $\tilde{G}_2$. 

\begin{Cor}\label{cor:computea}
In type $\tilde{G}_2$ we have $\ba(w)=\ba_{\pi_\Ga}$ if $w\in\Gamma$. 
\end{Cor}

\begin{proof}
This follows from Theorem~\ref{thm:balanced} and Theorem~\ref{thm:afn}.
\end{proof}

%%%%%%%%%%%%%%%%%%%%%%%%%%%%%%%%%%%
%%%%%%%%%%%%%%%%%%%%%%%%%%%%%%%%%%%
%%%%%%%%%%%%%%%%%%%%%%%%%%%%%%%%%%%
%%%%%%%%%%%%%%%%%%%%%%%%%%%%%%%%%%%
%%%%%%%%%%%%%%%%%%%%%%%%%%%%%%%%%%%
%%%%%%%%%%%%%%%%%%%%%%%%%%%%%%%%%%%
%%%%%%%%%%%%%%%%%%%%%%%%%%%%%%%%%%%
%%%%%%%%%%%%%%%%%%%%%%%%%%%%%%%%%%%
%%%%%%%%%%%%%%%%%%%%%%%%%%%%%%%%%%%
%%%%%%%%%%%%%%%%%%%%%%%%%%%%%%%%%%%

\section{The Plancherel Theorem, conjecture \conj{1}, and the Duflo involutions}\label{sec:plancherel}

In this section we prove \conj{1} and compute the set $\cD$ of Duflo ``involutions'' for all choices of parameters (and hence see that the elements of $\cD$ are indeed involutions). The main piece of machinery is the Plancherel Theorem of Opdam~\cite{Op:04} and the explicit $\tilde{G}_2$ formulation of this theorem computed by the second author in~\cite{Par:14}. 

\medskip

Let us first briefly recall the situation for finite dimensional Hecke algebras. In this case the \textit{canonical trace} $\mathrm{Tr}:\cH\to\sR$ with $\mathrm{Tr}(\sum a_wT_w)=a_e$ decomposes as
\begin{align}\label{eq:gendeg}
\mathrm{Tr}(h)=\sum_{\pi\in\mathrm{Irrep}(\cH)}m_{\pi}\chi_{\pi}(h)\quad\text{for all $h\in\cH$},
\end{align}
where the elements $m_{\pi}$ are the \textit{generic degrees} of $\cH$ (see~\cite[Chapter~11]{GP:00}). This formula is a crucial ingredient in Geck's proof \cite{Geck:11} of Lusztig's conjectures for spherical type $F_4$. In particular, the observation that the ``$\sq$-valuation'' $\nu_{\sq}(m_{\pi})$ (see below) of $m_{\pi}$ is equal to $2\ba_{\pi}$ played a central role in Geck's proof.

\medskip

There is an analogue of~(\ref{eq:gendeg}) for affine Hecke algebras in the form of the remarkable \textit{Plancherel formula} of Opdam~\cite{Op:04} (see also Opdam and Solleveld~\cite{OpSol:10}). The summation in~(\ref{eq:gendeg}) becomes an integral over irreducible representations of a $C^*$-algebra completion of $\cH$, and the generic degrees become the \textit{Plancherel measure}~$d\mu$. 

\medskip

In this section we recall the explicit formulation of the Plancherel formula in type $\tilde{G}_2$ computed by the second author in~\cite{Par:14}, and show that in this case there is an analogue of the formula $\nu_{\sq}(m_{\pi})=2\ba_{\pi}$ in terms of the Plancherel measure. We will use this observation to prove \conj{1}, \conj{7}, and compute the set $\cD$. Along the way we will also introduce the \textit{asymptotic Plancherel measure} (which we believe is a new, although it appears to be related to recent work of Braverman and Kazhdan~\cite{BK:17}) and show that this measure induces an inner product on Lusztig's asymptotic algebra~$\mathcal{J}$ (at least in type $\tilde{G}_2$). We believe that these observations provide an intriguing connection between Kazhdan-Lusztig cells and the Plancherel formula -- see also the conjectures listed at the end of Section~\ref{sec:lusztig}.

\subsection{The Plancherel formula}\label{sec:7.1}

The main references for this section are \cite{Op:04} and \cite{Par:14}. The Plancherel Theorem is an analytic concept, and therefore we now take a slightly different view of the affine Hecke algebra. We extend the scalars to $\mathbb{C}$, and specialise $\sq$ to a real number $q>1$. Thus $\cH$ is now an algebra over $\mathbb{C}$. We write $T_w$ and $C_w$ for the images of the standard basis and Kazhdan-Lusztig basis elements in~$\cH$. Note also that the representations $\cH_{\Upsilon}$ for any right cell $\Upsilon$ can naturally be regarded as representations of the Hecke algebra $\cH$ defined over~$\mathbb{C}$ by extending scalars and specialising~$\sq$. 

\medskip

Let $(\pi,V)$ be a finite dimensional $\cH$-module (now over $\mathbb{C}$). Recall that 
$$
V=\bigoplus_{\zeta\in\mathrm{Hom}(P,\mathbb{C}^{\times})}V_{\zeta}^{\mathrm{gen}}
$$
where $
V_{\zeta}^{\mathrm{gen}}=\{v\in V\mid \text{ for each $\lambda\in P$ we have $(X^{\lambda}-\zeta^{\lambda})^kv=0$ for some $k\in\mathbb{N}$}\}
$
is the \textit{generalised $\zeta$-weight space} of $V$. Let $\mathrm{supp}(\pi)=\{\zeta\in\mathrm{Hom}(P,\mathbb{C}^{\times})\mid V_{\zeta}^{\mathrm{gen}}\neq\{0\}\}$ be the \textit{support} of $(\pi,V)$. A representation $(\pi,V)$ is \textit{tempered} if $|\zeta^{\lambda}|\leq 1$ for all $\zeta\in\mathrm{supp}(\pi)$ and all $\lambda\in P^+$, and it is \textit{square integrable} if $|\zeta^{\lambda}|<1$ for all $\zeta\in\mathrm{supp}(\pi)$ and all $\lambda\in P^+\backslash\{0\}$. 

\medskip

Define an involution $*$ on $\cH$ and the \textit{canonical trace functional} $\mathrm{Tr}:\cH\to\mathbb{C}$ by
$$
\left(\sum_{w\in W}a_wT_w\right)^*=\sum_{w\in W}\overline{a_w}\,T_{w^{-1}}\quad\text{and}\quad \mathrm{Tr}\left(\sum_{w\in W}a_wT_w\right)=a_e
$$
where now $\overline{a_w}$ denotes complex conjugation. An induction on $\ell(v)$ shows that $\mathrm{Tr}(T_uT_v^*)=\delta_{u,v}$ for all $u,v\in W$, and hence $\mathrm{Tr}(h_1h_2)=\mathrm{Tr}(h_2h_1)$ for all $h_1,h_2\in \cH$. It follows that
$
(h_1,h_2)=\mathrm{Tr}(h_1h_2^*)
$
defines a Hermitian inner product on $\cH$. Let $\|h\|_2=\sqrt{(h,h)}$ be the $\ell^2$-norm. The algebra $\cH$ acts on itself by left multiplication, and the corresponding operator norm is $\|h\|=\sup\{\|hx\|_2\colon x\in\cH,\|x\|_2\leq 1\}$. Let $\overline{\cH}$ denote the completion of $\cH$ with respect to this norm. Thus $\overline{\cH}$ is a non-commutative $C^*$-algebra. The irreducible representations of $\overline{\cH}$ are the (unique) extensions of the irreducible representations of $\cH$ that are continuous with respect to the $\ell^2$-operator norm, and it is known that these are the irreducible tempered representations of $\cH$ (see \cite[\S2.7 and Corollary~6.2]{Op:04}). In particular, every irreducible representation of $\overline{\cH}$ is finite dimensional (since every irreducible representation of $\cH$ has degree at most $|W_0|$), and it follows from the general theory of traces on ``liminal'' $C^*$-algebras that there exists a unique positive Borel measure $\mu$, called the \textit{Plancherel measure}, such that (see \cite[\S8.8]{Dix:77})
$$
\mathrm{Tr}(h)=\int_{\irreps}\chi_{\pi}(h)\,d\mu(\pi)\quad\text{for all $h\in\overline{\cH}$}.
$$
The Plancherel formula has been obtained in general by Opdam~\cite{Op:04}. We now recall the explicit formulation in type~$\tilde{G}_2$ from~\cite{Par:14}. We first describe the representations that appear in the Plancherel formula.

\medskip

We define the representations $\pi_0$, $\pi_1$, and $\pi_2$ as in Sections~\ref{sec:pi0const} and~\ref{sec:piiconst}, however now the ring of scalars is $\mathbb{C}$, and $\zeta\in\mathrm{Hom}(P,\mathbb{C}^{\times})$ in the case $\pi_0$, and $\zeta\in\mathrm{Hom}(\mathbb{Z},\mathbb{C}^{\times})$ in the cases $\pi_1$ and $\pi_2$. To emphasise the dependence on the central character $\zeta$ we write $\pi_i=\pi_i^{\zeta}$ for $i=0,1,2$, and we write $\chi_i^{\zeta}$ for the corresponding characters. These representations are tempered if and only if $|\zeta^{\lambda}|=1$ for all $\lambda\in P$ (in the case $i=0$) and $|\zeta^n|=1$ for all $n\in\mathbb{Z}$ (in the cases $i=1,2$). Therefore the tempered representations correspond to $\zeta\in\mathbb{T}^2$ (in the case $i=0$) and $\zeta\in\mathbb{T}$ (in the case $i=1,2$), where $\mathbb{T}=\{t\in\mathbb{C}\mid |t|=1\}$.\medskip

Let $\pi_3=\rho_{\emptyset}$ be the $1$-dimensional representation of $\cH$ with $\pi_3(T_j)=-q^{-L(s_j)}$ for $j=0,1,2$ (using the notation of Example~\ref{exa:balanced-one-dim}). Let $\pi_4=\rho_3^+$ and $\pi_5=\rho_3^-$ be the two three dimensional irreducible representations constructed in Proposition~\ref{prop:decompositions}, and let $\pi_6\sim\Gamma_4$ in the case $r\neq 1$, and let $\pi_6=\rho_3''$ in the case $r=1$ (recall the $\sim$ notation from Section~\ref{sec:balanced} and the definition of $\rho_3''$ from Proposition~\ref{prop:decompositions}). Finally, let $\pi_7$ be the following representation, depending on the parameter regime (if $r\in\{3/2,2\}$ then $\pi_7$ is not defined, and does not appear in the Plancherel Theorem below):
\begin{align*}
\pi_7=\begin{cases}
\text{the $1$-dimensional representation $\rho_{\{1\}}$}&\text{if $r<3/2$}\\
\text{the $5$-dimensional representation $\rho\sim\Gamma_6$}&\text{if $3/2<r<2$}\\
\text{the $1$-dimensional representation $\rho_{\{0,2\}}$}&\text{if $2<r$.}
\end{cases}
\end{align*}
 Let $\chi_3,\ldots,\chi_7$ be the characters of the above representations. 
\medskip

We now describe the Plancherel measure. Let $\omega=e^{2\pi i/3}$ and define functions $c_j(\zeta)$, $j=0,1,2$, by
\begin{align*}
c_0(\zeta)&=\frac{(1-q^{-2a}\zeta_1^{-1})(1-q^{-2a}\zeta_1^{-2}\zeta_2^{-3})(1-q^{-2a}\zeta_1^{-1}\zeta_2^{-3})(1-q^{-2b}\zeta_2^{-1})(1-q^{-2b}\zeta_1^{-1}\zeta_2^{-2})(1-q^{-2b}\zeta_1^{-1}\zeta_2^{-1})}{(1-\zeta_1^{-1})(1-\zeta_1^{-2}\zeta_2^{-3})(1-\zeta_1^{-1}\zeta_2^{-3})(1-\zeta_2^{-1})(1-\zeta_1^{-1}\zeta_2^{-2})(1-\zeta_1^{-1}\zeta_2^{-1})}\\
c_1(\zeta)&=\frac{(1+q^{-a}\omega \zeta^{-1})(1+q^{-a}\omega^{-1}\zeta^{-1})(1-q^{-2b}\zeta^{-2})(1+q^{-a-2b}\zeta^{-1})(1+q^{a-2b}\zeta^{-1})}{(1-\zeta^{-2})(1+q^{-a}\zeta^{-1})(1+q^{a}\zeta^{-3})}\\
c_2(\zeta)&=\frac{(1-q^{-2a}\zeta^{-2})(1+q^{-2a-3b}\zeta^{-1})(1+q^{-2a+3b}\zeta^{-1})}{(1-\zeta^{-2})(1+q^{3b}\zeta^{-1})(1+q^b\zeta^{-1})}.
\end{align*}
(We note that there is a change in the formulae for $c_1(\zeta)$ and $c_2(\zeta)$ from those in \cite{Par:14} to reflect the fact that our representations $\pi_1^{\zeta}$ and $\pi_2^{\zeta}$ are related to the representations in \cite{Par:14} by $\zeta\to-\zeta$). Write $F(x)=x-1$, $G(x)=x+1$, $H(x)=x^2+x+1$, and $H'(x)=x^2-x+1$ and define
\begin{align*}
C_3&=\frac{F(q^{2a+4b})F(q^{4a+6b})}{G(q^{2a})G(q^{2b})H(q^{2b})H(q^{2a+2b})}&C_4&=\frac{q^{2a}F(q^{2a})F(q^{2b})}{2G(q^{2a})G(q^{2b})H(q^{a-b})H(q^{a+b})}&C_5&=\frac{q^{2a}F(q^{2a})F(q^{2b})}{2G(q^{2a})G(q^{2b})H'(q^{a-b})H'(q^{a+b})}\\
C_6&=\frac{q^{2b}F(q^{2a})F(q^{6a})}{H(q^{2b})H(q^{2a-2b})H(q^{2a+2b})}&C_7&=\frac{q^{2a-4b}F(q^{-2a+4b})F(q^{4a-6b})}{G(q^{2a})G(q^{-2b})H(q^{-2b})H(q^{2a-2b})}.
\end{align*}

\begin{Th}[{Plancherel Theorem for $\tilde{G}_2$, \cite[Theorem~4.7]{Par:14}}]\label{thm:planch}
For each $h\in\overline{\cH}$ we have
\begin{align*}
\Tr(h)&=\frac{1}{12q^{6a+6b}}\int_{\mathbb{T}^2}\frac{\chi^{\zeta}_0(h)}{|c_0(\zeta)|^2}\,d\zeta_1d\zeta_2+\frac{F(q^{2a})^2}{2q^{2a+6b}F(q^{4a})}\int_{\mathbb{T}}\frac{\chi_1^{\zeta}(h)}{|c_1(\zeta)|^2}\,d\zeta+\frac{F(q^{2b})^2}{2q^{6a+4b}F(q^{4b})}\int_{\mathbb{T}}\frac{\chi_2^{\zeta}(h)}{|c_2(\zeta)|^2}\,d\zeta+\sum_{k=3}^7 |C_k|\chi_k(h)
\end{align*}
where $d\zeta$ denotes the normalised Haar measure on the group $\mathbb{T}$ (thus $\int_{\mathbb{T}}f(\zeta)\,d\zeta=\frac{1}{2\pi}\int_{0}^{2\pi}f(e^{i\theta})\,d\theta$). 
\end{Th}

\begin{Rem}
The representations $\pi_4,\ldots,\pi_7$ were constructed differently in \cite{Par:14}, however it is an easy exercise to verify that they are isomorphic to the representations given above.
\end{Rem}

\subsection{The Plancherel formula and cell decompositions}

It is convenient to group the representations that appear under integral signs in the Plancherel formula (Theorem~\ref{thm:planch}) into ``classes'' $\Pi_0=\{\pi_0^{\zeta}\mid \zeta\in\mathbb{T}^2\}$ and $\Pi_i=\{\pi_i^{\zeta}\mid \zeta\in\mathbb{T}\}$ for $i=1,2$. The remaining representations (the ``point masses'' in the Plancherel formula) are taken to be in their own classes: $\Pi_j=\{\pi_j\}$ for $3\leq j\leq 7$. We make the following observation comparing the cell decomposition and the Plancherel formula in type $\tilde{G}_2$.

\begin{Prop}\label{propobs:1}
For each parameter regime there is a well defined surjective map $\Omega:\{\Pi_j\mid 0\leq j\leq 7\}$ given by 
$$
\Omega(\Pi_j)=\begin{cases}
\Gamma_j&\text{if $j\in\{0,1,2\}$}\\
\Gamma&\text{if $3\leq j\leq 7$ and $\pi_j$ is a submodule of a cell module $\cH_{\Up}$ for some finite right cell $\Up\subseteq\Gamma$.}
\end{cases}
$$
\end{Prop}

\begin{proof}
This follows immediately by comparing the Plancherel formula and the cell decomposition, using Proposition~\ref{prop:decompositions}. For example, if $2>r>3/2$ we have $\Omega(\Pi_3)=\Gamma_e$, $\Omega(\Pi_4)=\Omega(\Pi_5)=\Gamma_3$, $\Omega(\Pi_6)=\Gamma_4$, and $\Omega(\Pi_7)=\Gamma_6$, and if $r=1$ we have $\Omega(\Pi_3)=\Gamma_e$, and $\Omega(\Pi_4)=\Omega(\Pi_5)=\Omega(\Pi_6)=\Omega(\Pi_7)=\Gamma_3$. 
\end{proof}

We will sometimes write $\Omega(\pi)$ in place of $\Omega(\Pi)$ if $\pi\in\Pi$. 

\begin{Cor}\label{cor:nice}
Each representation $\pi$ appearing in the Plancherel Theorem for $\tilde{G}_2$ admits a basis such that \B{1} and \B{2} hold, with bound $\ba_{\pi}=\ba(w)$ for any $w\in\Omega(\pi)$. 
\end{Cor}

\begin{proof}
This follows immediately from Proposition~\ref{propobs:1}, Remark~\ref{rem:satisfyB}, and Theorem~\ref{thm:piiB1}.
\end{proof}

Henceforth we will assume that each representation appearing in the Plancherel Theorem is equipped with such a basis.

\medskip

We note that the tempered irreducible representations of $\cH$ are precisely the representations that appear in the Plancherel Theorem for $\tilde{G}_2$. This can be seen directly by classifying, via central characters and weight spaces, all irreducible tempered representations of $\cH$ in an analogous way to~\cite{Ram:02} and comparing with the Plancherel Theorem from Theorem~\ref{thm:planch}. Thus, using Proposition~\ref{propobs:1} we note the following.

\begin{Obs}\label{obs:2}
Every tempered irreducible representation $\pi$ in type $\tilde{G}_2$ satisfies \B{1} with respect to~$\Gamma=\Omega(\pi)$.  
\end{Obs}

\begin{Rem}
We note that any finite dimensional representation $\pi$ satisfying \B{1} with respect to a finite cell $\Gamma\in\tsc$ is necessarily tempered. To see this, note that since $\pi(C_w)=0$ whenever $w\notin\Gamma_{\geq_{\cLR}}$, and thus $\pi(C_w)$ is nonzero for only finitely many $w\in W$. Hence there is a bound on the matrix coefficients of $\pi(C_w)$, and then Casselman's Criterion applies (see Opdam \cite[Lemma~2.20]{Op:04}). In fact one can check that $\pi$ is square integrable (see Lusztig~\cite[\S3]{Lus4} and \cite{Lus:83}). 
\end{Rem}

\subsection{The asymptotic Plancherel measure}

Each rational function $f(\sq)=a(\sq)/b(\sq)$ can be written as $f(\sq)=\sq^{-N}a'(\sq^{-1})/b'(\sq^{-1})$ with $N\in\mathbb{Z}$  where $a'(\sq^{-1})$ and $b'(\sq^{-1})$  polynomials in $\sq^{-1}$ nonvanishing at $\sq^{-1}=0$. The integer $N$ in this expression is uniquely determined, and is called the \textit{$\sq$-valuation} of $f$, written $\nu_{\sq}(f)=N$. For example, $\nu_{\sq}((\sq^2+1)(\sq^3+1)/(\sq^7-\sq+1))=2$.

\begin{Def}
Let $\Pi$ be a class of representations appearing in the Plancherel Theorem, and let $C$ be the `coefficient' of a generic character $\chi_{\pi}$ with $\pi\in\Pi$. Consider this coefficient as a rational function $C=C(\sq)$ in $\sq$ by setting $q=\sq$. We define the \textit{$\sq$-valuation} of $\Pi$ to be $\nu_{\sq}(\Pi)=\nu_{\sq}(C(\sq))$. We also write $\nu_{\sq}(\pi)=\nu_{\sq}(\Pi)$ for any $\pi\in\Pi$. 
\end{Def}

For example, consider the class $\Pi_2$. The associated coefficient is
$$
\frac{(\sq^{2b}-1)^2(1-\zeta^{-2})(1-\zeta^2)(1-\sq^{3b}\zeta^{-1})(1-\sq^{3b}\zeta)(1-\sq^b\zeta^{-1})(1-\sq^b\zeta)}{2\sq^{6a+4b}(\sq^{4b}-1)(1-\sq^{-2a}\zeta^{-2})(1-\sq^{-2a}\zeta^2)(1-\sq^{-2a-3b}\zeta^{-1})(1-\sq^{-2a-3b}\zeta)(1-\sq^{-2a+3b}\zeta^{-1})(1-\sq^{-2a+3b}\zeta)},
$$
and thus
$$
\nu_{\sq}(\Pi_2)=\begin{cases}
2(a+b)&\text{if $a/b\leq 3/2$}\\
2(3a-2b)&\text{if $a/b>3/2$.}
\end{cases}
$$
For another example consider the class $\Pi_7=\{\pi_7\}$. We have
$$
\nu_{\sq}(\Pi_7)=\nu_{\sq}\left(\frac{\sq^{2a-4b}(\sq^{-2a+4b}-1)(\sq^{4a-6b}-1)}{(\sq^{2a}+1)(\sq^{-2b}+1)(\sq^{-4b}+\sq^{-2b}+1)(\sq^{4a-4b}+\sq^{2a-2b}+1)}\right)
=\begin{cases}
2a&\text{if $r\leq 1$}\\
2(3a-2b)&\text{if $1<r<3/2$}\\
2(a+b)&\text{if $3/2<r<2$}\\
2(3b)&\text{if $2<r$}.
\end{cases}
$$
Note that the values of the $\ba$-function are arising in these examples. Indeed we have the following theorem, where $\ba(\Gamma)$ denotes the constant value of Lusztig's $\ba$-function on the two-sided cell~$\Gamma$, and $\Omega$ is as in Proposition~\ref{propobs:1}. Note the similarity with the finite dimensional case described at the beginning of this section.

\begin{Th}\label{thm:nu}
For each classes $\Pi$ appearing the the Plancherel formula in type $\tilde{G}_2$ we have
$
\nu_{\sq}(\Pi)=2\ba(\Omega(\Pi)).
$
\end{Th}

\begin{proof}
This is by direct inspection using the formula in Theorem~\ref{thm:planch}.
\end{proof}

\begin{Def} Using Theorem~\ref{thm:nu} we can define an \textit{asymptotic Plancherel measure} on $\irreps$ by
$$
d\mu'(\pi)=\lim_{q\to\infty} q^{2\ba_{\pi}}d\mu(\pi).
$$
\end{Def}

\begin{Prop}\label{prop:asymp}
The asymptotic Plancherel measure on the classes $\Pi_0$ and $\Pi_i$ with $i=1,2$ is as follows:
\begin{align*}
d\mu'(\pi_0^{\zeta})&=\frac{1}{12}\prod_{\alpha\in\Phi^+}|1-\zeta^{-\alpha^{\vee}}|^2\,d\zeta_1d\zeta_2&d\mu'(\pi_i^{\zeta})&=\begin{cases}
\frac{1}{2}|1-\zeta^{-2}|^2\,d\zeta&\text{if $r\neq r_i$ is generic for $\Gamma_i$}\\
\frac{1}{2}|1-\zeta^{-1}|^2\,d\zeta&\text{if $r=r_i$ is non-generic for $\Gamma_i$}
\end{cases}
\end{align*}
For the classes of finite cells we have $\mu'(\pi_3)=1$, $\mu'(\pi_5)=\frac{1}{2}$, and 
\begin{align*}
d\mu'(\pi_4)&=\begin{cases}
\frac{1}{2}&\text{if $r\neq 1$}\\
\frac{1}{6}&\text{if $r=1$}
\end{cases}&
d\mu'(\pi_6)&=\begin{cases}
1&\text{if $r\neq 1$}\\
\frac{1}{3}&\text{if $r=1$}
\end{cases}&
d\mu'(\pi_7)&=\begin{cases}
1&\text{if $r\notin\{1,\frac{3}{2},2\}$}\\
\frac{1}{3}&\text{if $r=1$}
\end{cases}
\end{align*}

\end{Prop}

\begin{proof}
This is a straightforward calculation.
\end{proof}

\begin{Rem}\label{rem:on}
Note that the measure $d\mu'$ on $\Pi_0=\Omega^{-1}(\Gamma_0)$ is the Hall measure, and thus the Schur functions of type $G_2$ are orthonormal with respect to this measure (see, for example, \cite[Proposition~3.1]{Ram:03}). Similarly, in the generic cases for $\Pi_1$ and $\Pi_2$ the measure $d\mu'$ is the Hall measure of type $A_1$. In the non-generic cases $d\mu'$ is the Hall measure for the modified Schur functions $\ss_k(\zeta^{1/2})$. 
\end{Rem}

\subsection{The conjecture \conj{1}}\label{sec:P1}

We can now prove that \conj{1} holds for $\tilde{G}_2$. 

\begin{Th}\label{thm:P1} Lusztig's conjecture \conj{1} holds for $\tilde{G}_2$.
\end{Th}

\begin{proof}
Recall that $\Delta(w)$ is defined by $P_{e,w}=n_w\sq^{-\Delta(w)}+\text{(strictly smaller powers of $\sq$)}$, where $n_w\neq 0$. We are required to prove that $\ba(w)\leq\Delta(w)$. This is equivalent to showing that
$$
\lim_{q\to\infty}q^{\ba(w)}P_{e,w}(q)<\infty,
$$
where we write $P_{e,w}(q)$ for the specialisation of $P_{e,w}$ at $\sq=q$. By the Plancherel Theorem we have
\begin{align*}
q^{\ba(w)}P_{e,w}(q)=q^{\ba(w)}\mathrm{Tr}(C_w)&=\int_{\irreps}q^{\ba(w)}\chi_{\pi}(C_w)\,d\mu(\pi).
\end{align*}
Suppose that $w$ is in the two-sided cell~$\Ga$. In type $\tilde{G}_2$ it follows from Corollary~\ref{cor:nice}  that the integral above is over only those classes of representations associated to the cells $\Gamma'$ with $\Ga\geq_{\cLR}\Gamma'$. For each such class of representations the Plancherel measure is, by Corollary~\ref{cor:computea} and Theorem~\ref{thm:nu}, of the form 
$$
d\mu(\pi)=q^{-2\ba_{\pi_{\Ga'}}}(1+\mathcal{O}(q^{-1}))d\mu'(\pi)
$$ 
where $d\mu'$ is the asymptotic Plancherel measure. Thus the integrand (with respect to the asymptotic Plancherel measure) is $q^{\ba(w)-\ba_{\pi_{\Ga'}}}\mathrm{tr}(\fc_{\pi,w})(1+\mathcal{O}(q^{-1}))$. Since $\Gamma\geq_{\cLR}\Gamma'$ we use \B{5} to give $\ba_{\pi_{\Ga'}}\geq \ba_{\pi_{\Gamma}}=\ba(w)$ and thus the power of $q$ in the integrand is at most~$0$. It is clear from the explicit $\tilde{G}_2$ Plancherel Theorem that the limit may be passed under the integral sign, and the result follows. 
\end{proof}

\subsection{The Duflo elements}

In this section we extend the idea in the proof of \conj{1} to compute the set $\cD$ of Duflo elements. This  calculation will be used in Section~\ref{sec:lusztig} when dealing with the conjectures involving~$\cD$. We note that since we have proved \conj{1} and computed Lusztig's $\ba$-function it is also possible to use a technique of Xie~\cite{Xie:15} to compute~$\cD$.

\begin{Th}
\label{th:set-D}
For $\Ga\in\tsc$ let $\cD_{\Ga}=\cD\cap \Ga$. For the infinite cells we have
\begin{align*}
\cD_{\Ga_i}&=\{u^{-1}\sw_i u\mid w\in\sB_i\}&&\text{for $i\in\{0,1,2\}$ with $r$ generic for $\Ga_i$}\\
\cD_{\Ga_1}&=\{\sw_1^-\}\cup \{u^{-1}\sw_1^+u\mid u\in\sB_1^+\cap s_1\sB_1^-\}&&\text{if $r=2$}\\
\cD_{\Ga_2}&=\{\sw_2^+\}\cup \{u^{-1}\sw_2^-u\mid u\in\sB_2^-\cap s_0\sB_2^+\}&&\text{if $r=3/2$},
\end{align*}
and for the finite cells we have
\begin{align*}
\cD_{\Ga_3}&=\begin{cases}
\{1,212,02120\}&\text{if $r>1$}\\
\{0,1,2\}&\text{if $r=1$}\\
\{0,2,121\}&\text{if $r<1$}
\end{cases}&
\cD_{\Ga_4}&=\begin{cases}
\{0,2\}&\text{if $r>1$}\\
\{21212,0212120\}&\text{if $r<1$}
\end{cases}&
\cD_{\Ga_5}&=\{212\}\\
\cD_{\Ga_6}&=\{u^{-1}s_0s_1u\mid u\in\{e,2,21,212,2120\}\}&
\cD_{\Ga_7}&=\begin{cases}
\{12121\}&\text{if $1<r<3/2$}\\
\{1\}&\text{if $r<1$}
\end{cases}
\end{align*}
\end{Th}

\begin{proof}
Let $w\in W$, and let $n_w'$ denote the coefficient of $\sq^{-\ba(w)}$ in $P_{e,w}$. Thus $w\in\cD$ if and only if $n_w'\neq 0$. The formula 
$$
n_w'=\lim_{q\to\infty}q^{\ba(w)}P_{e,w}(q)
$$ 
gives, as in the proof of Theorem~\ref{thm:P1},
\begin{align}\label{eq:Dcal}
n_w'=\int_{\Omega^{-1}(\Ga)}\mathrm{tr}(\fc_{\pi,w})\,d\mu'(\pi)\quad\text{if $w\in\Gamma$}
\end{align}
(again, we are using Corollary~\ref{cor:nice} here). Thus in the case $w\in \Gamma=\Gamma_0$ we have, by Theorem~\ref{thm:mainlowest}, 
$$
n_w'=\frac{1}{12}\int_{\mathbb{T}^2}\ss_{\tau_{w}}(\zeta)\mathrm{tr}(E_{\su_w,\sv_w})\,\prod_{\alpha\in\Phi^+}|1-\zeta^{-\alpha^{\vee}}|^2\,d\zeta_1d\zeta_2.
$$
Thus $n_w'\neq 0$ if and only if $\su_w=\sv_w$ (due to the trace) and $\tau_w=0$ (since the measure is the Hall measure, see Remark~\ref{rem:on}). Thus $n_w'\neq 0$ if and only if $w=u^{-1}\sw_0u$ for some $u\in \sB_0$ (moreover, in this case $n_w'=1$). 
\medskip

The argument for $w\in\Gamma_1$ or $w\in\Gamma_2$ with $r$ generic for the cell is similar, noting that the measure in this case is the Hall measure for Schur functions of type~$A_1$. 
\medskip

Consider the case $w\in \Gamma_1$ with $r=2$. Recall the notation from Section~\ref{sec:equal}. Theorem~\ref{equal} again forces $\su_w=\sv_w$ if $n_w'\neq 0$ (where $w$ is written in ``$+$-form''). This forces $w$ to be either of $(+,+)$-type of $(-,-)$-type. In the former case we have $\fc_{\pi_1,w}=(\ss_N(\zeta)+\ss_{N-1}(\zeta))E_{\su_w,\su_w}$ with $N=\tau_w\geq 0$, and in the latter case we necessarily have $w=\sw_1^-$ where by definition $\tau_w=-1$ and hence $\fc_{\pi_1,w}=\ss_0(\zeta)=1$. Recall from the proof of Proposition~\ref{prop:Weyl-like} that $\ss_N(\zeta)+\ss_{N-1}(\zeta)=\ss_{2N}(\zeta^{1/2})$. The measure from Proposition~\ref{prop:asymp} is in this case is the Hall measure for these Schur functions, and thus we see that $n_w'\neq 0$ if and only if either $w$ is of $(+,+)$-type with $\tau_w=0$, or $w=\sw_1^-$. In the former case we have $w=u^{-1}\sw_1^+u$ for some $u\in \sB_1^+\cap s_1\sB_1^-$, and hence the result. The case $w\in\Gamma_2$ with $r=3/2$ is analogous. 
\medskip

The claims for finite cells follow by direct calculations. For example, consider the most complicated case $\Gamma=\Gamma_3$ and $r=1$. In this case~(\ref{eq:Dcal}) gives
\begin{align}\label{eq:traces}
n_w'=\frac{1}{6}\mathrm{tr}(\fc_{\rho_3^+,w})+\frac{1}{2}\mathrm{tr}(\fc_{\rho_3^-,w})+\frac{1}{3}\mathrm{tr}(\fc_{\rho_3',w})+\frac{1}{3}\mathrm{tr}(\fc_{\rho_3'',w}).
\end{align}
The matrices for $\rho_3^+$, $\rho_3^-$, $\rho_3'$, and $\rho_3''$ are computed from the decomposition in Proposition~\ref{prop:decompositions}, and the leading matrices are computed as 
$$
\fc_{\rho,w}=\lim_{q\to\infty}q^{-1}\rho(T_w)\quad\text{for $\rho\in\{\rho_3^+,\rho_3^-,\rho_3',\rho_3''\}$ and $w\in W$}.
$$
Thus a direct calculation using~(\ref{eq:traces}) gives $n_{s_0}'=n_{s_1}'=n_{s_2}'=1$ and $n_w'=0$ for all $w\in\Gamma_3\backslash\{s_0,s_1,s_2\}$, and hence the result. 

\medskip

The remaining cases are similar (in fact, easier). We note that some of the finite cells (for example $\Gamma=\Gamma_6$) can be handled using cell factorisation, in an analogous way to the infinite cells. 
\end{proof}

\subsection{An inner product on $\mathcal{J}$ and conjecture \conj{7}}

In this section we extend the above ideas further to endow Lusztig's asymptotic algebra $\mathcal{J}_{\Gamma}$ with a natural inner product inherited from the Plancherel Theorem (a kind of \textit{asymptotic Plancherel Theorem}). As a consequence we obtain a proof of \conj{7}. Recall that we have proved in Corollary~\ref{cor:J} that for each $\Gamma\in\tsc$ we have that $\mathcal{J}_{\Gamma}$ is isomorphic to the $\mathbb{Z}$-algebra spanned by the leading matrices $\{\fc_{\pi_{\Gamma},w}\mid w\in \Gamma\}$. We thus identify $\mathcal{J}_{\Gamma}$ with this concrete algebra, with $J_w\leftrightarrow \fc_{\pi_{\Gamma},w}$. Define an involution $*$ on $\mathcal{J}_{\Gamma}$ by linearly extending $J_w^*=J_{w^{-1}}$.  

\begin{Th}\label{thm:innerp}
Let $\Gamma$ be a two sided cell of $\tilde{G}_2$. The formula
$$
\langle g_1,g_2\rangle_{\Ga}=\int_{\Omega^{-1}(\Gamma)}\mathrm{tr}(g_1g_2^*)\,d\mu'(\pi)\quad\text{for $g_1,g_2\in \mathcal{J}_{\Gamma}$}
$$
defines an inner product on $\mathcal{J}_{\Gamma}$ with $\{J_w\mid w\in \Gamma\}$ an orthonormal basis.
\end{Th}

\begin{proof}
It is clear that this formula defines a skew linear form. 
For $x,y\in \Gamma$ we have 
\begin{align*}
\delta_{x,y}=\langle T_x,T_y\rangle&=\int_{\irreps}\mathrm{tr}(\pi(T_x)\pi(T_{y^{-1}}))d\mu(\pi)=\int_{\irreps}\mathrm{tr}(q^{-\ba_{\pi}}\pi(T_x)\cdot q^{-\ba_{\pi}}\pi(T_{y^{-1}}))(1+\mathcal{O}(q^{-1}))d\mu'(\pi).
\end{align*}
Taking limits as $q\to\infty$, and using the explicit expression for the Plancherel Theorem for $\tilde{G}_2$ to see that the limit may be passed inside the integral, we see that
$$
\delta_{x,y}=\int_{\irreps}\mathrm{tr}(\fc_{\pi,x}\fc_{\pi,y^{-1}})\,d\mu'(\pi).
$$
The terms $\fc_{\pi,x}\fc_{\pi,y^{-1}}$ are zero if $\pi\notin\Omega^{-1}(\Gamma)$, and hence the integral is over $\Omega^{-1}(\Gamma)$. Thus the formula $\langle\cdot,\cdot\rangle_{\Ga}$ given in the statement of the theorem defines an inner product on $\mathcal{J}_{\Gamma}$, and $\{J_w\mid w\in \Gamma\}$ is an orthonormal basis.
\end{proof}

\begin{Cor}\label{cor:P71}
Conjecture \conj{7} holds.
\end{Cor}

\begin{proof}
If $x,y,z\in\Gamma$ then
$
\gamma_{x,y,z}=\langle J_xJ_y,J_{z^{-1}}\rangle_{\Gamma}=\langle J_y,J_{x^{-1}}J_{z^{-1}}\rangle_{\Gamma}=\langle J_yJ_z,J_{x^{-1}}\rangle_{\Gamma}=\gamma_{y,z,x}.
$
\end{proof}

We will give a more combinatorial proof of \conj{7} in Section~\ref{sec:lusztig}.

%%%%%%%%%%%%%%%%%%%%%%%%%%%%%%%%%%%
%%%%%%%%%%%%%%%%%%%%%%%%%%%%%%%%%%%
%%%%%%%%%%%%%%%%%%%%%%%%%%%%%%%%%%%
%%%%%%%%%%%%%%%%%%%%%%%%%%%%%%%%%%%
%%%%%%%%%%%%%%%%%%%%%%%%%%%%%%%%%%%
%%%%%%%%%%%%%%%%%%%%%%%%%%%%%%%%%%%
%%%%%%%%%%%%%%%%%%%%%%%%%%%%%%%%%%%
%%%%%%%%%%%%%%%%%%%%%%%%%%%%%%%%%%%

\section{Proof of Lusztig's conjectures \conj{2}--\conj{15}}\label{sec:lusztig}

In this section we prove Lusztig's conjectures \conj{2}--\conj{15} for $\tilde{G}_2$.
We will denote by $\tsc_{\infty}$ the set of infinite two-sided cells and by $\tsc_{{\mathsf{f}}}$ the set of finite two-sided cells. Let $(\pi_\Ga)_{\Ga\in \tsc}$ be the system of balanced cell representations afforded by Theorem~\ref{thm:balanced}. When $\Ga_i\in \tsc_{\infty}$ we have $\pi_{\Ga_i}=\pi_i$ and when $\Ga\in \tsc_{{\mathsf{f}}}$ the representation $\pi_\Ga$ is the Kazhdan-Lusztig representation associated to $\Ga$ with its natural basis. By Corollary~\ref{cor:computea} we have $\ba_{\pi_\Ga}=\ba(w)$ for all $w\in \Ga$, and by Corollary~\ref{cor:J} we see that the coefficients $\ga_{x,y,z^{-1}}$ are the structure constants of the ring $\mathcal{J}_{\Ga}$ generated by $\{\fc_{\pi_\Ga,w}\mid w\in \Ga\}$.

\subsection{The conjectures \conj{4}, \conj{7}--\conj{12}, and \conj{14}}

Knowing the value of Lusztig's $\ba$-function (from Corollary~\ref{cor:computea}), and the partition of $W$ into cells (from Figure~\ref{partition}), it is elementary that ${\bf P4}$, ${\bf P9}$--${\bf P12}$ and $\conj{14}$ hold. We prove \conj{7} and \conj{8} in the following theorem (note that we obtained a different proof of \conj{7} in Corollary~\ref{cor:P71}).

\begin{Th}\label{thm:prev} Let $x,y,z\in W$. 
\ben
\item If 
$\ga_{x,y,w^{-1}}\neq 0$ then 
$x\sim_\cR w$, $y\sim_\cL w$ and $x\sim_{\cL} y^{-1}$.
\item  We have $\ga_{x,y,w}=\ga_{y,w,x}=\ga_{w,x,y}$.
\een
\end{Th}
\begin{proof}
Let $w\in \Ga$ and $x,y\in W$ be such that $\ga_{x,y,w^{-1}}\neq 0$. Suppose that $\Ga\in \tsc_{\mathsf{f}}$ and let $\cB_\Ga:=\{\fc_{\pi_\Ga,w}\mid w\in \Ga\}$ (a finite set of matrices). To prove (1), we simply need to check that if $\fc_{\pi_\Ga,w}$ appears in the expansion of $\fc_{\pi_\Ga,x}\fc_{\pi_\Ga,y}$ in the basis $\cB_{\Ga}$ then we have $x\sim_\cR w$, $y\sim_\cL w$ and $x\sim_{\cL} y^{-1}$. In the case that $\Gamma$ admits a cell factorisation we have $\fc_{\pi_{\Ga},w}=\mathfrak{f}_w\,E_{\su_w,\sv_w}$ for some nonzero constant $\mathfrak{f}_w$, and hence $\fc_{\pi_{\Ga},x}\fc_{\pi_{\Ga},y}=\mathfrak{f}_x\mathfrak{f}_y\,E_{\su_x,\sv_x}E_{\su_y,\sv_y}$. Thus if $\fc_{\pi_{\Ga},w}$ appears in this expansion we have $\sv_x=\su_y$, and hence $x\sim_{\cL} y^{-1}$. Moreover, $\su_w=\su_x$ and $\sv_w=\sv_y$, giving $w\sim_{\cR} x$ and $w\sim_{\cR} y$, and hence the result. In the case that $\Ga$ does not admit a cell factorisation the result is readily checked using the explicit formulae for the leading matrices (see Theorem~\ref{thm:finitecellbalanced}). Verifying (2) is similar.

\medskip

 We now prove (1) and (2) in the case that $\Ga\in \tsc_{\infty}$ and that $r$ is generic for $\Ga$. By Theorem~\ref{thm:mainlowest} and Corollary~\ref{cor:B3B4generic}, the equality $\fc_{\pi_\Ga,x}\fc_{\pi_\Ga,y}=\sum_{z}\ga_{x,y,z^{-1}}\fc_{\pi_\Ga,z}$ becomes
$$\ss_{\tau_x}E_{\su_{x},\sv_{x}}\cdot \ss_{\tau_y}E_{\su_y,\sv_y}=\sum_{z\in \Ga} \ga_{x,y,z^{-1}}\ss_{\tau_z}E_{\su_{z},\sv_z}.$$
Since $\ga_{x,y,w^{-1}}\neq 0$, the term indexed by $w$ on the righthand side is nonzero and this implies that the whole sum is nonzero by $\B{4}$. It follows that the lefthand side is nonzero hence it is equal to $\ss_{\tau_x}\ss_{\tau_y}E_{\su_x,\sv_y}$ and we have $\sv_x=\su_y$ (or in other words $x\sim_\cL y^{-1}$). From there we see that if $\ga_{x,y,z^{-1}}\neq 0$ then we must have (a) $\su_{z}=\su_{x}$ and $\sv_{z}=\sv_{y}$ and (b)  $c^{\tau_z}_{\tau_x,\tau_y}\neq 0$ where $c^{\tau_z}_{\tau_x,\tau_y}=\langle \ss_{\tau_x}\ss_{\tau_y},\ss_{\tau_z}\rangle$. In particular, since $\ga_{x,y,w^{-1}}\neq 0$ we have 
$\su_{w}=\su_{x}$ and $\sv_{w}=\sv_{y}$ or in other words $x\sim_\cR w$ and $y\sim_\cL w$.  This completes the proof of (1).

\medskip

We now show that $\ga_{x,y,w}=\ga_{y,w,x}=\ga_{w,x,y}$. 
We may assume that $x\sim_{\cL} y^{-1}$, $x\sim_\cR w^{-1}$ and $y\sim_\cL w^{-1}$ since otherwise  $\ga_{w,x,y}=\ga_{y,w,x}=0$ by (1). We know that $\ga_{x,y,w}$ is the coefficient of $\ss_{\tau_{w^{-1}}}$ in the product $\ss_{\tau_{x}}\ss_{\tau_y}$, which is equal to the coefficient of $\ss_{\tau_w}$ since by Remark \ref{rem:useful} we have $\ss_{\tau_{w^{-1}}}=\ss_{\tau_w}$. Then using standard results on Weyl characters we get that $\ga_{x,y,w}=\ga_{w,x,y}=\ga_{y,w,x}$. 

\medskip

Consider the case where $r$ is not generic for $\Ga_i$, with $i\in\{1,2\}$. Consider the case $i=1$, and so $r=2$ (the case $i=2$ is similar). Recall the notation of Theorem~\ref{equal}. Let $x$ be of $(\eps_1,\eps_2)$ type, and let $y$ be of $(\eps_2',\eps_3)$ type. If $\eps_2\neq\eps_2'$ then $\gamma_{x,y,z}=0$ (this follows from Theorem~\ref{equal} and the cell factorisation in Section~\ref{sec:equal}). Thus suppose that $\eps_2=\eps_2'$. Moreover, if $\gamma_{x,y,w}\neq 0$ then $w^{-1}$ is of type $(\eps_1,\eps_3)$.  Statement (1) now follows as in the generic case from Theorem~\ref{equal}. Next $\gamma_{x,y,w}$ is the coefficient of $\ss_{\tau_{w^{-1}}}^{(\eps_1,\eps_3)}$ in the expansion of $\ss_{\tau_x}^{(\eps_1,\eps_2)}\ss_{\tau_y}^{(\eps_2,\eps_3)}$ in the $(\eps_1,\eps_3)$ `basis'. Similarly $\gamma_{w,x,y}$ is the coefficient of $\ss_{\tau_{y^{-1}}}^{(\eps_3,\eps_2)}$ in the expansion of $\ss_{\tau_w}^{(\eps_3,\eps_1)}\ss_{\tau_x}^{(\eps_1,\eps_2)}$. Hence by Proposition~\ref{prop:Weyl-like} we have~$\gamma_{x,y,w}=\gamma_{w,x,y}$. 
\end{proof}
Hence \conj{7} and \conj{8} are proven. %

\subsection{The conjectures \conj{2}, \conj{3}, \conj{5}, \conj{6}, and \conj{13}}

We now consider the conjectures involving the set $\cD$, computed in Theorem \ref{th:set-D}. Note that for the infinite cells~$\Gamma$, if $d\in \cD_{\Ga}$ then $\su_d=\sv_{d}$ and $\tau_d=0$. Therefore  we have $\fc_{\pi_\Ga,d}=\ss_0(\zeta)$ for all $d\in \cD\cap(\Gamma_0\cup\Ga_1\cup\Ga_2)$, this will be of crucial importance in the proof below.

\begin{Th} We have the following.
\ben
\item If $d\in\cD$ then $d^2=1$ (hence \conj{6} holds). 
\item If $d\in \cD$ and $x,y\in W$ are such that $\ga_{x,y,d}\neq 0$ then $y=x^{-1}$ (hence \conj{2} holds).
\item If $y\in W$, there exists a unique $d\in \cD$ such that $\ga_{y,y^{-1},d}\neq 0$ (hence \conj{3} holds).
\item  If $d\in \cD$, $y\in W$, $\gamma_{y,y^{-1},d}\neq 0$, then
$\gamma_{y,y^{-1},d}=n_d=\pm 1$ (hence \conj{5} holds).
\item  Any right cell $\Up$ of $W$ contains a unique element
$d\in \cD$. We have $\gamma_{x,x^{-1},d}\neq 0$ for all $x\in \Up$ (hence \conj{13} holds).
\een
\end{Th}
\begin{proof}
The first statement follows immediately from the explicit calculation of $\cD$ given in  Theorem \ref{th:set-D}. 
For the remaining statements, note that if $\Ga\in \tsc_{\mathsf{f}}$ then the results can be proved by explicit matrix calculations, and thus we will focus here on the case where $\Ga\in \tsc_{\infty}$.
Let $d\in \cD_\Ga$ and assume that $r$ is generic for $\Ga$. 
Let $x,y\in W$ be such that $\ga_{x,y,d}\neq 0$. We have the equality
$$\ss_{\tau_x}E_{\su_{x},\sv_{x}}\cdot \ss_{\tau_y}E_{\su_y,\sv_y}=\sum \ga_{x,y,z^{-1}}\ss_{\tau_z}E_{\su_{z},\sv_z}.$$
Arguing as in the proof of Theorem~\ref{thm:prev} we obtain:
\bem
\item the lefthand side is equal to $\ss_{\tau_x}\ss_{\tau_y}E_{\su_d,\su_d}$;
\item If $\ga_{x,y,z^{-1}}\neq 0$ then $\su_z=\su_d=\sv_z$ and $c^{\tau_z}_{\tau_x,\tau_y}\neq 0$.
\eem
In particular since $\tau_d=0$ we have $c^{0}_{\tau_x,\tau_y}\neq 0$ which implies that $\tau_x=\tau_y$. Finally we have 
 $$x^{-1}=(\su^{-1}_x\sw_i\tau_x\sv_x)^{-1}=\sv_x^{-1}\tau_x^{-1}\sw_i\su^{-1}_y=\su_y^{-1}\sw_i\tau_y\sv_y=y$$
 as required in (2). In the case where $r$ is not generic, we can argue in the same fashion using the result of Proposition~\ref{prop:Weyl-like} to get that $\tau_x=\tau_y$. Hence (2). 
 
 \medskip
 
Let $y\in W$ and let $\Ga_i\in \tsc_{\infty}$ be such that $y\in \Ga_i$. If $r$ is generic for $\Ga_i$ then setting $d=\sv_y^{-1}\sw_i\sv_y$  we easily see arguing as above that $\ga_{y,y^{-1},d}=1$ since $c^{0}_{\tau_y,\tau_y}\neq 1$. In the case where $r=2$ and $y\in \Ga_1$ then  we have using Proposition~\ref{prop:Weyl-like}
\bem
\item if $y$ is of type $(\eps,-)$ then  $\ga_{y,y^{-1},d}=1$  where $d=s_0s_2s_0$;
\item if $y$ is of type $(\eps,+)$ then  $\ga_{y,y^{-1},d}=1$  where $d=\sv_y^{-1}s_1s_0\sv_y$.
\eem
The case $r=3/2$ and $y\in \Ga_2$ is similar. The statements (3), (4) and (5) follow readily.
 \end{proof}

\subsection{The conjecture \conj{15}}

We now prove \conj{15}. The technique here is somewhat different to the proofs of \conj{2}--\conj{14} given above, and relies on the process of \textit{generalised induction} introduced by the first author in \cite{guilhot3}. An alternative proof of \conj{15} can also be found in \cite[Theorem 6.2]{Xie:15}.

\medskip

In order to present a uniform proof of Theorem \ref{th:P15}, we will consider the two-sided cells $\Ga_i$ which are either infinite, or for which there is a cell factorisation. Thus the proof below applies to the infinite cells $\Gamma_i$, $i=0,1,2$, and also all finite cells except for $\Gamma_3$ with $r\leq 1$ and $\Gamma_4$ with $r\geq 1$ (see Remark~\ref{rem:cell-fact-finite}). In these few remaining finite cases we have checked~\conj{15} by explicit computations, and we omit the details.

\medskip

Let $\Up$ be the right cell in $\Ga_i$ that contains $\sw_i$. In the case where $r$ is not generic, we assume that $\sw_i=\sw_i^+$ or $\sw_i^-$ and we choose the positive or negative cell factorisation in the following definitions. To lighten the notation, we will not write the superscript $\pm$ when it is clear from the context. Most of the equalities in this section will hold modulo $\cH_{>_{\cLR}\Ga_i}$ and we sometimes write simply $\equiv$ and omit mod  $\cH_{>_{\cLR}\Ga_i}$.

\medskip

We set $\sT_0=\{t_{\om_1},t_{\om_2}\}$, $\sT_i=\{\st_i\}$ for $i=1,2,3$ and write $\sP_i$ for the set of monomials in the variables $\sT_i$ (see Remark~\ref{rem:cell-fact-finite} for the case $i=3$). When $i>3$ we simply set $\sP_i=\{e\}$. One can verify that $\Up=\{\sw_i \tau v\mid \tau \in\sP_i, v\in \sB_i\}$ where the set $\sB_i$ and $\sw_i$ have been defined in Sections~\ref{sec:factor0} and~\ref{sec:factor1}. 
For all $x=\tau v$ there exists an element $\sH(x)\in\cH$ such that 
$$C_{\sw_i}\sH(x)\equiv C_{\sw_ix}\quand \sH(x)\in T_x+\sum_{y<x,y\in X_i} \sq^{-1}\nZ[\sq^{-1}]T_y.$$
These elements can be constructed using the induction process; see \cite[Proposition 4.3]{geck} and the references therein. Using the anti-involution $\flat$, we easily see that $\sH(x)^\flat C_{\sw_i}\equiv C_{x^{-1}\sw_i}$. For $\tau\in \sP_i$, $v\in \sB_i$ and $x=\tau v$ we define
$$\sh_{\tau}=\begin{cases}
\sH(t_{\om_1})^m\sH(t_{\om_2})^n&\mbox{ if $\tau=m\om_1+n\om_2\in \sP_0$}\\
\sH(\st_i)^n&\mbox{ if $\tau=\st_i^n\in\sP_i$}
\end{cases}\quand \sh_{x}=\sh_{\tau}\sH(v).$$
It is important to notice here that we do not have $\sh_{\tau}=\sH(\tau)$.  
Some basic properties of these elements are presented in Section \cite[\S 4]{GM:12} where $\sh_{x}$ is denoted ${\bf P}(x)$. In particular, it is shown  that the $\sR$-module  of residues modulo $\cH_{\Ga_i}$ generated by $\{ C_{\sw_i}\sh_{\tau}\sh_{\sv}\mid \tau\in \sP_i,\sv\in \sB_i\}$  is a right $\cH$-module. We set 
$$
C_{\sw_i}\sh_{\sv}\cdot T_w\equiv \sum_{\tau\in \sP_i, \sv'\in\sB_i} \nu_{\sv,w}^{\tau,\sv'} C_{\sw_i}\sh_{\tau}\sh_{\sv'}
\quand 
T_wC_{\su^{-1}\sw_i}\equiv \sum_{\tau\in \sP_i, \sv'\in\sB_i} \la_{\sv,w}^{\tau,\su'} \sh^{\flat}_{\su'}\sh_\tau^{\flat}C_{\sw_i}.
$$
Since  $C_{\sw_i}\sh_{\tau}=\sh^\flat_{\tau}C_{\sw_i}$, we get that the $\sR$-module  of residues modulo $\cH_{\Ga_i}$ generated by $\{\sh^\flat_{\su}C_{\sw_i}\sh_{\tau}\sh_{\sv}\mid \tau\in \sP_i, \su,\sv\in \sB_i\}$ is a two-sided $\cH$-module. Further the coefficient $\la$ and $\nu$ completely determined the structure of this module. Indeed we have
\begin{align*}
\sh^\flat_{\su}C_{\sw_i}\sh_{\tau}\sh_{\sv} T_w&=\sh^\flat_{\su}\sh_{\tau}^\flat C_{\sw_i\sv} T_w=\sum_{\tau'\in \sP_i, \sv'\in\sB_i} \nu_{\sv,w}^{\tau',\sv'} \sh^\flat_{\su}\sh_{\tau'}^\flat C_{\sw_i}\sh_{\tau}\sh_{\sv'}=\sum_{\tau\in \sP_i, \sv'\in\sB_i} \nu_{\sv,w}^{\tau,\sv'} C_{\sw_i}\sh_{\tau+\tau'}\sh_{\sv'}.
\end{align*}
A similar formula holds for left multiplication. 
\begin{Rem}
In the case where $r$ is generic, it is shown in \cite[Proposition 4.6]{GM:12} that the two-sided $\cH$-module defined above is in fact equal to the two-sided cell module $\cH_{\Ga_i}$. 
\end{Rem}

\begin{Prop}
\label{prop:Zcoeff}
Let  $u,v\in \sB_i$ and $\tau\in \sP_i$. We have  $\sh^\flat_{\su}C_{\sw_i}\sh_{\tau}\sh_{\sv}\equiv \sum_{z\in \Ga_i} a_z C_z$ where $a_z\in \nZ$. 
\end{Prop}
\begin{proof}
We start by proving the result when $u$ and $v$ are equal to the identity.  
By a straightforward induction, it is enough to show that $C_{\sw_i\tau}\sh_{\st}$ is $\nZ$-linear combination of Kazhdan-Lusztig elements. Here $\st=t_{\om_1}$ or $t_{\om_2}$ if $i=0$ and $\st=\st_i$ if $i=1,2,3$. Since 
$$C_{\sw_i\tau}\sh_{\st}\equiv C_{\sw_i}\sH(\tau)\sh_{\st}\equiv \sH(\tau)^\flat \sh^\flat_\st C_{\sw_i}\mod \cH_{\Ga_i}$$
we obtain that $C_{\sw_i\tau}\sh_{\st}\equiv \sum_{\tau'\in \sP_i} b_{\tau'}  C_{\tau'\sw_i}$ where $b_{\tau'}\in \sR$; see \cite[Lemma 4.3]{GM:12}. We now show that the coefficients $b_{\tau'}$ lie in $\nZ$. The following argument is inspired by \cite[Proof of Theorem 6.2]{Xie:15}.  We have
$$h_{\sw_i,\sw_i,\sw_i}C_{\sw_i\tau}\sh_{\st}
\equiv \sH(\tau)^{\flat}C_{\sw_i}C_{\sw_i}\sh_{\st}
\equiv C_{\tau^{-1}\sw_i}C_{\sw_i\st}
\equiv \sum_{z\in \Ga_i} h_{\tau^{-1}\sw_i,\sw_i\st,z}C_z
\equiv \sum_{\tau'\in \sP_i} b_{\tau'}  h_{\sw_i,\sw_i,\sw_i}C_{\tau'\sw_i} \mod \cH_{\Ga_i}$$
which implies that $b_z\in \nZ$ since $\deg(h_{\sw_i,\sw_i,\sw_i})=\ba(\sw_i)$ and $\deg(h_{\tau^{-1}\sw_i,\sw_i\st,z})\leq \ba(z)=\ba(\sw_i)$ by \conj{11}.

\medskip

In order to prove the proposition, it is now enough to show that $\sh_u^{\flat}C_{\sw_i\tau}\sh_v$ is a $\nZ$-linear combination of Kazhdan-Lusztig element. We start by proving the result in the generic case. By  the generalised induction process \cite{guilhot3} and explicit computations in $\tG_2$ \cite{guilhot4} we have 
\bem
\item[(1)] $\sh_u^{\flat}C_{\sw_i\tau}\sh_v\equiv C_{u^{-1}\sw_i\tau}\sh_v$;
\item[(2)] $\sh_u^{\flat}C_{\sw_i\tau}\sh_v\equiv \sum_{\tau'\in \sP_i} b_{\tau'}C_{u^{-1}\sw_i\tau'v}$ where $b_{\tau'}\in \sR$.
\eem
The first statement was the key fact in \cite{guilhot4} to determine the partition of $\tG_2$ into cells. Then multiplying $C_{u^{-1}\sw_i\tau}\sh_v$ by $h_{\sw_i,\sw_i,\sw_i}$ we can conclude as above that $b_{\tau'}\in \nZ$. 

\medskip

The case where $r$ is not generic for $\Ga_i$ is more delicate. We will only treat the case $i=1$ and write $\st_-$ for the translation $\st_{1,-}$ and $\st_+$ for $\st_{1,+}$.  In the case where $u,v\in \sB_1^+\cap s_1\sB_1^-$ we can proceed exactly as above since~(1) and (2) still hold.  Next  we can show by explicit computations that 
$$\sh^\flat_uC_{\sw_i^+}\sh_v\equiv 
\begin{cases}
C_{u^{-1}\sw_i^+v}&\mbox{ if $u\in \sB_1^{+}\cap s_1\sB_1^-$,}\\
C_{\sw_i^-}+C_{\sw_i^-\st_{-}}&\mbox{ if $u=v=s_2s_0$}
\end{cases}
$$
so that the result holds in this case. Assume that $\tau=\st^n_{+}$ with $n\geq 1$ and $u=s_2s_0$. We have
\begin{align*}
\sh^{\flat}_{s_2s_0}C_{\sw_i^+\st^n_{+}}\sh_v&\equiv \sh^\flat_{s_2s_0}C_{s_1\sw_1^-\st^{n-1}_{-}s_1}\sh_v\\
&\equiv \sh^\flat_{s_2s_0}C_{s_1\sw_1^-}\sH(\st^{n-1}_{-}s_1)\sh_v\\
&\equiv \left(C_{\sw_i^-}+C_{\sw_i^-t_{-}}\right)\sH(\st^{n-1}_{-}s_1)\sh_v\quad\quad\mbox{(by explicit computations)}\\
&\equiv \left(C_{\sw_i^-\st^{n-1}_{-}s_1}+C_{\sw_i^-\st^{n}_{-}s_1}\right)\sh_v\\
&\equiv \begin{cases}
C_{s_0s_2\sw_i^+\st_+^{n-1}v}+C_{s_0s_2\sw_i^+\st_+^{n}v}&\mbox{if $v\in  \sB_1^{+}\cap s_1\sB_1^-$,}\\
C_{\sw_1^-\st_-^{n-1}}+2C_{\sw_1^-\st_-^{n}}+C_{\sw_1^-\st_-^{n+1}}&\mbox{if $u=v=s_2s_0$.}\\
\end{cases}
\end{align*}
The case where $\tau=\st^n_{+}$ with $n\geq 1$, $v=s_2s_0$ and $u\in \sB_1^+\cap s_1\sB_1^-$ can be dealt with using the $\flat$ anti-involution. 

\medskip

The case where $i=2$ and $r=3/2$ is completely analogous, based on the equalities
\begin{align*}
\sh^\flat_{s_1s_2s_1s_2}C_{\sw_2^-}\sh_{s_2s_1s_2s_1}&\equiv C_{\sw_2^+}+C_{\sw_2^+t_{2,+}}\quand \sh^\flat_{s_1s_2s_1s_2}C_{s_0\sw^-_2}\equiv C_{\sw_2^-}+C_{\sw_2^-t_{2,-}}.\qedhere
\end{align*}
\end{proof}

\begin{Th}
\label{th:P15}
Let $\Ga$ be a two-sided cell of $\tilde{G}_2$ and let $x,y,w,w'\in W$ be such that $x,y\in \Ga$. Then 
$$\sum_{z\in W} h_{x,w',z}\otimes h_{w,x,y}=\sum_{z\in W}  h_{x,w',y}\otimes h_{w,x,z}.$$
In other words, Conjecture ${\bf P15}$ holds. 
\end{Th}
\begin{proof}
First, we remark that the sum is in fact over $z\in \Ga$. Indeed, if there is a non-zero term in the left sum, that is $h_{x,w',z}\otimes h_{w,z,y}\neq 0$ then $h_{x,w',z}\neq 0$ which implies that $z\leq_{\cR} x$ and $h_{w,z,y}\neq 0$ which implies that $y\leq_{\cL} z$. Then we have $\ba(z)\geq \ba(x)$ and $\ba(y)\geq \ba(z)$ so that $\ba(y)=\ba(x)=\ba(y)$ since $x,y$ lie in the same cell. In turn, using \conj{9} and \conj{10} we get that $z\sim_{\cR} x$ and $y\sim_{\cL} z$ and therefore $z\in \Ga$. 

\medskip

Next, following \cite[Remark 2.3.7]{Geck:11book}, we note that Conjecture ${\bf P15}$ is in fact a statement of a certain bimodule structure. To see this, consider the ring $A:=\sR\otimes_{\nZ}\sR$ and let $\cE$ be a free $A$-module with basis $\{e_{w}\mid w\in \Ga\}$ where $\Ga$ is a two-sided cell of $W$. Let $\cH_1:=A\otimes_{\sR} \cH$ where $\sR$ is embedded into $A$ via $a\lmt 1\otimes a$
and $\cH_2:=\cH\otimes_{\sR} A$ where $\sR$ is embedded into $A$ via $a\lmt a\otimes1$. We can define a left $\cH_1$-action and a right $\cH_2$-action by 
$$C_w\cdot e_x=\sum_{z\in \Ga} (1\otimes h_{w,x,z})e_z\quand e_{x}\cdot C_w=\sum_{z\in \Ga} (h_{x,w,z}\otimes 1)e_z.$$
Then, \conj{15} states that $\cE$ is a two-sided $(\cH_1,\cH_2)$-bimodule.

\medskip

We have seen the set of residues modulo $\cH_{<_\cLR \Ga}$ of the form  $\sh_{\su_x}^\flat C_{\sw_i}\sh_{\tau_x}\sh_{\sv_x}$  is a two-sided submodule  of the cell module~$\cH_{\Ga_i}$. The right action (respectively the left action) of $\cH$ on this basis only depends on $\sv_x$ (respectively on $\su_x$) and is determined by the coefficients $\la$ and $\nu$. By Proposition \ref{prop:Zcoeff}, we can define a submodule $\cE'$ of $\cE$ with basis 
$\{e'_{\su_x\sw_i\tau_x\sv_x}\mid x\in \Ga_i\}$  and with action of $\cH_1$ and $\cH_2$ defined by 
$$
T_w\cdot e'_x=\sum_{(\sbb,\tau) \in \sB_i\times \sP_i}(1\otimes \la^{\sbb,\tau}_{\su_x,w})e'_{\sbb^{-1}\sw_i\tau_x\tau \sv_x}
\quand 
e'_{x}\cdot T_w=\sum_{(\sbb,\tau) \in \sB_i\times \sP_i}(\nu^{\sbb,\tau}_{\sv_x,w}\otimes 1)e'_{\su_x^{-1}\sw_i\tau_x\tau\sbb}.
$$
In the non-generic case, we defined two-submodules $\cE^\eps=\{e'_{\su_\eps(x)\sw^\eps_i\tau_\eps(x)\sv_\eps(x)}\mid x\in \Ga^\eps_i\}$ where $\eps=\pm$.
From there it is easy to see that the submodule of $\cE'$ is a two-sided $(\cH_1,\cH_2)$-module  since the coefficent of $e'_z$ in $T_{w}\cdot (e'_x\cdot T_{w'})$ and in $(T_{w}\cdot e'_x)\cdot  T_{w'}$  are equal to 
$$\underset{\tau+\tau'+\tau_x=\tau_z}{\sum_{\tau,\tau'\in \sP_i}} \nu^{\sv_z,\tau'}_{\sv_x,w}\otimes \la^{\su_z,\tau}_{\su_x,w}.$$
When the parameter $r$ is generic for $\Ga_i$, this concludes the proof since the submodule $\cE'$ is equal to $\cE$. When the parameter $r$ is not generic, we also get the result since $\cE^++\cE^-=\cE$. The inclusion $\cE\subset \cE^++\cE^-$ can be obtained using the fact that in Proposition \ref{prop:Zcoeff}, the decomposition of $\sh^\flat_{\su}C_{\sw_i}\sh_{\tau}\sh_{\sv}$ has to be of the form $\displaystyle C_{\su\sw_i\tau \sv}+\sum_{z\in \Ga_i,z<\su\sw_i\tau \sv} a_z C_z$.
\end{proof}

\subsection{Conjectures}\label{sec:7.3}

We conclude this paper with some conjectures.

\begin{conjecture}\label{conj:0}
For each affine Hecke algebra there exists a balanced system of cell representations for each choice of parameters. 
\end{conjecture}

As seen in Section~\ref{sec:def-balanced}, assuming the truth of this conjecture one can show that Lusztig's $\ba$-function satisfies $\ba(w)\leq \ba_{\Gamma}$ whenever $w\in \Gamma$. Further, we have equality if the system of balanced cell representations satisfies the extra axiom $\B{4}'$.

\begin{conjecture}\label{conj:1} 
There exists a well defined surjective map $\Omega$ from the set of classes of representations appearing in the Plancherel formula to the set of two-sided cells generalising the map from Proposition~\ref{propobs:1}. 
\end{conjecture}

\begin{conjecture}\label{conj:2}
Let $\Pi$ be class of representations appearing in the Plancherel Theorem, and let $\Gamma=\Omega(\Pi)$ where $\Omega$ is as in Conjecture~\ref{conj:1}. Then $\ba(w)=\nu_{\sq}(\Pi)/2$ for all $w\in \Gamma$.
\end{conjecture}

We note that \conj{1} can be deduced from Conjecture~\ref{conj:2} in a similar was as in Theorem~\ref{thm:P1}. Continuing in this direction, we make the following final conjecture.

\begin{conjecture}\label{conj:3}
The construction of the inner product in Theorem~\ref{thm:innerp} generalises to arbitrary affine type.
\end{conjecture}

The analysis of this paper proves all four conjectures in type $\tilde{G}_2$. 
%%%%%%%%%%%%%
%%%%%%%%%%%%%%%%%%%%%%%%%%%%%%%%%%%

\bibliographystyle{plain}
%\bibliography{biblioPhD}

\begin{thebibliography}{10}

\bibitem{Blas}
J.~Blasiak.
\newblock A factorization theorem for affine {K}azhdan-{L}usztig basis
  elements.
\newblock {\em Preprint available at http://arxiv.org/abs/0908.0340}, 2009.

\bibitem{Bon:09}
C.~{B}onnaf\'e.
\newblock {S}emicontinuity properties of {K}azhdan-{L}usztig cells.
\newblock {\em New Zealand J. Math.}, 39:171--192, 2009.

\bibitem{B-I}
C.~Bonnaf\'e and L.~Iancu.
\newblock Left cells in type ${B}_n$ with unequal parameters.
\newblock {\em Represent. Theory}, 7:587--609, 2003.

\bibitem{MAGMA}
W.~{B}osma, J.~{C}annon, and C.~{P}layoust.
\newblock The magma algebra system {I}: The user language.
\newblock {\em J. Symbolic Comput.}, 24:235--265, 1997.

\bibitem{BK:17}
A.~{B}raverman and D.~{K}azhdan.
\newblock Remarks on the asymptotic {H}ecke algebra.
\newblock {\em Preprint available at arXiv:1704.03019v3}, 2017.

\bibitem{Dix:77}
J.~{D}ixmier.
\newblock {\em $C^*$-algebras}, volume North-Holland Mathematical Library.
\newblock North-Holland Publishing Co., Amsterdam-New York-Oxford, 1977.

\bibitem{EW:14}
B.~{E}lias and G.~{W}illiamson.
\newblock The {H}odge theory of {S}oergel bimodules.
\newblock {\em Ann. of Math.}, 180(2):1089--1136, 2014.

\bibitem{Geck:02}
M.~{G}eck.
\newblock Constructible characters, leading coefficients and left cells for
  finite coxeter groups with unequal parameters.
\newblock {\em Represent. Theory}, 6:1--30, 2002.

\bibitem{geck}
M.~Geck.
\newblock On the induction of {K}azhdan-{L}usztig cells.
\newblock {\em Bull. London Math. Soc. {\bf 35}, 608--614}, 2003.

\bibitem{Geck:11}
M.~Geck.
\newblock On {I}wahori-{H}ecke algebras with unequal parameters and {L}usztig's
  isomorphism theorem.
\newblock {\em Pure Appl. Math. Q.}, 7:587--620, 2011.

\bibitem{chevie2}
M.~Geck, G.~Hiss, F.~L{\"u}beck, G.~Malle, and G.~Pfeiffer.
\newblock {CHEVIE} -- {A} system for computing and processing generic character
  tables for finite groups of {L}ie type, {W}eyl groups and {H}ecke algebras.
\newblock {\em Appl. Algebra Engrg. Comm. Comput.}, 7:175--210, 1996.

\bibitem{Geck:11book}
M.~{G}eck and N.~{J}acon.
\newblock {\em Representations of Hecke algebras at roots of unity}, volume~15
  of {\em Algebra and Applications}.
\newblock Springer-Verlas London, 2011.

\bibitem{Goe:07}
U.~{G}\"ortz.
\newblock Alcove walks and nearby cycles on affine flag manifolds.
\newblock {\em J. Algebraic Combin.}, 26:415--430, 2007.

\bibitem{guilhot3}
J.~Guilhot.
\newblock Generalized induction of {K}azhdan-{L}usztig cells.
\newblock {\em Ann. Inst. Fourier}, 59:1385--1412, 2009.

\bibitem{guilhot4}
J.~Guilhot.
\newblock Kazhdan-{L}usztig cells in affine {W}eyl groups of rank 2.
\newblock {\em Int Math Res Notices}, 2010:3422--3462, 2010.

\bibitem{GM:12}
J.~Guilhot and V.~Miemietz.
\newblock Affine cellularity of affine hecke algebras of rank two.
\newblock {\em Math. Z.}, 271:373--397, 2012.

\bibitem{GP:18}
J.~Guilhot and J.~Parkinson.
\newblock Balanced representations, the asymptotic Plancherel formula, and Lusztig's conjectures for $\tilde{C}_2$.
\newblock {\em Preprint}, 2018.

\bibitem{KL1}
D.~A. Kazhdan and G.~Lusztig.
\newblock Representations of {C}oxeter groups and {H}ecke algebras.
\newblock {\em Invent. Math}, 53:165--184, 1979.

\bibitem{KL2}
D.~A. Kazhdan and G.~Lusztig.
\newblock {S}chubert varieties and {P}oincar\'e duality.
\newblock {\em Proc. Sympos. Pure Math.}, 34:185--203, 1980.

\bibitem{KX}
S.~Koenig and C.~Xi.
\newblock Affine cellular algebras.
\newblock {\em Advances in Mathematics}, 229:139--182, 2011.

\bibitem{LP}
C.~Lenart and A.~Postnikov.
\newblock Affine weyl groups in k-theory and representation theory.
\newblock {\em Int. Math. Res. Not. IMRN}, 12, 2007.

\bibitem{LP2}
C.~Lenart and A.~Postnikov.
\newblock A combinatorial model for crystals of {K}ac-{M}oody algebras.
\newblock {\em Trans. Amer. Math. Soc.}, 360(8), 2008.

\bibitem{Lus1p}
G.~Lusztig.
\newblock Left cells in {W}eyl groups.
\newblock {\em Lecture Notes in Math.}, 1024:99--111, 1983.

\bibitem{Lus:83}
G.~Lusztig.
\newblock {S}ome examples of square integrable representations of semisimple
  $p$-adic groups.
\newblock {\em Trans. Amer. Math. Soc.}, 277(2):623--653, 1983.

\bibitem{Lus4}
G.~Lusztig.
\newblock Cells in affine {W}eyl groups {I}{I}.
\newblock {\em J. of Algebra}, 109:536--548, 1987.

\bibitem{bible}
G.~Lusztig.
\newblock {\em {H}ecke algebras with unequal parameters}.
\newblock CRM Monograph Series. Amer. Math. Soc., Providence, RI, 2003.

\bibitem{GP:00}
G.~{P}feiffer M.~{G}eck.
\newblock {\em Characters of finite {C}oxeter groups and {I}wahori-{H}ecke
  algebras}, volume~21 of {\em London Mathematical Society Monographs}.
\newblock Clarendon Press, Oxford, 2000.

\bibitem{Mac:71}
I.~G. {M}acdonald.
\newblock {\em Spherical functions on a group of $p$-adic type}, volume~2 of
  {\em Publications of the Ramanujan Institute}.
\newblock Ramanujan Institute, Centre for Advanced Studies in Mathematics,
  University of Madras, 1971.

\bibitem{chevie}
J.~Michel.
\newblock The development version of the {CHEVIE} package of {GAP}3.
\newblock {\em J. Algebra}, 435:308--336, 2015.

\bibitem{Op:04}
E.~{O}pdam.
\newblock On the spectral decomposition of affine {H}ecke algebras.
\newblock {\em J. Inst. Math. Jussieu}, 3:531--648, 2004.

\bibitem{OpSol:10}
E.~{O}pdam and M.~{S}olleveld.
\newblock Discrete series characters for affine {H}ecke algebras and their
  formal degrees.
\newblock {\em Acta Math.}, 205(1):105--187, 2010.

\bibitem{Par:14}
J.~{P}arkinson.
\newblock On calibrated representations and the {P}lancherel {T}heorem for
  affine {H}ecke algebras.
\newblock {\em J. Algebraic Combin.}, 40:331--371, 2014.

\bibitem{Par:11}
J.~{P}arkinson and B.~{S}chapira.
\newblock A local limit theorem for random walks on the chambers of
  $\tilde{A}_2$ buildings.
\newblock {\em Progr. Probab.}, 64:15--53, 2011.

\bibitem{Ram:02}
A.~{R}am.
\newblock Representations of rank two affine {H}ecke algebras.
\newblock {\em Advances in Algebra and Geometry}, pages 57--91, 2002.

\bibitem{Ram:06}
A.~{R}am.
\newblock Alcove walks, {H}ecke algebras, spherical functions, crystals and
  column strict tableaux.
\newblock {\em Pure and Applied Mathematics Quarterly (special issue in honor
  of Robert MacPherson)}, 2(4):963--1013, 2006.

\bibitem{Ram:03}
A.~{R}am and K.~{N}elsen.
\newblock {K}ostka-{F}oulkes polynomials and {M}acdonald spherical functions.
\newblock {\em Surveys in Combinatorics, London Math. Soc. Lecture Notes},
  307:325--370, 2003.

\bibitem{GAP}
Martin Sch{\"o}nert et~al.
\newblock {\em {GAP} -- {Groups}, {Algorithms}, and {Programming} -- version 3
  release 4 patchlevel 4}.
\newblock Lehrstuhl D f{\"u}r Mathematik, Rheinisch Westf{\"a}lische Technische
  Hoch\-schule, Aachen, Germany, 1997.

\bibitem{Shil2}
J.-Y. Shi.
\newblock A two-sided cell in an affine {Weyl} group {I}{I}.
\newblock {\em J. London Math. Soc.}, 37:253--264, 1988.


\bibitem{SY:15}
J.-Y. Shi and G.~{Y}ang.
\newblock The weighted universal {C}oxeter group and some related conjectures of {L}usztig.
\newblock {\em J. Algebra}, 441:678--694, 2015.


\bibitem{Xie:15}
X.~Xie.
\newblock A decomposition formula for the {K}azhdan-{L}usztig basis of affine
  {H}ecke algebras of rank 2.
\newblock {\em Preprint available at arXiv:1509.05991}, 2015.

\bibitem{Xie:17}
X.~Xie.
\newblock The based ring of the lowest generalized two-sided cell of an
  extended affine {W}eyl group.
\newblock {\em J. Algebra}, 477:1--28, 2017.

\end{thebibliography}

   \end{document}